\documentclass[preprint,pdftex,english]{imsart}
\usepackage{amsthm,amsmath,amssymb,amsfonts,epic,multirow,ulem}

\usepackage[numbers]{natbib}
\RequirePackage{natbib}
\RequirePackage[colorlinks,citecolor=blue,urlcolor=blue]{hyperref}

\usepackage{comment}
\usepackage{wrapfig}
\usepackage[english]{babel}
 \usepackage[T1]{fontenc} 

 \usepackage{makeidx}
 \usepackage{fancyhdr}
\usepackage{epsf}
\usepackage[dvips]{epsfig}  
\usepackage{subfig}
 \usepackage{graphicx}
\usepackage[nice]{nicefrac}
\usepackage{dsfont}
\usepackage{calrsfs} %Lettres plus jolies

\usepackage{pifont}
\usepackage{bm}

\usepackage{wasysym}
\usepackage{stmaryrd}
\usepackage{aeguill}
\usepackage{mathrsfs}
\usepackage{enumerate}
\usepackage{enumitem}
\usepackage[dvipsnames,table,xcdraw]{xcolor}
\usepackage{cancel}
\usepackage{ulem}

\startlocaldefs

\newtheorem{theo}{Theorem}[section]

\newtheorem{lemm}[theo]{Lemma}

\def\tt{\bm{\tau}}
\def\R{\mathbb{R}}
\def \N{\mathbb{N}}
\def \Z{\mathbb{Z}}

\def \Pe{\mathbb{P}^{\mathcal{E}}}
\def \P{\mathbb{P}} % proba 
\def \Ee{\mathbb{E}^{\mathcal{E}}}
\def \E{\mathbb{E}} % esperance 
\newcommand{ \un }{\mathds{1}}
\def \Ran{\mathcal{R}}
\def \bfa{\mathbf{f}}
\def \tS{\bm{S}}

\newcommand{\cro}[1]{\left[ \left. #1 \right. \right]}

\newcommand{\Var}{\mathbb{V}\mbox{ar}}

\def\ud{\text{d}}

\def\T{\mathbb{T}}

\newcommand{\red}[1]{{\color{red}{#1}}}

\newcommand{\pe}[1]{\p[\mathcal{E}]{#1}}

\endlocaldefs

\newcommand{\mto}{many-to-one Lemma}

\newtheorem{theorem}{Theorem}

\newtheorem{proposition}[theorem]{Proposition}

\newtheorem{remark}{Remark}
\renewcommand{\T}{\mathbb{T}}
\newcommand{\X}{\mathbb{X}}

\def \Eb{\mathbf{E}}
\def \Pb{\mathbf{P}}

\newcommand{\paren}[1]{\left( \left. #1 \right. \right)} 
\newcommand{\set}[1]{\left\{ \left. #1 \right. \right\}}
\newcommand{\tlp}{\ell^d}

\renewcommand{\P}{\mathbb{P}}

\newcommand{\seq}[2][n]{\left(#2_{#1}\right)_{#1\in\N}}

\newtheorem{postita}{Post-it}

\renewcommand{\P}{\mathbb{P}}
\newcommand{\p}[2][]{\P^{#1}\!\paren{#2}}
\newcommand{\bP}{\mathbf{P}}
\newcommand{\bp}[2][]{\bP^{#1}\!\paren{#2}}

\def\T{\mathbb{T}}

\newcommand{\Hline}[2]{\mathcal{O}_{{#1},{#2}}}

\newcommand{\bmlambda}{\bm{\lambda}_n}

\makeatletter
\newcommand*\bigcdot{\mathpalette\bigcdot@{.5}}
\newcommand*\bigcdot@[2]{\mathbin{\vcenter{\hbox{\scalebox{#2}{$\m@th#1\bullet$}}}}}
\makeatother

\newtheorem{innerCustomtheo}{Theorem}
\newenvironment{Customtheo}[1]
  {\renewcommand\theinnerCustomtheo{#1}\innerCustomtheo}
  {\endinnerCustomtheo}%Personnaliser l'intitulé d'un théorème

\begin{document}

\begin{frontmatter}

%%%%%%%%%
%%%%%%%%%%

\title{Generalized range of slow random walks on trees}

\author{\fnms{Pierre} \snm{Andreoletti}\ead[label=e1]{Pierre.Andreoletti@univ-orleans.fr}}
\address{
Institut Denis Poisson,   UMR CNRS 7013,
Universit\'e d'Orl\'eans, Orl\'eans, France. \printead{e1}}

\and 
\author{\fnms{Alexis} \snm{Kagan}\ead[label=e2]{Alexis.Kagan@univ-orleans.fr}}
\address{Institut Denis Poisson,   UMR CNRS 7013,
Universit\'e d'Orl\'eans, Orl\'eans, France. \printead{e2}} \vspace{0.5cm}

\runauthor{Andreoletti, Kagan} 

%Version : \today{}

\runtitle{Generalized range for slow random walks on trees}

\begin{abstract}
In this work, we are interested in the set of visited vertices of a tree $\mathbb{T}$ by a randomly biased random walk $\mathbb{X}:=(X_n,n \in \N)$. The aim is to study a generalized range, that is to say the volume of the trace of $\mathbb{X}$ with both constraints on the trajectories of $\mathbb{X}$ and on the trajectories of the underlying branching random potential $\mathbb{V}:=(V(x), x \in \mathbb{T})$.  Focusing on slow regime's random walks (see \cite{HuShi15}, \cite{AndChen}), we prove a general result and detail examples.  These examples exhibit many different behaviors for a wide variety of ranges, showing the interactions between the trajectories of $\mathbb{X}$ and the ones of $\mathbb{V}$.
%Among these examples one of them focus on simply or highly visited vertices with hight potential $V$, another one on ranges with a renormalization depending on the descent of $V$ along visited trajectory of $\mathbb{X}$.
%(Hu and Shi \cite{HuShi15b}, who proves that for slow regime's random walks , vertices with high potential can be visited
%. This leads naturally to the study of a range $R_n$ of this random walks on high potential. We exhibit a very interesting regime depending on the localisation of  the maximum of the visited potential as well as a general method to treat various ranges with quite general constraints.

\end{abstract}

% \begin{keyword}[class=AenMS]
 % \kwd[MSC : Primary ] 
% \end{keyword}
% % 62M05 Markov processes: estimation
% % 62F12 Asymptotic properties of estimators
% % 60J25 Markov processes with continuous parameter
% % 60J27 Markov chains with continuous parameter
% % 60J35 Transition functions, generators and resolvents


\end{frontmatter}

\section{Introduction}
The construction of the process we are interested in starts with a supercritical Galton-Watson tree $\T$ with offspring  distributed as a random variable $\nu$ such that $\E\cro{\nu}>1$. We adopt the following usual notations for tree-related quantities: the root  of $\T$ is denoted by $e$, for any $x \in \T$, $\nu_x$ denotes the number of descendants of $x$, %the variables $\nu_x$ are thus independent copies of $\nu$, 
 the parent of a vertex $x$ is denoted by $x^*$ and its children by $\set{x^i,\ 1\leq i\leq \nu_x}$. For technical reasons, we add to the root $e$, a parent $e^*$ which is not considered as a vertex of the tree. 
 % by \bl{$\llbracket x, y\rrbracket$} the sequence of vertices  realizing the unique shortest path between $x$ and $y$, by
 We denote by $|x|$ the generation of $x$, that is the length of the path from $e$ to $x$ and we write $x<y$ when $y$ is a descendant of $x$, also $x\leq y$ signifying  that $x$ can also be equal to $y$. Finally, we write $\T_n$ for the tree truncated at generation $n$. We then introduce a real-valued branching random walk indexed by $\T$: $\paren{V(x),x\in\T}$. We suppose that $V(e)=0$ and for any generation $n$, conditionally to $\mathcal{E}_n=\set{\T_n,(V(x),x\in\T_n)}$, the vectors of increments $((V(x^i)-V(x),i \leq \nu_x),|x|=n)$ are assumed to be i.i.d. 
 Finally, we denote by $\bP$ the distribution of $\mathcal{E}=\set{\T,\paren{V(x),x\in\T}}$ and $\bP^*$, the probability conditioned on the survival set of the tree $\T$.

We can now introduce the main process of this work which is a random walk $\seq{X}$ on $\T\cup\set{e^*}$ : for a given realization of the environment $\mathcal{E}$, $\seq{X}$ is a Markov chain  with transition probabilities given by
\begin{align*}
&\pe{X_{n+1}=e|X_n=e^*}=1\ ,\\
\forall x\in\T\smallsetminus\set{e^*},\quad &\pe{X_{n+1}=x^*|X_n=x}=\frac{e^{-V(x)}}{e^{-V(x)}+\sum_{i=1}^{\nu_x}e^{-V(x^i)}}, \\
	\forall j\leq\nu_x,\quad &\pe{X_{n+1}=x^j|X_n=x}=\frac{e^{-V(x^j)}}{e^{-V(x)}+\sum_{i=1}^{\nu_x} e^{-V(x^i)}}\enspace.
\end{align*}
The measure $\Pe$ is usually referred to as the quenched distribution of the walk $\seq{X}$ in contrast to the annealed distribution $\P$ which is the measure $\Pe$ integrated with respect to the law of $\mathcal{E}$: 
$$\p{\cdot}=\int\pe{\cdot}\bp{\ud \mathcal{E}}.$$
Similarly, $\P^*$ is the annealed probability conditioned on the survival set of the tree $\T$ (defined by replacing $\bf{P}$ by $\bf{P}^*$ in the above probability).  For $x\in\T\cup\set{e^*}$, we use the notation $\Pe_x$ for the conditional probability  $\Pe(\cdot | X_0=x)$; when there is no subscript, the walk is supposed to start at the root $e$. Recurrent criteria for these walks is determined from the fluctuations of log-Laplace transform 
\begin{align} \psi(s):=\log \Eb\Big[\sum_{|z|=1} e^{-sV(z)}\Big], s>0.\label{fcpsi} \end{align} If $\inf_{0\leq s \leq 1} \psi(s)>0$ then $(X_n,n)$ is $\bP$ almost surely transient and recurrent otherwise. It turns out that recurrent cases can be themselves classified, this can be found in the works of G. Faraud \cite{Faraud} and equivalently for transient cases in E. Aidekon \cite{Aidekon2008}.
%The walk $\seq{X}$, called biased random walk on a tree, was first introduced by R. Lyons (see \cite{Lyons} and \cite{Lyons2}). In our case where the bias is random, the first references go back to R. Lyons and R. Pemantle \cite{LyonPema} and M.V. Menshikov and D. Petritis \cite{MenPet}. Random walks in random environment on trees form a subclass of canonical models in the more general framework of random motions in random media that are widely used in physics. They are a natural extension of the one dimensional model, originally introduced in the works of \cite{Chernov}.  These models  have been intensively studied in  the last four decades, mostly   in   the    physics   and    probability   theory literature.%, seminal papers are \cite{Solomon}, \cite{KesKozSpi} and \cite{Sinai} for the one-dimensional case and \cite{Lyons}, \cite{HuShi10}, \cite{HuShi10a} for walks on trees. 

 %The behaviors of randomly biased walks on trees differ deeply from the behaviors of the RWRE in the one-dimensional case.  In particular there are several regimes for both recurrent and transient cases. 
 
% as resumed in Figure \ref{fig1}.
 %   \begin{figure}[ht]
%\begin{center} 
%{\scalebox{1.2} {\input{RecCrit.tex} }}
%\caption{Recurrence criteria for $\seq{X}$ on Galton-Watson trees} \label{fig1}
%\end{center}
%\end{figure}

Here we consider recurrent cases and more particularly the regime where the random walk is particularly slow (see \cite{HuShi15}), that is to say we put ourselves in the boundary case for which 
\begin{align}
\psi(1)=\psi'(1)=0. \label{hyp0}
\end{align}
\noindent  In this paper, we are interested in the trace of $\mathbb{X}$ which is the set of vertices visited by this random walk until a given instant. The literature on the subject initially started with the study of the range, that is to say the volume of the trace of the simple random walk on $\Z^d$, where $d\geq 2$ is the dimension. In particular P. Erdös and S. Taylor \cite{ErdTay} prove that the asymptotic in time of the trace depends on the dimension $d$. If we put ourselves in the present context of random walk in random environment on trees then the trace naturally depends on the hypothesis on the environment $\mathcal{E}$, see for example \cite{AndChen}, \cite{AidRap} and \cite{DeRap}. A first step in the extension of the notion of the range is to count, for example, the number of vertices visited a large number of time (instead of at least one time). This aspect has been studied for the simple random walk in \cite{Rosen} and in our context by \cite{AndDiel3} and \cite{Chen22}, about which we will give some details later in the paper. A second step in the study of the trace, especially in the case of random walk in random environment, is to select certain vertices not only with criteria on the trajectory of the walk but also on the underlying potential $V$.    
With this in mind we introduce a generalization of the range : for any $n$, let $\bfa^n=\{f^{n,k}:\R^k \rightarrow \R^+;\ k\in\N^*\}$ be a collection of bounded functions. Also, let ${g_n}:\R^+ \rightarrow \R $ be a positive function. Then, the generalized range $\Ran_n(g_n,\mathbf{f}^n)$ is given by \\
\begin{align}
& \Ran_n(g_n,\mathbf{f}^n):=\sum_{x \in \T}g_n(\mathcal{L}_x^n) f^{n,|x|}(V(x_1),V(x_2),\cdots,V(x)), \label{genRange} \textrm{ with }   \\
& \mathcal{L}_x^n:=\sum_{k=1}^n \un_{\{X_k=x\}}, \nonumber
\end{align}
 $(x_i,i \leq |x|)$ being  the sequence of vertices of the unique path from the root (excluded) to vertex $x$ and $\mathcal{L}_x^n$ is the usual local time of the walk at $x$ before the instant $n$. As we may see, $\Ran_n(g_n,\mathbf{f}^n)$ is quite  general and can not be treated in this form, at once for every of these functions $g_n$ and $\mathbf{f}^n$, so additional assumptions (involving $\mathbf{f}^n$, $g_n$ and distribution $\bP$)  will be introduced in Section \ref{sec1.2}. \\
The aim of studying this extended range is twofold, first it allows to understand the interactions between the trajectories of the main process $\mathbb{X}$ and of the underlying branching potential $\mathbb{V}$, second we  develop a general tool allowing to treat many examples (for chosen $\mathbf{f}^n$ and $g_n$).
\noindent Note that if we take, for example, $f^{n,k}=1$ and $g_n=\un_{[1,\infty)}$ for any integer $n$ and $k$, then we get the regular range (treated in \cite{AndChen}), and if $g_n= \un_{[n^b,\infty)}$ with $0<b<1$, then we get the heavy range (see  \cite{AndDiel3} and \cite{Chen22}). \\
The presentation of the results is divided into three subsections. In the first one below, we detail and comment particular examples showing a large variety of behaviors for the range for different $\mathbf{f}^n$ and $g_n$. In a second subsection, we present an informal statement of the general result, the aim of which is to give the main ideas without introducing too much technical material. Finally, in the last section, we introduce  assumptions which leads to the full statement of the main theorem.

\subsection{First results :  examples \label{sec1}}

The first two theorems (Theorems \ref{thm1} and \ref{thm2}) we present in this section derive from three other works : in the first one \cite{HuShi15b}, it is proved that, during its first $n$ steps, the walk can reach 
height of potential of order $(\log n)^2$. More precisely, it is proved that the random variable $(\nicefrac{\max_{1 \leq  k \leq n}V(X_k)}{(\log n)^2})_{n\geq 2}$ converges almost surely to one half. Note that this behavior can be quite disappointing if we have in mind the intuitive behavior of Sinai's one dimensional random walk in random environment \cite{Sinai} for which the highest height of potential reached by the walk is of order $\log n$. Of course the fact that the walk evolves on a tree instead of a one dimensional lattice changes the deal but at the same time it is also proved in \cite{HuShi15} that this walk has a similar behavior than Sinai's one (they are both at a distance of order $(\log n)^2$ from the origin at a given instant $n$). In both cases, the potential plays a crucial role. In the two other papers (\cite{AndChen} and \cite{AndDiel3}), the range is studied : in \cite{AndChen}, it is proved that regular range (the number of visited vertices up to the instant $n$) is of order $\nicefrac{n}{\log n}$, whereas in \cite{AndDiel3}, it is proved that the number of edges visited more than $n^b$ (with $0<b<1$)  times is typically of order $n^{1-b}$ (this particular range is called \guillemotleft heavy range\guillemotright in that paper, see also \cite{Chen22} for a refinement of this work). \\
Our first theorem below mixes the two approaches, showing the influence of a strong constraint on the trajectories of $V$ on both regular or heavy range. What we mean by strong constraint here is a condition of the form $V \geq (\log n)^{\alpha} $ with $1<\alpha<2$, that is to say when the potential is larger than what we can call regular height of potential for this walk (in the slow regime, a regular height is of order $\log n$ since it can be proved that $(\nicefrac{V(X_n)}{ \log n})_{n\geq 2}$ converges weakly, see \cite{HuShi15}) but smaller than the extreme value $(\log n)^2$ of \cite{HuShi15b}. %(which is the result we have just talked about : \cite{HuShi15b}).

%, that is to say volume of a particular trace before a given instant $n$ of random walk $\mathbb X$.
\noindent Before stating this result, let us introduce the  following hypothesis on the  distribution of the branching random walk : there exists $\theta>0$ and $\delta_1\in(0,1/2)$ such that
\begin{align}
\Eb\Big[\sum_{|z|=1}e^{-(1+\theta)V(z)}\Big]+\Eb\Big[\sum_{|z|=1}e^{\theta V(z)}\Big]<&\infty, \label{hyp0+}\\
 \Eb\Big[\Big(\sum_{|z|=1}(1+|V(z)|)e^{-V(z)}\Big)^2\Big]+\Eb\Big[\Big(\sum_{|z|=1}e^{-(1-\delta_1)V(z)}\Big)^2\Big]<&\infty, \label{hyp1}
\end{align}
these are common hypothesis used for example in \cite{AndChen}.

\vspace{0.3cm}
\noindent 

\begin{theo} \label{thm1}
Assume \eqref{hyp0}, \eqref{hyp0+} and \eqref{hyp1} hold. If for any $n$ and $k$, $f^{n,k}(t_1,t_2, \cdots,t_k)= \un_{\{t_k \geq (\log n)^{\alpha}\}}$ with $\alpha\in(1,2)$ and if $g_n(t)=\un_{\{t \geq n^{b}\}}$ with $b\in[0,1)$, then
\begin{center}
$\big(\frac{\log^+\Ran_n(g_n,\bfa^n)-(1-b)\log n}{(\log n)^{\alpha-1}}\big)_{n\geq 2}$ converges in $\P^*$-probability to $-1$,	\\
\end{center}
where  $\log ^+ x = \log (\max(1,x))$.
\end{theo}
This result shows that the number of vertices with high potential visited at least once (resp. strongly visited,  with $b>0$) is of the same order, though smaller, than the  regular  range (resp.  heavy-range). So visiting high potential is not just an accident appearing a couple of times on very specific paths of the tree. Far from that in fact, as the constraint of high potential creates a decrease of order $e^{- (\log n)^{\alpha-1}+ o(1)}$ and therefore appears as a second order correction comparing to ranges without constraint on the environment. \\
In the second theorem below, we add a slight different constraint which force the random walk to reach a high level of potential far from the ultimate visited vertices of given paths: 

\begin{theo} \label{thm2}
Assume \eqref{hyp0}, \eqref{hyp0+} and \eqref{hyp1} hold. If for any $n$ and $k$, $f^{n,k}(t_1,t_2, \cdots,t_k)= \\ \un_{\{t_{\lfloor k/\beta \rfloor} \geq (\log n)^{\alpha}\}}$ with $\beta>1$, $\alpha\in(1,2)$ ($\lfloor x \rfloor$ stands for the integer part of $x$) and for any $b\in[0,1)$, $g_n(t)=\un_{\{t\geq n^b\}}$ then
\begin{center}
$\big(\frac{\log^+\Ran_n(g_n,\mathbf{f}^n)-(1-b)\log n}{(\log n)^{\alpha-1}})_{n\geq 2}$ converges in $\P^*$-probability to $-1-\frac{\pi}{2}\sqrt{\beta-1}+\rho\big( (\beta-1) \frac{\pi^2}{4} \big)$, 	
\end{center}
where for any $c>0$,
\[ \rho(c):= \frac{c \sigma }{\sqrt{2\pi}} \int_{0}^{+ \infty} {e^{- \frac{c \sigma^2}{2} u}} \Big [ \frac{2}{u^{1/2}}\P(\overline{\mathfrak{m}}_1>1/\sqrt{u \sigma^2})  - \frac{1}{2}\int_u^{+\infty} \frac{1}{ y^{3/2}} \P(\overline{\mathfrak{m}}_1>1/\sqrt{y \sigma^2})dy \Big] du, \]
and  $\mathfrak{m}$ is a Brownian meander, $\overline{\mathfrak{m}}_1:=\sup_{s \leq 1} \mathfrak{m}_s $ and $\sigma^2:= \Eb[\sum_{|x|=1} V^2(x) e^{-V(x)}]$. 
\end{theo}
%$\rho(c)= \frac{c \sigma }{\sqrt{2\pi}} \int_{0}^{+ \infty} {e^{- \frac{c \sigma^2}{2} u} f(u)}du$, and $f(u)=\frac{2}{u^{1/2}}\P(\overline{\mathcal{M}}_1>1/\sqrt{u \sigma^2})  - \frac{1}{2}\int_u^{+√É¬É√Ç¬É√É¬Ç√Ç¬É√É¬É√Ç¬Ç√É¬Ç√Ç¬É√É¬É√Ç¬É√É¬Ç√Ç¬Ç√É¬É√Ç¬Ç√É¬Ç√Ç¬ä\infty} \frac{1}{ y^{3/2}} \P(\overline{\mathcal{M}}_1>1/\sqrt{y \sigma^2})dy$. 

\noindent As we may see, a slight change in function $\mathbf{f}^n$ (comparing to previous theorem) makes appear something new, as the constant in the limit is very different than in Theorem \ref{thm1}. Note that $\rho$ can be explicitly calculated : for any $c>0$
\begin{align}\label{Def_rho}
\rho(c)& =2 \sqrt{c}\Big ( \frac{1-e^{-\sqrt{c}}}{\sinh(\sqrt{c})} \Big) -2\Big( \sqrt c-\log( (e^{\sqrt{c}}+1)/2) \Big),
% & = 2 \sqrt{c}\Big ( \frac{1-e^{-\sqrt{c}}}{\sinh(\sqrt{c})} \Big)-2\Big( \frac{\sqrt c}{2}-\log(\cosh({\sqrt{c}/2})) \Big)  
\end{align}
so we clearly obtain continuity when $\beta$ converges to 1, getting back to the previous theorem. At this point, we also would like to discuss the appearance of the Brownian meander distribution in $\rho$. First, note that a Brownian meander appears  in the asymptotic distribution of the (correctly normalized) generation of $X_n$ (see \cite{HuShi15}) which is the consequence  of the positivity of $V$ (see Fact 4 below, page 11) together with an induced  constraint on the largest downfall of $V$ (we  call maximal downfall, for a given $x \in \T$, the quantity $\max_{y \leq x}(\overline{V}(y) -V(y))$, where $\overline{V}(y):=\max_{z \leq y} V(z)$)  visited by the walk before the instant $n$.  Also in \cite{AndChen}, the distribution of two independent Brownian meanders ($\mathfrak{m}^1$ and $\mathfrak{m}^2$) appears in the result for the regular range $\mathcal{R}_n$ (that is when $f^{n,k}=1$ and $g_n= \un_{[1,\infty)}$) : in $\P^*$-probability
 \begin{align} 
\lim_{n \rightarrow +\infty}\mathcal{R}_n  \frac{\log n}{n}=C(\mathcal{D}_{\mathfrak{m}^1},\mathcal{D}_{\mathfrak{m}^2}), \label{Rera1}
\end{align}  
one of these Brownian meanders also coming from the positivity of $V$ and the other one coming from the fact that for a given visited vertex $x$, the maximum of $V$ (on the unique path from the root  to $x$) is attained pretty near the generation of $x$. \\
Here, the Brownian meander appears as we ask a visited vertex $x$ to have reached a high level of potential in an early generation before the one of $x$ and it turns out that the constraint of low downfall of $V$ appearing in \cite{HuShi15}  ($\max_{y \leq x}(\overline{V}(y) -V(y) \leq \log n$)  along this kind of path produces this appearance of the Brownian meander. However, contrarily to \eqref{Rera1}, the Brownian meander is involved in the correction of the main fluctuation $(e^{-C(\mathcal{D}_{\mathfrak{m}})(\log n)^{\alpha-1}})$ and not just in the  constant 
of the limit ($C(\mathcal{D}_{\mathfrak{m}^1},\mathcal{D}_{\mathfrak{m}^2})$).  \\

%Note that we did not treat the heavy range in this second theorem, however we can expect a similar behavior that is to say taking $f$ as in Theorem  \eqref{thm1} and $g(t)=\un_{\{t \geq n^{b}\}}$ : 
%\begin{center}
%$\frac{\log^+\Ran_n(f,g)-(1-b)\log n}{(\log n)^{\alpha-1}} $ converges in $\P^*$-probability to $-\pi /2(\beta-1)$,  	
%\end{center}
In the third example below, we choose $f^{n,k}$ in such a way that an interaction appears between the trajectory of $\mathbb{X}$ and the downfalls of $V$, which have an important role in the behavior of these walks. More particularly, let us introduce, for a given $\mathbf{t}=(t_1,t_2, \cdots, t_k)$ with $k$ a positive integer, the following quantity
\[ H_k(\mathbf{t}) := \sum_{j=1}^k e^{t_j-t_k}, \]
then we call \textit{sum of exponential downfalls} of $V$ at $x \in \T$ with $|x|=k$ the quantity
\begin{align}
H_{|x|}(\mathbf{V}_x):=H_{|x|}({V}(x_1),\cdots,V(x_k))= \sum_{i=1}^{|x|} e^{V(x_i)-V(x_k)}. \label{Def_Hx}
\end{align}
In order to simplify the notation and when there is no possible confusion, we will simply write $H_x$ instead of $H_{|x|}(\mathbf{V}_x)$ in the sequel. %\red{[Bien verifier la coherence et la lecture des differentes notations]}

%the quantity which is nothing but a deterministic version of $H_x=H_k(\mathbf{V}_x)=H_k(\mathbf{V}(x_1),\cdots,V(x_k))$ if $|x|=k$ introduced  in \eqref{RegularLine}. 

\begin{theo} \label{thm3} 
Assume \eqref{hyp0}, \eqref{hyp0+} and \eqref{hyp1} hold.\\
For any $n$ and $k$ let $f^{n,k}(t_1,t_2, \cdots,t_k)=\un_{\{t_{k} \geq a(\log n)^{\alpha}\}}(\sum_{j=1}^kH_j(\mathbf{t}))^{-d}$ with $\alpha\in[1,2)$, $a\in \R$, $d \in \{0,1\}$ and  $g_n(t)=\un_{\{t \geq n^b\}}$ with $b\geq 0$. \\
 If $b\in[0,1/(1+d))$ and $\alpha = 1$ (with $a>1/\delta_1$ when $d=1$) then
\begin{center}
$\big(\frac{\log^+\mathcal{R}_n(g_n,\mathbf{f}^n)}{ \log n}\big)_{n\geq 2}$ converges in $\P^*$-probability to $1-(1+d)b$,
\end{center}
otherwise if $a=1$, $b=0$, $d=1$ and $1< \alpha <2$
\begin{center}
$\big(\frac{\log^+\mathcal{R}_n(g_n,\mathbf{f}^n)-\log n}{(\log n)^{\alpha/2}}\big)_{n\geq 2}$ converges in $\P^*$-probability to $-2$,
\end{center}
 finally if $a=1$, $0<b<1/2$, $d=1$ and $1< \alpha <2$
\begin{center}
$\big(\frac{\log^+ \mathcal{R}_n(g_n,\mathbf{f}^n)-(1-2b)\log n}{(\log n)^{\alpha-1}}\big)_{n\geq 2}$ converges in $\P^*$-probability to $-1/b$.
\end{center}
%\red{Continuit\'e $b \to 0$ ? exemple \`a revoir, d'autant que je crois que je pr\'ef\`ere la normalisation en $\exp(\overline V-V)$ .... mais cela doit pouvoir s'arranger.}
\end{theo}

%\bl{Dire que les cumultative exponential downfall ont une grande importance pour la trajectoire de $X$ d'o\`u l'id\'ee d'inclure une contrainte l'impliquant dans $f$.}
\noindent  For the first limit (when $ \alpha = 1$, implying that we have set a common height of potential - see Fact 1), by taking $d=0$, we obtain the limit ($1-b$) of the usual heavy range of \cite{AndDiel3}. Otherwise, if we add the penalization with the cumulative exponential downfalls $(\sum_{y \leq x} H_y)$, that is when $d=1$, then an extra cost $d*b=b$ appears. \\ %\red{This example can be used as a point of comparison to the two other more interesting cases that follow}.  \\
The second case (with $b=0$ but $1<\alpha < 2$) has two constraints on the environment so the normalization  $(\log n)^{\alpha/2}$  appears as a compromise  between the fact that high level of potential is asked ($\un_{ \{t_{k} \geq (\log n)^{\alpha}\}}$), which alone yields by Theorem \ref{thm1} a normalization $(\log n)^{\alpha-1}$, and the fact that cumulative exponential downfall fluctuations ($\sum_{m \leq k}H_m(\mathbf{t})$) can not be two large as it appears in the denominator of the range. This yields the $(\log n)^{\alpha/2}$ (note that as $\alpha<2$, $\alpha/2> \alpha-1$). \\ 
For the last case ($0<b<1/2$ and $1< \alpha <2$), the range is of order $n^{1-2b}e^{ -(\log n)^{ \alpha-1} /b}$ comparing to $n e^{ -2(\log n)^{ \alpha/2}}$ when $b=0$ of the previous case. In particular, the parameter $b$ of the heavy range appears in both the main normalization $n^{1-2b}$ and in the correction $e^{ -(\log n)^{\alpha-1} /b}$. This can be intuitively understood as follows : first  $n^{1-2b}=n*n^{-b}* n^{-b}$, one  $n^{-b}$ is classical from the heavy range when asking for a local time to be larger than $n^b$ (which already appears in  the first part of the Theorem), the second $n^{-b}$ comes from the fact that a local time at a given vertex $x$ can be larger than $n^b$ only if 
$\sum_{j=1}^{|x|} e^{V(x_j)-V(x)} \geq n^b$ and as this quantity appears in the normalization of the range (via $f^{n,k}(t_1,t_2, \cdots,t_k)$) this produced this second $n^{-b}$. So this part ($n^{1-2b}$) appears as a first interaction between the constraints on the trajectory of $\mathbb{X}$ and the one of $\mathbb V$. Let us now discuss about $e^{-\nicefrac{(\log n)^{\alpha- 1}}{b}} =e^{-\nicefrac{(\log n)^{\alpha}}{(b \log n)}}$. For this term, we see intuitively the constrains for the walk to reach height of potential of order $(\log n)^{\alpha}$ but a the same time,  in order to keep the denominator $\sum_{j \leq k} H_j(\mathbf{t})$ as low as possible, the maximal downfall has to remain  smaller than ${b \log n}$, thus producing the ratio $(\log n)^{\alpha} / (b \log n)$.  
\\

In the ultimate example below, we ask similar constraints for the environment than above but only in the early visited generations :

\begin{theo} \label{thm4}
Assume \eqref{hyp0}, \eqref{hyp0+} and \eqref{hyp1} hold. Let $\beta >1$. For any $n$ and $k$, let $f^{n,k}(t_1,t_2, \cdots,t_k)= \un_{\{t_{\lfloor k/\beta \rfloor} \geq (\log n)^{\alpha}\}}(\sum_{j=1}^{\lfloor k/\beta \rfloor}H_j(\mathbf{t}))^{-1}$, $\alpha\in(1,2)$ and if $g_n(t)=\un_{\{t \geq n^{b}\}}$ with $b\in(0,1)$, then
\begin{center}
$\big(\frac{\log^+ \mathcal{R}_n(g_n,\mathbf{f}^n)-(1-b)\log n}{(\log n)^{\alpha/2}}\big)_{n\geq 2}$ converges in $\P^*$-probability to $-2$. 	
\end{center}
\end{theo}
This last theorem just prove that if the factor $(\sum_{j=1}^{\lfloor k/\beta \rfloor}H_j(\mathbf{t}))^{-1}$ only concerns the beginning of the trajectory, that is the sites at a distance $\lfloor |x|/\beta\rfloor $ of the root (if  $x$ is a visited vertex), then things go back to normal: there is no more multiple interactions between $\mathbb{X}$ and $\mathbb V$.% in the final result like it was the case in the previous theorem.
\\

We can imagine more examples like the ones we present above (by acting more on the function $g_n$ as we did for example) but for now, let us introduce a more general result with  general hypothesis on $g_n$ and $\mathbf f^n$.

\subsection{A general result (informal statement)}

 %\red{[on a besoin d'un peu plus, les ajouter le cas \'ech\'eant]}.
In this section, we present an informal statement for the asymptotic in $n$ of   $\mathcal{R}_n({g_n},\mathbf{f}^n)$ for general  $g_n$ and $\mathbf{f}^n$ (including, in particular, the results of the preceding section). The aim, in a first step, is to introduce the result and the main ideas but to minimize the technical materials.
%which are 
%Comparing to above theorems it is less readable both in its statement and on the complexity of additional assumptions needed to introduce it. 
\noindent First recall the expression of the generalized range \eqref{genRange}
$$ \Ran_n(g_n,\mathbf{f}^n)=\sum_{x \in \T}g_n(\mathcal{L}_x^n) f^{n,|x|}(V(x_1),V(x_2),\cdots,V(x)),$$
with $\mathcal{L}_x^n$ the local time of $\mathbb X$ at $x$ before the instant $n$. \\
We assume that $g_n$ can be written as the product of an indicator function and a function $\varphi$ which is positive non-decreasing:  for any $b\geq 0$ and $t\geq 1$, $g_n(t):= \un_{\{t\geq n^b\}}\varphi(t)$. The indicator function is here to include all types of range (regular or heavy). Also, we ask the function $t\mapsto\varphi(t)/t$ to be non-increasing, so that $\varphi(\mathcal{L}_x^n)$ remains  reasonable (at most of the order of the local time itself). %\bl{[...]}
%Finally we also add the technical assumption : for all $\theta\in\R$, there exists $\gamma\in\R$, $\varphi\big(t(\log t)^{\theta})/\varphi(t)=O\big((\log t)^{\gamma}\big)$ \bl{[controler si la decroissance ne suffit pas]}. \\

\noindent
Let us introduce the branching object  $\Psi$ as follows : let $0\lor\lambda'<\lambda$ be two real numbers and $k\geq 1$ an integer, also let $\phi:\R^k\longrightarrow \R$ be a bounded function. $\Psi^k_{\lambda,\lambda'}(\phi)$ is then defined as a mean of $\phi$ along the trajectory of $V$ (with constraints) until generation $k$, that is 
\begin{align}
    \Psi^k_{\lambda,\lambda'}(\phi):= \mathbf{E}\Big[\underset{|x|=k}{\sum}e^{-V(x)}\phi\left(V\left(x_{1}\right),\ldots, V(x)\right)\mathds{1}_{\mathcal{O}_{\lambda,\lambda'}}(x)\Big], \label{Def_Psi}
\end{align} 
where $\mathcal{O}_{\lambda,\lambda'}$ is the set of $(\lambda,\lambda')$-\textit{regular lines} 
\begin{align}\label{RegularLine}
    \mathcal{O}_{\lambda,\lambda'} := \big\{x\in \mathbb{T};\; \underset{j\leq|x|}{\max}\; H_{x_j}\leq \lambda,\; H_x>\lambda'\big\},\ \textrm{ with } H_{x_j}=\sum_{i=1}^j e^{V(x_i)-V(x_j)}, %\ (H_x=H_{x_{|x|}}),
\end{align}
also we denote
\begin{align*}
    \mathcal{O}_{\lambda} := \big\{x\in \mathbb{T};\; \underset{j\leq|x|}{\max}\; H_{x_j}\leq \lambda\big\},\ \textrm{ and } \Psi^k_{\lambda}(\phi):= \mathbf{E}\Big[\underset{|x|=k}{\sum}e^{-V(x)}\phi\left(V\left(x_{1}\right),\ldots, V(x)\right)\mathds{1}_{\mathcal{O}_{\lambda}}(x)\Big].
\end{align*}
Note that since $H_x\geq 1$ ($H_x>1$ when $|x|>1$), we have, for all $\lambda'< 1$, $\mathcal{O}_{\lambda,\lambda'}=\mathcal{O}_{\lambda}$ and $\Psi^k_{\lambda,\lambda'}(\phi)=\Psi^k_{\lambda}(\phi)$. \\

\noindent The appearance of this set of regular lines $\mathcal{O}_{\lambda,\lambda'}$ is partly inspired from the works of  \cite{HuShi15} ($\lambda$ representing extreme exponential downfalls of $V$ related to a reflecting barrier for the walk $(X_k,k \leq n)$), and also (for $\lambda'$) from the constraint on the local time appearing in the function $g_n$. It turns out indeed that constraints on the value of the local time at a site $x$ imply constraints on $H_x$. In other words, there are constraints on the branching potential $\mathbb{V}$ induced by constraints on the random walk $\X$ and sometimes, these constraints have a major impact on the range. We call this type of contribution \guillemotleft   contribution of type one\guillemotright, that is of order $n^{\theta}$ where $\theta\in(0,1]$ (this actually appears for example in Theorem \ref{thm3}). To be more specific, let us introduce the following notations: first $C_{\infty}:=C_{\infty}(\{\mathbf{f}^n;n\geq 1\})$ stands for the supremum of $\{\mathbf{f}^n;n\geq 1\}$ that is
    $C_{\infty}:=\sup_{m,\ell}\|f^{m,\ell}\|_{\infty}$. Then,  define the set 
\begin{align}
\mathcal{U}_b := \big\{\kappa\in[0,1];\textrm{ for all } k\geq 1,\mathbf{t}\in\R^k,n\geq 1: \un_{\{H_k(\mathbf{t})>n^b\}} f^{n,k}(\mathbf{t})\leq C_{\infty} n^{-\kappa}\big\}, \label{Ub}
\end{align}
note that $\mathcal{U}_b\not = \emptyset$ because $0\in \mathcal{U}_b$ and as the supremum is attained, let
% where for any $k>0$
%\[ H_k(\mathbf{t}) := \sum_{j=1}^k e^{t_j-t_k}, \]
%which is nothing but a deterministic version of $H_x=H_k(\mathbf{V}_x)=H_k({V}(x_1),\cdots,V(x_k))$ if $|x|=k$ introduced  in \eqref{RegularLine}.
\begin{align} \kappa_b =: \max\mathcal{U}_b \label{kappab}. \end{align}
When $\kappa_b>0$, we say that a mixing between the constraints on trajectories of the random walk $\X$ and on those of the branching potential $\mathbb{V}$ produce a contribution of type one. \\
{To introduce a second type of contribution, which can be seen as the  second order comparing to the contribution of type one,  we present an important quantity which is the sum over all the generations of $\Psi^{\cdot}_{n,n^b}(f^{n,\cdot})$ :   $\sum_{k\geq 1}\Psi^{k}_{n,n^b}$ $(f^{n,k})$}. First, let us give an heuristic about the way it appears in the asymptotic of the range.  \\ 
For any $k \geq 1$, introduce the $k^{th}$ return time $T^k:=\inf\{k>T^{k-1}, X_k=e\}$ \label{defoftn} to $e$ and take $T^0=0$. Recall the definition of $\overline V$ before \eqref{Rera1} and let $\overline{\mathcal{R}}_{T^n}(g_n,\mathbf{f}^n):=\sum_{x\in \T}g_n(\mathcal{L}^{T^n}_x)f^{n,|x|}(V_x)\un_{\{\overline{V}(x)\geq A\log n\}}$ with $A>0$ \label{leRbar}. $\overline{\mathcal{R}}_{T^n}(g_n,\mathbf{f}^n)$ is a version of the generalized range where we have replaced the instant $n$ by $T^n$ and we have made appear the additional constraint $\overline{V}(x)\geq A\log n$. {Note that it is known (following Lemma 2.1 in \cite{AndChen} and its proof at the beginning of Section 4.2) that  this additional condition $\un_{\{\overline{V}(x)\geq A\log n\}}$ has no effect on the normalization of the range, that is \\
{\bf Fact 1}: There exists $0<c_1=c_1(A) \leq 1$ such that $\lim_{n \rightarrow +¬†\infty} \P^*\big(\frac{ \overline{\mathcal{R}}_{T^n}} {\mathcal{R}_{T^n}}=c_1\big)=1$. \\
So here, we typically consider collections of functions $\mathbf{f}^n$ such that $ \overline{\mathcal{R}}_{T^n}(g_n,\mathbf{f}^n)/ \mathcal{R}_{T^n}(g_n,\mathbf{f}^n) \rightarrow Cte>0$}. One of the main gain of this consideration is the fact that relatively high potential yields interesting quasi-independence in the trajectory of $(X_n,n)$. \\
With this fact, we have (see Section \ref{sec3.1}) something like $\overline{\mathcal{R}}_{T^n}(g_n,\mathbf{f}^n)\gtrsim n \Ee[\overline{\mathcal{R}}_{T^1}(g_n,\mathbf{f}^n)]$ in probability and thanks to the fact that $\varphi$ is non-decreasing and to the expression of the quenched mean of $\overline{\mathcal{R}}_{T^1}(g_n,\mathbf{f}^n)$,  in probability, for large $n$
 \begin{align*}
    \overline{\mathcal{R}}_{T^n}(g_n,\mathbf{f}^n) \gtrsim n\Ee[\overline{\mathcal{R}}_{T^1}(g_n,\mathbf{f}^n)]\gtrsim \frac{\varphi(n^b)}{n^b}n\sum_{k\geq 1}\Psi^k_{n,n^b}(f^{n,k}), 
 \end{align*}
which makes appear $\sum_{k\geq 1}\Psi^k_{n,n^b}(f^{n,k})$. It turns out that this lower bound is exactly the good quantity which leads to our main result. \\
%Keeping that in mind,\Modif\Modif let us now introduce the sequence $(b_n,n)$ which will parametrized the indicator function present in $g$ (in words $b_n$ is the lower bound asked for the local time of the walk) :
%\begin{align}
% b_n:=an^b (\log n)^{d}, \textrm{ for any } a\geq 1 \textrm{ and } b,d\geq 0. \Modif\Modif
%\end{align}
The following assumption ensures that $\sum_{k\geq 1}\Psi^k_{n,n^b}(f^{n,k})$ is not too small, which would correspond to an exaggerate penalization on the potential $V$: \\
\textit{Assumption 1}.\\
For all $b\in[0,1)$, $\varepsilon>0$ and $n$ large enough
\begin{align}\label{Hypothese1}
   \sum_{k\geq 1}\Psi^{k}_{n,n^b}(f^{n,k})\geq  \frac{1}{n^{(\kappa_b+\varepsilon)\land 1}}.
    \tag{\textbf{A1}}  
\end{align}
The second type of contribution that we call\textrm{ \guillemotleft contribution of type two\guillemotright } strongly involves the term $n^{\kappa_b}\sum_{k\geq 1}\Psi^k_{n,n^b}(f^{n,k})$.
It is negligible with respect to $n^{\varepsilon}$ for all $\varepsilon>0$ and  comes also from a mixing between the constraints on $\X$ and the constraints on $\mathbb{V}$.
 %We are now ready to introduce two assumptions, and then state the result. \\
\noindent So finally introduce $(h_n,n)$ which is certainly the most important sequence of the paper : for any $n \geq 2$
\begin{align}
    h_n :=\left \{
    \begin{array}{l c l}
      \big|\log\big(n^{\kappa_b}\sum_{k\geq 1}\Psi^{k}_{n,n^b} (f^{n,k})\big)\big|  \quad \text{if} \quad \exists\;\gamma\in(0,1):\;\frac{(\log n)^{\gamma}}{\log \big(n^{\kappa_b}\sum_{k\geq 1}\Psi^{k}_{n,n^b} (f^{n,k})\big)}\to 0 \\
      \log n   \qquad \text{otherwise} 
   \end{array}.
   \right.\label{Def_hn}
\end{align}

%\noindent \red{Discussion sur $h_n$ : Inclure la Remarque 1 ci-dessous.} %We discuss about the order of $h_n$ below (see Remark \ref{rem0}), before that let us introduce our first main assumption which  is a lower bound involving $\Psi$. 

%\RevAlexis{\textbf{Version 1, pas tout à fait ok! (discutée le 15/09)} The parameter $\kappa_b$ is dedicated to the contribution of both constraints on the trajectories of the random walk $\X$ and on those of the branching potential $\mathbb{V}$ via the term $\un_{\{H_k(\mathbf{t})>n^b\}} f^{n,k}(\mathbf{t})$ in \eqref{Ub}. \\
%The term $|\log(n^{\kappa_b}\sum_{k\geq 1}\Psi^{k}_{n,n^b} (f^{n,k}))|$ is actually the cost of the constraint $\mathbf{f}^n$. According to the asymptotic behavior of this term, we assign $h_n$ two possible expressions: if $(\log n)^{\gamma}$ is negligible with respect to $|\log(n^{\kappa_b}\sum_{k\geq 1}\Psi^{k}_{n,n^b} (f^{n,k}))|$ for some $\gamma\in(0,1)$ (which then remains smaller than $\varepsilon\log n $ by \eqref{Hypothese1}), the constraint $\mathbf{f}^n$ is considered penalizing and we set $h_n:=|\log(n^{\kappa_b}\sum_{k\geq 1}\Psi^{k}_{n,n^b} (f^{n,k}))|$, see Theorem \ref{thm1} for example. Otherwise, the constraint $\mathbf{f}^n$ is not penalizing enough and we set $h_n:=\log n$, see Theorem \ref{thm3} with $\alpha=1$ for example. In this latter case, the choice is significant since $\log n$ is the right order for the logarithm of the regular range, that is to say the range without any constraint on the trajectories of the branching random potential $\mathbb{V}$; Mettre la Remark 1 avant? } 

\noindent Let us start by a discussion about $(h_n,n)$ with the following remark in which we note that either $h_n=o(\log n)$ or $h_n=\log n$.
\begin{remark} \label{rem0}
By definition of $\kappa_b$, $$n^{\kappa_b}\sum_{k\geq 1}\Psi^{k}_{n,n^b} (f^{n,k}) \leq C_{\infty} \sum_{k\geq 1}\Psi^{k}_{n} (1)= C_{\infty} \Eb\Big[\sum_{x \in \T}e^{-V(x)}\un_{\{x\in\mathcal{O}_n\}}\Big]\leq C_{\infty}  (\log n)^3,$$ where the last inequality is a quite elementary fact that will be proved later (see {\bf Remark} \ref{rem1b}). This implies, in particular, that 
if there exists $0<\gamma<0$ such that ${(\log n)^{\gamma}}/{\log \big(n^{\kappa_b}\sum_{k\geq 1}\Psi^{k}_{n,n^b} (f^{n,k})\big)}\to 0$, then necessarily  $\log(n^{\kappa_b}\sum_{k\geq 1}\Psi^{k}_{n,n^b} (f^{n,k}))<0$ and $\lim_{n \rightarrow + \infty} \log(n^{\kappa_b}\sum_{k\geq 1}\Psi^{k}_{n,n^b} (f^{n,k})) =-\infty$. Moreover, in this case, there  exists $0<\gamma<1$ such that $ h_n \geq (\log n)^{\gamma}.$
Also assumption \eqref{Hypothese1} above ensures that 
\begin{align*}
\log(n^{\kappa_b}\sum_{k\geq 1}\Psi^{k}_{n,n^b} (f^{n,k}) ) \geq \log \Big ( \frac{n^{\kappa_b}}{n^{(\kappa_b+\varepsilon)\land 1}} \Big) \geq  -((\kappa_b+\varepsilon)\land 1- \kappa_b) \log n \geq -\varepsilon\log n, 
\end{align*}
overall, definition of $h_n$ implies, under \eqref{Hypothese1}, that
$$ (\log n)^{\gamma} \leq h_n \leq   \log n.$$
\end{remark}
\noindent The sequence $(h_n,n)$ is the quantity which gives the contribution of type two and produces the second order in our result. It is important to note that we carefully assign an expression to $h_n$ depending on whether constraints are penalizing or not. According to the asymptotic behavior of the term $n^{\kappa_b}\sum_{k\geq 1}\Psi^{k}_{n,n^b} (f^{n,k})$, we assign $h_n$ two possible expressions : if $(\log n)^{\gamma}$ is negligible with respect to $|\log(n^{\kappa_b}\sum_{k\geq 1}\Psi^{k}_{n,n^b} (f^{n,k}))|$ for some $\gamma\in(0,1)$ (which then remains smaller than $\varepsilon\log n $ by Remark \ref{rem0}), constraints are considered penalizing and we set $h_n:=|\log(n^{\kappa_b}\sum_{k\geq 1}\Psi^{k}_{n,n^b} (f^{n,k}))|$, see Theorem \ref{thm1} for example. Otherwise, constraints are not penalizing enough and we set $h_n:=\log n$, see Theorem \ref{thm3} with $\alpha=1$ for instance. In this latter case, the choice is significant since $\log n$ is the right order for the logarithm of the regular range, that is to say the range without any constraint on the trajectories of the branching random potential $\mathbb{V}$.\\
\noindent We are now almost ready to state a result. But first introduce two last values :  $L$ (with $L=\pm \infty$ possibly) and $\xi\in\{-1,0\} $ defined as follows
\begin{align}
    & L := \liminf_{n\rightarrow \infty}h_n^{-1}\log \big(n^{1-b-\kappa_b}\varphi(n^b)\big)\label{limite_L},  \textrm{ and}\\ &
    \xi := \lim_{n\rightarrow \infty}h_n^{-1}\log \big(n^{\kappa_b}\sum_{k\geq 1}\Psi^{k}_{n,n^b} (f^{n,k})\big), \label{Limite_xi}
\end{align}
and note that, following Remark \ref{rem0}, $\xi$ necessarily exists. \\
The full statement of our main result below need additional quite complex assumptions, involving $\mathbf{f}^n$ in particular,  they are described precisely in the next section (see \eqref{Hypothese2}, \eqref{Hypothese3} and \eqref{Hypothese4}). 
The interesting point is the fact that all of these assumptions concern $\sum_{k\geq 1}\Psi^{k}$. And more than that, we can  resume  the actions  of \eqref{Hypothese2}, \eqref{Hypothese3} and \eqref{Hypothese4} by  saying that $\Psi$ has to be stable for small perturbations of its parameters. In the informal statement below, we will say that $\Psi$ should have controlled fluctuations. \\
\begin{theo} [Informal statement] \label{Prop2bb} Assume \eqref{hyp0}, \eqref{hyp0+} and \eqref{hyp1} hold, $b \in [0,1)$, assume also that  \eqref{Hypothese1} is satisfied and $\Psi$ has controlled fluctuations.
%\noindent Assume \eqref{Hypothese2}, \eqref{Hypothese1}, \eqref{Hypothese3} and \eqref{Hypothese4} hold \DO{+ compl\'eter}. \\
If $L \in (-\xi,+\infty]$, then in $\P^*$-probability \begin{align*}%\label{ConvProba1}
      h_n^{-1}\big(\log^{+} \mathcal{R}_{n}(g_{n},\mathbf{f}^n)-\log( n^{1-b-\kappa_b}\varphi(n^b))\big)\underset{n\rightarrow\infty}{\longrightarrow}\xi,
\end{align*}
if $L=-\xi$, with $\underline{\Delta}_n:=h_n^{-1}\log(n^{1-b-\kappa_b}\varphi(n^b))-\inf_{\ell\geq n}h_{\ell}^{-1}\log(\ell^{1-b-\kappa_b}\varphi(\ell^b))$, then in $\P^*$-probability
\begin{align*}
    h_n^{-1}\log^{+} \mathcal{R}_{n}(g_{n},\mathbf{f}^n)-\underline{\Delta}_n\underset{n\rightarrow\infty}{\longrightarrow} 0,
\end{align*}
otherwise $L\in[-\infty,-\xi[$ and in $\P^*$-probability 
\begin{align*}%\label{ConvProba3}
       \mathcal{R}_{n_{\ell}}(g_{n_{\ell}},\mathbf{f}^{n_{\ell}})\underset{\ell\rightarrow\infty}{\longrightarrow} 0,
\end{align*}
for some increasing sequence $(n_{\ell})_{\ell}$ of positive integers. Note that when $\lim h_n^{-1}\log(n^{1-b-\kappa_b}\varphi(n^b))=L$, $n_{\ell}=\ell$.
\end{theo}

%\noindent \DO{One can notice that when $L \in (- \xi,+\infty) $ we simply have ( 03/09 √É¬É√Ç¬† retirer car s√É¬É√Ç¬ªrement faux)
%\begin{align*}
%    h_n^{-1}\log^{+} \mathcal{R}_{T^{n}}(g_{n},\mathbf{f}^n)\underset{n\rightarrow\infty}{\longrightarrow}\xi+L, 
%\end{align*}
%so no second order appears in the estimation. \\
\noindent We now present particular examples which lead to different values of $L$ and $\xi$. First, note that all theorems presented in the previous section satisfy $L=+\infty$ {and $\xi=-1$}, corresponding, from our point of view, to the most interesting case. 
Let us take, for example, $g_n(t)= \un_{\{t\geq n^b\}}$ and $f^{n,k}(t_1,t_2 $ $, \cdots,t_k)= \un_{ \{t_{k} \geq a(\log n)^{\alpha}\}}(\sum_{l \leq k }H_l(\mathbf{t}))^{-1}$ as in Theorem \ref{thm3}, with $a>0$, $\alpha\in[1,2)$, but $b \in [1/2,1)$. 
%\DO{Let us link this general result with some of the examples we have treated in the previous section. 
%When $\mathcal{R}_{T^{n}}(g_{n},\mathbf{f}^n)=\sum_{x\in\T}\un_{\{\mathcal{L}_x^{T^n}\geq n^b\}}\un_{\{V(x)\geq(\log n)^{\alpha}\}}(\sum_{j\leq |x|}H_{x_j})^{-1}$, that is the same function as the one in Theroem \ref{thm3}, 
When $\alpha>1$ and $b>1/2$, we can prove that $h_n\sim a(\log n)^{\alpha-1}/b$ (with the usual notation { $t_n\sim s_n$ if and only if $t_n/s_n\to 1$}) and $n^{1-b-\kappa_b}\varphi(n^b)=n^{1-2b}$ so we obtain $\lim h_n^{-1}\log(n^{1-b-\kappa_b}\varphi(n^b))$ $=L=-\infty$. % \red{which tells that ???} \red{\underline{r√É¬É√Ç¬©ponse A}: cela implique que le cas $L=-\infty$ existe, en voici un exemple.}. \\
\noindent However, when $\alpha=1$ and $a>1/\delta_1$, we can prove that for all $b\in[1/2,1)$, $\kappa_b=b$ and $h_n=\log n$ thus giving $L=1-2b$ and $\xi=0$. In other words, $L\in(-\infty,-\xi]$ (with $L=-\xi$ if and only if $b=1/2$). \\
\noindent Let us finally take the simple example  $g_n(t)=t\un_{\{t\geq n^b\}}$ and $f^{n,k}=1$. We can prove that for all $b\in(0,1)$, $h_n=\log n$, $\xi=0$ and $n^{1-b-\kappa_b}\varphi(n^b)=n$ so $\lim h_n^{-1}\log(n^{1-b-\kappa_b}\varphi(n^b))=L=1$ and we are in the case $L\in (-\xi,+\infty$).

\noindent To finish, we present an example for which $f^{n,k}$ is quite general but with a simple form. Assume
\begin{itemize}
    \item $f^{n,k}=\un_{\bm{A}^{n,k}}$ with $\bm{A}^{n,k}\subset\R^k$ and $\bm{A}^{n,k}_b:=\bm{A}^{n,k}\cap\{\mathbf{t}\in\R^k;\; \max_{1\leq j\leq k}H_j(\mathbf{t})\leq n,\; H_k(\mathbf{t})>n^b\}$;
    \item $(\bm{A}^{n,k}_b\times\R^{k'-k})\cap \bm{A}^{n,k'}_b=\varnothing$ for all $k<k'$;
    \item $\kappa_b=0$.
\end{itemize}
We obtain the following simple expression for $n^{\kappa_b}\sum_{k\geq 1}\Psi^k_{n,n^b}(f^{n,k})=\Pb(\cup_{k\geq 1}\{(S_1,\ldots,S_k)\in \bm{A}^{n,k}_b\})$, where $(S_i,i)$ is a sum of i.i.d random variables with mean $0$ and variance $\psi''(1)$ (this comes from the so-called \mto, see Lemma \ref{manytoone}). So
$$ \big|\log(n^{\kappa_b}\sum_{k\geq 1}\Psi^k_{n,n^b}(f^{n,k}))\big|=-\log\Pb(\cup_{k\geq 1}\{(S_1,\ldots,S_k)\in \bm{A}^{n,k}_b\}). $$\\
Consequently, if the probability $\Pb(\cup_{k\geq 1}\{(S_1,\ldots,S_k)\in \bm{A}^{n,k}_b\})$ is small enough, that is to say such that $(\log n)^{\gamma}$ is negligible comparing to $-\log\Pb(\cup_{k\geq 1}\{(S_1,\ldots,S_k)\in \bm{A}^{n,k}_b\})$ for a  certain $\gamma\in(0,1)$, then the constraint is penalizing enough and $h_n=-\log\Pb(\cup_{k\geq 1}\{(S_1,\ldots,S_k)\in \bm{A}^{n,k}_b\})$. Otherwise, $h_n=\log n$. For example, take
$ \bm{A}^{n,k}=\{\mathbf{t}\in\R^k;\; \inf\{j\leq k;\; t_j\geq(\log n)^{\alpha}\}=k\}$ which leads to an example similar to Theorem \ref{thm1}.

%we discuss the appearance of $\sum_{k\geq 1}\Psi^{k}_{n,n^b}$ below but first let us focus on sequence $(h_n,n)$.

\subsection{A general result (full statement) \label{sec1.2}} 

In this section, we explain precisely what\textrm{ \guillemotleft$\Psi$ has controlled fluctuations\guillemotright } means. For that, we present the assumptions \eqref{Hypothese2},  \eqref{Hypothese3} and \eqref{Hypothese4} mentioned in the previous section. We start with \eqref{Hypothese2}, and then state a preliminary result (Proposition \ref{Prop1}) of the main theorem (Theorem \ref{Prop2bb}). This proposition is quite technical especially in its statement. However, it stresses on the fact that all the expressions  involved depend deeply on $\sum_{k} \Psi^k_{.,.}(f^{n,k}) $ and therefore justify the last two Assumptions  \eqref{Hypothese3} and \eqref{Hypothese4} which leads to the formal statement of Theorem \ref{Prop2bb}. \\
%These assumptions that we describe below mix in particular $\mathbf{f}^n$ and distribution of $V$. \\ 

%\noindent\DO{J'aimerais ajouter un fact qui regroupe des propri\'et\'es sur $h_n$ qui reviennent assez souvent, dont certaines sont des cons\'equences de l'assumption \eqref{Hypothese1} et/ou de la d\'efinition de $h_n$ \eqref{Def_hn}: \\
%\textbf{Facts sur $h_n$}:
%\begin{enumerate}[label=\roman*)]
%    \item il existe $\gamma\in(0,1)$ tel que $h_n\geq(\log n)^{\gamma}$ pour $n$ assez grand,
%    \item d'apr\`es l'assumption \eqref{Hypothese1}, soit $h_n=\log n$ pour tout $n\geq 2$ soit $h_n=o(\log n)$,
%    \item toujours d'apr\`es l'assumption \eqref{Hypothese1}, $\log(n^{\kappa_b}\sum_{k\geq 1}\Psi^k _{n,n^b}(f^{n,k}))\to -\infty$.
%\end{enumerate}}

\noindent\textit{Assumption 2}. \\
Assumption \eqref{Hypothese2} below is an upper bound for a conditional version of $\sum_{k\geq 1}\Psi^{k}_{n,n^b}(f^{n,k})$ actually requiring in order to be introduced two facts and additional notations. %So we start with two elementary results already present and proved in \cite{AndDiel3} \\
\\
{\bf Fact 2} : 
By Lemma 2.3 in \cite{AndDiel3}, there exists two real numbers $c_2,\tilde{c}_2>0$ such that for any $h>0$ 
\begin{align}\label{GenPotMax}
    \Pb^*\big(\max_{|w|\leq\lceil h/c_2\rceil}|V(w)|>h\big)\leq he^{-\tilde{c}_2h}.
\end{align}
This fact, that will be useful when cutting on early generations of the tree, justifies the introduction of the following notation : for any $n$ and $k$, $f^{n,k}_{h}$ is the function defined by 
%\noindent Let $\left\{F^{n,k};\;k,n\geq 1\right\}$ a family of uniformly bounded functions ($\sup_{k,n\in \mathbb{N}^*}\|F^{n,k}\|_{\infty}<\infty$),   we assume that $f^{n,k}$ can be written as 
%\begin{align}
%& f^{n,k}(t_1,\cdots,t_k)=F^{n,k}(t_1,\cdots,t_k)\un_{t_k \geq v_n},  \textrm{with }  (v_n,n),   \textrm{ a positive sequence.}
%\end{align}
%We will explain this condition later and precise the sequence $v_n$ \red{on va voir comment introduire \c{c}a car cela restreint $f$, d'autant que cela n'apparait pas dans le Th√É¬É√Ç¬É√É¬Ç√Ç¬É√É¬É√Ç¬Ç√É¬Ç√Ç¬É√É¬É√Ç¬É√É¬Ç√Ç¬Ç√É¬É√Ç¬Ç√É¬Ç√Ç¬Ç√É¬É√Ç¬É√É¬Ç√Ç¬É√É¬É√Ç¬Ç√É¬Ç√Ç¬Ç√É¬É√Ç¬É√É¬Ç√Ç¬Ç√É¬É√Ç¬Ç√É¬Ç√Ç¬éor√É¬É√Ç¬É√É¬Ç√Ç¬É√É¬É√Ç¬Ç√É¬Ç√Ç¬É√É¬É√Ç¬É√É¬Ç√Ç¬Ç√É¬É√Ç¬Ç√É¬Ç√Ç¬Ç√É¬É√Ç¬É√É¬Ç√Ç¬É√É¬É√Ç¬Ç√É¬Ç√Ç¬Ç√É¬É√Ç¬É√É¬Ç√Ç¬Ç√É¬É√Ç¬Ç√É¬Ç√Ç¬ème 1.2 ni le futur 1.3 ....}. A last notation 
\begin{align}
 f^{n,k}_{h}(t_1,\ldots, t_k) :=  \underset{\mathbf{s}\in [- h, h]^{m}}{\inf}f^{n,k+m}\left(s_1,\ldots,s_m,t_1+s_m,\ldots, t_k+s_m\right), \label{Def_f} 
\end{align}
with $m = \lceil h/c_2\rceil$  and $\mathbf{s} = (s_1,\ldots,s_m)\in\R^m$.  \\
The second fact is about the largest generation visited by the walk before the instant $n$ or before $n$ excursions to the vertex $e$. \\
{\bf Fact 3 }:  Let  $(\ell_n= (\log n)^3,n\geq 2)$, by Lemma 3.2 in \cite{AndDiel3}, there exists $A>0$ such that :
\begin{align*}
   \lim_{n \rightarrow +\infty} \P( \max_{ k\leq T^n } |X_k| \leq A \ell_n )=1.
\end{align*}
This fact is here essentially to justify the introduction of the sequence $(\ell_n,n)$ which appears in our second assumption and all along the paper. Note that a very precise result on the largest generation visited by the walk before the instant $n$ can be found in  \cite{HuShi10b} .

\noindent A last notation we need to introduce is a conditional and translated version of $\Psi^{k}_{\lambda,\lambda'}(F)$ for a given bounded function $F$. For all $k\in\N^*$, $l\in\N^*$,  $F:\R^{l+k}\longrightarrow \R$ bounded and $\mathbf{t}=(t_1,\ldots,t_l)\in\R^l$
\begin{align}
    \Psi^{k}_{\lambda,\lambda'}(F|\mathbf{t}):= \mathbf{E}\Big[\underset{|x|=k}{\sum}e^{-V(x)}F(t_1,\ldots,t_l,V(x_{1})+t_l,\ldots, V(x)+t_l)\mathds{1}_{\mathcal{O}_{\lambda,\lambda'}}(x)\Big],
    \label{PsiCondi}
\end{align}
where ${\sum}_{|x|=k}$ is the sum over all the individuals $x$ of generation $k$. Otherwise, if  $l=0$, then $\Psi^{k}_{\lambda,\lambda'}(F|\mathbf{t}):=\Psi^{k}_{\lambda,\lambda'}(F)$. \\
We are now ready to introduce the second assumption : for all $\delta,\varepsilon,A,B>0$ and $b\in[0,1)$, there exists $n_0\in\N^*$ such that for any $n\geq n_0$, $l\leq \lfloor A\ell_n\rfloor$ and any $\mathbf{t}=\left(t_1,\ldots,t_l\right)\in\R^{l}$ with $t_l\geq -B$ and $H_l(\mathbf{t})\leq n$
\begin{align}\label{Hypothese2}
    \sum_{k\geq 1}\Psi^{k}_{n,n^b-H_l(\mathbf{t})}\big(f^{n,l+k}_{\varepsilon h_n}|\mathbf{t}\big)\leq e^{\delta t_l+\frac{\varepsilon}{A}h_n}\sum_{k\geq 1}\Psi^k_{n,n^b}(f^{n,k}).
    \tag{\textbf{A2}}
\end{align}
Let us comment this inequality which plays two roles. A first one ensures that the fluctuations of $V$ in the early  generations of the 
tree have minor influence, this yields the presence of $e^{\frac{\varepsilon}{A}h_n}$. 

\noindent {The second point is technical and aims to show that $\Ee[\overline{\mathcal{R}}_{T^1}(g_n,\mathbf{f}^n)]\gtrsim n ^{-b}\varphi(n^b)\sum_{k\geq 1}\Psi^k_{n,n^b}(f^{n,k})$ in probability. For that, the second moment of $$Z_n:=\sum_{x\in\Hline{n}{n^b}}e^{-V(x)}f^{n,k}_{\varepsilon h_n}(\bm{V}_x)\un_{\{\overline{V}(x)\geq A\log n,\underline{V}(x)\geq-B,\overline{V}(x)=V(x)\}}$$
has to be controlled, with $\bm{V}_x:=(V(x_1),\ldots,V(x))$ and $\underline{V}(x):= \min_{v \leq x }V(v)$. We first observe that
\begin{align*}
    Z_n^2\approx\sum_{z\in\T}\sum_{x,y>z}\prod_{u\in\{x,y\}}\un_{\{u\in\Hline{n}{n^b}\}}e^{-V(u)}f^{n,|u|}_{\varepsilon h_n}(\bm{V}_u)\un_{\{V(u)\geq A\log n,\underline{V}(u)\geq-B,\overline{V}(u)=V(u)\}}.
\end{align*}
Then taking the expectation of $ Z_n^2$, $\sum_{k\geq 1}\Psi^{k}_{n,n^b-H_l(\mathbf{t})}(f^{n,l+k}_{\varepsilon h_n}|\mathbf{t})$ in \eqref{Hypothese2} actually appears as the conditional expectation of a well chosen function of the translated potential $(V_z(u):=V(u)-V(z))_{u>z}$. Indeed, note that $u\in\Hline{n}{n^b}$ together with $\overline{V}(u)=V(u)$ implies that $u\in\Hline{n}{n^b-H_z}^z:=\{u>z: \max_{z<v\leq u}H_{z,v}\leq n, H_{z,u}>n^b-H_z\}$ with $H_{z,v}:=\sum_{z<w\leq v}e^{V_z(w)-V_z(v)}$. Hence, for all $\delta\in(0,1/2)$, by independence of the increments of the branching random walk $(\T,V(u);u\in\T)$
\begin{align*}
    \Eb[Z_n^2]&\lesssim e^{(1-2\delta)B}\Eb\Big[\sum_{z\in\mathcal{O}_n}e^{-V(z)}\sum_{x,y>z}\prod_{u\in\{x,y\}}\un_{\{u\in\Hline{n}{n^b-H_z}^z\}}e^{-V_z(u)}F_{\bm{V}_z}^{n,|u|}(V_z(u_{|z|+1}),\ldots,V_z(u))\Big] \\ & \approx e^{(1-2\delta)B}\Eb\Big[\sum_{z\in\mathcal{O}_n}e^{-V(z)}\Big(e^{\delta V(z)}\sum_{k\geq 1}\Psi^{k}_{n,n^b-H_z}(f^{n,l+k}|\bm{V}_z)\Big)^2\Big],
\end{align*}
where, for $|z|=l$ and any $\mathbf{t}=(t_1,\ldots,t_l)\in\R^l$ 
\begin{align*}
    F_{\mathbf{t}}^{n,|u|}(V_z(u_{l+1}),\ldots,V_z(u)):=e^{\delta t_l}f^{n,|u|}_{\varepsilon h_n}(t_1,\ldots,t_l,V_z(u_{l+1})+t_l,\ldots,V_z(u)+t_l).
\end{align*}
Assumption \eqref{Hypothese2} finally allows to say that $\Eb[Z_n^2]\lesssim e^{\varepsilon h_n}(\sum_{k\geq 1}\Psi^k_{n,n^b}(f^{n,k}))^2$ for all $\varepsilon>0$ and $n$ large enough.}

\noindent We are now almost ready to state an intermediate result which is a proposition giving a lower and an upper bound for the generalized range stopped at $T^n$. This proposition is followed by the theorem, much easier to read, but requiring extra assumptions. First, let us introduce for any $\mathfrak{z}>0$
\begin{align}  \label{HighSet}
    \mathcal{H}^k_{\mathfrak{z}}:= \left\{\left(t_1,\ldots,t_k\right)\in \mathbb{R}^k;\; t_k\geq\mathfrak{z}\right\},\  \mathcal{H}^k_{B,\mathfrak{z}}:= \{\left(t_1,\ldots,t_k\right)\in \mathbb{R}^k;\; t_k\geq\mathfrak{z}, \min_{1\leq i\leq k} t_i \geq -B \}, 
\end{align} 
respectively the set of vectors such that its last coordinate is larger than $\mathfrak{z}$ and additionally with all coordinates larger than $-B$.
The introduction of these last two objects is justified by  \\
{\bf Fact 4 :} for any $\varepsilon>0$, there exists $a>0$ such that (see\ \cite{Aid13})  
\begin{align*}
\Pb( \min_{u \in \T} V(u) \geq  -a)\geq 1- \varepsilon,\ 
\end{align*}
and Fact 1 we have already talked about saying that, in $\P^*$-probability, $\ell_n^{1/3}$ is a height of potential usually reached by the walk.
%\begin{align}
% \mathcal{R}_{T^n}(\un_{\cdot \geq 1},\un_{\mathcal{H}_{B,v_n}})=o(\mathcal{R}_{T^n}) 
%\end{align}
%\bl{[pas facile de voir si bien coherent avec def range generalised]} in words $v_n$ is the height of potential usually reach by the walk (see \cite{AndChen}) and moreover (see \cite{AndDiel3}) having the nice property to split excursions (to the root) of the walk  into independent excursions.
%\noindent \Amodif{which is a truncated version of the reflecting barrier $\mathcal{L}_{\lambda}$
%\begin{align*}
%    \mathcal{L}_{\lambda} := \left\{x\in \mathbb{T};\; \underset{j<|x|}{\max}\; H_{x_j}\leq \lambda,\; H_x>\lambda\right\}.

%\noindent \red{A definir  $\lambda_n$, $\Phi$, $h_n$ $\Psi^{k}_{\lambda_n} \left(F_n^{k}\right)$ (je ne comprends pas bien cette notation)}

%\bl{mettre la def de $W$ dans la Proposition}
%\noindent Let $W$ be the random variable defined by 
%\begin{align}\label{Def_W}
%    W:=\sum_{|z|=1}e^{-V(z)}.
%\end{align}

\begin{proposition}\label{Prop1}
Recall \eqref{kappab}, let $\varepsilon_b:=\min(b+\un_{\{b=0\}},1-b)/13$ and $W:=\sum_{|z|=1}e^{-V(z)}$. Assume \eqref{hyp0}, \eqref{hyp0+} and \eqref{hyp1} hold as well as \eqref{Hypothese1} and \eqref{Hypothese2}. \\
\textit{Lower bound}: there exists $c_5>0$ such that for all $b\in[0,1)$, $\varepsilon\in(0,\varepsilon_b)$, $B>0$ and $n$ large enough 
\begin{align}\label{LowerProp1}
    \P\Big(\frac{\mathcal{R}_{T^n}(g_n,\mathbf{f}^n)}{n^{1-b}\varphi(n^b) u_{1,n}}<e^{-5\varepsilon h_n}\Big)\leq \frac{e^{-c_5\varepsilon h_n}}{(u_{1,n})^2}\big(\sum_{k\geq 1}\Psi^k_{n,n^b}(f^{n,k}_{\varepsilon h_n})\big)^2 + h_ne^{-\varepsilon \tilde{c}_2h_n}+\frac{e^{-\min(\varepsilon\log n,3h_n)}}{(n^{\kappa_b}u_{1,n})^2},
\end{align}
with
\begin{align*}
    u_{1,n}=u_{1,n}(\varepsilon):=\sum_{k\geq 1}\Psi^{k}_{\lambda_n/2,n^b}\big(f^{n,k}_{\varepsilon h_n}\mathds{1}_{\Upsilon^k_n}\big),
\end{align*}
 $\Upsilon^k_n:=\{\mathbf{t}\in\R^k;\;H_k(\mathbf{t})\leq n^be^{\varepsilon h_n}\}\cap\mathcal{H}^k_{B,2\ell_n^{1/3}/\delta_1}$ and $\lambda_n=ne^{-\min(10\varepsilon\log n,5h_n)}$. \\
\noindent Upper bound: for any $\varepsilon>0$ and $n$ large enough
\begin{align}\label{UpperProp1}
    \P\Big(\frac{\mathcal{R}_{T^n}(g_n,\mathbf{f}^n)}{n^{1-b}\varphi(n^b)u_{2,n}}>e^{\varepsilon h_n}\Big)\leq e^{-\frac{\varepsilon}{2}h_n}+o(1)
\end{align}
with
\begin{align*}
    u_{2,n}:=\sum_{k\geq 1}\big(\Psi^{k}_{n} \big(f^{n,k}\un_{\R^k\setminus \mathcal{H}^k_{\ell_n^{1/3}/\delta_1}}\big)+\Psi^{k}_{n,n^b/(\log n)^2} (f^{n,k})+\Eb\big[W\Psi^k_{n,n^b/(W(\log n)^2)}(f^{n,k})\big]\big).
\end{align*}
Note that \eqref{LowerProp1} and \eqref{UpperProp1} remain true replacing $\mathcal{R}_{T^n}(g_n,\mathbf{f}^n)$ by $\mathcal{R}_{T^{k_n}}(g_n,\mathbf{f}^n)$ with {$ k_n = \lfloor n/ (\log n)^p \rfloor $, $p>0$}.
\end{proposition}

\noindent This proposition is technical and difficult to read, we present it here however because it shows that all the estimations depend deeply on $\Psi_{.,.}^.(f)$ and $g_n$, recall indeed that the key sequence $(h_n,n)$ defined in \eqref{Def_hn} depends both on $\Psi_{.,.}^.(f)$ and $\kappa_b$ (with $b$ coming from the function $g_n$).  
This also means that without any more information on $\Psi_{.,.}^.(f)$, it is difficult to state a more explicit result. Finally, note that  {the exact role of  \eqref{Hypothese1} and \eqref{Hypothese2} will appear clearly in the proof of the lower bound} (Section \ref{sec3.2}). \\

We now present two new assumptions \eqref{Hypothese3}  and \eqref{Hypothese4} which lead to the formal statement of the result. These assumptions  tell essentially that quantities $u_{1,n}$ and $u_{2,n}$, which appear in the previous proposition, are actually very similar.   \noindent Now introduce \eqref{Hypothese3}  and \eqref{Hypothese4} : \\

\noindent\textit{Assumption 3} %\DO{(lower bound bis?)}}\\
: for all ${b\in[0,1)}$, $\varepsilon\in(0,\varepsilon_b)$, {$\varepsilon_1 \in(0, \varepsilon) $} and $n$ large enough
\begin{align}\label{Hypothese3}
    u_{1,n}\geq e^{-\varepsilon_1h_n}\sum_{k\geq 1}\Psi^k_{n,n^b}(f^{n,k}).
    \tag{\textbf{A3}}
\end{align}

\noindent\textit{Assumption 4} %\DO{(upper bound bis?)}}\\
 : for all ${\varepsilon_1>0}$, ${b\in[0,1)}$ and $n$ large enough 
\begin{align}\label{Hypothese4}
    u_{2,n}\leq e^{\varepsilon_1h_n}\sum_{k\geq 1}\Psi^k_{n,n^b}(f^{n,k}).
    \tag{\textbf{A4}}
\end{align}

\noindent The full statement of Theorem \ref{Prop2} then writes as follows:

\begin{Customtheo}{1.5}[Full statement]\label{Prop2} Assume \eqref{hyp0}, \eqref{hyp0+} and \eqref{hyp1} hold, $b \in [0,1)$ and  \eqref{Hypothese1}, \eqref{Hypothese2}, \eqref{Hypothese3} and \eqref{Hypothese4} are satisfied.
%\noindent Assume \eqref{Hypothese2}, \eqref{Hypothese1}, \eqref{Hypothese3} and \eqref{Hypothese4} hold \DO{+ compl\'eter}. \\
If $L \in (-\xi,+\infty]$, then in $\P^*$-probability \begin{align*}%\label{ConvProba1}
      h_n^{-1}\big(\log^{+} \mathcal{R}_{n}(g_{n},\mathbf{f}^n)-\log( n^{1-b-\kappa_b}\varphi(n^b))\big)\underset{n\rightarrow\infty}{\longrightarrow}\xi,
\end{align*}
if $L=-\xi$, with $\underline{\Delta}_n:=h_n^{-1}\log(n^{1-b-\kappa_b}\varphi(n^b))-\inf_{\ell\geq n}h_{\ell}^{-1}\log(\ell^{1-b-\kappa_b}\varphi(\ell^b))$, then in $\P^*$-probability
\begin{align*}
    h_n^{-1}\log^{+} \mathcal{R}_{n}(g_{n},\mathbf{f}^n)-\underline{\Delta}_n\underset{n\rightarrow\infty}{\longrightarrow} 0,
\end{align*}
otherwise $L\in[-\infty,-\xi[$ and in $\P^*$-probability 
\begin{align*}%\label{ConvProba3}
       \mathcal{R}_{n_{\ell}}(g_{n_{\ell}},\mathbf{f}^{n_{\ell}})\underset{\ell\rightarrow\infty}{\longrightarrow} 0,
\end{align*}
for some increasing sequence $(n_{\ell})_{\ell}$ of positive integers. Note that when $\lim h_n^{-1}\log(n^{1-b-\kappa_b}\varphi(n^b))=L$, $n_{\ell}=\ell$.
\end{Customtheo}

%The appearance of the set regular lines $\mathcal{O}_{\lambda,\lambda'}$ is partly inspired from the works of \cite{HuShiFar} and  \cite{HuSHi15}, we recall these facts in following Remark  

%$\underset{k\geq 1}{\sum}\Psi^k_{\lambda_n} \left(F^k_n\right)$

%\subsection{Main steps of the proof}

\noindent {The rest of the paper is decomposed as follows: in Section \ref{sec2}, after short preliminaries (Section \ref{sec2.1}), we prove the theorems of Section \ref{sec1}. For these proofs (Section \ref{sec2.2}), we check that the four assumptions (\ref{Hypothese1}-\ref{Hypothese4}) of Theorem  \ref{Prop2} are realized, obtaining simultaneously the asymptotic of $h_n$. In section \ref{sec2.3}, we prove Theorem \ref{Prop2} : essentially, Proposition \ref{Prop1} is assumed to be true and we only check that if Assumptions \eqref{Hypothese3} and \eqref{Hypothese4} are true then the theorem comes. \\ 
We prove Proposition \ref{Prop1} in section \ref{sec3}, this is the most technical part of the paper which can be read independently of the other parts : in Section \ref{sec3.1}, we summarize usual facts, in a second sub-section we prove a lower bound for stopped generalized range $\mathcal{R}_{T^n}(g_n,\mathbf{f}^n)$ and finally in a last one an upper bound. \\
In section \ref{sec4bis} we present some estimates on sums of i.i.d. random variables useful for the proof of the examples of Section \ref{sec1}. Finally, we resume in a last section (page \pageref{notations}) the notations which are transversal along the paper.  }

%{\bf Ici on discute de la proposition g\'en\'erale et des hypoth\`eses sur $f$ et $g$. } \\

%\noindent We are now ready to state the main assumptions on the familly $\left\{f^k_n;\;k\geq 1\right\}$:

\vspace{0.2cm}

\section{Proof of the theorems} \label{sec2}

This section is decomposed in three parts: in the first section below, one can find preliminaries that are useful all along the rest of the paper. In the second sub-section, we prove the four theorems presented as examples. Finally, the last section is devoted to the proof of Theorem \ref{Prop2}.

\subsection{Preliminary material} \label{sec2.1}

%\red{[Section \'a reprendre avec la re-lecture des sections 3 et 4]} 

\noindent We recall the many-to-one formula (see \cite{Shi2015} Chapter 1, and \cite{HuShi10b} equation 2.1) which will be used several times in the paper to compute  expectations related to the environment. {Note that the identity below comes from a change of probability measure (see references above), however we still keep $\Pb$ and $\Eb$ for simplicity.}
\begin{lemm}[Many-to-one Lemma]\label{manytoone} Recall the definition of $\psi$ in \eqref{fcpsi}. For any $t>0$,
	\begin{align*}
\Eb\Big[\sum_{|x|=m}f(V(x_i),1 \leq i \leq m)\Big]= \Eb(e^{tS_m+\psi(t) m }f(S_i,1 \leq i \leq m)),  
\end{align*}
where $\seq{S}$ is the random walk starting at 0, such that the increments $\paren{S_{n+1}-S_n}_{n\in\N}$ are i.i.d. and for any measurable function $h:\R^m\to [0,\infty)$, 
$$
\Eb[h(S_1)]=e^{-\psi(t)}\Eb(\sum_{|x|=1}e^{-tV(x)}h(V(x))).
%\frac{\e{\sum_{|x|=1}e^{-tV(x)}h(V(x))}}{\e{\sum_{|x|=1}e^{-tV(x)}}}=e^{-\psi(t)}\e{\sum_{|x|=1}e^{-tV(x)}h(V(x))}\ .
$$
\end{lemm}

{A second very useful fact is contained in the following remark, it tells essentially that, in probability, the $e^{-V(x)}$-weighted  number of vertices $x$ such that $x \in \mathcal{O}_n$ (recall \eqref{RegularLine}) can be found in a quite small quantity when $|x| \leq A\ell_n$ and can not be found when $|x| > A\ell_n$. This remark is not precise at all but will be enough for our purpose.

%\bl{Ci dessous, remarque a completer et mettre au bon endroit : v\'erifier qu'elle n'est pas deja cite plus haut ...}
\begin{remark}  \label{rem1b}
%In words $\ell_n$ is (typically) the largest generation visited by the walk before instant $n$.
There exists $c_3\in(0,1)$ such that for any $A>0$ and $n$ large enough 
\begin{align*}
   \Eb\Big[\sum_{|x|>\lfloor A\ell_n\rfloor}e^{-V(x)}\un_{\{x\in\mathcal{O}_n\}}\Big]\leq {n^{-Ac_3}} \textrm{ and }   \Eb\Big[\sum_{|x| \leq \lfloor A\ell_n\rfloor}e^{-V(x)}\un_{\{x\in\mathcal{O}_n\}}\Big]\leq  \ell_n/2,
\end{align*}
which implies   $\Eb\Big[\sum_{x \in \T}e^{-V(x)}\un_{\{x\in\mathcal{O}_n\}}\Big]\leq  \ell_n$.
\label{Rem1}
\end{remark}

\begin{proof} 
We give a proof here which essentially use technical Lemma \ref{LemmBorneUnif} (for the second inequality below), indeed by Lemma  \ref{manytoone} above  \begin{align*}
\Eb\Big[\sum_{|x|>\lfloor A\ell_n\rfloor}e^{-V(x)}\un_{\{x\in\mathcal{O}_n\}}\Big] & \leq \sum_{k > \lfloor A\ell_n\rfloor} \P( \sup_{i \leq k}(\overline S_i-S_i) \leq \log n ) \\ 
& \leq \sum_{k > \lfloor A\ell_n\rfloor}  \exp({-\frac{k \pi^2 \sigma^2(1- \varepsilon)}{8 \log n}}) \leq n^{-Ac_3}.
\end{align*}
 A similar computation gives the second fact and both of them the last one.
\end{proof} }

%\DO{Petite remarque: $\mathcal{O}_n$ est inutile quand $|x| \leq \lfloor A\ell_n\rfloor$, faut-il quand m√É¬É√Ç¬™me garder? moi j'ai enlev√É¬É√Ç¬© dans certains calculs.} \\ 

%Comme son nom l'indique, on utilise la Proposition g/'e/'erale et les propi\'et\'se de l'environnement de la Section 5.

\subsection{Proofs of Theorems \ref{thm1} to \ref{thm4}} \label{sec2.2}

The pattern of the proofs of each theorem is the following : we first prove two facts (an upper and a lower bound) about the sum $ \sum_{k\geq 1}\Psi^{k}_{\cdot,\cdot}(F)$ with specific $F$, depending on the considered function $f^{n,k}$ and on a slightly different version of the latter whether we are looking for an upper or a lower bound. Then we use this two facts to prove that \eqref{Hypothese1}, \eqref{Hypothese2}, \eqref{Hypothese3} and \eqref{Hypothese4} are satisfied. \\ 
In these proofs, we use several times the notation  $\varepsilon_b=\min(b+\un_{\{b=0\}},1-b)/13$ which was introduced in Proposition \ref{Prop1}.

\begin{proof}[Proof of Theorem \ref{thm1}]
Recall that $f^{n,k}(t_1,t_2, \cdots,t_k)= \un_{\{t_k\geq(\log n)^{\alpha}\}}$, $\alpha\in(1,2)$ and see \eqref{RegularLine} for the definition of $\Hline{\lambda}{\lambda'}$. All along the proof, we assume that $B,\delta>0$, $\varepsilon\in(0,\varepsilon_b)$, $n$ is large enough and  $t\geq -B$. %The proof is devided into two steps. The first one is to prove some facts about $\mathbf{f}^n$. The second is to check that for all $n\geq 1$, $\mathbf{f}^n$ actually satisfies main assumptions \eqref{Hypothese1}, \eqref{Hypothese2} and \eqref{Hypothese3}.\\ 
Let us start with the proof of the following two facts: 
\begin{align}
    \Eb\Big[\sum_{x\in\mathcal{O}_n}e^{-V(x)}\un_{\{V(x)\geq (\log n)^{\alpha}-t\}}\Big]\leq e^{\delta t-(\log n)^{\alpha-1}(1-\varepsilon)},
    \label{Fact1TH1}
\end{align}
and for any $0\leq m\leq \log n$ 
\begin{align}
    \Eb\Big[\sum_{x\in\Hline{\bmlambda}{n^b}}e^{-V(x)}\un_{\{V(x)\geq (\log n)^{\alpha}+m,\; H_x\leq n^be^{\varepsilon(\log n)^{\alpha-1}},\;\underline{V}(x)\geq -B\}}\Big]\geq e^{-(\log n)^{\alpha-1}(1+\varepsilon)}, 
    \label{Fact2TH1}
\end{align} 
with $\bmlambda=ne^{-6(\log n)^{\alpha-1}}$ and recall $\underline{V}(x)=\min_{u \leq x} V(u)$. We first deal with the upper bound \eqref{Fact1TH1}. Recall $\ell_n=(\log n)^3$, 
\begin{align*}
    \Eb\Big[\sum_{x\in\mathcal{O}_n}e^{-V(x)}\un_{\{V(x)\geq (\log n)^{\alpha}-t\}}\Big]\leq & \sum_{k\leq \lfloor A\ell_n \rfloor}\Eb\Big[\sum_{|x|=k}e^{-V(x)}\un_{\{V(x)\geq(\log n)^{\alpha}-t\}}\un_{\{x\in\mathcal{O}_n\}}\Big] \\ & + \Eb\Big[\sum_{|x|>\lfloor A\ell_n \rfloor}e^{-V(x)}\un_{\{x\in\mathcal{O}_n\}}\Big], 
\end{align*}
where $A>0$ is chosen such that $\Eb[\sum_{|x|>\lfloor A\ell_n \rfloor}e^{-V(x)}\un_{\{x\in\mathcal{O}_n\}}]\leq 1/n$ (see Remark \ref{Rem1}). This yields, as $\alpha\in(1,2)$, $\Eb[\sum_{|x|>\lfloor A\ell_n \rfloor}e^{-V(x)}\un_{\{x\in\mathcal{O}_n\}}]\leq  e^{\delta (t+B)}\frac{1}{n}\leq \frac{1}{2}e^{\delta t-(\log n)^{\alpha-1}(1-\varepsilon)}$ for $n$ large enough and any $t\geq-B$. Thanks to \mto\ \ref{manytoone}, the first sum in the above inequality is smaller than 
\begin{align*}
    \sum_{k\leq\lfloor A\ell_n \rfloor}\Pb\big(S_k\geq(\log n)^{\alpha}-t,\;\max_{j\leq k}\;\overline{S}_j-S_j\leq\log n\big)\leq \lfloor A\ell_n \rfloor \Pb\big(\max_{j\leq \tau_{(\log n)^{\alpha}-t}}\;\overline{S}_j-S_j\leq \log n\big),
\end{align*}
with {$\tau_r=\inf\{i\geq 1; S_i\geq r\}$}. Then, thanks to Lemma A.3 in \cite{HuShi15b}, and  %\bl{[mettre dans les fact de la fin ?]\bo{√É¬É√Ç¬É√É¬Ç√Ç¬É√É¬É√Ç¬Ç√É¬Ç√Ç¬ßa vaut le coup? on ne l'utilise que deux fois.}} 
as  $t\geq-B$
\begin{align*}
    \lfloor A\ell_n\rfloor\Pb\big(\max_{j\leq \tau_{(\log n)^{\alpha}-t}}\;\overline{S}_j-S_j\leq\log n\big) \leq\lfloor A\ell_n\rfloor e^{\frac{t}{\log n}-(\log n)^{\alpha-1}(1-\frac{\varepsilon}{2})}&\leq \lfloor A\ell_n\rfloor e^{\frac{t+B}{\log n}-(\log n)^{\alpha-1}(1-\frac{\varepsilon}{2})} \\ & \leq e^{\delta t+\delta B-(\log n)^{\alpha-1}(1-\frac{\varepsilon}{2})} \\ & \leq \frac{1}{2}e^{\delta t-(\log n)^{\alpha-1}(1-\varepsilon)},
\end{align*}
so we get exactly \eqref{Fact1TH1}. \\
We now turn to the lower bound \eqref{Fact2TH1}. Let $\ell'_n=(\log n)^4$ and $\alpha_n=(\log n)^{\alpha}+\log n$. By the \mto, for any $m\leq\log n$, the expectation in \eqref{Fact2TH1} is larger than 
\begin{align*}
    \sum_{k\geq 1}\Pb\big(S_k\geq\alpha_n,\;\max_{j\leq k}\;H_j^S\leq \bmlambda, n^b<H_k^{S}\leq n^be^{\varepsilon(\log n)^{\alpha-1}} ,\underline{S}_k\geq-B\big),
\end{align*}
with $H_j^S:=\sum_{i=1}^j e^{S_i-S_j}$ \label{HS}. For any $b\in(0,1)$, by Lemma \ref{Minor_HRHP1} \eqref{lem4.1b} (with $\ell=(\log n)^2$, $t_{\ell}=\alpha_n$, $q=1$, $a_b=a=6$, $d=(\alpha-1)/2$ and $c=\varepsilon$), above sum is larger than $e^{-(\log n)^{\alpha-1}(1+\varepsilon)}$. Otherwise, if $b=0$, observe that for all $k\leq\ell'_n$, $\overline{S}_k=S_k$ implies $H_k^{{S}}\leq k\leq\ell'_n$ so the sum is larger than $\sum_{k\leq\ell'_n}\Pb\big(S_k\geq\alpha_n,\;\max_{j\leq k}\;H_j^S\leq\bmlambda,\overline{S}_k=S_k, \underline{S}_k\geq-B \big)$. Lemma \ref{MinorRangeHP2} (with $\ell=(\log n)^2$, $t_{\ell}=\alpha_n$, $d=1/2$, $a=6$ and $d'=(\alpha-1)/2$) leads to \eqref{Fact2TH1} also for $b=0$. \\
We are now ready to prove that $\mathbf{f}^n$ satisfies assumptions  \eqref{Hypothese1}, \eqref{Hypothese2}, \eqref{Hypothese3} and \eqref{Hypothese4}. Recall that $\Psi^k_{n,n^b}(f^{n,k})=\Eb\big[\sum_{|x|=k}e^{-V(x)}f^{n,k}(V(x_1),\ldots,V(x))\un_{\{x\in\Hline{n}{n^b}\}}\big]$ where $x\in\mathcal{O}_{n,n^b}$ if and only if $\max_{j\leq|x|}H_{x_j}\leq n$ and $H_x>n^b$, also $\mathcal{U}_b =\{\kappa\in[0,1];\textrm{ for all } k\geq 1,\mathbf{t}\in\R^k,n\geq 1: \un_{\{H_k(\mathbf{t})>n^b\}} f^{n,k}(\mathbf{t})\leq C_{\infty} n^{-\kappa}\}$ with $C_{\infty}=\sup_{n,\ell}\|f^{n,\ell}\|_{\infty}$. \\
$\bullet$ {Check of \eqref{Hypothese1} and asymptotic of $h_n$}. We obtain from \eqref{Fact2TH1} with $m=0$ that for any $\varepsilon\in(0,\varepsilon_b)$ and $n$ large enough, $\Eb[\sum_{x\in\Hline{n}{n^b}}e^{-V(x)}\un_{\{V(x)\geq(\log n)^{\alpha}\}}]$ is larger than (as $\bmlambda\leq n$)
\begin{align*}
    \Eb\Big[\sum_{x\in\mathcal{O}_{\bmlambda,n^b}}e^{-V(x)}\un_{\{V(x)\geq(\log n)^{\alpha}\}}\un_{\{H_x\leq n^be^{\varepsilon(\log n)^{\alpha-1}},\underline{V}(x)\geq-B\}}\Big]\geq e^{-(\log n)^{\alpha-1}(1+\varepsilon)}.
\end{align*}
Note that above inequality implies that for all $b\in[0,1)$, $\kappa_b=\max\mathcal{U}_b=0$. Indeed, if we had $\kappa_b>0$, then this should imply that for any $x \in \T$
\begin{align*}
e^{-V(x)}\un_{\{x\in\mathcal{O}_{n,n^b}\}} f^{n,k}(V(x_1),\ldots,V(x))\leq C_{\infty} n^{- \kappa_b}e^{-V(x)}  \un_{\{x\in\mathcal{O}_{n,n^b}\}}, 
\end{align*}
which gives that  $\Eb[\sum_{x\in\Hline{n}{n^b}}e^{-V(x)} \un_{\{V(x)\geq(\log n)^{\alpha}\}} \un_{\{H_x\leq n^be^{\varepsilon(\log n)^{\alpha-1}},\underline{V}(x)\geq-B\}}]$ is smaller than $C_{\infty}n^{-\kappa_b}\Eb[\sum_{x\in\mathcal{O}_{n}}e^{-V(x)}]\leq C_{\infty}\ell_n n^{-\kappa_b}$ by Remark \ref{Rem1}, but this contradicts the above lower bound \eqref{Fact2TH1} as $\alpha\in(1,2)$. \\
Then, by definition of $\Psi^{k}_{n,n^b}$, 
\begin{align*}
    \sum_{k\geq 1}\Psi^{k}_{n,n^b}(f^{n,k}) =\Eb\Big[\sum_{x\in\mathcal{O}_{n,n^b}}e^{-V(x)}\un_{\{V(x)\geq (\log n)^{\alpha}\}}\Big]\geq e^{-(\log n)^{\alpha-1}(1+\varepsilon)}, 
\end{align*}
and additionally  with \eqref{Fact1TH1} (taking $t=0$), asymptotic of $h_n$ is given by
\begin{align*}
    h_n = \Big|n^{\kappa_b}\log \big(\sum_{k\geq 1}\Psi^k_{n,n^b}(f^{n,k})\big)\Big|=\Big|\log\Eb\Big[\sum_{x\in\mathcal{O}_{n,n^b}}e^{-V(x)}\un_{\{V(x)\geq (\log n)^{\alpha}\}}\Big]\Big|\sim (\log n)^{\alpha-1}. 
\end{align*}
We also deduce from the previous lower bound that \eqref{Hypothese1} is satisfied, indeed, as $\alpha\in(1,2)$, $\sum_{k\geq 1}\Psi^{k}_{n,n^b}(f^{n,k})\geq n^{-(\kappa_b+\varepsilon_1)\land 1}$ for any $\varepsilon_1>0$ and $n$ large enough. \\
$\bullet$ For \eqref{Hypothese2}, recalling $m_n=\lceil \varepsilon h_n/c_2\rceil$ (see \eqref{GenPotMax}), then by definition 
\begin{align*}
    f^{n,j}_{\varepsilon h_n}(t_1,\ldots,t_j)&=\inf_{\mathbf{s}\in[-\varepsilon h_n,\varepsilon h_n]^{m_n}}f^{n,m_n+j}(s_1,\ldots,s_{m_n},t_1+s_{m_n},\ldots,t_j+s_{m_n}) \\ & = \inf_{s_{m_n}\in[-\varepsilon h_n,\varepsilon h_n]}\un_{\{t_j+s_m \geq (\log n)^{\alpha}\}}=\un_{\{t_j  \geq (\log n)^{\alpha}+\varepsilon h_n \}}.
\end{align*}
Observe that for $A>0$, $n$ large enough, any $l\in\N$ and $\mathbf{t}=(t_1,\ldots,t_l)$, by definition of  $\Psi^{k}_{n}(F|\mathbf{t})$ (see \eqref{PsiCondi}) and \eqref{Fact1TH1} with ${\varepsilon}/{3A}$ instead of $\varepsilon$
%=\Eb[\sum_{|x|=k}e^{-V(x)}F(t_1,\ldots,t_l,V(x_1)+t_l,\ldots,V(x)+t_l)\un_{\{n^b<H_x\lor H_l(\mathbf{t})\leq n\}}\un_{x\in\mathcal{O}_n}]$. Then, for all $\varepsilon,A,B>0$ and $n$ large enough, any $l\in\N^*$ and $\mathbf{t}=(t_1,\ldots,t_l)\in\R^l$ with $t_l\geq -B$
\begin{align*}
    \sum_{k\geq 1}\Psi^{k}_{n,n^b-H_l(\mathbf{t})}\big(f^{n,l+k}_{\varepsilon h_n}|\mathbf{t}\big)&=\Eb\Big[\sum_{x\in\mathcal{O}_n}e^{-V(x)}\un_{\{V(x)+t_l\geq(\log n)^{\alpha}+\varepsilon h_n\}}\un_{\{H_x>n^b-H_l(\mathbf{t})\}}\Big] \\ & \leq\Eb\Big[\sum_{x\in\mathcal{O}_n}e^{-V(x)}\un_{\{V(x)\geq (\log n)^{\alpha}-t_l\}}\Big]\leq e^{\delta t_l-(\log n)^{\alpha-1}(1-\frac{\varepsilon}{3A})}.
\end{align*}
Moreover, $e^{-(\log n)^{\alpha-1}(1-\frac{\varepsilon}{3A})}=e^{\frac{2\varepsilon}{3A}(\log n)^{\alpha-1}}e^{-(\log n)^{\alpha-1}(1+\frac{\varepsilon}{3A})}\leq e^{\frac{\varepsilon}{A}h_n}\sum_{k\geq 1}\Psi^k_{n,n^b}(f^{n,k})$, the last inequality coming from the fact that $h_n\sim (\log n)^{\alpha-1}$ and \eqref{Fact2TH1} with $m=0$ and as above $\frac{\varepsilon}{3A}$ instead of $\varepsilon$. So \eqref{Hypothese2} is satisfied. \\

\noindent We are left to prove that technical assumptions \eqref{Hypothese3} and \eqref{Hypothese4} are realized. \\
$\bullet$ For \eqref{Hypothese3}, recall first, from Proposition \ref{Prop1}, that for all $b\in[0,1)$, $\Upsilon^k_n$ is the set
\begin{align*}
    \{\mathbf{t}=(t_1,\ldots,t_k)\in\R^k; H_k(\mathbf{t})\leq n^be^{\varepsilon h_n}, t_k\geq 2\ell_n^{1/3}/\delta_1, \min_{j\leq k}t_j\geq-B\},
\end{align*}
with $\lambda_n=ne^{\min(10 \varepsilon \log n,-5h_n)} =ne^{-5h_n}$ for large $n$. Let $0<\varepsilon_1<\varepsilon$, note that $\lambda_n/2\geq\bmlambda=ne^{-6(\log n)^{\alpha-1}}$ so for $n$ large enough 
\begin{align*}
 u_{1,n}=   \sum_{k\geq 1}\Psi^k_{\lambda_n/2,n^b}(f^{n,k}_{\varepsilon h_n}\un_{\Upsilon^k_n})&=\Eb\Big[\sum_{x\in\Hline{\lambda_n/2}{n^b}}e^{-V(x)}\un_{\{V(x)\geq(\log n)^{\alpha}+\varepsilon h_n,H_x\leq n^be^{\varepsilon h_n},\underline{V}(x)\geq-B\}}\Big] \\ & \geq\Eb\Big[\sum_{x\in\Hline{\bmlambda}{n^b}}e^{-V(x)}\un_{\{V(x)\geq(\log n)^{\alpha}+h_n,H_x\leq n^be^{\frac{\varepsilon_1}{3}(\log n)^{\alpha-1}},\underline{V}(x)\geq-B\}}\Big] \\ & \geq e^{-(\log n)^{\alpha-1}(1+\frac{\varepsilon_1}{3})},
\end{align*}
where we use that $(\log n)^{\alpha}>2\ell_n^{1/3}/\delta_1$ for the second equality and the last inequality comes from \eqref{Fact2TH1},  with $m=h_n$ and $\varepsilon_1/3$ instead of $\varepsilon$. \\ 
 Moreover, $e^{-(\log n)^{\alpha-1}(1+\frac{\varepsilon_1}{3})}=e^{-\frac{2\varepsilon_1}{3}(\log n)^{\alpha-1}}e^{-(\log n)^{\alpha-1}(1-\frac{\varepsilon_1}{3})}\geq e^{-\varepsilon_1 h_n}\sum_{k\geq 1}\Psi^k_{n,n^b}(f^{n,k})$ which comes from the  fact that $h_n\sim (\log n)^{\alpha-1}$ and \eqref{Fact1TH1} with $t=0$, $\frac{\varepsilon_1}{3}$ instead of $\varepsilon$. \\ 
$\bullet$ Finally for \eqref{Hypothese4}, recall the definition of $u_{2,n}$ just below \eqref{UpperProp1}. First observe that as $\alpha\in(1,2)$, for $n$ large enough, $(\log n)^{\alpha}>\ell_n^{1/3}/\delta_1$ so  for  any $k$
\begin{align*}
    \Psi^k_n(f^{n,k}\un_{\R\setminus\mathcal{H}^k_{\ell_n^{1/3}/\delta_1}})=\Eb\Big[\sum_{|x|=k}e^{-V(x)}\un_{\{V(x)\geq(\log n)^{\alpha},V(x)<\ell_n^{1/3}/\delta_1\}}\un_{\{x\in\mathcal{O}_n\}}\Big]=0.
\end{align*}
Recall that $\Eb[W]=e^{\psi(1)}=1$  so 
\begin{align*}
    \sum_{k\geq 1}\big(\Psi^k_{n,n^b/(\log n)^2}(f^{n,k}) +\Eb\big[W\Psi^k_{n,n^b/(W(\log n)^2)}(f^{n,k})\big]\big)  & \leq \sum_{k\geq 1}\big(\Psi^k_{n}(f^{n,k}) +\Eb\big[W\Psi^k_{n}(f^{n,k})\big]\big) \\ & =2\sum_{k\geq 1}\Psi^k_n(f^{n,k}), 
\end{align*}
$2\sum_{k\geq 1}\Psi^k_n(f^{n,k})=2\Eb[\sum_{x\in\mathcal{O}_n}e^{-V(x)}\un_{\{V(x)\geq(\log n)^{\alpha}\}}]\leq 2e^{-(\log n)^{\alpha-1}(1-\frac{\varepsilon_1}{6})}\leq e^{-(\log n)^{\alpha-1}(1-\frac{\varepsilon_1}{3})}$ thanks to \eqref{Fact1TH1} with $t=0$, $\varepsilon_1/6$ instead of $\varepsilon$. \\ 
Moreover, $e^{-(\log n)^{\alpha-1}(1-\frac{\varepsilon_1}{3})}=e^{\frac{2\varepsilon_1}{3}(\log n)^{\alpha-1}}e^{-(\log n)^{\alpha-1}(1+\frac{\varepsilon_1}{3})}\leq e^{\varepsilon_1 h_n}\sum_{k\geq 1}\Psi^k_{n,n^b}(f^{n,k})$. The last inequality comes from the fact that $h_n\sim (\log n)^{\alpha-1}$ and \eqref{Fact2TH1} with $m=0$ and  $\frac{\varepsilon_1}{3}$ instead of $\varepsilon$.
\end{proof}

\begin{proof}[Proof of Theorem \ref{thm2}] Here $f^{n,k}(t_1,t_2, \cdots,t_k)= \un_{\{t_{\lfloor k/\beta \rfloor} \geq (\log n)^{\alpha}\}}$ with $\beta>1$ and $\alpha\in(1,2)$, let us start with the proof of the two following facts, for all $B,\delta>0$, $\varepsilon\in(0,\varepsilon_b)$, $n$ large enough, any $t\geq -B$ and $i\in\N$
%Recall that $f^{n,k}(t_1,t_2, \cdots,t_k)= \un_{\{t_{\lfloor k/\beta \rfloor} \geq (\log n)^{\alpha}\}}$, $\alpha\in(1,2)$ and $\beta>1$. The proof is devided into two steps. The first one is to prove some facts about $\mathbf{f}^n$. 
%The second is to check that for all $n\geq 1$, $\mathbf{f}^n$ actually satisfies main assumptions \eqref{Hypothese1}, \eqref{Hypothese2} and \eqref{Hypothese3}.\\ Let's start with the proof of the two following facts: for all $B,\varepsilon,\delta>0$ and $n$ large enough, for any $t\geq -B$ and $i\in\N$ 
\begin{align}
    \Eb\Big[\sum_{x,\lfloor (|x|+i)/\beta\rfloor>i}e^{-V(x)}\un_{\{V(x_{\lfloor (|x|+i)/\beta\rfloor-i})\geq (\log n)^{\alpha}-t\}}\un_{\{x\in\mathcal{O}_n\}}\Big]\leq e^{\delta t-c_{\beta}(\log n)^{\alpha-1}(1-\varepsilon)},
    \label{Fact1TH2}
\end{align}
and for any $m\leq \log n$ 
\begin{align}
    \Eb\Big[\sum_{x,\lfloor (|x|+i)/\beta\rfloor>i}e^{-V(x)}\un_{\{V(x_{\lfloor (|x|+i)/\beta\rfloor-i})\geq (\log n)^{\alpha}+m\}}\un_{\{x\in\Upsilon_n\cap\Hline{\bmlambda}{n^b}\}}\Big]\geq e^{-c_{\beta}(\log n)^{\alpha-1}(1+\varepsilon)},
    \label{Fact2TH2}
\end{align}
with $\bmlambda=ne^{-6c_{\beta}(\log n)^{\alpha-1}}$, for any $a>\frac{1}{\delta_1}$
\begin{align*}
   \Upsilon_n= \Upsilon_n(\varepsilon):=\{x\in\T;H_x\leq n^be^{\varepsilon c_{\beta}(\log n)^{\alpha-1}}, V(x)\geq a\log n,\underline{V}(x)\geq-B\},
\end{align*}
and $c_{\beta}=-1-\pi\sqrt{\beta-1}/2+\rho((\beta-1)\pi^2/4)$ (for $\rho$ see  \eqref{Def_rho}). Recall $\ell_n=(\log n)^3$ and introduce $L_n:=\lfloor (\log n)^{2+\varepsilon_{\alpha}}\rfloor$ with $\varepsilon_{\alpha}\in(0,\alpha-1)$. \\
\textit{Proof of \eqref{Fact1TH2}}  : first note that if $t>(\log n)^{\alpha}/2$, \eqref{Fact1TH2} is obviously satisfied, indeed 
\begin{align*}
    \Eb\Big[\sum_{x,\lfloor (|x|+i)/\beta\rfloor>i}e^{-V(x)}\un_{\{V(x_{(\lfloor (|x|+i)/\beta\rfloor-i)})\geq (\log n)^{\alpha}-t\}}\un_{\{x\in\mathcal{O}_n\}}\Big]\leq \Eb\Big[\sum_{x\in\mathcal{O}_n}e^{-V(x)}\Big],
\end{align*}
and by Remark \ref{Rem1}, $\Eb[\sum_{x\in\mathcal{O}_n}e^{-V(x)}]=\Eb[\sum_{x\in\mathcal{O}_n}e^{-V(x)}]e^{\delta t-\delta t}\leq \ell_n e^{\delta t-\frac{\delta}{2}(\log n)^{\alpha}}\leq e^{\delta t-c_{\beta}(\log n)^{\alpha-1}}$ for $n$ large enough. Now assume $t\leq (\log n)^{\alpha}/2$. The expectation in \eqref{Fact1TH2} is smaller than
\begin{align*}
    \sum_{k\leq\lfloor A\ell_n \rfloor}\sum_{p\geq 1}\un_{\{p=\lfloor\frac{k+i}{\beta}\rfloor-i\}}\Eb\Big[\sum_{|x|=k}e^{-V(x)}\un_{\{V(x_p)\geq(\log n)^{\alpha}-t\}}\un_{\{x\in\mathcal{O}_n\}}\Big] +\Eb\Big[\sum_{|x|>\lfloor A\ell_n \rfloor}e^{-V(x)}\un_{\{x\in\mathcal{O}_n\}}\Big], 
\end{align*}
with $A>0$ such that the last term is smaller than $1/n$ (Remark \ref{Rem1}). Note that $p=\lfloor\frac{k+i}{\beta}\rfloor-i$ implies $k\geq\lceil\beta p\rceil$ and as $\sum_{k\leq\lfloor A\ell_n \rfloor}\un_{\{p=\lfloor\frac{k+i}{\beta}\rfloor-i\}}\leq\beta$ for any $p\geq 1$, the above sum is smaller, by the \mto, than
\begin{align}
    \beta\sum_{p\leq\lfloor A\ell_n \rfloor}\Pb\big(S_p\geq (\log n)^{\alpha}-t,\max_{j\leq\lceil p\beta\rceil }H_j^S\leq n\big)+\frac{1}{n}\leq&\beta\sum_{p=L_n}^{\lfloor A\ell_n \rfloor}\Pb\big(S_p\geq (\log n)^{\alpha}-t,\max_{j\leq\lceil p\beta\rceil }H_j^S\leq n\big)\label{eqp12} \\ & +\beta\sum_{p<L_n}\Pb\big(S_p\geq (\log n)^{\alpha}-t\big)+\frac{1}{n}.\nonumber
\end{align}
For the second sum in \eqref{eqp12}, by  the exponential Markov inequality, for $n$ large enough, all $p<L_n$ and $t\geq-B$
\begin{align*}
    \Pb\big(S_p\geq (\log n)^{\alpha}-t\big)\leq e^{\bm{\delta}_nt-\bm{\delta}_n(\log n)^{\alpha}+p\psi(1-\bm{\delta}_n)}\leq e^{\bm{\delta}_n(t+B)-\frac{(\log n)^{2\alpha}}{2\sigma^2L_n}+L_n\psi(1-\bm{\delta}_n)}\leq e^{\delta t-(1-\varepsilon)\frac{(\log n)^{2\alpha}}{2\sigma^2L_n}}, 
\end{align*}
with $\bm{\delta}_n:=(\log n)^{\alpha}/\sigma^2L_n$, and we have used that $\psi(1-\bm{\delta}_n)\in\R^+$ for the second inequality and that $\bm{\delta}_n\to 0$ ($\alpha\in(1,2)$) together with $\psi(1)=\psi'(1)=0$ and $\psi''(1)=\sigma^2$ for the last one. \\ 
For the first sum in \eqref{eqp12}, which gives the main contribution,  by the Markov property at time $p$, $\Pb\big(S_p\geq (\log n)^{\alpha}-t,\max_{j\leq \lceil \beta p\rceil}H_j^S\leq n\big)$ is smaller than $\Pb\big(S_p\geq (\log n)^{\alpha}-t,\max_{j\leq p}H_j^S\leq n\big)\Pb\big(\max_{j\leq \lceil (\beta-1) p\rceil}H_j^S\leq n\big)$. Then thanks to Lemma \ref{LemmBorneUnif} \eqref{LemmBorneUnif1} (with $\ell=(\log n)^2$, $\lceil(\beta-1)p\rceil$ and $\varepsilon/2$ in place of, respectively, $k$ and $\varepsilon$), for $n$ large enough and any $p\in\{L_n,\ldots,\lfloor A\ell_n \rfloor\}$
\begin{align*}
  \Pb\big(\max_{j\leq \lceil (\beta-1) p\rceil}H_j^S\leq n\big)\leq e^{-p\frac{\pi^2\sigma^2(\beta-1)}{8(\log n)^2}(1-\frac{\varepsilon}{2})}=e^{-p\frac{\pi^2\sigma^2(\beta-1)}{8((1-\varepsilon/2)^{-1/2}\log n)^2}}.
\end{align*}
Hence, as $\log n\leq (1-\varepsilon/2)^{-1/2}\log n$, $\sum_{p=L_n}^{\lfloor A\ell_n \rfloor}\Pb\big(S_p\geq (\log n)^{\alpha}-t,\max_{j\leq\lceil p\beta\rceil }H_j^S\leq n\big)$ is smaller than 
\begin{align*}
     &\sum_{p=L_n}^{\lfloor A\ell_n \rfloor}\Eb\Big[\un_{\{\tau_{(\log n)^{\alpha}-t}\leq p,\; \max_{j\leq k}\overline{S}_j-S_j\leq (1-\varepsilon/2)^{-1/2}\log n\}}e^{-p\frac{\pi^2\sigma^2(\beta-1)}{8((1-\varepsilon/2)^{-1/2}\log n)^2}}\Big] \\ & \leq A\ell_n\Eb\Big[\un_{\{\max_{j\leq \tau_{(\log n)^{\alpha}-t}}\overline{S}_j-S_j\leq(1-\varepsilon/2)^{-1/2}\log n\}}e^{-\tau_{(\log n)^{\alpha}-t}\frac{\pi^2\sigma^2(\beta-1)}{8((1-\varepsilon/2)^{-1/2}\log n)^2}}\Big] \\ & \leq A\ell_ne^{\sqrt{1-\frac{\varepsilon}{2}}\frac{c_{\beta}t}{\log n}-c_{\beta}(\log n)^{\alpha-1}(1-\frac{\varepsilon}{2})}\leq A\ell_ne^{\frac{c_{\beta}(t+B)}{\log n}-c_{\beta}(\log n)^{\alpha-1}(1-\frac{\varepsilon}{2})}\leq \frac{1}{3}e^{\delta t-c_{\beta}(\log n)^{\alpha-1}(1-\varepsilon)},
\end{align*}
where Lemma \ref{Laplace2} (with $\ell=((1-\varepsilon/2)^{-1/2}\log n)^2$, $r(\ell)=(\log n)^{\alpha}-t$, $c=\pi^2(\beta-1)/4$ and $1-\sqrt{1-\varepsilon/2}$ instead of $\varepsilon$) provides the second inequality. Finally collecting all the upper bounds of the three sums in \eqref{eqp12}, for $n$ large enough 
\begin{align*}
    &\Eb\Big[\sum_{x;\lfloor (|x|+i)/\beta\rfloor>i}e^{-V(x)}\un_{\{V(x_{(\lfloor (|x|+i)/\beta\rfloor-i)})\geq (\log n)^{\alpha}-t\}}\un_{\{x\in\mathcal{O}_n\}}\Big] \\ & \leq\frac{1}{3}e^{\delta t-c_{\beta}(\log n)^{\alpha-1}(1-\varepsilon)} + \beta e^{\delta t-(1-\varepsilon)\frac{(\log n)^{2\alpha}}{2\sigma^2L_n}}+ \frac{1}{n}\leq \frac{2}{3}e^{\delta t-c_{\beta}(\log n)^{\alpha-1}(1-\varepsilon)}+\frac{e^{\delta(t+B)}}{n}, 
\end{align*}
which is smaller than $e^{\delta t-c_{\beta}(\log n)^{\alpha-1}(1-\varepsilon)}$ (we have used that $(\log n)^{2\alpha}/L_n\geq (\log n)^{2(\alpha-1)-\varepsilon_{\alpha}}$ and $(\log n)^{\alpha-1}=o((\log n)^{2(\alpha-1)-\varepsilon_{\alpha}})$). This yields the upper bound in \eqref{Fact1TH2}. \\
 %We used that $(\log n)^{2\alpha}/L_n\geq (\log n)^{2(\alpha-1)-\varepsilon_{\alpha}}=o((\log n)^{\alpha-1})$ (recall $\varepsilon_{\alpha}\in(0,\alpha-1)$ and $\alpha\in(1,2)$). \\
\textit{Proof of \eqref{Fact2TH2}}. Let  $\alpha_n:=(\log n)^{\alpha}+\log n$. For all $m\leq \log n$,  by the \mto, the expectation in \eqref{Fact2TH2} is larger than 
\begin{align*}
    \sum_{p,k\geq 1}\un_{\{p=\lfloor(k+i)/\beta\rfloor-i\}}\Pb\big(S_{p}\geq\alpha_n, n^b<H^S_{k}\leq n^be^{\varepsilon c_{\beta}(\log n)^{\alpha-1}}, \max_{j\leq k}\;H^S_j\leq \bmlambda, S_k\geq \frac{2\ell_n^{\frac{1}{3}}}{\delta_1},\underline{S}_k\geq-B\big). 
\end{align*}
The above probability is larger than (as $\alpha_n>a\log n$ for all $a>\frac{1}{\delta_1}$)
\begin{align*}
    \Pb\big(S_{p}\geq\alpha_n,\underline{S}_p\geq-B,\overline{S}_p=S_p, n^b<H^S_{k}\leq n^be^{\varepsilon c_{\beta}(\log n)^{\alpha-1}}, \max_{j\leq k}\;H^S_j\leq \bmlambda,\min_{p<j\leq k}S_j\geq S_p\big). 
\end{align*}
Recall that $H_j^S=\sum_{i=1}^je^{S_i-S_j}$ so we have, for any $p<j\leq k$, $H_j^S=e^{S_p-S_j} H_p^S+H_{p,j}^S$ where $H_{p,j}^S= \sum_{i=p+1}^j e^{S_i-S_j}$. Note that $\overline{S}_p=S_p$ and $\min_{p<j\leq k}S_j\geq S_p$ implies $H_{j}^S\leq p+H_{p,j}^S$ so the previous probability is larger than
\begin{align*}
    \Pb\big(S_{p}\geq\alpha_n,\underline{S}_p\geq-B,\overline{S}_p=S_p, \max_{j\leq p}H_j^S\leq \bmlambda,n^b<H^S_{p,k}\leq &n^be^{\varepsilon c_{\beta}(\log n)^{\alpha-1}}-p \\ &, \max_{p<j\leq k}\;H^S_{p,j}\leq \bmlambda-p,\min_{p<j\leq k}S_j\geq S_p\big), 
\end{align*}
which, thanks to the Markov property at time $p$, is nothing but the product of $\Pb\big(S_{p}\geq\alpha_n,\underline{S}_p\geq-B,\overline{S}_p=S_p, \max_{j\leq p}H_j^S\leq \bmlambda\big)$ and $\Pb\big(n^b<H^S_{k-p}\leq n^be^{\varepsilon c_{\beta}(\log n)^{\alpha-1}}-p, \max_{j\leq k-p}\;H^S_{j}\leq \bmlambda-p=ne^{-6(\log n)^{\alpha-1}}-p,\underline{S}_{k-p}\geq 0\big)$. From now, let $p\in\{L_n,\ldots,\ell'_n= (\log n)^4\}$. We first deal with the second probability. Observe that for all $i\geq 0$, $p=\lfloor(k+i)/\beta\rfloor-i$ implies $k-p\geq\lceil(\beta-1)L_n\rceil$. It follows that for all $\varepsilon\in(0,\varepsilon_b)$, $n$ large enough, for all $L_n\leq p\leq\ell'_n$, $k\geq 1$, $i\geq 0$ such that $p=\lfloor(k+i)/\beta\rfloor-i$, $\Pb\big(n^b<H^S_{k-p}\leq n^be^{\varepsilon c_{\beta}(\log n)^{\alpha-1}}-p,\max_{j\leq k-p}\;H^S_{j}\leq \bmlambda-p,\underline{S}_{k-p}\geq 0\big)$ is larger than (as $\bmlambda-p\geq \bmlambda-\ell'_n\geq ne^{-7c_{\beta}(\log n)^{\alpha-1}}$)
\begin{align*}
    \Pb\big(n^b<H^S_{k-p}\leq n^be^{\frac{\varepsilon}{2} c_{\beta}(\log n)^{\alpha-1}}, \max_{j\leq k-p}\;H^S_{j}\leq ne^{-7c_{\beta}(\log n)^{\alpha-1}},\underline{S}_{k-p}\geq 0\big)\geq e^{-\frac{\pi^2\sigma^2}{8}\frac{(k-p)}{(\log\bmlambda')^2}},
\end{align*} 
with $\bmlambda':=n^{(1+\varepsilon/2)^{-1/2}}$ \label{blambdap}. The  last inequality comes from  Lemma \ref{LemmBorneUnif} \eqref{LemmBorneUnif2} (with $\ell=(\log n)^2$, $a=7$, $c=\frac{\varepsilon c_{\beta}}{2}$, $d=\frac{\alpha-1}{2}$, $k-p$ and $\varepsilon/2$ instead respectively of $k$ and $\varepsilon$). The equality $p=\lfloor(k+i)/\beta\rfloor-i$ also implies, for any $0\leq i\leq\log n$ that $k-p\leq(p+\log n)(\beta-1)+\beta$ so it follows that the above probability is larger than $C\exp(\frac{\pi^2\sigma^2(\beta-1)}{8(\log\bmlambda')^2}p)$ for some positive constant $C\in(0,1)$. Collecting the previous inequalities together with Lemma \ref{LemmFKG} gives, as $\sum_{k\geq 1}\un_{\{p=\lfloor(k+i)/\beta\rfloor-i\}}\geq 1$, that for $n$ large enough, the mean in \eqref{Fact2TH2} is larger than 
\begin{align*}
    &C\sum_{p=L_n}^{\ell'_n}\Eb\Big[e^{-\frac{\pi^2\sigma^2(\beta-1)}{8(\log\bmlambda')^2}p}\un_{\{S_{p}\geq\alpha_n, \underline{S}_p\geq-B,\overline{S}_p=S_p,\max_{j\leq p}H_j^S\leq ne^{-7c_{\beta}(\log n)^{\alpha-1}}\}}\Big]\sum_{k\geq 1}\un_{\{p=\lfloor(k+i)/\beta\rfloor-i\}} \\ & \geq C\Pb(\underline{S}_{\ell'_n}\geq 0)^2\Eb\Big[e^{-\frac{\pi^2\sigma^2(\beta-1)}{8(\log\bmlambda')^2}\tau_{\alpha_n}}\un_{\{L_n\leq\tau_{\alpha_n}\leq\ell'_n,\forall j\leq \tau_{\alpha_n}:\overline{S}_j-S_j\leq\log\bmlambda'\}}\Big] \\ & \geq C\Pb(\underline{S}_{\ell'_n}\geq 0)^2\Pb(\overline{S}_{\ell'_n}\geq\alpha_n)\Eb\Big[e^{-\frac{\pi^2\sigma^2(\beta-1)}{8(\log\bmlambda')^2}\tau_{\alpha_n}}\un_{\{\forall j\leq \tau_{\alpha_n}:\overline{S}_j-S_j\leq\log\bmlambda'\}}\Big]-\Pb(\overline{S}_{L_n}\geq\alpha_n).
\end{align*}
Note that thanks to \eqref{lim45} and the fact that $\alpha\in(1,2)$, we can find a constant $c_{(1.2)}>0$ such that
$C\Pb\big(\underline{S}_{\ell'_n}\geq 0\big)^2\Pb\big(\overline{S}_{\ell'_n}\geq \alpha_n\big) \geq c_{(1.2)}(\ell'_n)^{-1}\geq 2e^{-\frac{\varepsilon}{2}(\log n)^{\alpha-1}}$. Then applying Lemma \ref{Laplace2} (with $\ell=\log\bmlambda'$, $r=\alpha_n$, $c=\pi^2(\beta-1)/4$ and $\sqrt{1+\varepsilon/2}-1$ instead of $\varepsilon$), for $n$ large enough
\begin{align*}
    \Eb\Big[e^{-\frac{\pi^2\sigma^2(\beta-1)}{8(\log\bmlambda')^2}\tau_{\alpha_n}}\un_{\{\forall j\leq \tau_{\alpha_n}:\overline{S}_j-S_j\leq\log\bmlambda'\}}\Big]\geq e^{-c_{\beta}(\log n)^{\alpha-1}(1+\frac{\varepsilon}{2})}.
\end{align*}
Finally, by Markov inequality, $\Pb(\overline{S}_{L_n}\geq\alpha_n)\leq L_ne^{-c'_{(1.2)}\alpha_n^2/L_n}$ for some constant $c'_{(1.2)}>0$. Since $\alpha_n^2/L_n\geq(\log n)^{2(\alpha-1)-\varepsilon_{\alpha}}$ and $(\log n)^{\alpha-1}=o((\log n)^{2(\alpha-1)-\varepsilon_{\alpha}})$, %, recall indeed that $\varepsilon_{\alpha}\in(0,\alpha-1)$ and $\alpha\in(1,2)$,
we get that $\Pb(\overline{S}_{L_n}\geq\alpha_n)\leq e^{-c_{\beta}(\log n)^{\alpha-1}(1+\varepsilon)}$. Collecting the different estimates  yields \eqref{Fact2TH2}. \\ \\
We are ready to prove that $\mathbf{f}^n$ satisfies assumptions \eqref{Hypothese1}, \eqref{Hypothese2}, \eqref{Hypothese3} and \eqref{Hypothese4}.
Recall that $\Psi^k_{n,n^b}(f^{n,k})=\Eb\big[\sum_{|x|=k}e^{-V(x)}f^{n,k}(V(x_1),\ldots,V(x))\un_{\{x\in\Hline{n}{n^b}\}}\big]$ where $x\in\mathcal{O}_{n,n^b}$ if and only if $\max_{j\leq|x|}H_{x_j}\leq n$ and $H_x>n^b$, $\mathcal{U}_b =\{\kappa\in[0,1];\textrm{ for all } k\geq 1,\mathbf{t}\in\R^k,n\geq 1: \un_{\{H_k(\mathbf{t})>n^b\}} f^{n,k}(\mathbf{t})\leq C_{\infty} n^{-\kappa}\}$ with $C_{\infty}=\sup_{n,\ell}\|f^{n,\ell}\|_{\infty}$. \\
$\bullet$ Check of \eqref{Hypothese1} and asymptotic of $h_n$. We obtain from \eqref{Fact2TH2} with $i=m=0$ and $n$ large enough
\begin{align*}
   \Eb\Big[\sum_{x\in\mathcal{O}_{n,n^b}}e^{-V(x)}\un_{\{V(x_{\lfloor |x|/\beta \rfloor})\geq(\log n)^{\alpha}\}}\Big]& \geq\Eb\Big[\sum_{x\in\T}e^{-V(x)}\un_{\{V(x_{\lfloor |x|/\beta \rfloor})\geq (\log n)^{\alpha}\}}\un_{\{x\in\Upsilon_n\cap\mathcal{O}_{\bmlambda,n^b}\}}\Big] \\ & \geq e^{-c_{\beta}(\log n)^{\alpha-1}(1+\varepsilon)}.
\end{align*}
This implies that for all $b\in[0,1)$, $\kappa_b=\max\mathcal{U}_b=0$ {(we use a similar argument than in the proof of Theorem \ref{thm1})} and additionally with \eqref{Fact1TH2}, gives, taking $i=t=0$ 
\begin{align*}
    h_n = \Big|n^{\kappa_b}\log \big(\sum_{k\geq 1}\Psi^k_{n,n^b}(f^{n,k})\big)\Big|=\Big|\log\Eb\Big[\sum_{x\in\mathcal{O}_{n,n^b}}e^{-V(x)}\un_{\{V(x_{\lfloor |x|/\beta \rfloor})\geq (\log n)^{\alpha}\}}\Big]\Big|\sim c_{\beta}(\log n)^{\alpha-1}.
\end{align*}
We also deduce from the previous lower bound that \eqref{Hypothese1} is satisfied. \\
$\bullet$ For \eqref{Hypothese2}, recalling $m_n=\lceil \varepsilon h_n/c_2\rceil$ ($c_2$ is defined in \eqref{GenPotMax}), by definition,  for any $j>0$
\begin{align*}
    f^{n,j}_{\varepsilon h_n}(t_1,\ldots,t_j) & =  \inf_{\mathbf{s}\in[-\varepsilon h_n,\varepsilon h_n]^{m_n}}f^{n,m_n+j}(s_1,\ldots,s_{m_n},t_1+s_{m_n},\ldots,t_j+s_{m_n}) \\ & = \inf_{{s_{m_n}}\in[-\varepsilon h_n,\varepsilon h_n]} \un_{\{t_{\lfloor(m_n+j)/\beta\rfloor-m_n} 
    \geq (\log n)^{\alpha}-s_{m_n}\}} \\ &  = \un_{\{\lfloor(m_n+j)/\beta\rfloor>m_n\}}\un_{\big\{t_{\lfloor(m_n+j)/\beta\rfloor-m_n}\geq (\log n)^{\alpha}+\varepsilon h_n\big\}}.
\end{align*}
Then for any $l\in\N^*$ and all $\mathbf{t}=(t_1,\ldots,t_l)\in\R^l$, $f^{n,l+k}_{\varepsilon h_n}(t_1,\ldots,t_l,V(x_1)+t_l,\ldots,V(x)+t_l)$, with $|x|=k$, is equal to 
\begin{align*}
    \un_{\{m_n<\lfloor(k+i)/\beta\rfloor\leq i\}}\un_{\{t_{\lfloor(k+i)/\beta\rfloor-m_n}\geq (\log n)^{\alpha}+\varepsilon h_n\}}+\un_{\{\lfloor (k+i)/\beta\rfloor>i\}}\un_{\{V(x_{(\lfloor (k+i)/\beta\rfloor-i)})+t_l\geq (\log n)^{\alpha}+\varepsilon h_n\}},
\end{align*}
with $i=m_n+l$. Recall the definition of  $\Psi^{.}_{.,.}(F|\mathbf{t})$ in \eqref{PsiCondi}, we have
\begin{align*}
   \sum_{k\geq 1}\Psi^{k}_{n,n^b-H_l(\mathbf{t})}\big(f^{n,l+k}_{\varepsilon h_n}|\mathbf{t}\big)&\leq \mathbf{E}\Big[\underset{x\in\mathcal{O}_n}{\sum}e^{-V(x)}f^{n,l+k}_{\varepsilon h_n}(t_1,\ldots,t_l,V(x_{1})+t_l,\ldots, V(x)+t_l)\Big] \\ & \leq \sum_{k\geq 1}\un_{\{m_n<\lfloor(i+k)/\beta\rfloor\leq i\}}\un_{\{t_{\lfloor(i+k)/\beta\rfloor-m_n}\geq (\log n)^{\alpha}\}}\Psi^k_n(1) \\  & +  \Eb\Big[\sum_{x;\lfloor (|x|+i)/\beta\rfloor>i}e^{-V(x)}\un_{\{V(x_{(\lfloor (|x|+i)/\beta\rfloor-i)})\geq (\log n)^{\alpha}-t_l\}}\un_{\{x\in\mathcal{O}_n\}}\Big].
\end{align*}
$\sum_{k\geq 1}\un_{\{m_n<\lfloor(i+k)/\beta\rfloor\leq i\}}\un_{\{t_{\lfloor(i+k)/\beta\rfloor-m_n}\geq (\log n)^{\alpha}\}}\Psi^k_n(1)$ is equal to
\begin{align*}
    \sum_{p=1}^l\un_{\{t_p\geq(\log n)^{\alpha}\}}\sum_{k\geq 1}\Psi^k_n(1)\un_{\{p=\lfloor\frac{i+k}{\beta}\rfloor-m_n\}}\leq \beta \sum_{p=1}^l\un_{\{t_p\geq(\log n)^{\alpha}\}},
\end{align*}
where we have used that $\sum_{k\geq 1}\Psi^k_n(1)\un_{\{p=\lfloor\frac{i+k}{\beta}\rfloor-m_n\}}\leq e^{\psi(1)}\sum_{k\geq 1}\un_{\{p=\lfloor\frac{i+k}{\beta}\rfloor-m_n\}}\leq\beta$. Also by \eqref{Fact1TH2} with $i=m_n+l$, $t=t_l$ and $\frac{\varepsilon}{4A}$ instead of $\varepsilon$, $$\Eb\Big[\sum_{x;\lfloor (|x|+i)/\beta\rfloor>i}e^{-V(x)}\un_{\{V(x_{(\lfloor (|x|+i)/\beta\rfloor-i)})\geq (\log n)^{\alpha}-t_l\}}\un_{\{x\in\mathcal{O}_n\}}\Big]\leq e^{\delta t_l-c_{\beta}(\log n)^{\alpha-1}(1-\frac{\varepsilon}{4A})},$$ so
\begin{align*}
  \sum_{k\geq 1}\Psi^{k}_{n,n^b-H_l(\mathbf{t})}\big(f^{n,l+k}_{\varepsilon h_n}|\mathbf{t}\big)\leq  \beta\sum_{p=1}^l\un_{\{t_p\geq(\log n)^{\alpha}\}} + \frac{1}{2}e^{\delta t_l-c_{\beta}(\log n)^{\alpha-1}(1-\frac{\varepsilon}{3A})}.
\end{align*}
Note that $\beta\sum_{p=1}^l\un_{\{t_p\geq(\log n)^{\alpha}\}}$ is very small for $n$ large enough, any $l<\lfloor A\ell_n\rfloor$ and $H_l(\mathbf{t})\leq n$. Indeed, $\sum_{p=1}^le^{\delta(t_p-t_l)}\leq lH_l(\mathbf{t})^{\delta}\leq A\ell_n n^{\delta}$ so 
\begin{align*}
    \beta\sum_{p=1}^l\un_{\{t_p\geq(\log n)^{\alpha}\}}=e^{\delta t_l}\beta\sum_{p=1}^le^{\delta(t_p-t_l)}e^{-\delta t_p}\un_{\{t_p\geq(\log n)^{\alpha}\}}\leq e^{\delta t_l}\beta A\ell_n n^{\delta}e^{-\delta(\log n)^{\alpha}},
\end{align*}
which, as $\alpha\in(1,2)$, is smaller than $\frac{1}{2}e^{\delta t_l-c_{\beta}(\log n)^{\alpha-1}(1-\frac{\varepsilon}{3A})}$. Finally observe that
\begin{align*}
    e^{-c_{\beta}(\log n)^{\alpha-1}(1-\frac{\varepsilon}{3A})}= e^{c_{\beta}(\log n)^{\alpha-1}\frac{2\varepsilon}{3A}}e^{-c_{\beta}(\log n)^{\alpha-1}(1+\frac{\varepsilon}{3A})}\leq e^{\frac{\varepsilon}{A}h_n}\sum_{k\geq 1}\Psi^k_{n,n^b}(f^{n,k}),
\end{align*}
where we have used that $h_n\sim c_{\beta}(\log n)^{\alpha-1}$ and \eqref{Fact2TH2} with $i=m=0$. \\

%%%% ICI %%%% 
\noindent We are left to prove that the technical assumptions \eqref{Hypothese3} and \eqref{Hypothese4} are realized. The ideas are very similar to those of the proof of these two assumptions in the previous theorem, we give details here however to keep the proofs independent from one another.  \\
$\bullet$ For \eqref{Hypothese3}, recall that $\Upsilon^k_n$ is the set
\begin{align*}
   \{\mathbf{t}=(t_1,\ldots,t_k)\in\R^k; H_k(\mathbf{t})\leq n^be^{\varepsilon h_n}, t_k\geq 2\ell_n^{1/3}/\delta_1, \min_{j\leq k}t_j\geq-B\}.
\end{align*}
Let $0<\varepsilon_1< \varepsilon$ and recall that $\lambda_n=ne^{-5h_n}$. Note that $\lambda_n/2\geq\bmlambda=ne^{-6(\log n)^{\alpha-1}}$ so the sum $\sum_{k\geq 1}\Psi^k_{\lambda_n/2,n^b}(f^{n,k}_{\varepsilon h_n}\un_{\Upsilon^k_n})$ is larger than $\sum_{k\geq 1}\Psi^k_{\bmlambda,n^b}(f^{n,k}_{\varepsilon h_n}\un_{\Upsilon^k_n})$ which is nothing but
\begin{align*}
    &\Eb\Big[\sum_{x\in\Hline{\bmlambda}{n^b}}e^{-V(x)}\un_{\{\lfloor\frac{|x|+m_n}{\beta}\rfloor>m_n\}}\un_{\{V(x_{\lfloor(|x|+m_n)/\beta\rfloor-m_n})\geq(\log n)^{\alpha}+\varepsilon h_n,H_x\leq n^be^{\varepsilon h_n}, V(x)\geq \frac{2}{\delta_1}\ell_n^{\frac{1}{3}},\underline{V}(x)\geq-B\}}\Big] \\ & \geq\Eb\Big[\sum_{x,\lfloor(|x|+m_n)/\beta\rfloor>m_n}e^{-V(x)}\un_{\{V(x_{\lfloor(|x|+m_n)/\beta\rfloor-m_n})\geq(\log n)^{\alpha}+h_n\}}\un_{\{x\in\Upsilon_n(\frac{\varepsilon_1}{3})\cap\Hline{\bmlambda}{n^b}\}}\Big] \\ & \geq e^{-c_{\beta}(\log n)^{\alpha-1}(1+\frac{\varepsilon_1}{3})},
\end{align*}
where this last inequality comes from \eqref{Fact2TH2} with $i=m=m_n$ and $\varepsilon_1/3$ instead of $\varepsilon$. Moreover, $e^{-c_{\beta}(\log n)^{\alpha-1}(1+\frac{\varepsilon_1}{3})}=e^{-\frac{2\varepsilon_1}{3}c_{\beta}(\log n)^{\alpha-1}}e^{-c_{\beta}(\log n)^{\alpha-1}(1-\frac{\varepsilon_1}{3})}\geq e^{-\varepsilon_1 h_n}\sum_{k\geq 1}\Psi^k_{n,n^b}(f^{n,k})$, the last inequality comes from the fact that $h_n\sim c_{\beta}(\log n)^{\alpha-1}$ and \eqref{Fact1TH2} with $i=t=0$. \\ 
$\bullet$ For \eqref{Hypothese4}, first observe that for all $k\in\N^*$ and $\alpha\in(1,2)$, $(\log n)^{\alpha}-\ell_n^{1/3}/\delta_1>\log n$ for $n$ large enough so
\begin{align*}
    \Psi^k_n(f^{n,k}\un_{\R\setminus\mathcal{H}^k_{\ell_n^{1/3}/\delta_1}})&=\Eb\Big[\sum_{|x|=k}e^{-V(x)}\un_{\{V(x_{\lfloor|x|/\beta\rfloor})\geq(\log n)^{\alpha},V(x)<\ell_n^{1/3}/\delta_1\}}\un_{\{x\in\mathcal{O}_n\}}\Big] \\ & \leq \Eb\Big[\sum_{|x|=k}e^{-V(x)}\un_{\{\overline{V}(x)\geq(\log n)^{\alpha},V(x)<\ell_n^{1/3}/\delta_1\}}\un_{\{\overline{V}(x)-V(x)\leq \log n\}}\Big]=0.
\end{align*}
Recall that $W=\sum_{|z|=1}e^{-V(z)}$ and  
\begin{align*}
    \sum_{k\geq 1}\big(\Psi^k_{n,n^b/(\log n)^2}(f^{n,k}) +\Eb\big[W\Psi^k_{n,n^b/(W(\log n)^2)}(f^{n,k})\big]\big)\leq \sum_{k\geq 1}\big(\Psi^k_{n}(f^{n,k}) +\Eb\big[W\Psi^k_{n}(f^{n,k})\big]\big),
\end{align*}
which is equal to $2\sum_{k\geq 1}\Psi^k_n(f^{n,k})$ since $\Eb[W]=e^{\psi(1)}=1$ and thanks to \eqref{Fact1TH2} with $i=t=0$ and $\frac{\varepsilon_1}{4}$ in place of $\varepsilon$
\begin{align*}
    2\sum_{k\geq 1}\Psi^k_n(f^{n,k})=2\Eb\Big[\sum_{x\in\mathcal{O}_n}e^{-V(x)}\un_{\{V(x_{|x|/\beta})\geq(\log n)^{\alpha}\}}\Big]&\leq 2e^{-c_{\beta}(\log n)^{\alpha-1}(1-\frac{\varepsilon_1}{4})} \\ & \leq e^{-c_{\beta}(\log n)^{\alpha-1}(1-\frac{\varepsilon_1}{3})}.
\end{align*}
Moreover, $e^{-c_{\beta}(\log n)^{\alpha-1}(1-\frac{\varepsilon_1}{3})}=e^{\frac{2\varepsilon_1}{3}c_{\beta}(\log n)^{\alpha-1}}e^{-c_{\beta}(\log n)^{\alpha-1}(1+\frac{\varepsilon_1}{3})}\leq e^{\varepsilon_1 h_n}\sum_{k\geq 1}\Psi^k_{n,n^b}(f^{n,k})$, the last inequality comes from the fact that $h_n\sim c_{\beta}(\log n)^{\alpha-1}$ and \eqref{Fact2TH2} with $i=m=0$.\end{proof}

\begin{proof}[Proof of Theorem \ref{thm3}] 
\noindent \\
\textbf{Assume first that $a=d=1$ and $\alpha\in(1,2)$} which corresponds to the second and third case of the theorem. 
Let us start with the proof of the two facts,  note that we distinguish whether $b=0$ or $b \in(0,1/2)$. \\
\underline{Facts for the case $b=0$} : for all $B,\delta>0$, $\varepsilon\in(0,\varepsilon_b)$ and $n$ large enough, for any $t\geq -B$, 
\begin{align}
    \mathbf{E}\Big[\sum_{x\in\mathcal{O}_n}\frac{e^{-V(x)}}{\sum_{j\leq |x|} H_{x_j}}\un_{\{V(x) \geq (\log n)^{\alpha}-t\}}\Big]\leq e^{\delta t-2(\log n)^{\alpha/2}(1-\varepsilon)},
    \label{Fact1TH3}
\end{align}
and for all $0\leq m\leq\log n$, $0\leq M\leq e^{(\log n)^{\alpha/2}}$
\begin{align}
    \mathbf{E}\Big[\sum_{x\in\T}\frac{e^{-V(x)}\un_{\{x\in\Upsilon_{n,1}\cap\mathcal{O}_{\bm{\lambda}_{n,1}}\}}}{M|x|+\sum_{j\leq |x|} H_{x_j}}\un_{\{V(x)\geq(\log n)^{\alpha}+m\}}\Big]\geq e^{-2(\log n)^{\alpha/2}(1+\varepsilon)},
    \label{Fact2TH3}
\end{align}
 with $\bm{\lambda}_{n,1}=ne^{-12(\log n)^{\alpha/2}}$ and
\begin{align*}
    \Upsilon_{n,1}=\Upsilon_{n,1}(\varepsilon):=\{x\in\T;H_x\leq e^{2\varepsilon(\log n)^{\alpha/2}},\underline{V}(x)\geq-B\}. %\red{ \varepsilon \ ici\ ? \ probleme\ plus\ bas}
\end{align*}
We first deal with the upper bound \eqref{Fact1TH3}. Note that if $t>(\log n)^{\alpha}/2$, then \eqref{Fact1TH3} is obviously satisfied. Indeed, $(\sum_{j\leq |x|} H_{x_j})^{-1}\un_{\{V(x)\geq(\log n)^{\alpha}-t\}}\leq 1$ so for $n$ large enough
\begin{align*}
    \mathbf{E}\Big[\sum_{x\in\mathcal{O}_n}\frac{e^{-V(x)}}{\sum_{j\leq |x|} H_{x_j}}\un_{\{V(x) \geq (\log n)^{\alpha}-t\}}\Big] & \leq \Eb\Big[\sum_{x\in\mathcal{O}_n}e^{-V(x)}\Big]e^{-\delta t}e^{\delta t}  \leq \ell_n e^{\delta t-\frac{\delta}{2}(\log n)^{\alpha}} \\ & \leq e^{\delta t-2(\log n)^{\alpha/2}(1-\varepsilon)},
\end{align*}
where we have used Remark \ref{Rem1}. Now assume $t\leq (\log n)^{\alpha}/2$, by the \mto, the expectation in \eqref{Fact1TH3} is smaller than
\begin{align}
    \sum_{k\leq\lfloor A\ell_n \rfloor}\mathbf{E}\Big[\frac{1}{\sum_{j=1}^{k}H_j^S}\un_{\{\tau_{(\log n)^{\alpha}-t}\leq k,\;\max_{j\leq k}H^S_j\leq n\}}\Big]+\Eb\Big[\sum_{|x|>\lfloor A\ell_n\rfloor}e^{-V(x)}\un_{\{x\in\mathcal{O}_n\}}\Big], \label{lessum1.2}
\end{align}
the second sum is treated as usual : Remark \ref{Rem1} with a chosen $A$, together with the fact that $\alpha \in (1,2)$ and $t\geq-B$ implies that $\Eb\Big[\sum_{|x|>\lfloor A\ell_n\rfloor}e^{-V(x)}\un_{\{x\in\mathcal{O}_n\}}\Big]\leq 1/n \leq \frac{1}{2}e^{\delta t-2(\log n)^{\alpha/2}(1-\varepsilon)}.$
%\gr{where $A>0$ is chosen, thanks to Remark \ref{Rem1}, such that $\sum_{k>\lfloor A\ell_n \rfloor}\Pb\big(\max_{j\leq k}\overline{S}_j-S_j\leq \log n\big)\leq 1/n$}. 
Also using that $(\sum_{j=1}^{k} H_j^S)^{-1}\leq e^{-\max_{j\leq k}\overline{S}_j-S_j}$ leads to
\begin{align*}
    \sum_{k\leq \lfloor A\ell_n \rfloor}\mathbf{E}\Big[\frac{1}{\sum_{j=1}^{k}H_j^S}\un_{\{\tau_{(\log n)^{\alpha}-t}\leq k,\;\max_{j\leq k}H^S_j\leq n\}}\Big] & \leq  \lfloor A\ell_n \rfloor\mathbf{E}\Big[e^{-\max_{j\leq \tau_{(\log n)^{\alpha}-t}}\overline{S}_j-S_j}\Big].
\end{align*}
Since $t<(\log n)^{\alpha}/2$, $(\log n)^{\alpha}-t>(\log n)^{\alpha}/2$ so by Lemma \ref{EspMaxLine} with $\frac{\varepsilon}{2}$ instead of $\varepsilon$ and any $t\geq-B$ %, for all $\delta,\varepsilon,B>0$, $n$ large enough and any $t\geq -B$ 
\begin{align*}
    \mathbf{E}\Big[e^{-\max_{j\leq \tau_{(\log n)^{\alpha}-t}}\overline{S}_j-S_j}\Big]\leq e^{-2(1-\frac{\varepsilon}{2})\sqrt{(\log n)^{\alpha}-t}}\leq e^{-2(1-\frac{\varepsilon}{2})\frac{(\log n)^{\alpha}-(t+B)}{\sqrt{(\log n)^{\alpha}+B}}}\leq \frac{1}{2}e^{\delta t-2(\log n)^{\alpha/2}(1-\varepsilon)}.
\end{align*}
%Finally, since $\alpha\in(1,2)$, we have, for all $\varepsilon,\delta>0$, $n$ large enough and any $t\geq-B$ we have, $1/n\leq e^{\delta t-2(\log n)^{\alpha/2}(1-\varepsilon/2)}$
This treats the first sum in \eqref{lessum1.2} and yields \eqref{Fact1TH3}. \\
We now turn to the lower bound \eqref{Fact2TH3}. Recall $\ell'_n=(\log n)^4$, using that $\sum_{j=1}^kH_j^S\leq k\max_{j\leq k}H_j^S$ and the fact that $m\leq \log n$, $0\leq M\leq e^{(\log n)^{\alpha/2}}$ and  $\bm{\lambda}_{n,1}>e^{(\log n)^{\alpha/2}}$, we obtain thanks to the \mto\;
\begin{align*}
     &\mathbf{E}\Big[\sum_{x\in\T}\frac{e^{-V(x)}\un_{\{x\in\Upsilon_{n,1}\cap\mathcal{O}_{\bm{\lambda}_{n,1}}\}}}{M|x|+\sum_{j\leq |x|} H_{x_j}}\un_{\{V(x)\geq(\log n)^{\alpha}+m\}}\Big] \\ & \geq \sum_{k\leq\ell'_n} \mathbf{E}\Big[\frac{1}{2ke^{(\log n)^{\alpha/2}}}\un_{\{S_k\geq\alpha_n,\;\max_{j\leq k}H_j^S\leq e^{ (\log n)^{\alpha/2}},\; \underline{S}_k\geq-B,\; \overline{S}_k=S_k\}}\Big] \\ & \geq \frac{e^{-(\log n)^{\alpha/2}}}{2\ell'_n}\sum_{k\leq\ell'_n}\Pb\big(S_{k}\geq\alpha_n,\;\max_{j\leq k}H_j^S\leq e^{(\log n)^{\alpha/2}}, \underline{S}_k\geq-B,\overline{S}_k=S_k\big),
\end{align*}
where $\alpha_n=(\log n)^{\alpha}+\log n$. By Lemma \ref{MinorRangeHP2} (with $\ell=(\log n)^2$, $t_{\ell}=\alpha_n$, $d=\alpha/4$ and $a=0$), the previous probability is larger than $e^{-(\log n)^{\alpha/2}(1+\frac{\varepsilon}{2})}$ . Finally collecting the inequalities, we get  \eqref{Fact2TH3}. %for $n$ large enough $\Pb\big(\forall j\leq \tau_{\alpha_n}:\overline{S}_j-S_j\leq(\log n)^{\alpha/2}, S_j\geq-B\big)\geq e^{-(\log n)^{\alpha/2}(1+\varepsilon/3)}$. We conclude by recalling that $\alpha\in(1,2)$.  \\ 
% Besides, by theorem A in \cite{Kozlov1976}, $\Pb\big(\overline{S}_{\ell'_n}\geq\alpha_n\big)=1-\Pb\big(\overline{S}_{\ell'_n}<\alpha_n\big)\geq1-C_{4,K}\alpha_n/\sqrt{\ell'_n}$ and

\noindent \underline{Facts for the case $b \in (0,1/2)$} : for any $t\geq -B$, $r\geq 0$ and $w>0$
\begin{align}
    \mathbf{E}\Big[\sum_{x\in\mathcal{O}_n}\frac{e^{-V(x)}\un_{\{r+H_x>n^b/(w(\log n)^2)\}}}{r+\sum_{j\leq |x|} H_{x_j}}\un_{\{V(x) \geq (\log n)^{\alpha}-t\}}\Big]\leq (w+1)n^{-b}e^{\delta t-\frac{1-\varepsilon}{b}(\log n)^{\alpha-1}}.
    \label{Fact1TH3bis}
\end{align}
Also for all $0\leq m\leq\log n$, $0\leq M\leq n^b$ 
\begin{align}
    \mathbf{E}\Big[\sum_{x\in\T}\frac{e^{-V(x)}\un_{\{x\in\Upsilon_{n,2}\cap\Hline{\bm{\lambda}_{n,2}}{n^b}\}}}{M|x|+\sum_{j\leq |x|} H_{x_j}}\un_{\{V(x)\geq (\log n)^{\alpha}+m\}}\Big]\geq n^{-b}e^{-\frac{1+\varepsilon}{b}(\log n)^{\alpha-1}},
    \label{Fact2TH3bis}
\end{align}
with $\bm{\lambda}_{n,2}=ne^{-\frac{6}{b}(\log n)^{\alpha-1}}$ and  
\begin{align*}
    \Upsilon_{n,2}=\Upsilon_{n,2}(\varepsilon):=\{x\in\T;H_x\leq n^be^{\frac{\varepsilon}{b}(\log n)^{\alpha-1}},\underline{V}(x)\geq-B\}.
\end{align*} 
We first deal with the upper bound \eqref{Fact1TH3bis}. We split the sum according to the generation of $x$: when $|x|>\lfloor A\ell_n\rfloor$, we use that $\un_{\{r+H_x>n^b/(w(\log n)^2),V(x)\geq(\log n)^{\alpha}-t\}}(r+\sum_{j\leq |x|}H_{x_j})^{-1}\leq 1$ so the expectation in \eqref{Fact1TH3bis} is smaller than
\begin{align*}
    \Eb\Big[\sum_{|x|>\lfloor A\ell_n\rfloor}e^{-V(x)}\un_{\{x\in\mathcal{O}_n\}}\Big]+ \mathbf{E}\Big[\sum_{|x|\leq\lfloor A\ell_n\rfloor}\frac{e^{-V(x)}\un_{\{r+H_x>n^b/(w(\log n)^2)\}}}{r+\sum_{j\leq |x|} H_{x_j}}\un_{\{V(x) \geq (\log n)^{\alpha}-t\}}\un_{\{x\in\mathcal{O}_n\}}\Big].
\end{align*}
Then, when $|x|\leq\lfloor A\ell_n\rfloor$, we again split the sum but this time according to $\max_{j\leq |x|}H_{x_j}$: when $\max_{j\leq |x|}H_{x_j}>n^be^{\frac{1}{b}(\log n)^{\alpha-1}}$, we use that $\un_{\{r+H_x>n^b/(w(\log n)^2),V(x)\geq(\log n)^{\alpha}-t\}}(r+\sum_{j\leq |x|}H_{x_j})^{-1}\leq (\max_{j\leq |x|}H_{x_j})^{-1}\leq n^{-b}e^{-\frac{1}{b}(\log n)^{\alpha-1}}$. Otherwise, one can observe that $\un_{\{r+H_x>n^b/(w(\log n)^2)\}}(r+\sum_{j\leq |x|}H_{x_j})^{-1}\leq\un_{\{r+H_x>n^b/(w(\log n)^2)\}}(r+H_{x})^{-1}\leq wn^{-b}(\log n)^2$. Therefore, the expectation in \eqref{Fact1TH3bis}  is smaller than 
\begin{align*}
    &\Eb\Big[\sum_{|x|>\lfloor A\ell_n\rfloor}e^{-V(x)}\un_{\{x\in\mathcal{O}_n\}}\Big]+\Eb\Big[\sum_{|x|\leq\lfloor A\ell_n\rfloor}e^{-V(x)}\un_{\{x\in\mathcal{O}_n\}}\Big]n^{-b}e^{-\frac{1}{b}(\log n)^{\alpha-1}} \\ & + wn^{-b}(\log n)^2\mathbf{E}\Big[\sum_{|x|\leq\lfloor A\ell_n\rfloor}e^{-V(x)}\un_{\{V(x) \geq (\log n)^{\alpha}-t,\underset{j\leq |x|}{\max}H_{x_j}\leq n^be^{\frac{1}{b}(\log n)^{\alpha-1}}\}}\Big],
\end{align*}
which, by Remark \ref{Rem1} and the \mto, is smaller, for $n$ large enough, than
\begin{align*}
    \frac{1}{n}+\ell_n n^{-b}e^{-\frac{1}{b}(\log n)^{\alpha-1}}+wn^{-b}(\log n)^2\sum_{k\leq \lfloor A\ell_n\rfloor}\Pb\big(S_k\geq(\log n)^{\alpha}-t,\underset{j\leq k}{\max}\;H^S_{j}\leq n^be^{\frac{1}{b}(\log n)^{\alpha-1}}\big).
\end{align*}
Also, $\sum_{k\leq \lfloor A\ell_n\rfloor}\Pb\big(S_k\geq(\log n)^{\alpha}-t,\underset{j\leq k}{\max}\;H^S_{j}\leq n^be^{-\frac{1}{b}(\log n)^{\alpha-1}}\big)$ is smaller than
\begin{align*}
    \lfloor A\ell_n\rfloor\Pb\big(\underset{j\leq \tau_{(\log n)^{\alpha}-t}}{\max}\;\overline{S}_j-S_j\leq b\log n+\frac{1}{b}(\log n)^{\alpha-1}\big)\leq e^{\frac{t}{\log n}-\frac{1}{b}(\log n)^{\alpha-1}(1-\frac{\varepsilon}{2})}, 
\end{align*}
where  Lemma A.3 in \cite{HuShi15b} provides us the last inequality for $n$ large enough and any $t$. Finally, note that for any $\delta>0$, $n$ large enough, any $w>0$ and any $t\geq-B$, $1/n\leq\frac{1}{3} n^{-b}e^{-\delta B-\frac{1-\varepsilon}{b}(\log n)^{\alpha-1}}\leq\frac{w+1}{3} n^{-b}e^{\delta t-\frac{1-\varepsilon}{b}(\log n)^{\alpha-1}}$, $\ell_n n^{-b}e^{-\frac{1}{b}(\log n)^{\alpha-1}}\leq \frac{1}{3}n^{-b}e^{-\delta B-\frac{1-\varepsilon}{b}(\log n)^{\alpha-1}}\leq \frac{w+1}{3}n^{-b}e^{\delta t-\frac{1-\varepsilon}{b}(\log n)^{\alpha-1}}$, $wn^{-b}(\log n)^2e^{\frac{t}{\log n}-\frac{1}{b}(\log n)^{\alpha-1}(1-\frac{\varepsilon}{2})}\leq\frac{w+1}{n^b}(\log n)^2e^{\frac{t+B}{\log n}-\frac{1}{b}(\log n)^{\alpha-1}(1-\frac{\varepsilon}{2})}\leq \frac{w+1}{3n^b}e^{\delta t-\frac{1-\varepsilon}{b}(\log n)^{\alpha-1}}$ and this finish the proof of the first fact.
 %As  $1/n=o(n^{-b} e^{\delta t-\frac{1}{b}(\log n)^{\alpha-1}})$ this ends the proof of the fact. \\
We now turn to the lower bound \eqref{Fact2TH3bis}. By the \mto, for any $m\leq\log n$, $0\leq M\leq n^b$ and $A>0$, the mean in \eqref{Fact2TH3bis}  is larger than (as $\bm{\lambda}_{n,2}>n^be^{\frac{\varepsilon}{3b}(\log n)^{\alpha-1}}$)
\begin{align*}
    &\sum_{k\leq\lfloor A\ell_n \rfloor}\Eb\Big[\frac{1}{kn^b+\sum_{j=1}^kH_j^S}\un_{\{S_k\geq\alpha_n,\max_{1\leq j\leq k}H_j^S\leq n^be^{\frac{\varepsilon}{3b}(\log n)^{\alpha-1}},H_k^S>n^b,\underline{S}_k\geq-B\}}\Big] \\ & \geq \frac{n^{-b}}{2A\ell_n}e^{-\frac{\varepsilon}{3b}(\log n)^{\alpha-1}}\sum_{k\leq\lfloor A\ell_n \rfloor}\Pb\big(S_k\geq\alpha_n,\max_{1\leq j\leq k}H_j^S\leq n^be^{\frac{\varepsilon}{3b}(\log n)^{\alpha-1}},H_k^S>n^b,\underline{S}_k\geq-B\big),
\end{align*}
with $\alpha_n:=(\log n)^{\alpha}+\log n$. By Lemma \ref{Minor_HRHP1} \eqref{lem4.1b} (with $\ell=(\log n)^2$, $t_{\ell}=\alpha_n$, $q=b$, $a_b=-a=-\frac{\varepsilon}{3b}$, $d=\frac{\alpha-1}{2}$ and $c=\frac{\varepsilon}{3b}$) the above sum is larger, for $n$ large enough, than $e^{-\frac{1}{b}(\log n)^{\alpha-1}(1+\frac{\varepsilon}{2})}\geq 2A\ell_ne^{-\frac{1}{b}(\log n)^{\alpha-1}(1+\varepsilon)}$, which completes the proof of the upper bound. \\ 

\noindent We are ready to prove that $\mathbf{f}^n$ satisfies assumptions \eqref{Hypothese1}, \eqref{Hypothese2}, \eqref{Hypothese3} and \eqref{Hypothese4}.
%Recall that $\Psi^k_{n,n^b}(f^{n,k})=\Eb\big[\sum_{|x|=k}e^{-V(x)}f^{n,k}(V(x_1),\ldots,V(x))\un_{\{x\in\Hline{n}{n^b}\}}\big]$ where $x\in\mathcal{O}_{n,n^b}$ if and only if $\max_{j\leq|x|}H_{x_j}\leq n$ and $H_x>n^b$, $\mathcal{U}_b =\{\kappa\geq 0;\textrm{ for all } k\geq 1,\mathbf{t}\in\R^k,n\geq 1: \un_{\{H_k(\mathbf{t})>n^b\}} f^{n,k}(\mathbf{t})\leq C_{\infty} n^{-\kappa}\}$ with $C_{\infty}=\sup_{n,\ell}\|f^{n,\ell}\|_{\infty}$. \\ 

\noindent $\bullet$ Check of \eqref{Hypothese1} and asymptotic of $h_n$. \eqref{Fact2TH3} with $m=M=0$ implies, for $b=0$ and $n$ large enough
\begin{align*}
   \sum_{k\geq 1}\Psi^k_{n}(f^{n,k})\geq\mathbf{E}\Big[\sum_{x\in\T}\frac{e^{-V(x)}}{\sum_{j\leq |x|} H_{x_j}}\un_{\{V(x)\geq(\log n)^{\alpha}\}}\un_{\{x\in\Upsilon_{n,1}\cap\mathcal{O}_{\bm{\lambda}_{n,1}}\}}\Big]\geq e^{-2(\log n)^{\alpha/2}(1+\varepsilon)}.
\end{align*}
This implies that $\kappa_0=\max\mathcal{U}_0=0$ (see the part concerning $\kappa_b$ in the proof of Theorem \ref{thm1} for details) and additionally with \eqref{Fact1TH3} and $t=0$
\begin{align*}
    h_n = \Big|n^{\kappa_b}\log \big(\sum_{k\geq 1}\Psi^k_{n,n^b}(f^{n,k})\big)\Big|\sim 2(\log n)^{\alpha/2}.
\end{align*}
We also deduce from the previous lower bound that \eqref{Hypothese1} is satisfied.\\
From \eqref{Fact1TH3bis} with $r=t=0$, $w=1$ and $\frac{\varepsilon}{2}$ instead of $\varepsilon$, we get for all $b\in(0,1)$ and $n$ large enough
\begin{align*}
    \sum_{k\geq 1}\Psi^k_{n,n^b}(f^{n,k})\leq\mathbf{E}\Big[\sum_{x\in\mathcal{O}_n}\frac{e^{-V(x)}\un_{\{H_x>n^b/(\log n)^2\}}}{\sum_{j\leq |x|} H_{x_j}}\un_{\{V(x) \geq (\log n)^{\alpha}\}}\Big]\leq n^{-b}e^{-\frac{1-\varepsilon}{b}(\log n)^{\alpha-1}}.
\end{align*}
This implies that for all $b\in(0,1)$, $\kappa_b\geq b$. From \eqref{Fact2TH3bis} with $m=M=0$, we get that for all $b\in(0,1)$
\begin{align*}
   \sum_{k\geq 1}\Psi^k_{n,n^b}(f^{n,k})\geq\mathbf{E}\Big[\sum_{x\in\T}\frac{e^{-V(x)}}{\sum_{j\leq |x|} H_{x_j}}\un_{\{V(x)\geq(\log n)^{\alpha}\}}\un_{\{x\in\Upsilon_{n,2}\cap\Hline{\bm{\lambda}_{n,2}}{n^b}\}}\Big]\geq n^{-b}e^{-\frac{1+\varepsilon}{b}(\log n)^{\alpha-1}}.
\end{align*}
This implies that for all $b\in(0,1/2)$, $\kappa_b\leq b$. Finally, for any $b\in(0,1/2)$, $\kappa_b=b$ and
\begin{align*}
    h_n = \Big|n^{\kappa_b}\log \big(\sum_{k\geq 1}\Psi^k_{n,n^b}(f^{n,k})\big)\Big|\sim \frac{1}{b}(\log n)^{\alpha-1}.
\end{align*}
We also deduce from the previous lower bound that \eqref{Hypothese1} is satisfied. \\
$\bullet$ For \eqref{Hypothese2}, recalling $m_n=\lceil \varepsilon h_n/c_2\rceil$ (see  \eqref{GenPotMax}) and for all $\mathbf{s}=(s_1,\ldots,s_{m_n})\in\R^{m_n}$, $\mathbf{t}=(t_1,\ldots,t_k)\in\R^k$, with $\mathbf{u}=(s_1,\ldots,s_{m_n},t_1+s_{m_n},\ldots,t_k+s_{m_n})$
\begin{align}\label{GrandH}
f^{n,m_n+k}(s_1,\ldots,s_{m_n},t_1+s_{m_n},\ldots,t_k+s_{m_n}) = \un_{\{t_k+s_{m_n}\geq (\log n)^{\alpha}\}} \frac{1}{\sum_{j=1}^{m_n+k}H_j(\mathbf{u})}.
\end{align}
Note that $\sum_{j=1}^{m_n+k}H_j(\mathbf{u})=\sum_{j=1}^{m_n}H_j(\mathbf{s})+\sum_{j=1}^{k}\big(e^{-t_j}H_{m_n}(\mathbf{s})+H_j(\mathbf{t}))\geq\sum_{j=1}^kH_j(\mathbf{t}\big)$ so 
\begin{align*}
    f^{n,k}_{\varepsilon h_n}(t_1,\ldots,t_k) & =  \inf_{\mathbf{s}\in[-\varepsilon h_n,\varepsilon h_n]^{m_n}}f^{n,m_n+k}(s_1,\ldots,s_{m_n},t_1+s_{m_n},\ldots,t_k+s_{m_n}) \\ & \leq \inf_{s_{m_n}\in[-\varepsilon h_n,\varepsilon h_n]}\un_{\{t_k+s_{m_n}\geq (\log n)^{\alpha}\}} \frac{1}{\sum_{j=1}^kH_j(\mathbf{t})}= \frac{\un_{\{t_k\geq (\log n)^{\alpha}+\varepsilon h_n\}}}{\sum_{j=1}^kH_j(\mathbf{t})}.
\end{align*} 
It follows that $f^{n,k}_{\varepsilon h_n}(t_1,\ldots,t_k)\leq \un_{\{t_k\geq(\log n)^{\alpha}\}}\big(\sum_{j=1}^kH_j(\mathbf{t}_j)\big)^{-1}$ and for $|x|=k$ with $\mathbf{u}_x=(t_1,\ldots,t_l,V(x_1)+t_l,\ldots,V(x)+t_l)$
\begin{align*}
  f^{n,l+k}_{\varepsilon h_n}(t_1,\ldots,t_l,V(x_1)+t_l,\ldots,V(x)+t_l)\leq\un_{\{V(x)\geq(\log n)^{\alpha}-t_l\}}\frac{1}{\sum_{j=1}^{l+k}(\mathbf{u}_x)}.
\end{align*}
Assume $b=0$. Observe again that $\sum_{j=1}^{l+k}(\mathbf{u}_x)=\sum_{j=1}^{l}H_j(\mathbf{t})+\sum_{j=1}^{k}\big(e^{-V(x_j)}H_{l}(\mathbf{t})+H_{x_j}\big)\geq\sum_{j\leq k}H_{x_j}$. Then, by definition of $\Psi^{k}_{n}\big(F|\mathbf{t}\big)$ (see \eqref{PsiCondi}), for all $A,B,\varepsilon,\delta>0$, $n$ large enough, for any $l\in\N^*$ and all $\mathbf{t}=(t_1,\ldots,t_l)\in\R^l$ with $t_l\geq -B$
\begin{align*}
      \sum_{k\geq 1}\Psi^{k}_{n,n^b-H_l(\mathbf{t})}\big(f^{n,l+k}_{\varepsilon h_n}|\mathbf{t}\big)&\leq\mathbf{E}\Big[\underset{x\in\mathcal{O}_n}{\sum}e^{-V(x)}f^{n,l+k}_{\varepsilon h_n}(t_1,\ldots,t_l,V(x_{1})+t_l,\ldots, V(x)+t_l)\Big] \\ & \leq \mathbf{E}\Big[\underset{x\in\mathcal{O}_n}{\sum}e^{-V(x)}\un_{\{V(x)\geq(\log n)^{\alpha}-t_l\}}\frac{1}{\sum_{j\leq k}H_{x_j}}\Big]\leq e^{\delta t_l-2(\log n)^{\alpha/2}(1-\frac{\varepsilon}{3A})},
\end{align*}
where we have used \eqref{Fact1TH3} with $t=t_l$ and replaced $\varepsilon$ by $\frac{\varepsilon}{3A}$ for the last inequality. Finally, observe that 
\begin{align*}
    e^{-2(\log n)^{\alpha/2}(1-\frac{\varepsilon}{3A})}=e^{\frac{4\varepsilon}{3A}(\log n)^{\alpha/2}}e^{-2(\log n)^{\alpha/2}(1+\frac{\varepsilon}{3A})}
    \leq e^{\frac{\varepsilon}{A}h_n}\sum_{k\geq 1}\Psi^k_{n,n^b}(f^{n,k}),
\end{align*}
where we have used that $h_n\sim 2(\log n)^{\alpha/2}$ and \eqref{Fact2TH3} with $m=M=0$. \\
Assume $b\in(0,1/2)$. Note that $\sum_{j=1}^{l+k}H_j(\mathbf{u}_x)\geq H_l(\mathbf{t})+\sum_{j\leq k}H_{x_j}$. Then for all $A,B,\varepsilon,\delta>0$, $n$ large enough, for any $l\in\N^*$ and all $\mathbf{t}=(t_1,\ldots,t_l)\in\R^l$ with $t_l\geq -B$ 
\begin{align*}
      \sum_{k\geq 1}\Psi^{k}_{n,n^b-H_l(\mathbf{t})}\big(f^{n,l+k}_{\varepsilon h_n}|\mathbf{t}\big)&\leq \mathbf{E}\Big[\underset{x\in\mathcal{O}_n}{\sum}e^{-V(x)}\frac{\un_{\{V(x)\geq(\log n)^{\alpha}-t_l\}}}{H_l(\mathbf{t})+\sum_{j\leq |x|}H_{x_j}}\un_{\{H_l(\mathbf{t})+H_x>n^b/(\log n)^2\}}\Big] \\ & \leq 2n^{-b}e^{\delta t_l-\frac{1}{b}(\log n)^{\alpha-1}(1-\frac{\varepsilon}{4A})}\leq n^{-b}e^{\delta t_l-\frac{1}{b}(\log n)^{\alpha-1}(1-\frac{\varepsilon}{3A})},
\end{align*}
where we have used \eqref{Fact1TH3bis} with $r=H_l(\mathbf{t})$, $w=1$, $t=t_l$ and $\frac{\varepsilon}{4A}$ instead of $\varepsilon$ for the last inequality. Finally, observe that 
\begin{align*}
    n^{-b}e^{\delta t_l-\frac{1}{b}(\log n)^{\alpha-1}(1-\frac{\varepsilon}{3A})}=e^{\frac{2\varepsilon}{3bA}(\log n)^{\alpha-1}}n^{-b}e^{-\frac{1}{b}(\log n)^{\alpha-1}(1+\frac{\varepsilon}{3A})}\leq e^{\frac{\varepsilon}{A}h_n}\sum_{k\geq 1}\Psi^k_{n,n^b}(f^{n,k}),
\end{align*}
where we have used that $h_n\sim\frac{1}{b}(\log n)^{\alpha-1}$ and \eqref{Fact2TH3bis} with $m=M=0$. \\  

\noindent We are left to prove that technical assumptions \eqref{Hypothese3} and \eqref{Hypothese4} are realized. \\ 
$\bullet$ For $\eqref{Hypothese3}$, recall that  $\Upsilon^k_n =\{\mathbf{t}=(t_1,\ldots,t_k)\in\R^k; H_k(\mathbf{t})\leq n^be^{\varepsilon h_n}, V(x)\geq 2\ell_n^{1/3}/\delta_1, t_k\geq-B\}$. By \eqref{GrandH}, for $|x|=k$ with $\mathbf{v}_x=(s_1,\ldots,s_{m_n},V(x_1)+s_{m_n},\ldots,V(x)+s_{m_n})$
\begin{align*}
    f^{n,k}_{\varepsilon h_n}(V(x_1),\ldots,V(x))  =\inf_{\mathbf{s}\in[-\varepsilon h_n,\varepsilon h_n]^{m_n}}\un_{\{V(x)+s_{m_n}\geq (\log n)^{\alpha}\}} \frac{1}{\sum_{j=1}^{m_n+k}H_j(\mathbf{v}_x)},
\end{align*} 
and recall that $\sum_{j=1}^{m_n+k}H_j(\mathbf{v}_x)=\sum_{j=1}^{m_n}H_j(\mathbf{s})+\sum_{j=1}^{k}\big(e^{-V(x_j)}H_{m_n}(\mathbf{s})+H_{x_j}\big)$. For $|x|=k$ such that $\underline{V}(x)\geq-B$, observe, as $\mathbf{s}\in[-\varepsilon h_n,\varepsilon h_n]^{m_n}$, that $\sum_{j=1}^{m_n+k}H_j(\mathbf{v}_x)\leq m_ne^{2\varepsilon h_n}+km_n^2e^{2\varepsilon h_n+B}+\sum_{j=1}^kH_{x_j}$. Also recall, by definition, that $h_n\geq(\log n)^{\gamma}$ for $\gamma\in(0,1)$ so $\sum_{j=1}^{m_n+k}H_j(\mathbf{v}_x)\leq 2km_n^2e^{2\varepsilon h_n+B}+\sum_{j=1}^kH_{x_j}\leq ke^{3\varepsilon h_n}+\sum_{j=1}^kH_{x_j}$. It follows that 
\begin{align*}
    f^{n,k}_{\varepsilon h_n}(V(x_1),\ldots,V(x))\geq 
    \un_{\{V(x)\geq (\log n)^{\alpha}+\varepsilon h_n\}}\Big(ke^{3\varepsilon h_n}+\sum_{j=1}^kH_{x_j}\Big)^{-1}.
\end{align*}
Let $0<\varepsilon_1<\varepsilon$ and recall $\lambda_n=ne^{-5 h_n} \geq 2 \bm{\lambda}_{n,i}$, $i\in\{1,2\}$. Thanks to the previous inequality and the fact that $(\log n)^{\alpha}>2\ell_n^{1/3}/\delta_1$, we have 
\begin{align*}
    \sum_{k\geq 1}\Psi^k_{\lambda_n/2,n^b}(f^{n,k}_{\varepsilon h_n}\un_{\Upsilon^k_n})\geq\Eb\Big[\sum_{x\in\Hline{\bm{\lambda}_{n,i}}{n^b}}\frac{e^{-V(x)}\un_{\{V(x)\geq(\log n)^{\alpha}+\varepsilon h_n\}}}{|x|e^{3\varepsilon h_n}+\sum_{j\leq|x|}H_{x_j}}\un_{\{H_x\leq n^be^{\varepsilon h_n},\underline{V}(x)\geq-B\}}\Big].
\end{align*}
Assume $b=0$. By \eqref{Fact2TH3} with $m=h_n$, $M=e^{(\log n)^{\alpha/2}}$ and $\frac{\varepsilon_1}{3}$ instead of $\varepsilon$, together with the fact that $h_n\sim 2(\log n)^{\alpha/2}$, for $n$ large enough 
\begin{align*}
    \sum_{k\geq 1}\Psi^k_{\lambda_n/2}(f^{n,k}_{\varepsilon h_n}\un_{\Upsilon^k_n})&  \geq  \Eb\Big[\sum_{x\in\T}\frac{e^{-V(x)}\un_{\{V(x)\geq(\log n)^{\alpha}+ h_n\}}}{|x|e^{(\log n)^{\alpha/2}}+\sum_{j\leq|x|}H_{x_j}}\un_{\{x\in\Upsilon_{n,1}(\frac{\varepsilon_1}{3})\cap\mathcal{O}_{\bm{\lambda}_{n,1}}\}}\Big]\\ & \geq e^{-2(\log n)^{\alpha/2}(1+\frac{\varepsilon_1}{3})}.
\end{align*}
Moreover, $e^{-2(\log n)^{\alpha/2}(1+\frac{\varepsilon_1}{3})}=e^{-\frac{4\varepsilon_1}{3}(\log n)^{\alpha/2}}e^{-2(\log n)^{\alpha/2}(1-\frac{\varepsilon_1}{3})}\geq e^{-\varepsilon_1 h_n}\sum_{k\geq 1}\Psi^k_{n,n^b}(f^{n,k})$, the last inequality comes from the fact that $h_n\sim 2(\log n)^{\alpha/2}$ and \eqref{Fact1TH3} with $t=0$. \\ 
Assume $b\in(0,1/2)$. By \eqref{Fact2TH3bis} with $m= h_n$ and $M=n^b$, together with the fact that $h_n\sim\frac{1}{b}(\log n)^{\alpha-1}$, for $n$ large enough
\begin{align*} 
    \sum_{k\geq 1}\Psi^k_{\lambda_n/2,n^b}(f^{n,k}_{\varepsilon h_n}\un_{\Upsilon^k_n})& \geq\Eb\Big[\sum_{x\in\T}\frac{e^{-V(x)}\un_{\{V(x)\geq(\log n)^{\alpha}+ h_n\}}}{|x|n^b+\sum_{j\leq|x|}H_{x_j}}\un_{\{x\in\Upsilon_{n,2}(\frac{\varepsilon_1}{3})\cap\mathcal{O}_{\bm{\lambda}_{n,2}}\}}\Big] \\ & \geq n^{-b}e^{-\frac{1}{b}(\log n)^{\alpha-1}(1+\frac{\varepsilon_1}{3})}.
\end{align*}
Moreover, $e^{-\frac{1}{b}(\log n)^{\alpha-1}(1+\frac{\varepsilon_1}{3})}=e^{-\frac{2\varepsilon_1}{3b}(\log n)^{\alpha-1}}e^{-\frac{1}{b}(\log n)^{\alpha-1}(1-\frac{\varepsilon_1}{3})}\geq n^be^{-\varepsilon_1 h_n}\sum_{k\geq 1}\Psi^k_{n,n^b}(f^{n,k})$, the last inequality comes from the fact that $h_n\sim\frac{1}{b}(\log n)^{\alpha-1}$ and \eqref{Fact1TH3bis} with $r=t=0$,  $w=1$ and we have used that  $n^{b}(\log n)^{-2}<n^b$.  \\
$\bullet$ Finally for \eqref{Hypothese4}, we first observe that for all $k\in\N^*$ and $\alpha\in(1,2)$, $(\log n)^{\alpha}>2\ell_n^{1/3}/\delta_1$ for $n$ large enough so
\begin{align*}
    \Psi^k_n(f^{n,k}\un_{\R\setminus\mathcal{H}^k_{\ell_n^{1/3}/\delta_1}})&=\Eb\Big[\sum_{|x|=k}\frac{e^{-V(x)}}{\sum_{j=1}^kH_{x_j}}\un_{\{V(x)\geq(\log n)^{\alpha},V(x)<\ell_n^{1/3}/\delta_1\}}\un_{\{x\in\mathcal{O}_n\}}\Big]=0. 
\end{align*}
Recall that $W=\sum_{|z|=1}e^{-V(z)}$ and $\Eb[W]=e^{\psi(1)}=1$ so when $b=0$
\begin{align*}
    \sum_{k\geq 1}\big(\Psi^k_{n,n^b/(\log n)^2}(f^{n,k}) +\Eb\big[W\Psi^k_{n,n^b/(W(\log n)^2)}(f^{n,k})\big]\big)&\leq\sum_{k\geq 1}\big(\Psi^k_{n}(f^{n,k}) +\Eb\big[W\Psi^k_{n}(f^{n,k})\big]\big) \\ & =2\sum_{k\geq 1}\Psi^k_n(f^{n,k}),
\end{align*}
and thanks to \eqref{Fact1TH3} for $n$ large enough with $t=0$
\begin{align*}
    2\sum_{k\geq 1}\Psi^k_n(f^{n,k})=2\Eb\Big[\sum_{x\in\mathcal{O}_n}\frac{e^{-V(x)}}{\sum_{j\leq |x|}H_{x_j}}\un_{\{V(x)\geq(\log n)^{\alpha}\}}\Big]&\leq 2e^{-2(\log n)^{\alpha/2}(1-\frac{\varepsilon_1}{4})} \\ & \leq e^{-2(\log n)^{\alpha/2}(1-\frac{\varepsilon_1}{3})}.
\end{align*}
Moreover, $e^{-2(\log n)^{\alpha/2}(1-\frac{\varepsilon_1}{3})}=e^{\frac{4\varepsilon_1}{3}(\log n)^{\alpha/2}}e^{-2(\log n)^{\alpha/2}(1+\frac{\varepsilon_1}{3})}\leq e^{\varepsilon_1 h_n}\sum_{k\geq 1}\Psi^k_{n,n^b}(f^{n,k})$, the last inequality comes from the fact that $h_n\sim 2(\log n)^{\alpha/2}$ and \eqref{Fact2TH3} with $m=M=0$. \\  Otherwise, $b\in(0,1/2)$ and thanks to \eqref{Fact1TH3bis} for $n$ large enough with $r=t=0$, $w=1$ and $\frac{\varepsilon_1}{4}$ instead of $\varepsilon$
\begin{align*}
    \sum_{k\geq 1}\Psi^k_{n,n^b/(\log n)^2}(f^{n,k})=\Big[\sum_{x\in\mathcal{O}_n}\frac{e^{-V(x)}\un_{\{H_x>n^b/(\log n)^2\}}}{r+\sum_{j\leq |x|} H_{x_j}}\un_{\{V(x) \geq (\log n)^{\alpha}\}}\Big]\leq \frac{1}{n^b}e^{-\frac{1}{b}(\log n)^{\alpha-1}(1-\frac{\varepsilon_1}{3})},
\end{align*}
and we also get from \eqref{Fact1TH3bis} with $r=t=0$ and $w=W$ that for $n$ large enough
\begin{align*}
    \Psi^k_{n,n^b/(W(\log n)^2)}(f^{n,k})&=\mathbf{E}\Big[\sum_{x\in\mathcal{O}_n}\frac{e^{-V(x)}\un_{\{H_x>n^b/(W(\log n)^2)\}}}{r+\sum_{j\leq |x|} H_{x_j}}\un_{\{V(x) \geq (\log n)^{\alpha}\}}\Big] \\ & \leq \frac{W+1}{n^b}e^{-\frac{1}{b}(\log n)^{\alpha-1}(1-\frac{\varepsilon_1}{4})}.
\end{align*}
By \eqref{hyp1}, telling that $\Eb[W^2]<\infty$, we have $C_4:=\Eb[W(W+1)+1]=\Eb[W^2+2]<\infty$ and then
\begin{align*}
    \sum_{k\geq 1}\big(\Psi^k_{n,n^b/(\log n)^2}(f^{n,k}) +\Eb\big[W\Psi^k_{n,n^b/(W(\log n)^2)}(f^{n,k})\big]\big)&\leq\frac{C_4}{n^b}e^{-\frac{1}{b}(\log n)^{\alpha-1}(1-\frac{\varepsilon_1}{4})} \\ & \leq \frac{1}{2n^b}e^{-\frac{1}{b}(\log n)^{\alpha-1}(1-\frac{\varepsilon_1}{3})}.
\end{align*}
Moreover, $e^{-\frac{1}{b}(\log n)^{\alpha-1}(1-\frac{\varepsilon_1}{3})}=e^{\frac{\varepsilon_1}{3b}(\log n)^{\alpha-1}}e^{-\frac{1}{b}(\log n)^{\alpha-1}(1+\frac{\varepsilon_1}{3})}\leq n^be^{\varepsilon_1 h_n}\sum_{k\geq 1}\Psi^k_{n,n^b}(f^{n,k})$, the last inequality comes from the fact that $h_n\sim\frac{1}{b}(\log n)^{\alpha-1}$ and \eqref{Fact2TH3bis} with $m=M=0$. This completes the proof for these two cases.

\noindent\textbf{Assume now $\alpha=1$ and $a\in\R$ (with $a>1/\delta_1$ when $d=1$)}, which corresponds to the first case of the theorem.
%Recall that $f^{n,k}(t_1,t_2, \cdots,t_k)= \un_{\{t_{k}\geq a(\log n)^{\alpha}\}}(\sum_{j=1}^{k} e^{t_j-t_k})^{-d}$, $a>0$, $\alpha\in[1,2)$ and $d\in\{0,1\}$. %The proof is devided into two steps. The first one is to prove some facts about $\mathbf{f}^n$. The second is to check that for all $n\geq 1$, $\mathbf{f}^n$ actually satisfies main assumptions \eqref{Hypothese1}, \eqref{Hypothese2} and \eqref{Hypothese3}.\\
%When $\alpha=1$ and $a>d\un_{\{b>0\}}/\delta_1$, 
As usual, let us first state the following two facts: \\
for all $b\in[0,1/(d+1))$, $B,\delta>0$, $\varepsilon\in(0,\varepsilon_b)$ and $n$ large enough, for any $t\geq -B$, $r\geq 0$ and $w>0$
\begin{align}
    \mathbf{E}\Big[\sum_{x\in\mathcal{O}_n}\frac{e^{-V(x)}\un_{\{r+H_x>n^b/(w(\log n)^2)\}}}{\big(r+\sum_{j\leq |x|} H_{x_j}\big)^d}\un_{\{V(x)\geq a\log n-t\}}\Big]\leq (w+1)\ell_n^2 e^{\delta t}n^{-bd},
    \label{Fact1TH3'}
\end{align}
For any $0\leq M\leq n^b$, $\varepsilon<b/3$ (when $b>0$)
\begin{align}
    \mathbf{E}\Big[\sum_{x\in\T}\frac{e^{-V(x)}\un_{\{x\in\Upsilon_n\cap\Hline{\bmlambda}{n^b}\}}}{\big(M|x|+\sum_{j\leq |x|} H_{x_j}\big)^d}\un_{\{V(x)\geq a\log n\}}\Big]\geq \frac{1}{\ell_n^2}n^{-bd},
    \label{Fact2TH3'}
\end{align}
with $\bmlambda=n^{1-11\varepsilon}$ and for any $a'>1/\delta_1$
\begin{align*}
    \Upsilon_{n}=\Upsilon_{n}(\varepsilon):=\{x\in\T;H_x\leq n^{b+\varepsilon}, V(x)\geq a'\log n,\underline{V}(x)\geq-B\}.
\end{align*}
These facts ensure that $\mathbf{f}^n$ satisfies assumptions \eqref{Hypothese1}, \eqref{Hypothese2}, \eqref{Hypothese3} and \eqref{Hypothese4} for $b\in(0,1/(d+1))$. \eqref{Hypothese3} does not hold exactly when $b=0$ so we use \eqref{UpperTH5} (which appears in the proof of Theorem \ref{Prop2}) together with the result when $b>0$ to conclude this case. \\ %(so this part $\alpha=1$ of the  theorem is still a corollary of the main general Theorem).  \\ \\
$\bullet$ Check of \eqref{Hypothese1} and asymptotic of $h_n$. We get from \eqref{Fact2TH3'} that $\kappa_b=\max\mathcal{U}_b\leq bd$ and \eqref{Fact1TH3'} gives $\kappa_b\geq bd$. It follows that for all $b\in[0,1/(d+1))$, $\kappa_b=bd$ and for any $n\geq 2$, $h_n=\log n$. Indeed, on the one hand, \eqref{Fact1TH3'} with $r=t=0$ and $w=1$ gives, for $n$ large enough
\begin{align*}
    n^{\kappa_b}\sum_{k\geq 1}\Psi^k_{n,n^b}(f^{n,k})=n^{bd}\Eb\Big[\sum_{x\in\Hline{n}{n^b}}\frac{e^{-V(x)}}{\big(\sum_{j=1}^kH_{x_j}\big)^d}\un_{\{V(x)\geq a\log n\}}\Big]\leq 2^d\ell_n^2,
\end{align*}
and on the other hand, we get from \eqref{Fact2TH3'}, for $n$ large enough that 
\begin{align*}
    n^{\kappa_b}\sum_{k\geq 1}\Psi^k_{n,n^b}(f^{n,k})\geq\frac{1}{\ell_n^2}.
\end{align*}
From these inequalities, we get that for any $\gamma\in(0,1)$, $|\log(n^{\kappa_b}\sum_{k\geq 1}\Psi^k_{n,n^b}(f^{n,k}))|\leq 3\log\ell_n=o((\log n)^{\gamma})$. Then $h_n=\log n$ and we also deduce that \eqref{Hypothese1} is satisfied. \\
$\bullet$ For \eqref{Hypothese2}, let $|x|=k$ and observe that $f^{n,l+k}_{\varepsilon h_n}(t_1,\ldots,t_l,V(x_1),\ldots,V(x))\leq (H_l(\mathbf{t})+H_x)^{-d}$ so it follows, for all $\varepsilon\in(0,\varepsilon_b)$, $A,\delta,B>0$, $n$ large enough, any $l\in\N^*$, $\mathbf{t}=(t_1,\ldots,t_l)\in\R^l$ and $t_l\geq-B$, by \eqref{Fact1TH3'} with $r=H_l(\mathbf{t})$, $t=t_l$ and $w=1$
\begin{align*}
    \sum_{k\geq 1}\Psi^{k}_{n,n^b-H_l(\mathbf{t})}\big(f^{n,l+k}_{\varepsilon h_n}|\mathbf{t}\big)\leq 2^d\ell_n^2 e^{\delta t_l}n^{-bd}\leq e^{\delta t_l+\frac{\varepsilon}{A}h_n}\sum_{k\geq 1}\Psi^k_{n,n^b}(f^{n,k}),
\end{align*}
where the last inequality comes from \eqref{Fact2TH3'}. \\
$\bullet$ For $\eqref{Hypothese3}$, recall that  $\Upsilon^k_n =\{\mathbf{t}=(t_1,\ldots,t_k)\in\R^k; H_k(\mathbf{t})\leq n^{b+\varepsilon}, V(x)\geq 2\ell_n^{1/3}/\delta_1, t_k\geq-B\}$. For $|x|=k$, we have
\begin{align*}
    f^{n,k}_{\varepsilon h_n}(V(x_1),\ldots,V(x))\geq 
    \un_{\{V(x)\geq (a+\varepsilon)(\log n)\}}\Big(kn^{3\varepsilon}+\sum_{j=1}^kH_{x_j}\Big)^{-d},
\end{align*}
and thanks to \eqref{Fact2TH3'} with $M=n^b$, $b\in(0,1/(d+1))$
\begin{align*}
    \sum_{k\geq 1}\Psi^k_{\lambda_n/2,n^b}(f^{n,k}_{\varepsilon h_n}\un_{\Upsilon^k_n})\geq\Eb\Big[\sum_{x\in\T}\frac{e^{-V(x)}\un_{\{\Upsilon_n\cap\Hline{\bmlambda}{n^b}\}}}{(|x|n^{b}+\sum_{j\leq|x|}H_{x_j})^d}\Big]\geq \frac{1}{\ell_n^2}n^{-bd}\geq e^{-\varepsilon_1h_n}\sum_{k\geq 1}\Psi^k_{n,n^b}(f^{n,k}),
\end{align*}
where we recall $\lambda_n=n^{1-10\varepsilon}$. \\
$\bullet$ Finally, for \eqref{Hypothese4} with $d=1$ (and then $a>1/\delta_1$)
\begin{align*}
    \Psi^k_n(f^{n,k}\un_{\R\setminus\mathcal{H}^k_{\ell_n^{1/3}/\delta_1}})&=\Eb\Big[\sum_{|x|=k}\frac{e^{-V(x)}\un_{\{x\in\mathcal{O}_n\}}}{\big(\sum_{j=1}^kH_{x_j}\big)^d}\un_{\{V(x)\geq a\log n,V(x)<\ell_n^{1/3}/\delta_1\}}\Big]=0. 
\end{align*}
Otherwise, $d=0$ and for any $a\in\R$, thanks to Remark \ref{Rem1}
\begin{align*}
    \sum_{k\geq 1}\Psi^k_n(f^{n,k}\un_{\R\setminus\mathcal{H}^k_{\ell_n^{1/3}/\delta_1}})&=\Eb\Big[\sum_{x\in\mathcal{O}_n}e^{-V(x)}\un_{\{V(x)\geq a\log n,V(x)<\ell_n^{1/3}/\delta_1\}}\Big]\leq\Eb\Big[\sum_{x\in\mathcal{O}_n}e^{-V(x)}\Big]\leq\ell_n,
\end{align*}
which, thanks to \eqref{Fact2TH3'}, is smaller than $e^{\varepsilon_1h_n}\sum_{k\geq 1}\Psi^k_{n,n^b}(f^{n,k})$ for all $\varepsilon_1>0$. We get from \eqref{Fact1TH3'} with $r=t=0$ and $w=W$ that for $n$ large enough
\begin{align*}
    \Psi^k_{n,n^b/(W(\log n)^2)}(f^{n,k})=\mathbf{E}\Big[\sum_{x\in\mathcal{O}_n}\frac{e^{-V(x)}\un_{\{H_x>n^b/(W(\log n)^2)\}}}{(\sum_{j\leq |x|} H_{x_j})^d}\un_{\{V(x) \geq a\log n\}}\Big]\leq (W+1)\ell_n^2 n^{-bd}.
\end{align*}
By \eqref{hyp1}, telling that $\Eb[W^2]<\infty$, we have $C_4:=\Eb[W(W+1)+1]=\Eb[W^2+2]<\infty$ and then
\begin{align*}
    \sum_{k\geq 1}\big(\Psi^k_{n,n^b/(\log n)^2}(f^{n,k}) +\Eb\big[W\Psi^k_{n, n^b/(W(\log n)^2)}(f^{n,k})\big]\big)\leq 2C_4\ell n^{-bd}\leq e^{\varepsilon_1h_n}\sum_{k\geq 1}\Psi^k_{n,n^b}(f^{n,k}),
\end{align*}
where, again, the last inequality comes from \eqref{Fact2TH3'}. This finishes the proof of the result of the theorem for $b\in(0,1/(d+1))$. \\
Now assume $b=0$ and let $\varepsilon>0$. Using the result of the theorem with $b_{\varepsilon}=\varepsilon/(2+d)$ and the fact that $\mathcal{R}_n(\un_{[n^{b_{\varepsilon}},\infty)},\mathbf{f}^n)\leq\mathcal{R}_n(\un_{[1,\infty)},\mathbf{f}^n)$, we get the following lower bound for $\mathcal{R}_n(\un_{[1,\infty)},\mathbf{f}^n)$: $\P(\log^+\mathcal{R}_n(\un_{[1,\infty)},\mathbf{f}^n)<(1-\varepsilon)\log n)$ is smaller than
\begin{align*}
    \P\big(\log^+\mathcal{R}_n(\un_{[n^{b_{\varepsilon}},\infty)},\mathbf{f}^n)<(1-(1+d)b_{\varepsilon}-\varepsilon/(2+d))\log n\big)\to 0, 
\end{align*}
where we have used the case $b>0$.
For the upper bound, we use an intermediate result in the proof of Theorem \ref{Prop2}:  recall that $\kappa_{0}=0$ and $h_n=\log n$. \\ Also recall $\xi=\lim_{n\rightarrow \infty}h_n^{-1}\log(n^{\kappa_b}\sum_{k\geq 1}\Psi^{k}_{n,n^b} (f^{n,k}))$. It's easy to see that $\xi=0$ and by \eqref{UpperTH5}
\begin{align*}
    \P(\log^+\mathcal{R}_n(\un_{[1,\infty)},\mathbf{f}^n)>(1+\varepsilon)\log n)\leq \P\Big(\frac{1}{n}\mathcal{R}_{n}(\un_{[1,\infty)},\mathbf{f}^n)>e^{\varepsilon h_n}\Big)\to 0,
\end{align*}
 this ends the proof of the theorem for all $b\in[0,1/(d+1))$. \end{proof}

\begin{proof}[Proof of Theorem \ref{thm4}] Here $f^{n,k}(t_1,t_2, \cdots,t_k)= \un_{\{t_{\lfloor k/\beta \rfloor} \geq (\log n)^{\alpha}\}}(\sum_{j=1}^{\lfloor k/\beta \rfloor}H_j(\mathbf{t}))^{-1}$ with $\beta>1$ and $\alpha\in(1,2)$. We state the following facts: for all $B,\delta>0$, $\varepsilon\in(0,\varepsilon_b)$, $n$ large enough, any $t\geq -B$ and $i\in\N$
\begin{align}
    \mathbf{E}\Big[\sum_{x;\lfloor |x|+i/\beta\rfloor>i}\frac{e^{-V(x)}\un_{\{x\in\mathcal{O}_n\}}}{\sum_{j=1}^{\lfloor |x|/\beta\rfloor} H_{x_j}}\un_{\{V(x_{\lfloor (|x|+i)/\beta\rfloor-i})\geq (\log n)^{\alpha}-t\}}\Big]\leq e^{\delta t-2(\log n)^{\alpha/2}(1-\varepsilon)},
    \label{Fact1TH4}
\end{align}
and for all $0\leq i,m\leq\log n$, $0\leq M\leq e^{(\log n)^{\alpha/2}}$
\begin{align}
    \mathbf{E}\Big[\sum_{x;\lfloor |x|+i/\beta\rfloor>i}\frac{e^{-V(x)}\un_{\{x\in\Upsilon_{n}\cap\mathcal{O}_{\bm{\lambda}_{n}}\}}}{M|x|+\sum_{j=1}^{\lfloor |x|/\beta\rfloor} H_{x_j}}\un_{\{V(x_{\lfloor (|x|+i)/\beta\rfloor-i})\geq (\log n)^{\alpha}+m\}}\Big]\geq e^{-2(\log n)^{\alpha/2}(1+\varepsilon)},
    \label{Fact2TH4}
\end{align}
 with $\bm{\lambda}_{n}=ne^{-12(\log n)^{\alpha/2}}$ and for any $a>\frac{1}{\delta_1}$
\begin{align*}
    \Upsilon_{n}=\Upsilon_{n}(\varepsilon):=\{x\in\T;H_x\leq e^{2\varepsilon(\log n)^{\alpha/2}}, V(x)\geq a\log n,\underline{V}(x)\geq-B\}.
\end{align*}
Using these two facts, we follow the same lines as in the previous theorem to prove that $h_n\sim 2(\log n)^{\alpha/2}$ and that \eqref{Hypothese1} to \eqref{Hypothese4} are satisfied.
%\DO{Que dire? Gr√¢ce aux facts, $\kappa_b=0$ pour tout $b\in[0,1)$, $h_n\sim 2(\log n)^{\alpha/2}$ et de m√™me que dans la preuve du TH3 quand $b$ valait 0, les assumptions sont verifiees.} 
\end{proof}

\subsection{Proof of Theorem \ref{Prop2}} \label{sec2.3}

First, note that Remark \ref{rem0}  implies that $\xi=\lim_{n\rightarrow \infty}h_n^{-1}\log( n^{\kappa_b}\sum_{k\geq 1}\Psi^{k}_{n,n^b}(f^{n,k}))$ well exists. To prove Theorem \ref{Prop2}, we first show that Assumptions \eqref{Hypothese3} and \eqref{Hypothese4} yield a simpler statement for both lower and upper bound of Proposition \ref{Prop1}. This implies a convergence in probability for stopped ranges $\mathcal{R}_{T^{k_n}}$  with $k_{n} =\lceil n/(\log n)^{3/2}\rceil$ and $\mathcal{R}_{T^{n}}$. %\noindent\DO{Note also that  by  $\mathcal{R}_{T^{k_n}}(g_n,\mathbf{f}^n)$ with $k_n:=\lceil n/(\log n)^{\theta_1}\rceil$, $\theta_1\geq 0$.} \bl{Ici ???} \\
%Existence of $\xi$ :  recall \eqref{Def_hn}, so if $\gamma>0$ is such that $(\log n)^{\gamma}/\log(n^{\kappa_b}\sum_{k\geq 1}\Psi^{k}_{n,n^b} (f^{n,k}))\to 0$ as by definition of $\mathcal{U}_b$ and Remark \ref{Rem1} $\log(n^{\kappa_b}\sum_{k\geq 1}\Psi^{k}_{n,n^b} (f^{n,k}))\leq \log(C_{\infty} \sum_{k\geq 1} \Psi^{k}_{n}(1) ) \leq \log(C_{\infty}\ell_n)$ so necessarily for $n$ large enough $\log(n^{\kappa_b}\sum_{k\geq 1}\Psi^{k}_{n,n^b} (f^{n,k}))<-(\log n)^{\gamma}$, which implies that $\xi$  exists and is equal to $-1$. Otherwise $h_n=\log n$ (by assumption \eqref{Hypothese1}   \red{pourquoi on a besoin de cette assumption ?}) so $\xi$ exists and is equal to $0$. \\ 
%\red{\underline{R√É¬É√Ç¬©ponse A}: on a besoin de \eqref{Hypothese1} car cela implique que $\liminf(\log n)^{-1}\log\big(n^{\kappa_b}\sum_{k\geq 1}\Psi^k_{n,n^b}\big)\geq 0$. De plus, par d√É¬É√Ç¬©finition de $\kappa_b$ et d'apr√É¬É√Ç¬®s la remarque \ref{Rem1}, $\limsup(\log n)^{-1}\log\big(n^{\kappa_b}\sum_{k\geq 1}\Psi^k_{n,n^b}\big)\leq 0$ donc dans le cas $h_n=\log n$ pour tout $n\geq 2$, on obtient bien que $\xi=0$.} \\
%\noindent\DO{Version du 25/09}
%\begin{proof}
Then, we use a result of \cite{HuShi15} (Proposition 2.4) implying that $T_n/(n \log n)$ converges in probability to a positive limit in order to obtain the result for $\mathcal{R}_{{n}}$. Let us start with the

%Recall that $\xi$. \\
\noindent \textit{Lower bound :} Recalling the expression of $u_{1,n}=\sum_{k\geq 1}\Psi^{k}_{\lambda_n/2,n^b}(f^{n,k}_{\varepsilon h_n}\mathds{1}_{\Upsilon^k_n})$ (see below \eqref{LowerProp1}), together with  \eqref{Hypothese3} choosing $\varepsilon_1=\min(1,c_5)\frac{\varepsilon}{4}$ (see Proposition \ref{Prop1} for $c_5$), we get    
\begin{align*}
 u_{1,n}\geq e^{-\min(1,c_5)\frac{\varepsilon}{4}} \sum_{k\geq 1}\Psi^{k}_{n,n^b}\big(f^{n,k}\big).
  %  u_{1,n}\geq e^{-\frac{a\varepsilon}{4}h_n} \sum_{k\geq 1}\Psi^{k}_{n,n^b}\big(f^{n,k}\big)\hspace{0.5cm}  \forall a\in\{1,c_5\}
\end{align*}
This, together with the fact that, by definition of $\xi$, $n^{\kappa_b}\sum_{k\geq 1}\Psi^k_{n,n^b}(f^{n,k})>e^{(\xi-\varepsilon)h_n}$ for $n$ large enough, implies 
\begin{align*}  
    \P\Big(\frac{\mathcal{R}_{T^{k_n}}(g_n,\mathbf{f}^n)}{n^{1-b-\kappa_b}\varphi(n^b)}<e^{(\xi-7\varepsilon) h_n}\Big)&\leq\P\Big(\frac{\mathcal{R}_{T^{k_n}}(g_n,\mathbf{f}^n)}{n^{1-b}\varphi(n^b)\sum_{k\geq 1}\Psi^{k}_{n,n^b}(f^{n,k})}<e^{-6\varepsilon h_n}\Big) \\ & 
    \leq\P\Big(\frac{\mathcal{R}_{T^{k_n}}(g_n,\mathbf{f}^n)}{n^{1-b}\varphi(n^b) u_{1,n}}<e^{-5\varepsilon h_n}\Big). %\leq e^{-\frac{c_5}{10}\varepsilon h_n}+ h_ne^{-\frac{\varepsilon}{5c_6}h_n}+e^{-3h_n+2|\log(n^{\kappa_b}\sum_{k\geq 1}\Psi^{k}_{n,n^b}(f^{n,k}))|}
\end{align*}
%implies that (as $\min(1,c_5)\frac{\varepsilon}{4}\leq \varepsilon$), and 
Also considering \eqref{LowerProp1}, $\P\big(\frac{\mathcal{R}_{T^{k_n}}(g_n,\mathbf{f}^n)}{n^{1-b}\varphi(n^b) u_{1,n}}<e^{-5\varepsilon h_n}\big)$ is smaller than
\begin{align*}
  &e^{(-c_5+\frac{\min(1,c_5)}{2})\varepsilon h_n}+ h_ne^{-\varepsilon h_n}+ \frac{e^{-\min(\varepsilon\log n,3h_n)+\min(1,c_5)\frac{\varepsilon}{2}h_n}}{\big(n^{\kappa_b} \sum_{k\geq 1}\Psi^{k}_{n,n^b}(f^{n,k})\big)^2}  \\
 &  \leq e^{-\frac{\varepsilon c_5}{2}h_n}+ h_ne^{-\varepsilon h_n}+e^{-\min(\varepsilon\log n,3h_n)+\frac{\varepsilon}{2}h_n+2|\log(n^{\kappa_b}\sum_{k\geq 1}\Psi^{k}_{n,n^b}(f^{n,k}))|}.
 %\leq e^{-\frac{c_5}{10}\varepsilon h_n}+ h_ne^{-\frac{\varepsilon}{5c_6}h_n}+e^{-3h_n+2|\log(n^{\kappa_b}\sum_{k\geq 1}\Psi^{k}_{n,n^b}(f^{n,k}))|}
\end{align*}
Now, thanks to Remark \ref{rem0},  for $n$ large enough, $|\log(n^{\kappa_b}\sum_{k\geq 1}\Psi^{k}_{n,n^b}(f^{n,k}))|$ is smaller than $\frac{\varepsilon}{8}\log n\leq
-\min(-\frac{\varepsilon}{8}\log n,-h_n)$ and $\frac{\varepsilon}{2}h_n$ is smaller than $\leq-\frac{1}{2}\min(-\varepsilon\log n,-h_n)$ so $-\min(\varepsilon\log n,3h_n)+\frac{\varepsilon}{2}h_n+2|\log(n^{\kappa_b}\sum_{k\geq 1}\Psi^{k}_{n,n^b}(f^{n,k}))|$ is smaller than  $-\frac{1}{2}\min(\varepsilon\log n,h_n)$. Finally, for all $\varepsilon\in(0,\varepsilon_b)$ and $n$ large enough
\begin{align*}
    \P\Big(\frac{\mathcal{R}_{T^{k_n}}(g_n,\mathbf{f}^n)}{n^{1-b-\kappa_b}\varphi(n^b)}<e^{(\xi-7\varepsilon) h_n}\Big)\leq  e^{-\frac{\varepsilon c_5}{2}h_n}+ h_ne^{-\varepsilon \tilde{c}_2h_n}+e^{-\frac{1}{4}\min(\varepsilon\log n,h_n)},
\end{align*}
then switching $\varepsilon$ by   ${\varepsilon}/{7}$ in the above probability, we obtain  as $h_n \rightarrow+ \infty$,  the desired expression : for all $\varepsilon\in(0,7\varepsilon_b)$
\begin{align*}
    \lim_{n\to\infty}\P\Big(\frac{\mathcal{R}_{T^{k_n}}(g_n,\mathbf{f}^n)}{n^{1-b-\kappa_b}\varphi(n^b)}<e^{(\xi-\varepsilon) h_n}\Big)=0.
\end{align*}
We are now ready to move from $\mathcal{R}_{T^{k_n}}$ to $\mathcal{R}_{n}$. First note that 
\begin{align*}
    \P\Big(\frac{\mathcal{R}_{n}(g_n,\mathbf{f}^n)}{n^{1-b-\kappa_b}\varphi(n^b)}<e^{(\xi-\varepsilon) h_n}\Big)\leq \P\Big(\frac{\mathcal{R}_{n}(g_n,\mathbf{f}^n)}{n^{1-b-\kappa_b}\varphi(n^b)}<e^{(\xi-\varepsilon) h_n},T^{k_n}\leq n\Big) + \P(T^{k_n}>n),
\end{align*}
recalling that $\Ran_n(g_n,\mathbf{f}^n)=\sum_{x \in \T}g_n(\mathcal{L}_x^n) f^{n,|x|}(V(x_1),V(x_2),\cdots,V(x))$ and $g_n(t)=\varphi(t)\un_{\{t\geq n^b\}}$ with $b\in[0,1)$. Then, as $\varphi$ is non-decreasing and positive, so is $g_n$, hence $T^{k_n}\leq n$ implies $g_n(\mathcal{L}_x^{T^{k_n}})\leq g_n(\mathcal{L}_x^n)$ and therefore $\Ran_{T^{k_n}}(g_n,\mathbf{f}^n)\leq \Ran_n(g_n,\mathbf{f}^n)$ since $f^{n,k}\geq 0$. It follows that 
\begin{align*}
     \P\Big(\frac{\mathcal{R}_{n}(g_n,\mathbf{f}^n)}{n^{1-b-\kappa_b}\varphi(n^b)}<e^{(\xi-\varepsilon) h_n}\Big)\leq  \P\Big(\frac{\mathcal{R}_{T^{k_n}}(g_n,\mathbf{f}^n)}{n^{1-b-\kappa_b}\varphi(n^b)}<e^{(\xi-\varepsilon) h_n}\Big)+\P(T^{k_n}>n),
\end{align*}
and thanks to the above convergence, together with the fact that $(T^n/(n\log n))_n$ convergences in $\P$-probability to an almost surely finite and positive random variable, we obtain the desired expression: for all $\varepsilon\in(0,7\varepsilon_b)$:
\begin{align}\label{LowerTH5}
    \lim_{n\to\infty}\P\Big(\frac{\mathcal{R}_{n}(g_n,\mathbf{f}^n)}{n^{1-b-\kappa_b}\varphi(n^b)}<e^{(\xi-\varepsilon) h_n}\Big)=0.
   % \tag{\textbf{L2}}
\end{align}

\noindent \textit{Upper bound}: we prove the following statement, for all $\varepsilon>0$
\begin{align}\label{UpperTH5}
    \lim_{n\to\infty}\P\Big(\frac{\mathcal{R}_{n}(g_n,\mathbf{f}^n)}{n^{1-b-\kappa_b}\varphi(n^b)}>e^{(\xi+\varepsilon) h_n}\Big)=0.
    %\tag{\textbf{U2}}
\end{align}
Recall that $u_{2,n}=\sum_{k\geq 1}(\Psi^{k}_{n} (f^{n,k}\un_{\R^k\setminus \mathcal{H}^k_{\ell_n^{1/3}/\delta_1}})+\Psi^{k}_{n,n^b/(\log n)^2}(f^{n,k})+\Eb[W\Psi^k_{n,n^b/(W(\log n)^2)}(f^{n,k})])$. Assumption \eqref{Hypothese4} with $\varepsilon_1=\frac{\varepsilon}{4}$ gives that 
\begin{align*}
    u_{2,n}\leq e^{\frac{\varepsilon}{4}h_n}\sum_{k\geq 1}\Psi^k_{n,n^b}(f^{n,k}), 
\end{align*}
so for $n$ large enough, as $n^{\kappa_b }\sum_{k\geq 1}\Psi^k_{n,n^b}(f^{n,k})\leq e^{(\xi+\frac{\varepsilon}{2})h_n}$ and $T^n\geq n$
\begin{align*}
   \P\Big(\frac{\mathcal{R}_{n}(g_n,\mathbf{f}^n)}{n^{1-b-\kappa_b}\varphi(n^b)}>e^{(\xi+\varepsilon) h_n}\Big)&\leq\P\Big(\frac{\mathcal{R}_{T^n}(g_n,\mathbf{f}^n)}{n^{1-b-\kappa_b}\varphi(n^b)}>e^{(\xi+\varepsilon) h_n}\Big) \\ & \leq  \P\Big(\frac{\mathcal{R}_{T^n}(g_n,\mathbf{f}^n)}{n^{1-b}\varphi(n^b)\sum_{k\geq 1}\Psi^{k}_{n,n^b}(f^{n,k})}>e^{\frac{\varepsilon}{2}h_n}\Big) \\ & \leq\P\Big(\frac{\mathcal{R}_{T^n}(g_n,\mathbf{f}^n)}{n^{1-b}\varphi(n^b)u_{2,n}}>e^{\frac{\varepsilon}{4} h_n}\Big)\leq e^{-\frac{\varepsilon}{8}h_n}+o(1),
\end{align*}
where the last inequality comes from \eqref{UpperProp1} replacing  $\varepsilon$ by $\frac{\varepsilon}{4}$. Then, taking the limit, we get \eqref{UpperTH5}. \\ 
We are now ready to prove the theorem. We split this proof in three parts depending on the values of (recall) $L=\liminf_{n\rightarrow \infty}h_n^{-1}\log \big(n^{1-b-\kappa_b}\varphi(n^b)\big)$. \\
$\bullet$ Assume $L\in(-\xi,+\infty]$. For any $t\in\R$, $e^{\log^+t}=e^{\log(t\lor 1)}\geq t$ so for any $\varepsilon\in(0,\varepsilon_b)$ and $n$ large enough, $\P\big(\log^+\mathcal{R}_{n}(g_n,\mathbf{f}^n)-\log(n^{1-b-\kappa_b}\varphi(n^b))<(\xi-\varepsilon)h_n\big)$ is smaller than
\begin{align*}
    \P\big(e^{\log^+\mathcal{R}_{n}(g_n,\mathbf{f}^n)}<n^{1-b-\kappa_b}\varphi(n^b)e^{(\xi-\varepsilon) h_n}\big)\leq \P\Big(\frac{\mathcal{R}_{n}(g_n,\mathbf{f}^n)}{n^{1-b-\kappa_b}\varphi(n^b)}<e^{(\xi-\varepsilon) h_n}\Big)\to 0,
\end{align*}
where the limit comes from \eqref{LowerTH5}. Note that this lower bound remains true even when $L\not\in(-\xi,+\infty]$. However, we need that $L\in(-\xi,+\infty]$ for the upper bound. Indeed, in this case, for $n$ large enough, $n^{1-b-\kappa_b}\varphi(n^b)>e^{-\xi h_n}$ and for any $\varepsilon>0$, $n^{1-b-\kappa_b}\varphi(n^b)e^{(\xi+\varepsilon)h_n}>e^{\varepsilon h_n}>1$ so for $n$ large enough $\P(\log^+\mathcal{R}_{n}(g_n,\mathbf{f}^n)-\log(n^{1-b-\kappa_b}\varphi(n^b))>(\xi+\varepsilon)h_n)=\P(\log^+\mathcal{R}_{n}(g_n,\mathbf{f}^n)>\log(n^{1-b-\kappa_b}\varphi(n^b)e^{(\xi+\varepsilon)h_n}),\mathcal{R}_{n}(g_n,\mathbf{f}^n)>1)$. Moreover, when $\mathcal{R}_{n}(g_n,\mathbf{f}^n)>1$, $\log^+\mathcal{R}_{n}(g_n,\mathbf{f}^n)=\log\mathcal{R}_{n}(g_n,\mathbf{f}^n)$ so the previous probability is equal to 
\begin{align*}
   \P\big(\log\mathcal{R}_{n}(g_n,\mathbf{f}^n)>\log(n^{1-b-\kappa_b}\varphi(n^b)e^{(\xi+\varepsilon)h_n}),\mathcal{R}_{n}(g_n,\mathbf{f}^n)>1\big)\leq\P\Big(\frac{\mathcal{R}_{n}(g_n,\mathbf{f}^n)}{n^{1-b-\kappa_b}\varphi(n^b)}>e^{(\xi+\varepsilon) h_n}\Big).
\end{align*}
Then, taking the limit, we get the result thanks to \eqref{UpperTH5}. \\
$\bullet$ Assume $L=-\xi$. Recall that $\underline{\Delta}_n=h_n^{-1}\log(n^{1-b-\kappa_b}\varphi(n^b))-\inf_{\ell\geq n}h_{\ell}^{-1} \log(\ell^{1-b-\kappa_b}\varphi(\ell^b))$. $L=-\xi$ implies that for any $\varepsilon\in(0,\varepsilon_b)$ and $n$ large enough, $\inf_{\ell\geq n}h_{\ell}^{-1}\log(\ell^{1-b-\kappa_b}\varphi(\ell^b))>-\xi-\frac{\varepsilon}{2}$ so $h_n\underline{\Delta}_n<\log(n^{1-b-\kappa_b}\varphi(n^{b}))+(\xi+\frac{\varepsilon}{2})h_n$ and as $e^{\log^+t}\geq t$
\begin{align*}
    \P\big(h_n^{-1}\log^+\mathcal{R}_{n}(g_n,\mathbf{f}^n)<-\varepsilon+\underline{\Delta}_n\big)&\leq \P\big(\mathcal{R}_{n}(g_n,\mathbf{f}^n)<e^{-\varepsilon h_n+h_n\underline{\Delta}_n}\big) \\ & \leq \P\Big(\frac{\mathcal{R}_{n}(g_n,\mathbf{f}^n)}{n^{1-b-\kappa_b}\varphi(n^b)}<e^{(\xi-\frac{\varepsilon}{2}) h_n}\Big)\to 0,
\end{align*}
where the limit comes from \eqref{LowerTH5}. Also, $L=-\xi$ implies that for any $\varepsilon\in(0,\varepsilon_b)$ and $n$ large enough, $\inf_{\ell\geq n}h_{\ell}^{-1}\log(\ell^{1-b-\kappa_b}\varphi(\ell^b))<-\xi+\frac{\varepsilon}{2}$ so $h_n\underline{\Delta}_n>\log(n^{1-b-\kappa_b}\varphi(n^{b}))+(\xi-\frac{\varepsilon}{2})h_n$ and as $h_n(\varepsilon+\underline{\Delta}_n)>0$, $\P(h_n^{-1}\log^+\mathcal{R}_{n}(g_n,\mathbf{f}^n)>\varepsilon+\underline{\Delta}_n)=\P(\log\mathcal{R}_{n}(g_n,\mathbf{f}^n)>h_n(\varepsilon+\underline{\Delta}_n),\mathcal{R}_{n}(g_n,\mathbf{f}^n)>1)$ which is smaller than
\begin{align*}
   \P\big(\mathcal{R}_{n}(g_n,\mathbf{f}^n)>e^{\varepsilon h_n+h_n\underline{\Delta}_n}\big)\leq\P\Big(\frac{\mathcal{R}_{n}(g_n,\mathbf{f}^n)}{n^{1-b-\kappa_b}\varphi(n^b)}>e^{(\xi+\frac{\varepsilon}{2}) h_n}\Big)\to 0,
\end{align*}
where the limit comes from \eqref{UpperTH5}. \\
$\bullet$ Assume $L\in[-\infty,-\xi)$. In this case, there exists an increasing sequence $(n_{\ell})_{\ell}$ of positive integers (with $n_{\ell}=\ell$ when $\lim h_n^{-1}\log(n^{1-b-\kappa_b}\varphi(n^b))=L$) and $\varepsilon_{L}>0$ such that for any $\ell\in\N^*$, $n_{\ell}^{1-b-\kappa_b}\varphi(n_{\ell}^b)<e^{-(\xi+2\varepsilon_{L})h_n}$ and for any $\varepsilon'>0$
\begin{align*}
    \P\big(\mathcal{R}_{n_{\ell}}(g_{n_{\ell}},\mathbf{f}^{n_{\ell}})>\varepsilon'\big)\leq\P\big(\mathcal{R}_{n_{\ell}}(g_{n_{\ell}},\mathbf{f}^{n_{\ell}})>e^{-\varepsilon_L h_n}\big)\leq\P\Big(\frac{\mathcal{R}_{n_{\ell}}(g_{n_{\ell}},\mathbf{f}^{n_{\ell}})}{n_{\ell}^{1-b-\kappa_b}\varphi(n_{\ell}^b)}>e^{(\xi+\varepsilon_L)h_n}\Big)\to 0,
\end{align*}
where the limit comes from \eqref{UpperTH5} with $\varepsilon=\varepsilon_L$, which ends the proof. \hfill $\square$
%\end{proof}

%\Amodif{*Changement de $b_n$ en $n^b$ \\
%*plus de d√É¬É√Ç¬É√É¬Ç√Ç¬É√É¬É√Ç¬Ç√É¬Ç√Ç¬©tails.}

\section{Proof of Proposition \ref{Prop1} \label{sec3}} 
The proof of Proposition  \ref{Prop1} is decomposed as follows. In the first short section below, we present the expression of the generating function with constraint of edge local time. In a second sub-section, we prove the lower bound \eqref{LowerProp1}, this section is itself decomposed in different steps treating successively the random walk at fixed environment and then an important quantity of the environment. Finally, in a third section, we obtain the upper bound \eqref{UpperProp1}. Note that the fact that the upper and the lower bounds are robust when replacing $T^n$ by $T^{k_n}$ with $k_n = \lfloor n/ (\log n)^p \rfloor $, with $p>0$, does not need extra arguments than the ones that follow.
\subsection{Preliminary} \label{sec3.1}

We first introduce the edge local time $N^{n}_x$ of a vertex $x\in\T$,  that is the number of times the random walk $\X$ visits the edge $(x^*,x)$ before the instant $n$:
\begin{align}
    N^{n}_x := \sum_{i=1}^n\un_{\{X_{i-1}=x^*,\;X_i=x\}}, \label{Nxb}
\end{align}
 the law of $N^{T^1}_x$ (recall that $T^1$ is the instant of the first return to the root $e$) and $\sum_{y;y^*=x}N^{T^1}_y$ at fixed environment, that is under $\Pe$, are given by
\begin{lemm}\label{LawGeo}
Let $x\in\T$, and $T_x := \inf\{k>0, X_k=x\}$,  then $\P^{\mathcal{E}}(T_x<T^1)= e^{-V(x)}/H_x$ and for any  $i\in\N^*$, $s\in[0,1]$ and $\nu\geq 0$, %, mettre en premier car on s'en sert dans les preuves de i) et ii)??}
\begin{enumerate}[label=\roman*)]
    \item The distribution of $N^{T^1}_x$ under $\P^{\mathcal{E}}_x(\cdot) =\P^{\mathcal{E}}(\cdot |X_0=x)$ is geometrical on $\N$ with mean $H_x-1=\sum_{1\leq j<|x|}e^{V(x_j)-V(x)}$. In particular
    \begin{align*}
        \E^{\mathcal{E}}\big[s^{\nu N_x^{T^1}}\un_{\{N_x^{T^1}\geq i\}}\big] = \frac{e^{-V(x)}}{H_x^2}\Big(1-\frac{1}{H_x}\Big)^{i-1}\frac{s^{i\nu}}{1-s^{\nu}(1-\frac{1}{H_x})}.
    \end{align*}
    \item For any $z\in\T$ such that $z^*=x$, the distribution of $\sum_{y;y^*=x}N^{T^1}_y$ under $\P^{\mathcal{E}}_z$ is geometrical on $\N$ with mean $\tilde{H}_x:= H_x\sum_{y;y*=x}e^{-V_x(y)}$ with $V_x(y)=V(y)-V(x)$ . In particular
    \begin{align*}
        \E^{\mathcal{E}}\big[s^{\nu \sum_{y;y*=x}N_y^{T^1}}\un_{\{\sum_{y;y*=x}N_y^{T^1}\geq i\}}\big] = \frac{e^{-V(x)}}{H_x}\frac{\tilde{H}_x}{(1+\tilde{H}_x)^2}\Big(1-\frac{1}{1+\tilde{H}_x}\Big)^{i-1}\frac{s^{i\nu}}{1-s^{\nu}(1-\frac{1}{1+\tilde{H}_x})}.
    \end{align*} 
\end{enumerate}
\end{lemm}
\begin{proof} The fact that $\P^{\mathcal{E}}(T_x<T^1)= e^{-V(x)}/H_x$ comes from a standard result for one-dimensional random walks in random environment, see for example \cite{Golosov}.
The proofs of points $i)$ and $ii)$  are very similar and elements for the first one can be found in \cite{AndDiel3} so we will only deal with the second one. \\ 
For any $x\in\T$, $\min_{y;y^*=x}T_y$ is the first hitting time of the set $\{y\in\T;\; y^*=x\}$ of children of $x$ and let $\beta_x:=\P_x^{\mathcal{E}}(\min_{y;y^*=x}T_y<T^1)$ be the quenched probability, starting from $x$, to reach a children of $x$ before hitting the root $e$. Hence, $\sum_{y;y^*=x}N^{T^1}_y$ is the number of times the random walk $\X$ visits the \textrm{\guillemotleft edge\guillemotright } $(x,\{y\in\T;\; y^*=x\})$ before the instant $T^1$. It follows, thanks to the strong Markov property, that for all $z\in\T$ such that $x^*=z$ and $k\in\N$
\begin{align}\label{SumEdge}
    \P_z^{\mathcal{E}}\Big(\sum_{y;y^*=x}N^{T^1}_y=k\Big)=\beta_x^k(1-\beta_x).
\end{align}
Note that the right part above does  not depend on $z$. We now compute $\beta_x$. On the one hand, thanks to \eqref{SumEdge}, we have $\E_z^{\mathcal{E}}[\sum_{y;y^*=x}N^{T^1}_y]=\beta_x/(1-\beta_x)$ and on the other hand, thanks to the first point, $\E_z^{\mathcal{E}}[\sum_{y;y^*=x}N^{T^1}_y]=\sum_{y;y^*=x}\E_z^{\mathcal{E}}[N^{T^1}_y]=\sum_{y;y^*=x}(H_y-1)=H_x\sum_{y;y^*=x}e^{-V_x(y)}=\tilde{H}_x$. $\sum_{y;y^*=x}N^{T^1}_y$ is finally geometrical on $\N$ under $\P_z^{\mathcal{E}}$ with mean $\tilde{H}_x$ and $\beta_x=\tilde{H}_x/(1+\tilde{H}_x)$. \\
Introduce $\alpha_x:=\P^{\mathcal{E}}(\min_{y;y^*=x}T_y<T^1)$, the quenched probability to reach the set $\{y\in\T;\; y^*=x\}$ during the first excursion. Thanks to \eqref{SumEdge}, we have for all $k\in\N^*$ 
\begin{align*}
    \P^{\mathcal{E}}\Big(\sum_{y;y^*=x}N^{T^1}_y=k\Big)=\alpha_x\beta_x^{k-1}(1-\beta_x)\;\;\textrm{ and }\;\; \P^{\mathcal{E}}\Big(\sum_{y;y^*=x}N^{T^1}_y=0\Big)=1-\alpha_x,
\end{align*}
so on the one hand, $\E^{\mathcal{E}}[\sum_{y;y^*=x}N^{T^1}_y]=\alpha_x/(1-\beta_x)$ and on the other hand, thanks to the first point, $\E^{\mathcal{E}}[\sum_{y;y^*=x}N^{T^1}_y]=\sum_{y;y^*=x}\E^{\mathcal{E}}[N^{T^1}_y]=\sum_{y;y^*=x}e^{-V(y)}$. It follows that $\alpha_x=\sum_{y;y^*=x}e^{-V(y)}/(1+\tilde{H}_x)$ and the result is proved. \\
\end{proof}

\subsection{Lower bound for $\mathcal{R}_{T^n}(g_n,\mathbf{f}^n)$} \label{sec3.2}

\noindent Let us first introduce  two key random variables denoted $\mathcal{R}_{T^n}(\mathbf{f}^n)$ and $R(\mathbf{f}^n)$. $\mathcal{R}_{T^n}(\mathbf{f}^n)$ is  a simplified version of $\mathcal{R}_{T^n}(g_n,\mathbf{f}^n)$  which does not depend on the function $g_n$ and with a constraint on $V$ : recall $\lambda_n=ne^{-\min(10\varepsilon\log n,5h_n)}$ and $\mathcal{H}^k_{\mathfrak{z}_n}=\{(t_1,\ldots,t_k)\in \mathbb{R}^k;\; t_k\geq \mathfrak{z}_n\}$ where we set for convenience $\mathfrak{z}_n:=\ell_n^{1/3}/\delta_1$ with $\ell_n=(\log n)^3$ and $\delta_1\in(0,1/2)$ (see \eqref{hyp1}), then
%as indeed this range does not depend explicitly of function  $g$ and more than that this range is focusing on visited vertices excursion by excursion : 
\begin{align*}
    \mathcal{R}_{T^n}(\mathbf{f}^n) &:= \sum_{i=1}^n\;\sum_{x\in\mathcal{O}_{\lambda_n,n^b}}\un_{\{N_x^{T^i}-N_x^{T^{i-1}}\geq n^b\}} \un_{\{\forall j\not=i:N_x^{T^j}-N_x^{T^{j-1}}=0\}} f^{n,|x|}\un_{\mathcal{H}^{|x|}_{\mathfrak{z}_n}}(\mathbf{V}_x), \\
 % &  \bl{N_x^k:=\sum_{m=1}^k \un_{X_{k-1}=x^*,X_{k}=x}}
\end{align*}
where we use the notation $F(\mathbf{V}_x)=F(V(x_1),\cdots,V(x))$. Note that the local time until $T^n$ which appears in $\mathcal{R}_{T^n}(g_n,\mathbf{f}^n)$ is replaced in $\mathcal{R}_{T^n}(\mathbf{f}^n)$ by edge local times excursion by excursion. Also, visited vertices are restricted to some $V$-regular lines $\mathcal{O}_{\lambda_n,{n^b}}$. $\mathcal{R}_{T^n}(g_n,\mathbf{f}^n)$ and $\mathcal{R}_{T^n}(\mathbf{f}^n)$ are related as follows, first since $\varphi$ is non-decreasing
\begin{align*}%\label{MinorRanGen1}
   \mathcal{R}_{T^n}(g_n,\mathbf{f}^n)\geq \varphi(n^b) \sum_{x\in\T}\un_{\{\mathcal{L}_x^{T^n}\geq n^b\}}f^{n,|x|}\un_{\mathcal{H}^{|x|}_{\mathfrak{z}_n}}(\mathbf{V}_x).
\end{align*}
Then, introduce $E_x^n = \sum_{i=1}^n\un_{\{\mathcal{L}_x^{T^i}-\mathcal{L}_x^{T^{i-1}}\geq 1\}}$, the number of excursions to the root where the walk hits vertex $x$. %Since we are looking for a lower bound, it is enough to work on the event $\left\{E_x^n= 1\right\}$\textcolor{red}{(dire ici que c'est l√É¬É√Ç¬É√É¬Ç√Ç¬É√É¬É√Ç¬Ç√É¬Ç√Ç¬É√É¬É√Ç¬É√É¬Ç√Ç¬Ç√É¬É√Ç¬Ç√É¬Ç√Ç¬† o√É¬É√Ç¬É√É¬Ç√Ç¬É√É¬É√Ç¬Ç√É¬Ç√Ç¬É√É¬É√Ç¬É√É¬Ç√Ç¬Ç√É¬É√Ç¬Ç√É¬Ç√Ç¬π tout se passe d'apr√É¬É√Ç¬É√É¬Ç√Ç¬É√É¬É√Ç¬Ç√É¬Ç√Ç¬É√É¬É√Ç¬É√É¬Ç√Ç¬Ö%√É¬É√Ç¬Ç√É¬Ç√Ç¬°s lemma 3.5 dans \cite{ADHeavyRange})} and on the optional line $\mathcal{O}_{\lambda_n}$\textcolor{red}{(pareil)}. 
Notice that $E_x^n= 1$ if and only if there exists $i\in\{1,\ldots,n\}$ such that $\mathcal{L}_x^{T^i}-\mathcal{L}_x^{T^{i-1}}\geq 1$ and for any $j\in\{1,\ldots,n\}, j\not = i$, $\mathcal{L}_x^{T^j}-\mathcal{L}_x^{T^{j-1}} = 0$ that is $N_x^{T^j}-N_x^{T^{j-1}}=0$. Thus
\begin{align*}
% \begin{array}{l c l}
& \sum_{x\in\T}\un_{\{\mathcal{L}_x^{T^n}\geq n^b\}}f^{n,|x|}\un_{\mathcal{H}^{|x|}_{\mathfrak{z}_n}}(\mathbf{V}_x) \geq \sum_{x\in\mathcal{O}_{\lambda_n,n^b}}\un_{\left\{\mathcal{L}_x^{T^n}\geq n^b,E_x^n= 1\right\}}f^{n,|x|}\un_{\mathcal{H}^{|x|}_{\mathfrak{z}_n}}(\mathbf{V}_x) \\
  & \geq \sum_{i=1}^n\;\sum_{x\in\mathcal{O}_{\lambda_n,n^b}}\un_{\{\mathcal{L}_x^{T^i}-\mathcal{L}_x^{T^{i-1}}\geq n^b\}}\un_{\{\forall j\not=i:\;N_x^{T^i}-N_x^{T^{i-1}} = 0\}}f^{n,|x|}\un_{\mathcal{H}^{|x|}_{\mathfrak{z}_n}}(\mathbf{V}_x) ,
%  & \geq & \displaystyle \overset{n}{\underset{i=1}{\sum}}\;\underset{x\in\mathcal{O}_{\lambda_n}}{\underset{|x|\leq l_n}{\sum}}\un_{\left\{N_x^{T^i}-N_x^{T^{i-1}}\geq b_n\right\}}F^{|x|}\un_{\mathcal{H}^{|x|}_{v_n}}(V(x_1),\cdots,V(x))\un_{\left\{\forall j\not=i\;N_x^{T^j}-N_x^{T^{j-1}} = 0\right\}}.
  %  \end{array}
\end{align*}
so finally, as $\mathcal{L}_x^{T^i}-\mathcal{L}_x^{T^{i-1}}\geq N_x^{T^i}-N_x^{T^{i-1}}$, we have the following relation 
\begin{align}
  \mathcal{R}_{T^n}(g_n,\mathbf{f}^n)\geq \varphi(n^b)   \mathcal{R}_{T^n}(\mathbf{f}^n)  \label{MinorRanGen1} .
\end{align}
The second random variable $R(\mathbf{f}^n)$ depends only on the environment : 
\begin{align*}
     R(\mathbf{f}^n) :=\sum_{x\in\mathcal{O}_{\lambda_n,n^b}} e^{-V(x)}\frac{1}{H_x}\Big(1-\frac{1}{H_x}\Big)^{n^b-1}f^{n,|x|}\un_{\mathcal{H}^{|x|}_{\mathfrak{z}_n}}(\mathbf{V}_x), 
\end{align*}  
it can be related to the quenched mean of $\mathcal{R}_{T^n}(\mathbf{f}^n)$ as follows 
\begin{align}\label{EquivMoy}
    1\leq \frac{nR(\mathbf{f}^n)}{\E^{\mathcal{E}}[\mathcal{R}_{T^n}(\mathbf{f}^n)]}\leq (1-e^{-\mathfrak{z}_n})^{-(n-1)}.
\end{align}
Indeed, the random variables $N_x^{T^i}-N_x^{T^{i-1}}, i\in\left\{1,\ldots,n\right\}$, are i.i.d under $\P^{\mathcal{E}}$ so, 
\begin{align*}
     \E^{\mathcal{E}}\left[\mathcal{R}_{T^n}(\mathbf{f}^n)\right] = n\sum_{x\in\mathcal{O}_{\lambda_n,n^b}} \P^{\mathcal{E}}(N_x^{T^1}\geq n ^b)\P^{\mathcal{E}}(N_x^{T^1}=0)^{n-1}f^{n,|x|}\un_{\mathcal{H}^{|x|}_{\mathfrak{z}_n}}(\mathbf{V}_x).
\end{align*}
Moreover, on the event $\{V(x)\geq \mathfrak{z}_n\}$, thanks to Lemma \ref{LawGeo}, $\P^{\mathcal{E}}(N_x^{T^1}=0)^{n-1} = \P^{\mathcal{E}}(T_x>T^1)^{n-1}= (1-e^{-V(x)}/H_x)^{n-1}\geq(1-e^{-V(x)})^{n-1}\geq (1-e^{-\mathfrak{z}_n})^{n-1}$ since $H_x\geq 1$, and thanks to Lemma \ref{LawGeo} $i)$ with $\nu=0$, $\P^{\mathcal{E}}(N_x^{T^1}\geq n^b) = e^{-V(x)}(1-1/H_x)^{n^b-1}/H_x$ which gives \eqref{EquivMoy}.
We are now ready to obtain a relation between a lower bound for $ \mathcal{R}_{T^n}(g_n,\mathbf{f}^n)$ and a lower bound for $R(\mathbf{f}^n)$.
 %The following Lemma makes a link between the tail of the range and the one of $R_n$ :

\begin{lemm}\label{LemmMinorProbaRange1}
Recall $\varepsilon_b=\min(b+\un_{\{b=0\}},1-b)/13$ and let $(a_n)$ be a sequence  of positive numbers. For all $\varepsilon\in(0,\varepsilon_b)$ and $n$ large enough
\begin{align}
   \P^*\big(\mathcal{R}_{T^n}(g_n,\mathbf{f}^n)<n\varphi(n^b)a_n/4n^b\big) \leq  \mathbf{P}^*\left(R(\mathbf{f}^n)<a_n/n^b\right) + \frac{ne^{-\min(9\varepsilon\log n,4h_n)}}{n^{2\kappa_b}a_n^2}.
   \label{MinorProbaRange1}
\end{align}
\end{lemm}

\begin{proof}

Note that thanks to \eqref{EquivMoy}, for $n$ large enough, $nR(\mathbf{f}^n)\leq 2\E^{\mathcal{E}}[\mathcal{R}_{T^n}(\mathbf{f}^n)]$, so by \eqref{MinorRanGen1}, on the event $\{R(\mathbf{f}^n)\geq a_n/n^b\}$
\begin{align*}
   \P^{\mathcal{E}}\big(\mathcal{R}_{T^n}(g_n,\mathbf{f}^n)<n\varphi(n^b)a_n/4n^b\big)\leq  \P^{\mathcal{E}}\big(\mathcal{R}_{T^n}(\mathbf{f}^n)<\E^{\mathcal{E}}[\mathcal{R}_{T^n}(\mathbf{f}^n)]/2\big).
\end{align*}
 Using Bienaym\'e-Tchebychev inequality and the fact that $N_x^{T^i}-N_x^{T^{i-1}}, i\in\{1,\ldots,n\},$ are i.i.d under $\P^{\mathcal{E}}$ implies, on the event $\{R(\mathbf{f}^n)\geq a_n/n^b\}$
\begin{align}
    &\P^{\mathcal{E}}(\mathcal{R}_{T^n}(\mathbf{f}^n)<\E^{\mathcal{E}}\left[\mathcal{R}_{T^n}(\mathbf{f}^n)\right]/2)
    \leq  \frac{4}{\E^{\mathcal{E}}[\mathcal{R}_{T^n}(\mathbf{f}^n)]^2} n\Var^{\mathcal{E}}(\mathcal{R}_{T^1}(\mathbf{f}^n)) \nonumber  \\
    \leq & \frac{16n^{2b}}{a_n^2 n} \sum_{x,y\in\mathcal{O}_{\lambda_n,n^b}}\P^{\mathcal{E}}(N_x^{T^1}\land N_y^{T^1}\geq n ^b)f^{n,|x|}\un_{\mathcal{H}^{|x|}_{\mathfrak{z}_n}}(\mathbf{V}_x)f^{n,|y|}\un_{\mathcal{H}^{|y|}_{\mathfrak{z}_n}}(\mathbf{V}_y). \label{MajorVariance1}
\end{align} 
%indeed \bo{by  independence}, 
%\begin{align*}
%\Var^{\mathcal{E}}(\mathcal{R}_{T^n}(\mathbf{f}^n))=n\Var^{\mathcal{E}}(\mathcal{R}_{T^1}(\mathbf{f}^n)) \leq \sum_{x,y\in\mathcal{O}_{\lambda_n,n^b}}\P^{\mathcal{E}}(N_x^{T^1}\land N_y^{T^1}\geq n ^b)f^{n,|x|}\un_{\mathcal{H}^{|x|}_{v_n}}(\mathbf{V}_x)f^{n,|y|}\un_{\mathcal{H}^{|y|}_{v_n}}(\mathbf{V}_y).
%\end{align*}
The last inequality coming from the fact that, on $\{R(\mathbf{f}^n)\geq a_n/n^b\}$, thanks to (\ref{EquivMoy}) $\E^{\mathcal{E}}[\mathcal{R}_{T^n}(\mathbf{f}^n)]^2 \geq n^2R(\mathbf{f}^n)^2/4\geq n^2a_n^2/4n ^{2b}$. 
%\begin{align*}
%\end{align*}
 Markov inequality in (\ref{MajorVariance1}) yields $  \P^{\mathcal{E}}(N_x^{T^1}\land N_y^{T^1}\geq n^b)\leq \E^{\mathcal{E}}[N_x^{T^1}N_y^{T^1}]/n^{2b} $,

%\noindent Now define the "independant version" $\mathcal{R}_n(\mathbf{f}^n)$ of the range
%\begin{align*}
%    \mathcal{R}_n(\mathbf{f}^n) := \overset{n}{\underset{i=1}{\sum}}\;\underset{x\in\mathcal{O}_{\lambda_n}}{\underset{|x|\leq l_n}{\sum}}\un_{\left\{N_x^{T^i}-N_x^{T^{i-1}}\geq b_n\right\}}F^{|x|}\un_{\mathcal{H}^{|x|}_{v_n}}(V(x_1),\cdots,V(x))\un_{\left\{\forall j\not=i\;N_x^{T^j}-N_x^{T^{j-1}} =  0\right\}}
%\end{align*}

%\noindent and the rescaled mean $R(\mathbf{f}^n)$ of the range
%\begin{align*}
%     R(\mathbf{f}^n) := \underset{x\in\mathcal{O}_{\lambda_n}}{\underset{|x|\leq l_n}{\sum}}e^{-V(x)}\frac{1}{H_x}\left(1-\frac{1}{H_x}\right)^{b_n-1}F^{|x|}\un_{\mathcal{H}^{|x|}_{v_n}}(V(x_1),\cdots,V(x))
%\end{align*}
\noindent so finally, on the event $\{R(\mathbf{f}^n)\geq a_n/n^b\}$
\begin{align}
    \P^{\mathcal{E}}(\mathcal{R}_{T^n}(g_n,\mathbf{f}^n)<n\varphi(n^b)a_n/4n^b) \leq  \frac{16}{na_n^2} \sum_{x,y\in\mathcal{O}_{\lambda_n,n^b}}\E^{\mathcal{E}}[N^{T^1}_xN^{T^1}_y]f^{n,|x|}\un_{\mathcal{H}^{|x|}_{\mathfrak{z}_n}}(\mathbf{V}_x)f^{n,|y|}\un_{\mathcal{H}^{|y|}_{\mathfrak{z}_n}}(\mathbf{V}_y).
    \label{Variance1}
\end{align}
To treat the above sum, we first make a simplification by using the uniform upper bound of the set $\mathcal{U}_b$, see \eqref{Ub}
\begin{align}\label{simpli}
    \sum_{x,y\in\mathcal{O}_{\lambda_n,n ^b}}\E^{\mathcal{E}}[N^{T^1}_xN^{T^1}_y]f^{n,|x|}\un_{\mathcal{H}^{|x|}_{\mathfrak{z}_n}}(\mathbf{V}_x)f^{n,|y|}\un_{\mathcal{H}^{|y|}_{\mathfrak{z}_n}}(\mathbf{V}_y)\leq \frac{C_{\infty}^2}{n^{2\kappa_b}} \sum_{x,y\in\mathcal{O}_{\lambda_n}}\E^{\mathcal{E}}[N^{T^1}_xN^{T^1}_y].
\end{align}
We then split the computations in two distinct steps: the first step is dedicated to the cases $x\leq y$ or $y\leq x$ and the second one to the cases nor $x\leq y$ neither $y\leq x$. The key here is to take into account that we are only interested in vertices belonging to $\lambda_n$-regular lines $\mathcal{O}_{\lambda_n}$ with $\lambda_n=ne^{-\min(10\varepsilon\log n,5h_n)}$  for $\varepsilon\in(0,\varepsilon_b)$. \\
We start with the cases $x\leq y$ and $y\leq x$ and as they are symmetrical, we only deal with the first one. First note that as $\E^{\mathcal{E}}\left[N^{T^1}_xN^{T^1}_y\right]\leq 2e^{-V(y)}H_x = 2H_xe^{-V(x)}e^{-V_x(y)}$ (see \cite{AndDiel3} Lemma 3.6) 
\begin{align*}
    \mathbf{E}\Big[\underset{x,y\in\mathcal{O}_{\lambda_n}}{\sum_{x\leq y}}\E^{\mathcal{E}}[N^{T^1}_xN^{T^1}_y]\Big] & \leq  2\mathbf{E}\Big[\sum_{x\in\mathcal{O}_{\lambda_n}}e^{-V(x)}H_x\underset{y\in \mathcal{O}^x_{\lambda_n}}{\sum_{y\geq x}}e^{-V_x(y)}\Big] \leq 2\mathbf{E}\Big[\sum_{x\in\mathcal{O}_{\lambda_n}}e^{-V(x)}\Big]^2\lambda_n \\ & \leq  2\ell_n^2\lambda_n,
\end{align*}
where for all $\lambda>0$, $\mathcal{O}^x_{\lambda}$ is translated set of $\lambda$-regular lines
\begin{align*}
    \mathcal{O}^x_{\lambda} = \big\{y\in \mathbb{T}, y>x;\; \underset{|x|< j\leq|y|}{\max}\; H_{x,y_j}\leq \lambda\big\},\ H_{x,y_j}=\sum_{ x<w \leq  y_j} e^{V_x(w)-V_x(y_j)}, 
\end{align*} 
also, the second inequality is obtained thanks to the regular line which yields $H_x\un_{\mathcal{O}_{\lambda_n}}(x)\leq \lambda_n$ and the last one comes from Remark \ref{Rem1}.

%\vspace{0.3cm}

\noindent We then move to the second case, neither $x\leq y$ nor $y\leq x$, that we denote $x\not \sim y$. In this case, $\E^{\mathcal{E}}\left[N^{T^1}_xN^{T^1}_y\right] = 2H_{x\land y}e^{V(x\land y)-V(x)-V(y)}$ (see \cite{AndDiel3} Lemma 3.6). Thus
\begin{align*}
    \E^{\mathcal{E}}[N^{T^1}_xN^{T^1}_y] \leq 2\lambda_n\sum_{l\geq 1}\sum_{|z|=l}e^{-V(z)}\un_{\{z\in\mathcal{O}_{\lambda_n}\}}\underset{u^*=v^*=z}{\sum_{u\not = v}}e^{-V_z(u)}e^{-V_z(v)}\underset{x\in\mathcal{O}^u_{\lambda_n}}{\sum_{x\geq u}}e^{-V_u(x)}\underset{y\in\mathcal{O}^v_{\lambda_n}}{\sum_{y\geq v}}e^{-V_v(y)},
\end{align*}

\noindent where we have used again the regular line $\mathcal{O}_{\lambda_n}$ which gives an upper bound for $H_{x\land y}$. Finally, independence of the increments of $V$ conditionally to $(\T,V(w); w\in\T,|w|\leq l+1)$ and Remark \ref{Rem1} yields
\begin{align*}
    \mathbf{E}\Big[\underset{x,y\in\mathcal{O}_{\lambda_n}}{\sum_{x\not \sim y}}\E^{\mathcal{E}}[N^{T^1}_xN^{T^1}_y]\Big]& \leq  2\lambda_n \mathbf{E}\Big[\Big(\sum_{|u|=1}e^{-V(u)}\Big)^2\Big] \mathbf{E}\Big[ \sum_{z\in\mathcal{O}_{\lambda_n}}e^{-V(z)}\Big]^3 \\ 
   & \leq  2\lambda_n \mathbf{E}\Big[\Big(\sum_{|u|=1}e^{-V(u)}\Big)^2\Big](\ell_n)^3,  
\end{align*}
and thanks to \eqref{hyp1}, the second moment above is finite. Collecting the upper bounds for the two cases and moving back to \eqref{simpli}, we get for $n$ large enough
\begin{align}
  \mathbf{E}\Big[\sum_{x,y\in\mathcal{O}_{\lambda_n,n^b}}\E^{\mathcal{E}}[N^{T^1}_xN^{T^1}_y]f^{n,|x|}\un_{\mathcal{H}^{|x|}_{\mathfrak{z}_n}}(V_x)f^{n,|y|}\un_{\mathcal{H}^{|y|}_{\mathfrak{z}_n}}(V_y)\Big]\leq   \frac{(\ell_n)^4\lambda_n}{n^{2\kappa_b}}\leq\frac{ne^{-\min(9\varepsilon\log n,4h_n)}}{n^{2\kappa_b}},
    \label{VarianceNum}
\end{align}
the last inequality is justified by the fact (see Remark \ref{rem0}) that $(\ell_n)^4=o(e^{h_n})$ and $(\ell_n)^4=o(e^{\varepsilon \log n})$. We are now ready to conclude the proof of the lemma : $\P^*\big(\mathcal{R}_{T^n}(g_n,\mathbf{f}^n)<n\varphi(n^b)a_n/4n^b\big)$ is smaller than 
\begin{align*}
   \P^*(R(\mathbf{f}^n)<a_n/n^b) + \P^*\big(\mathcal{R}_{T^n}(g_n,\mathbf{f}^n)<n\varphi(n^b)a_n/4n^b,R(\mathbf{f}^n)\geq a_n/n^b\big),
 \end{align*}
 \noindent then, as the second term in the above inequality is nothing but
 \begin{align*}
     \mathbf{E}^*\big[\P^{\mathcal{E}}\big(\mathcal{R}_{T^n}(g_n,\mathbf{f}^n)<n\varphi(n^b)a_n/4n^b\big)\un_{\{R(\mathbf{f}^n)\geq a_n/n^b\}}\big], 
 \end{align*}
the proof ends thanks to \eqref{Variance1} and \eqref{VarianceNum}.
\end{proof}

\subsubsection{Lower bound for $R(\mathbf{f}^n)$}
This is the most technical part of the proof of Proposition \ref{Prop1}. For any $n\geq 2$ and $\varepsilon\in(0,\varepsilon_b)$, recall that  $\lambda_n=ne^{-\min(10\varepsilon\log n,5h_n)}$ and $\mathfrak{z}_n=\ell_n^{1/3}/\delta_1$, $\delta_1\in(0,1/2)$ (see \eqref{hyp1}) with $\ell_n=(\log n)^3$. For any $\varepsilon>0$, let us choose $(a_n)$ as follows
%The main difficulty is to obtain the upper bound of following random variable $Z_n(B,\mathcal{A})$. Let $(m_n)$ a positive sequence and $v'_n := v_n + \varepsilon h_n$: 
%\begin{align}\label{RVMinor2}
 %   Z_n(B,\mathcal{A}) := \overset{l_n-m_n}{\underset{k=1}{\sum}}\underset{x\in\mathcal{O}_{\lambda_n}}{\underset{|x|=k}{\sum}}e^{-V(x)}F^{k,\mathcal{A}}_{m_n,\varepsilon h_n}\mathbf{H}^{k}_{b_n,\varepsilon h_n}\un_{\mathcal{H}^k_{B,v'_n}}\left(V\left(x_{1}\right), \ldots, V(x)\right).
%\end{align}
\begin{align}\label{Defa_n}
    a_n := e^{-4\varepsilon h_n}\sum_{k\geq 1}\Psi^k_{\lambda_n/2,n^b}\big(f^{n,k}_{\varepsilon h_n}\un_{\Upsilon^k_n}\big)
\end{align}
with $\Upsilon^k_n=\{\mathbf{t}\in\R^k;\;H_k(\mathbf{t})\leq n^be^{\varepsilon h_n}\}\cap\mathcal{H}^k_{B,2\mathfrak{z}_n}$. Recall that $\Psi^k_{\lambda,\lambda'}$, $h_n$, $\mathcal{H}^k_{B,2\mathfrak{z}_n}$ and $f^{n,k}_{\varepsilon h_n}$ can be found  respectively in \eqref{Def_Psi}, \eqref{Def_hn}, \eqref{HighSet} and \eqref{Def_f}.
%$\mathbf{P}\left(R(\mathbf{f}^n)<a_n/b_n\right)$ in (\ref{MinorProbaRange1}), with the right $a_n$, is the most technical point of the section. Indeed, unlike Lemma \ref{Variance2}, the optional line $\mathcal{O}_{\lambda_n}$ is not the key here and can't be used directly to bound $\mathbf{P}\left(R(\mathbf{f}^n)<a_n/b_n\right)$. Therefore, we need to be much more precise.

%\noindent The second step is to bound $\mathbf{P}\left(R(\mathbf{f}^n)<a_n/b_n\right)$ using random variables $Z^u_n(B,\mathcal{A})$ and $Z_n(B,\mathcal{A})$:
\begin{lemm}\label{lemmeMarjorProba}
There exists $c_4>0$ such that for any $\varepsilon\in(0,\varepsilon_b)$ and $n$ large enough
\begin{align}
    \mathbf{P}^*\left(R(\mathbf{f}^n)<a_n/n^b\right)\leq\frac{e^{-\varepsilon\frac{c_4}{c_2}h_n}\mathbf{E}[Z_n^2]}{\big(\sum_{k\geq 1}\Psi^k_{\lambda_n/2,n^b}\big(f^{n,k}_{\varepsilon h_n}\mathds{1}_{\Upsilon^k_n}\big)\big)^2} + h_ne^{-\varepsilon \tilde{c}_2h_n},
    \label{MarjorProba}
\end{align}
with, recall, $m_n = \lceil \varepsilon h_n/c_2\rceil$ (see \eqref{GenPotMax}).
\end{lemm}

%\begin{lemm}\label{Lemm_Minor_Esp_du_carre}
%Suppose Assumption (\ref{Hypoth√É¬É√Ç¬É√É¬Ç√Ç¬É√É¬É√Ç¬Ç√É¬Ç√Ç¬É√É¬É√Ç¬É√É¬Ç√Ç¬Ö%√É¬É√Ç¬Ç√É¬Ç√Ç¬°se1}) holds. For $n$ large enough
%\begin{align}\label{Minor_Esp_du_carre}
% \mathbf{E}\left[\left(Z_n(B,\mathcal{A})\right)^2\right]\leq C_7\frac{(\log b_n)^2}{b_n^2}e^{2\varepsilon'h_n}\underset{l\leq l_n'}{\sum}\Psi^{l}_{\lambda_n} \left(\left(\Phi^l_n\right)^2\right)
%\end{align}
%\noindent where, for any $n\in\N^*$ (see (\ref{PhiSuite}))
%\begin{align*}
%\Psi_{\lambda_n} = \mathbf{E}\left[\underset{z\in\mathcal{O}_{\lambda_n}}{\underset{|z|<l'_n}{\sum}}e^{-V(z)}\un_{\left\{\underline{V}(z)\geq -B\right\}}\left(\underset{k\geq 1}{\sum}\Psi^k_{\lambda_n}\left(F^{k+|z|+1,\mathcal{A}}_{m_n,\varepsilon h_n}\cdot(V(z_1),\cdots,V(z))\right)\right)^2\right]   
%\end{align*}
%\end{lemm}

\begin{proof}

\noindent Recall the expression of $R(\mathbf{f}^n)$:
\begin{align*}
    R(\mathbf{f}^n) = \sum_{x\in\mathcal{O}_{\lambda_n,n^b}}e^{-V(x)}\frac{1}{H_x}\Big(1-\frac{1}{H_x}\Big)^{n^b-1}f^{n,|x|}\un_{\mathcal{H}^{|x|}_{\mathfrak{z}_n}}(V(x_1),\cdots,V(x)), 
\end{align*}
with $H_x$ and $\mathcal{H}^{|x|}_{\mathfrak{z}_n}$ respectively defined in \eqref{Def_Hx} and  \eqref{HighSet}.
%For any $n\in\N^*$, let $\varepsilon h_n$ be a positive number and l
\noindent The main idea here is to cut the tree at the generation $m_n$ to introduce independence between  generations. First note that %by cutting the tree at generation $m_n$, 
\begin{align*}
    R(\mathbf{f}^n) & \geq &  \sum_{|u|=m_n}\sum_{k\geq 1}\underset{x>u;\;x\in\mathcal{O}_{\lambda_n,n^b}}{\sum_{|x|=k+m_n}}\frac{e^{-V(x)}}{H_x}\Big(1-\frac{1}{H_x}\Big)^{n^b}f^{n,k+m_n}\un_{\mathcal{H}^{k+m_n}_{\mathfrak{z}_n}}(V(x_{1}),\ldots, V(x)), 
\end{align*}      
from here we would like to make a translation to decompose the trajectories of $V$ before and after the generation $m_n$ and to do that, we have in particular to re-write $H_{x_j}$ for $j \leq |x|$. Let $u<x$ with $|u|=m_n$. For all $m_n<j\leq |x|$, we have $H_{x_j}=H_ue^{-V_u(x_j)}+H_{u,x_j}$ where, for any $z<v$, $H_{z,v}:=\sum_{z<w\leq v}e^{V_z(w)-V_z(v)}$. 

\noindent So on the events $\{\max_{|w|\leq m_n}|V(w)|\leq \varepsilon h_n\}$ and $\{\underline{V}_u(x){:=\min_{ u < w \leq x }(V(w)-V(u))} \geq -B\}$, for any $B>0$ :
\begin{align*}
   \forall i\leq m_n:\;H_{x_i}\leq m_ne^{2\varepsilon h_n}\;\;\textrm{ and }\;\; \forall\; m_n<j\leq |x|:\;
    H_{x_j}\leq m_ne^{2\varepsilon h_n+B}+H_{u,x_j}.
\end{align*}
Assume $n^b<H_{u,x}\leq n^be^{\varepsilon h_n}$. Then, $H_x>n^b$ and for $n$ large enough (recall $h_n\leq \log n$ for $n$ large enough, $h_n\to\infty$ and $\varepsilon\in(0,\varepsilon_b)$)
\begin{align*}
    \frac{1}{H_x}\Big(1-\frac{1}{H_x}\Big)^{n^b}\geq \frac{(1-1/n^b)^{n^b}}{m_ne^{2\varepsilon h_n+B}+H_{u,x}}\geq \frac{(1-1/n^b)^{n^b}}{m_ne^{2\varepsilon h_n+B}+n^be^{\varepsilon h_n}}\geq \frac{e^{-3\varepsilon h_n}}{n^b}.
\end{align*}
Now introduce the translated $(\lambda,\lambda')$-regular lines
\begin{align*}
    \mathcal{O}^v_{\lambda,\lambda'} := \big\{y\in \mathbb{T}, y>v;\; \underset{|v|< j\leq|y|}{\max}\; H_{v,y_j}\leq \lambda,\; H_{v,y}>\lambda'\big\}.
\end{align*}
Note that for $n$ large enough, $\mathcal{O}^u_{\lambda_n/2,n^b}\subset \mathcal{O}_{\lambda_n,n^b}$. Indeed, if $|u|=m_n$ and $m_n<j\leq |x|$, then $H_{x_j}\leq m_ne^{2\varepsilon h_n+B}+H_{u,x_j}$. Moreover, $m_ne^{2\varepsilon h_n+B}\leq e^{3\varepsilon h_n}\leq\lambda_n/2$ for $n$ large since $\varepsilon\in(0,1/13)$, so $H_{u,x_j}\leq \lambda_n/2$ implies $H_{x_j}\leq\lambda_n$. \\ 
For $f^{n,m_n+k}$,  we simply write (still on the event $\{\underset{|w|\leq m_n}{\max}|V(w)|\leq \varepsilon h_n \}$)
\begin{align*}
    & f^{n,m_n+k}(V(x_{1}),\ldots, V(x)) \geq f^{n,k}_{\varepsilon h_n}(V_u(x_{m_n+1}),\ldots, V_u(x)),
\end{align*}
where we recall that $f^{n,k}_{h}(t_1,\ldots, t_k) =  \inf_{\mathbf{s}\in [- h, h]^{m}}f^{n,m+k}\left(s_1,\ldots,s_m,t_1+s_m,\ldots, t_k+s_m\right)$ with $m = \lceil h/c_2\rceil$. In the same way, if $|V(u)|\leq\varepsilon h_n$ then $\un_{\left\{V(x)\geq \mathfrak{z}_n\right\}}\geq \un_{\{V_u(x)\geq 2\mathfrak{z}_n\}}$ since $\varepsilon<1$ and $h_n\leq\ell_n^{1/3}$. We finally obtain, for $n$ large enough (independently of the environment)
%\begin{align*}
%    \un_{\left\{V(x)\geq v_n\right\}}\geq \un_{\left\{V_u(x)\geq v'_n\right\}}\un_{\left\{\underline{V}_u(x)\geq -B\right\}} =\un_{\mathcal{H}^k_{B,v'_n}}\left(V_u\left(x_{m_n+1}\right), \ldots, V_u(x)\right) 
%\end{align*}
%and the event of positive potential $\{\min_{x \in \T}V(x) \geq - B\}$, and obtain 
on $\{\max_{|w|\leq m_n}|V(w)|\leq \varepsilon h_n \}$ that  $R(\mathbf{f}^n)$ is larger than
 \begin{align}      
     &\frac{e^{-3\varepsilon h_n}}{n^b} \sum_{|u|=m_n}e^{-V(u)}\sum_{k\geq 1}\underset{x>u;\;x\in\mathcal{O}^u_{\lambda_n/2,n^b}}{\sum_{|x|=k+m_n}}e^{-V_u(x)}\un_{\{H_{u,x}\leq n^be^{\varepsilon h_n}\}} f^{n,k}_{\varepsilon h_n}\un_{\mathcal{H}^k_{B,2\mathfrak{z}_n}} (V_u(x_{m_n+1}), \ldots, V_u(x)) \nonumber \\ & \geq \frac{e^{-4\varepsilon h_n}}{n^b} \sum_{|u|=m_n}\sum_{k\geq 1}\underset{x>u;\;x\in\mathcal{O}^u_{\lambda_n/2,n^b}}{\sum_{|x|=k+m_n}}e^{-V_u(x)}f^{n,k}_{\varepsilon h_n}\un_{\Upsilon^k_n}(V_u(x_{m_n+1}), \ldots, V_u(x)). \label{CutMinor}
\end{align}
Now, introduce the random variable $Z^u_n$
\begin{align*}%\label{Def_Zn}
    Z^u_n :=\sum_{k\geq 1}\underset{x>u;\;x\in\mathcal{O}^u_{\lambda_n/2,n^b}}{\sum_{|x|=k+m_n}}e^{-V_u(x)} f^{n,k}_{\varepsilon h_n}\un_{\Upsilon^k_n}(V_u(x_{m_n+1}), \ldots, V_u(x)),
\end{align*}
we obtain 
\begin{align*}
    \mathbf{P}\big(R(\mathbf{f}^n)<e^{-4\varepsilon h_n}\mathbf{E}[Z_n]/n^b,\max_{|w|\leq m_n}|V(w)|\leq \varepsilon h_n\big)\leq \mathbf{P}\Big(\sum_{|u|=m_n}Z_n^u<\Eb[Z_n]\Big),
\end{align*}
with 
\begin{align}
Z_n:=\sum_{x\in\mathcal{O}_{\lambda_n/2,n^b}}e^{-V(x)} f^{n,|x|}_{\varepsilon h_n}\un_{\Upsilon^{|x|}_n}(V(x_{1}), \ldots, V(x))
 \label{Def_Zn}. 
 \end{align}
  Hence, by Lemma 2.4 in \cite{{AndDiel3}}, there exists $c_4>0$ such that for $n$ large enough
\begin{align}\label{InegConcent}
    \mathbf{P}^*\big(R(\mathbf{f}^n)<e^{-4\varepsilon h_n}\mathbf{E}[Z_n]/n^b,\max_{|w|\leq m_n}|V(w)|\leq \varepsilon h_n\big)\leq e^{-c_4m_n}\frac{\mathbf{E}[Z_n^2]}{\mathbf{E}[Z_n]^2},   
\end{align}
and finally, \eqref{Defa_n} yields
\begin{align*}
     \mathbf{P}^*\big(R(\mathbf{f}^n)<a_n/n^b,\max_{|w|\leq m_n}|V(w)|\leq \varepsilon h_n\big)\leq \frac{e^{-\varepsilon\frac{c_4}{c_2}h_n}\mathbf{E}[Z_n^2]}{\big(\sum_{k\geq 1}\Psi^k_{\lambda_n/2,n^b}\big(f^{n,k}_{\varepsilon h_n}\un_{\Upsilon^k_n}\big)\big)^2},
\end{align*}
we have used that $\Eb[Z_n]=\sum_{k\geq 1}\Psi^k_{\lambda_n/2,n^b}\big(f^{n,k}_{\varepsilon h_n}\un_{\Upsilon^k_n}\big)$ and $m_n=\lceil \varepsilon h_n/c_2\rceil$. Finally, \eqref{GenPotMax} finishes the proof.
\end{proof}

\noindent The next step is to give a lower bound for $\mathbf{E}[Z_n^2]$, we do that in the dedicated section below.
%\vspace{0.3cm}

\subsubsection{Control of the second moment $\mathbf{E}[Z_n^2]$}

In this section we prove the following lemma, 

\begin{lemm}\label{Lemm_Minor_Esp_du_carre}
Assume \eqref{Hypothese1} and \eqref{Hypothese2}  hold. For all $\varepsilon\in(0,\varepsilon_b)$, $A>2/c_3$ and $n$ large enough
\begin{align*}
    \mathbf{E}[Z_n^2]\leq e^{\frac{6\varepsilon}{A} h_n}\Big(\sum_{k\geq 1}\Psi^k_{n,n^b}(f^{n,k})\Big)^2,
    %\label{Minor_Esp_du_carre}
\end{align*}
recall also that $c_3$ comes from Remark \ref{rem1b}. 
\end{lemm}

\begin{proof}
The expression of $Z_n^2$ is given by $\sum_{x,y\in\mathcal{O}_{\lambda_n/2,n^b}}e^{-V(x)-V(y)}f^{n,|x|}_{\varepsilon h_n}\un_{\Upsilon^{|x|}_n}(\bm{V}_x)f^{n,|y|}_{\varepsilon h_n}\un_{\Upsilon^{|y|}_n}(\bm{V}_y)$ (see \eqref{Def_Zn}) and $\lambda_n\leq n$ so
\begin{align}
    Z_n^2\leq\sum_{x,y\in\mathcal{O}_{n,n^b}}e^{-V(x)}e^{-V(y)}f^{n,|x|}_{\varepsilon h_n}\un_{\mathcal{H}^{|x|}_{B,2\mathfrak{z}_n}}(\mathbf{V}_x)f^{n,|y|}_{\varepsilon h_n}\un_{\mathcal{H}^{|y|}_{B,2\mathfrak{z}_n}}(\mathbf{V}_y),
    \label{lasomme_bis}
\end{align}
with (recall) $F(\bm{V}_w)=F(V(w_1),\ldots,V(w))$. Let us split the computations of the upper bound of the mean of $Z_n^2$ into two main cases : the first one is when $x$ and $y$ in the sum \eqref{lasomme_bis} are directly related in the tree and the second one when it is not: \\
\textit{Cases 1} ($x\leq y$ or $y\leq x$) : recall $\mathfrak{z}_n=\ell_n^{1/3}/\delta_1$ with $\ell_n=(\log n)^3$, $\delta_1\in(0,1/2)$ (see \eqref{hyp1}). For this case, we simply use the fact that $f^{n,i}_{\varepsilon h_n}\leq C_{\infty}$ and $e^{-2V(w)}\un_{\{V(w) \geq 2\mathfrak{z}_n\}}\leq e^{-V(w)}/n^2$ so by symmetry 
\begin{align*}
    \mathbf{E}\Big[\underset{x,y\in \Hline{n}{n^b}}{\sum_{x\leq y \textrm{ or } y\leq x}} e^{-V(x)-V(y)} \un_{\{V(x) \geq 2\mathfrak{z}_n\}}\Big] &\leq 2\mathbf{E}\Big[\sum_{x\in\mathcal{O}_{n}}e^{-2V(x)}\un_{\{V(x) \geq 2\mathfrak{z}_n\}}\underset{y\in\mathcal{O}^x_{n}}{\sum_{y\geq x}}e^{-V_x(y)}\Big] \nonumber \\ &\leq \frac{2}{n^2}\mathbf{E}\Big[\sum_{x\in\mathcal{O}_{n}}e^{-V(x)}\underset{y\in\mathcal{O}^x_{n}}{\sum_{y\geq x}}e^{-V_x(y)}\Big], 
\end{align*}
which is equal, by using that the increments of $V$ are conditionally independent and stationary, to $2\mathbf{E}[\sum_{x\in\mathcal{O}_{n}}e^{-V(x)}]^2/n^2$. Then, thanks to Remark \ref{Rem1} and the fact that $h_n \geq (\log n)^{\gamma}$ with $0 <\gamma \leq 1$,   $2\mathbf{E}[\sum_{x\in\mathcal{O}_{n}}e^{-V(x)} ]^2\leq \ell_n \leq e^{\varepsilon h_n/A}$. In addition with assumption \eqref{Hypothese1}, the part $\{x\leq y$ or $y\leq x$\} in the sum \eqref{lasomme_bis} is smaller than $e^{\frac{\varepsilon}{A} h_n}\big(\sum_{k\geq 1}\Psi^k_{n,n^b}(f^{n,k})\big)^2$.

\vspace{0.3cm}

\noindent\textit{Cases 2} ($x\not\sim y$) : recall that $x\not\sim y$ if and only if neither $x\leq y$ nor $y\leq x$. First let
\begin{align*}
    \Sigma_0(z):=\underset{x,y\in\mathcal{O}_{n,n^b}}{\sum_{x\not\sim y}}\un_{\{x\land y=z\}}e^{-V(x)}e^{-V(y)}f^{n,|x|}_{\varepsilon h_n}\un_{\mathcal{H}^{|x|}_{B,2\mathfrak{z}_n}}(\mathbf{V}_x)f^{n,|y|}_{\varepsilon h_n}\un_{\mathcal{H}^{|y|}_{B,2\mathfrak{z}_n}}(\mathbf{V}_y).
\end{align*}
We decompose $\Sigma_0(z)$ as follows: for all $A>2/c_3$
\begin{align}
    \sum_{z\in\T}\Sigma_0(z) = \sum_{|z|\geq\lfloor A\ell_n\rfloor}\Sigma_0(z) + \sum_{|z|<\lfloor A\ell_n\rfloor}(\Sigma_1(z)+\Sigma_2(z)), \label{lem3.4eq1}
\end{align}
and for any $i\in\{1,2\}$, 
\begin{align*}
    \Sigma_i(z):=\underset{x,y\in\mathcal{O}_{n,n^b}}{\sum_{x\not\sim y}}\un_{\{x\land y=z\}}e^{-V(x)}e^{-V(y)}f^{n,|x|}_{\varepsilon h_n}\un_{\mathcal{H}^{|x|}_{B,2\mathfrak{z}_n}}(\mathbf{V}_x)f^{n,|y|}_{\varepsilon h_n}\un_{\mathcal{H}^{|y|}_{B,2\mathfrak{z}_n}}(\mathbf{V}_y)\un_{\{(x,y)\in\mathcal{C}_{i,z}\}},
\end{align*}
with $\mathcal{C}_{1,z}:=\{(x,y)\in\T^2;x^*>z\textrm{ and }y^*>z\}$ and $\mathcal{C}_{2,z}:=\{(x,y)\in\T^2;x^*=z\textrm{ or }y^*=z\}$. \\
Let us  start with the easiest part: $\sum_{|z|\geq\lfloor A\ell_n\rfloor}\Sigma_0(z)$. Observe that
\begin{align*}
    \sum_{|z|\geq\lfloor A\ell_n\rfloor}\Sigma_0(z)\leq C_{\infty}^2 \sum_{l\geq\lfloor A\ell_n\rfloor}\sum_{|z|=l}\un_{\{V(z)\geq -B,\;z\in\mathcal{O}_n\}}\underset{u^*=v^*=z}{\sum_{u\not=v}}\;\;\underset{x\in\mathcal{O}_{n}}{\sum_{x\geq u}}e^{-V(x)}\underset{y\in\mathcal{O}_{n}}{\sum_{y\geq v}}e^{-V(y)}.
\end{align*}
By conditional independence of the increments of $V$ and Remark \ref{Rem1}, for any  $n$ large enough
\begin{align}
    \Eb\Big[\sum_{|z|\geq\lfloor A\ell_n\rfloor}\Sigma_0(z)\Big]&\leq C_{\infty}^2e^{B}\Eb\Big[\Big(\sum_{|u|=1}e^{-V(u)}\Big)^2\Big]\Eb\Big[\sum_{x\in\mathcal{O}_n}e^{-V(x)}\Big]^2\sum_{l\geq\lfloor A\ell_n\rfloor}\Eb\Big[\sum_{|z|=l}e^{-V(z)}\un_{\{z\in\mathcal{O}_n\}}\Big] \nonumber \\ & \leq C_{\infty}^2e^{B}\Eb\Big[\Big(\sum_{|u|=1}e^{-V(u)}\Big)^2\Big]\ell_n^2n^{-2}\leq \sum_{k\geq 1}\Psi^k_{n,n^b}(f^{n,k}), \label{lem3.4eq2}
\end{align}
where we have used \eqref{Hypothese1} and \eqref{hyp1} for the last inequality. \\
For $\Sigma_1(z)$, $|z|<\lfloor A\ell_n\rfloor$, we decompose according to the value of $V(w)$ with $w\in\{u,v\}$: $\Sigma_1(z)=\Sigma_{1,1}(z)+\Sigma_{1,2}(z)$ with 
\begin{align*}
    \Sigma_{1,1}(z):= \underset{u^*=v^*=z}{\sum_{u\not=v}}\un_{\{V(u)\lor V(v)<2\mathfrak{z}_n\}}\underset{x\in\mathcal{O}_{n,n^b}}{\sum_{x> u}}e^{-V(x)}f^{n,|x|}_{\varepsilon h_n}\un_{\mathcal{H}^{|x|}_{B,2\mathfrak{z}_n}}(\mathbf{V}_x)\underset{y\in\mathcal{O}_{n,n^b}}{\sum_{y> v}}e^{-V(y)}f^{n,|y|}_{\varepsilon  h_n}\un_{\mathcal{H}^{|y|}_{B,2\mathfrak{z}_n}}(\mathbf{V}_y),
\end{align*}
and 
\begin{align*}
    \Sigma_{1,2}(z):= \underset{u^*=v^*=z}{\sum_{u\not=v}}\un_{\{V(u)\lor V(v)\geq 2\mathfrak{z}_n\}}\underset{x\in\mathcal{O}_{n,n^b}}{\sum_{x> u}}e^{-V(x)}f^{n,|x|}_{\varepsilon h_n}\un_{\mathcal{H}^{|x|}_{B,2\mathfrak{z}_n}}(\mathbf{V}_x)\underset{y\in\mathcal{O}_{n,n^b}}{\sum_{y> v}}e^{-V(y)}f^{n,|y|}_{\varepsilon  h_n}\un_{\mathcal{H}^{|y|}_{B,2\mathfrak{z}_n}}(\mathbf{V}_y).
\end{align*}
We first deal with $\Sigma_{1,1}(z)$. Observe that $x\in\Hline{n}{n^b}$ (resp. $y\in\Hline{n}{n^b}$) means $H_u\leq n$ (resp. $H_v\leq n$), $x\in\mathcal{O}^u_n$ (resp. $y\in\mathcal{O}^v_n$) and $n^b-H_ue^{-V_u(x)}<H_{u,x}$ (resp. $n^b-H_ve^{-V_v(y)}<H_{v,y}$). Besides, $V(u)<2\mathfrak{z}_n$ and $V(x)>2\mathfrak{z}_n$ (resp. $V(v)<2\mathfrak{z}_n$ and $V(y)>2\mathfrak{z}_n$) implies $V_u(x)>0$ (resp. $V_v(y)>0$) that is $n^b-H_u<H_{u,x}$ (resp. $n^b-H_v<H_{v,y}$), so $\Sigma_{1,1}(z)$ is smaller than
\begin{align}\label{lasomme_bis'}
   \underset{u^*=v^*=z}{\sum_{u\not=v}}\un_{\{\underline{V}(u)\land \underline{V}(v)\geq -B,H_u\lor H_v\leq n\}}\underset{x\in\mathcal{O}^u_{n,n^b-H_u}}{\sum_{x>u}}e^{-V(x)}f^{n,|x|}_{\varepsilon h_n}(\mathbf{V}_x)\underset{y\in\mathcal{O}^v_{n,n^b-H_v}}{\sum_{y>v}} e^{-V(y)}f^{n,|y|}_{\varepsilon h_n}(\mathbf{V}_y). 
\end{align}
We now move to $\Sigma_{1,2}(z)$. Note that $\{V(u)\lor V(v)\geq 2\mathfrak{z}_n\}=\{V(u)\geq 2\mathfrak{z}_n,V(v)<2\mathfrak{z}_n\}\cup\{V(v)\geq 2\mathfrak{z}_n,V(u)<2\mathfrak{z}_n\}\cup\{V(u)\land V(v)\geq 2\mathfrak{z}_n\}$. By symmetry, $\Sigma_{1,2}(z)$ is equal to 
\begin{align*}
    &2\underset{u^*=v^*=z}{\sum_{u\not=v}}\un_{\{V(u)\geq 2\mathfrak{z}_n,V(v)<2\mathfrak{z}_n\}}\underset{x\in\mathcal{O}_{n,n^b}}{\sum_{x> u}}e^{-V(x)}f^{n,|x|}_{\varepsilon h_n}\un_{\mathcal{H}^{|x|}_{B,2\mathfrak{z}_n}}(\mathbf{V}_x)\underset{y\in\mathcal{O}_{n,n^b}}{\sum_{y> v}}e^{-V(y)}f^{n,|y|}_{\varepsilon  h_n}\un_{\mathcal{H}^{|y|}_{B,2\mathfrak{z}_n}}(\mathbf{V}_y) \\ & +\underset{u^*=v^*=z}{\sum_{u\not=v}}\un_{\{V(u)\land V(v)\geq 2\mathfrak{z}_n\}}\underset{x\in\mathcal{O}_{n,n^b}}{\sum_{x> u}}e^{-V(x)}f^{n,|x|}_{\varepsilon h_n}\un_{\mathcal{H}^{|x|}_{B,2\mathfrak{z}_n}}(\mathbf{V}_x)\underset{y\in\mathcal{O}_{n,n^b}}{\sum_{y> v}}e^{-V(y)}f^{n,|y|}_{\varepsilon  h_n}\un_{\mathcal{H}^{|y|}_{B,2\mathfrak{z}_n}}(\mathbf{V}_y).
\end{align*}
The same decomposition of $H_y$ we used for $\Sigma_{1,1}(z)$ also works for the part $\{V(v)<2\mathfrak{z}_n\}$ in the above sum, so as in \eqref{lasomme_bis'} and first using that on $\{V(u)\geq 2\mathfrak{z_n}\}\cap\{\underline{V}(y)\geq-B\}$ ,$V(u) \geq (1-\delta_1)V(u)+ 2\log n$ and $\delta_1V(v)\geq -\delta_1B$, then using that on $\{V(u)\land V(v)\geq 2\mathfrak{z}_n\}$, $V(u)+V(v)\geq (1-\delta_1)(V(u)+V(v))+4\log n$, $\Sigma_{1,2}(z)$ is smaller than 
\begin{align*}
    &\un_{\{V(z)\geq-B,  z \in \mathcal{O}_{n}\}}2e^{\delta_1B}\frac{C_{\infty}^2}{n^2}\underset{u^*=v^*=z}{\sum_{u\not=v}}e^{-(1-\delta_1)(V(u)+V(v))}\underset{x\in\mathcal{O}_{n}^u}{\sum_{x>u}}e^{-V_u(x)}\underset{y\in\mathcal{O}^v_{n}}{\sum_{y>v}}e^{-V_v(y)} \\ & + \un_{\{V(z)\geq-B,  z \in \mathcal{O}_{n}\}}\frac{C_{\infty}^2}{n^4}\underset{u^*=v^*=z}{\sum_{u\not=v}}e^{-(1-\delta_1)(V(u)+V(v))}\underset{x\in\mathcal{O}_{n}^u}{\sum_{x>u}}e^{-V_u(x)}\underset{y\in\mathcal{O}^v_{n}}{\sum_{y>v}}e^{-V_v(y)} \\ &  \leq\un_{\{V(z)\geq-B,  z \in \mathcal{O}_{n}\}}3e^{\delta_1B}\frac{C_{\infty}^2}{n^2}\underset{u^*=v^*=z}{\sum_{u\not=v}}e^{-(1-\delta_1)(V(u)+V(v))}\underset{x\in\mathcal{O}_{n}^u}{\sum_{x>u}}e^{-V_u(x)}\underset{y\in\mathcal{O}^v_{n}}{\sum_{y>v}}e^{-V_v(y)}.
\end{align*}
Note that the genealogical common line between $x$ and $y$ is the common line of individuals before $u$ and $v$ so for any $p\leq |z|$, $x_p=y_p=u_p=v_p$ and
\begin{align*}
   f^{n,|x|}_{\varepsilon h_n}(\mathbf{V}_x)= f^{n,|x|}_{\varepsilon h_n}(V(u_1),\cdots,V(u),V_u(x_{|u|+1})+V(u),\cdots,V_u(x)+V(u)), 
\end{align*}
and
\begin{align*}
   f^{n,|y|}_{\varepsilon h_n}(\mathbf{V}_y) = f^{n,|y|}_{\varepsilon h_n}(V(v_1),\cdots,V(v),V_v(y_{|v|+1})+V(v),\cdots,V_v(y)+V(v)). 
\end{align*}
Recall that for all $q\geq 1$ and $\mathbf{t}_q=(t_1,\ldots,t_q)\in\R^q$, \begin{align*}
    \Psi^k_n(F|\mathbf{t}_p)=\Eb\Big[\sum_{|x|=k}e^{-V(x)}F(t_1,\ldots,t_p,V(x_1)+t_p,\ldots,V(x)+t_p)\un_{\mathcal{O}_n}(x)\Big].
\end{align*}
We naturally note $\Psi^k_n(F|\bm{V}_w)$ when we evaluate the function $\Psi^k_n(F|\cdot)$ at $(V(w_1),\ldots,V(w))$.\\ By  conditional independence of the increments of $V$, $\Eb[\sum_{|z|=l}\Sigma_{1}(z)]=\Eb[\sum_{|z|=l}\Sigma_{1,1}(z)+\Sigma_{1,2}(z)]$ is smaller, for $n$ large enough with $l<\lfloor A\ell_n\rfloor$, than
\begin{align*}
    &\Eb\Big[\sum_{|z|=l}\underset{u^*=v^*=z}{\sum_{u\not=v}}\un_{\{\underline{V}(u)\land \underline{V}(v)\geq -B,H_u\lor H_v\leq n \}}\sum_{i,j\geq 1}\prod_{(k,w)\in{\{(i,u);(j,v)\}}}e^{-V(w)}\Psi^k_{n,n^b-H_w}\big(f^{n,|w|+k}_{\varepsilon h_n}|\mathbf{V}_w\big)\Big] \\ & + 3e^{\delta_1B}\ell_n^2\frac{C_{\infty}^2}{n^2}\Eb\Big[\sum_{|z|=l}\un_{\{V(z)\geq-B,  z \in \mathcal{O}_{n}\}}\underset{u^*=v^*=z}{\sum_{u\not=v}}e^{-(1-\delta_1)(V(u)+V(v))}\Big],
\end{align*}
where we have used that $\Eb[\sum_{x\in\mathcal{O}_{n}}e^{-V(x)}]\leq\ell_n$. Then, by assumption \eqref{Hypothese2} with $\delta=\delta_1$ (see \eqref{hyp1} for the definition of $\delta_1$), for all $l<\lfloor A\ell_n\rfloor$ $(|u|=|v|=l+1)$ and $n$ large enough, on the event $\{V(u)\land V(v)\geq -B, H_u\lor H_v\leq n\}$
\begin{align*} 
    \sum_{i,j\geq 1}\prod_{(k,w)\in{\{(i,u);(j,v)\}}}\Psi^k_{n,n^b-H_w}\big(f^{n,|w|+k}_{\varepsilon h_n}|\mathbf{V}_w\big)\leq e^{\delta_1V(u)+\delta_1 V(v)+\frac{2\varepsilon}{A}h_n}\Big(\sum_{k\geq 1}\Psi^k_{n,n^b}( f^{n,k})\Big)^2.
\end{align*}
Hence, $\Eb[\sum_{|z|<\lfloor A\ell_n\rfloor}\Sigma_{1}]$ is smaller, for $n$ large enough, than
\begin{align*}
    & e^{\frac{2\varepsilon}{A}h_n}\Eb\Big[\Big(\sum_{|w|=1}e^{-(1-\delta_1)V(w)}\Big)^2\Big]\Eb\Big[\sum_{|z|<\lfloor A\ell_n\rfloor}e^{-V(z)-(1-2\delta_1)V(z)}\un_{\{V(z)\geq -B\}}\Big]\Big(\sum_{k\geq 1}\Psi^k_{n,n^b}( f^{n,k})\Big)^2 \\ & +3e^{\delta_1B}\ell_n^2\frac{C_{\infty}^2}{n^2}\Eb\Big[\Big(\sum_{|w|=1}e^{-(1-\delta_1)V(w)}\Big)^2\Big]\Eb\Big[\sum_{z\in\mathcal{O}_n}e^{-V(z)-(1-2\delta_1)V(z)}\un_{\{V(z)\geq -B\}}\Big].
\end{align*} 
Finally, thanks to assumption \eqref{Hypothese1}, \eqref{hyp1} and by Remark \ref{Rem1}, for $n$ large enough
\begin{align}
    \Eb\Big[\sum_{|z|<\lfloor A\ell_n\rfloor}\Sigma_{1}(z)\Big]\leq e^{\frac{5\varepsilon}{A}h_n}\Big(\sum_{k\geq 1}\Psi^k_{n,n^b}( f^{n,k})\Big)^2.
    \label{Sigma1n}
\end{align}
We now turn to $\Sigma_2(z)$, that is the sum
\begin{align*}
    \underset{x,y\in\mathcal{O}_{n,n^b}}{\sum_{x\not\sim y}}\un_{\{x\land y=z\}}e^{-V(x)}e^{-V(y)}f^{n,|x|}_{\varepsilon h_n}\un_{\mathcal{H}^{|x|}_{B,2\mathfrak{z}_n}}(\mathbf{V}_x)f^{n,|y|}_{\varepsilon h_n}\un_{\mathcal{H}^{|y|}_{B,2\mathfrak{z}_n}}(\mathbf{V}_y)\un_{\{(x,y)\in\mathcal{C}_{2,z}\}},
\end{align*}
with $\mathcal{C}_{2,z}:=\{(x,y)\in\T^2;x^*=z\textrm{ or }y^*=z\}$. The first step is to split the set $\{x^*=z \textrm{ or } y^*=z\}$ into three disjoint sets: $\{x^*=z \textrm{ and } y^*>z\}$, $\{x^*>z \textrm{ and } y^*=z\}$ and $\{x^*=z \textrm{ and } y^*=z\}$. By symmetry, the previous sum is equal to 
\begin{align*}
    &2\underset{x^*=v^*=z}{\sum_{x\not=v}}\un_{\{x\in\Hline{n}{n^b}\}}e^{-V(x)}f^{n,|x|}_{\varepsilon h_n}\un_{\mathcal{H}^{|x|}_{B,2\mathfrak{z}_n}}(\mathbf{V}_x)\underset{y\in\Hline{n}{n^b}}{\sum_{y>v}}e^{-V(y)}f^{n,|y|}_{\varepsilon h_n}\un_{\mathcal{H}^{|y|}_{B,2\mathfrak{z}_n}}(\mathbf{V}_y) \\ & + \underset{x^*=y^*=z}{\sum_{x\not=y}}e^{-V(x)}e^{-V(y)}\un_{\{x,y\in\Hline{n}{n^b}\}}f^{n,|x|}_{\varepsilon h_n}\un_{\mathcal{H}^{|x|}_{B,2\mathfrak{z}_n}}(\bm{V}_x)f^{n,|y|}_{\varepsilon h_n}\un_{\mathcal{H}^{|y|}_{B,2\mathfrak{z}_n}}(\bm{V}_y).
\end{align*}
We then use a similar approach as the one we used for $\Sigma_1(z)$ to obtain
\begin{align*}
    \sum_{|z|=l}\Sigma_2(z)\leq&\frac{2C_{\infty}^2}{n^2}\sum_{|z|=l}\underset{x^*=v^*=z}{\sum_{x\not=v}}e^{-(1-\delta_1)V(x)}\un_{\{\underline{V}(v)\geq -B\}}\underset{y\in\mathcal{O}^v_{n}}{\sum_{y>v}}e^{-V(y)} \\ & + \frac{C_{\infty}^2}{n^4}\sum_{|z|=l}\un_{\{V(z)\geq -B\}}\underset{x^*=y^*=z}{\sum_{x\not=y}}e^{-(1-\delta_1)(V(x)+V(y))}.
\end{align*}
Hence, by using conditional independence of the increments of $V$, $\Eb[\sum_{|z|=l}\Sigma_{2}(z)]$ is smaller, for $n$ large enough, than
\begin{align*}
    &2e^{\delta_1B}\ell_n\frac{C_{\infty}^2}{n^2}\Eb\Big[\sum_{|z|=l}\un_{\{V(z)\geq -B\}}\underset{x^*=v^*=z}{\sum_{x\not=v}}e^{-(1-\delta_1)(V(x)+V(v))}\Big] \\ & + \frac{C_{\infty}^2}{n^4}\Eb\Big[\sum_{|z|=l}\un_{\{V(z)\geq -B\}}\underset{x^*=y^*=z}{\sum_{x\not=v}}e^{-(1-\delta_1)(V(x)+V(y))}\Big],
\end{align*} 
where we used as usual $\Eb[\sum_{x\in\mathcal{O}_{n}}e^{-V(x)}]\leq\ell_n$. Hence, thanks to assumption \eqref{Hypothese1} and \eqref{hyp1}, for $n$ large enough
\begin{align}
    \Eb\Big[\sum_{|z|<\lfloor A\ell_n\rfloor}\Sigma_{2}(z)\Big]\leq e^{\frac{3\varepsilon}{A}h_n}\Big(\sum_{k\geq 1}\Psi^k_{n,n^b}( f^{n,k})\Big)^2.
    \label{Sigma2n}
\end{align}
Collecting Case 1, Case 2 (\eqref{lem3.4eq1}, inequalities \eqref{lem3.4eq2}, \eqref{Sigma1n} and \eqref{Sigma2n}) and considering \eqref{lasomme_bis} give the lemma. 

\end{proof}

\noindent We are now ready to prove the lower bound of $\mathcal{R}_{T^n}(g_n,\mathbf{f}^n)$ in Proposition \ref{Prop1}. Recall $u_{1,n}=\sum_{k\geq 1}\Psi^{k}_{\lambda_n/2,n^b}\big(f^{n,k}_{\varepsilon h_n}\mathds{1}_{\Upsilon^k_n}\big)$ where $\Upsilon^k_n=\{\mathbf{t}\in\R^k;\;H_k(\mathbf{t})\leq n^be^{\varepsilon h_n}\}\cap\mathcal{H}^k_{B,2\mathfrak{z}_n}$, $\mathcal{H}^k_{B,2\mathfrak{z}_n}$  is defined in \eqref{HighSet} and $\mathfrak{z}_n=2\ell_n^{1/3}/\delta_1$. Thanks to Lemmata \ref{LemmMinorProbaRange1},  \ref{lemmeMarjorProba} and the expression of $a_n$ \eqref{Defa_n}, for $n$ large enough, as $e^{-\varepsilon h_n}\leq\frac{1}{4}$, the probability $\P(\mathcal{R}_{T^n}(g_n,\mathbf{f}^n)<n^{1-b}\varphi(n^b)e^{-5\varepsilon h_n}u_{1,n})$ is smaller than
\begin{align*}
    \P\big(\mathcal{R}_{T^n}(g_n,\mathbf{f}^n)<n\varphi(n^b)e^{-4\varepsilon h_n}u_{1,n}/4n^b\big)\leq e^{-\varepsilon\frac{c_4}{c_2}h_n}\frac{\mathbf{E}[Z_n^2]}{u_{1,n}^2} + h_ne^{-\varepsilon \tilde{c}_2h_n}+\frac{e^{8\varepsilon h_n-\min(9\varepsilon\log n,4h_n)}}{n^{2\kappa_b}u_{1,n}^2}.
\end{align*} 
Then, Lemma \ref{Lemm_Minor_Esp_du_carre} provides the upper bound of $\Eb[Z_n^2]$ so $ \P\big(\mathcal{R}_{T^n}(g_n,\mathbf{f}^n)<n\varphi(n^b)e^{-4\varepsilon h_n}u_{1,n}/4n^b\big)$ is smaller, for $n$ large enough, than (recall that $h_n\leq \log n$)
\begin{align*}
    e^{-(\frac{c_4}{c_2}-\frac{6}{A})\varepsilon h_n}\Big(\sum_{k\geq 1}\Psi^k_{n,n^b}(f^{n,k})/u_{1,n}\Big)^2+h_ne^{-\varepsilon \tilde{c}_2h_n}+\frac{e^{-\min(\varepsilon\log n,3h_n)}}{n^{2\kappa_b}u_{1,n}^2}, 
\end{align*}
 which yields the lower bound of Proposition \ref{Prop1}.

\subsection{Upper bound for $\mathcal{R}_{T^n}(g_n,\mathbf{f}^n)$} \label{sec3.3}

\noindent For all $n\geq 1$ and $x\in\T$, recall that $E^n_x$ is the number of excursions, among the first $n$ excursions to the root, for which the  edge $(x^*,x)$ is reached. In a similar way, $\tilde{E}^{n}_x$ is the number of excursions such that the vertex $x$ is reached more often from above than from below : 
\begin{align*}
    E^n_x=\sum_{i=1}^n\un_{\{ N^{T^{i}}_x-N^{T^{i-1}}_x\geq 1\}}\;\;\textrm{ and }\;\; \tilde{E}^{n}_x:=\sum_{i=1}^n\un_{\{\sum_{y;y^*=x} N^{T^i}_y-N^{T^{i-1}}_y>N^{T^i}_x-N^{T^{i-1}}_x\}}.
\end{align*}
Also introduce the event $\mathcal{A}_n$ such that all vertices of the trace of $\{ X_k,k \leq T^n\}$ have exponential downfall fluctuation lower than $n$, potential larger than $\mathfrak{z}_n=\ell_n^{1/3}/\delta_1$ and which are visited during a single excursion to the root
\begin{align}
     \mathcal{A}_n := \Big\{\forall \; j\leq T^n, X_j\in\mathcal{O}_n,\sum_{x\in\mathcal{O}_n}(\un_{\{E^{n}_x \geq 2\}}+\un_{\{\tilde{E}^{n}_x \geq 2\}})\un_{\mathcal{H}^{|x|}_{\mathfrak{z}_n}}(\bm{V}_x)=0\Big\}. \label{A_n}
\end{align}
Note that $\lim_{n\to\infty}\P(\mathcal{A}_n)=1$. Indeed, $\tilde{E}^n_x\geq 2$ implies $E^n_x\geq 2$ so
\begin{align*}
   1-\P(\mathcal{A}_n)\leq\P(\exists\;j\leq T^n: X_j\not\in\mathcal{O}_n)+\P\Big(\sum_{x\in\mathcal{O}_n}\un_{\{E^{n}_x \geq 2\}}\un_{\mathcal{H}^{|x|}_{\mathfrak{z}_n}}(\bm{V}_x)> 0\Big). 
\end{align*}
By \cite{AndChen} (equation 2.2), $\P(\exists\;j\leq T^n: X_j\not\in\mathcal{O}_n)\to 0$. Moreover, $\P(\sum_{x\in\mathcal{O}_n}\un_{\{E^{n}_x \geq 2\}}\un_{\mathcal{H}^{|x|}_{\mathfrak{z}_n}}(\bm{V}_x)>0)$ is smaller than
\begin{align*}
    \Eb\Big[\sum_{x\in\mathcal{O}_n}\P^{\mathcal{E}}(E^n_x\geq 2)\un_{\mathcal{H}^{|x|}_{\mathfrak{z}_n}}(\bm{V}_x)\Big]=\Eb\Big[\sum_{x\in\mathcal{O}_n}\big(\P^{\mathcal{E}}(E^n_x\geq 1)-\P^{\mathcal{E}}(E^n_x=1)\big)\un_{\mathcal{H}^{|x|}_{\mathfrak{z}_n}}(\bm{V}_x)\Big].
\end{align*}
Thanks to the strong Markov property, $N_x^{T^i}-N_x^{T^{i-1}}, i\in\{1,\ldots,n\},$ are i.i.d under $\P^{\mathcal{E}}$ so $\P^{\mathcal{E}}(E^n_x\geq 1)-\P^{\mathcal{E}}(E^n_x=1)\leq \E^{\mathcal{E}}[E^n_x]-\P^{\mathcal{E}}(E^n_x=1)=n\P^{\mathcal{E}}(N_x^{T^1}\geq 1)(1-\P^{\mathcal{E}}(N_x^{T^1}=0)^{n-1})\leq n^2\P^{\mathcal{E}}(N_x^{T^1}\geq 1)^2$ and by Lemma \ref{LawGeo}, for all $x$ with $V(x)\geq \mathfrak{z}_n$, $n^2\P^{\mathcal{E}}(N_x^{T^1}\geq 1)^2\leq n^2e^{-2V(x)}\leq n^{2-1/\delta_1}e^{-V(x)}/\ell_n^{1/\delta_1}$. $\delta_1\in(0,1/2)$, hence, by Remark \ref{Rem1}
\begin{align*}
    \P\Big(\sum_{x\in\mathcal{O}_n}\un_{\{E^{n}_x \geq 2\}}\un_{\mathcal{H}^{|x|}_{\mathfrak{z}_n}}(\bm{V}_x)>0\Big)\leq\frac{n^{2-1/\delta_1}}{\ell_n^{1/\delta_1}}\Eb\Big[\sum_{x\in\mathcal{O}_n}e^{-V(x)}\Big]\leq \frac{n^{2-1/\delta_1}}{\ell_n^{1/\delta_1-1}}\to 0.
\end{align*}

%\bl{[... j ai enleve tous les rappels qui suivaient, pas pertinent  cet endroit ?]}

%\noindent\bo{Recall that for any $n\in\N^*$ and $b\in[0,1)$
%\begin{align*}
%    \mathcal{R}_{T^n}(g_n,\mathbf{f}^n) = \sum_{x\in\T}g_n\big(\mathcal{L}_x^{T^n}\big)f^{n,|x|}(V(x_1),\ldots,V(x)),
%\end{align*}
%with $g_n(t)=\varphi(t)\un_{\{t\geq n^b\}}$. $\varphi:\R^+\longrightarrow\R^+$ is an increasing function and the application $t\in[1,\infty)\longmapsto\varphi(t)/t$ is non-increasing.} \\

%\noindent\bo{Recall $N_x^{T^n}$ is the number of times the edge $(x^*,x)$ is visited by the walk during the first $n$ excursions and $\mathcal{L}_x^{T^n}$ is the number of times the walk visits the vertex $x$ during the first $n$ excursions that is: $\mathcal{L}_x^{T^n}=\sum_{j=1}^n(N_x^{T^j}-N_x^{T^{j-1}})+\sum_{j=1}^n\sum_{y;y^*=x}(N_y^{T^j}-N_y^{T^{j-1}})$.} \\

%\noindent\bo{recall $h_n$? } 

 \begin{lemm}\label{RangeUpBound1}
Let $(u_n,n)$ be a sequence of positive numbers, then 
\begin{align*}
    \P^{\mathcal{E}}(\mathcal{R}_{T^n}(g_n,\mathbf{f}^n)>u_n,\mathcal{A}_n)\leq \frac{2n^{1-b}\varphi(n^b)}{u_n}(\mathcal{X}_{1,n}+\mathcal{X}_{2,n}+\mathcal{X}_{3,n}), 
\end{align*}
where
\begin{align}\label{DefX_1,n}
   \mathcal{X}_{1,n} := \sum_{x\in\mathcal{O}_n}\un_{\{V(x)<\mathfrak{z}_n\}}\Big(e^{-V(x)}+\sum_{y;y^*=x}e^{-V(y)}\Big)f^{n,|x|}(\bm{V}_x),
\end{align}
\begin{align}\label{DefX_2,n}
    \mathcal{X}_{2,n} := \sum_{x\in\mathcal{O}_n}\un_{\{V(x)\geq \mathfrak{z}_n\}}\frac{e^{-V(x)}}{H_x}\Big(1-\frac{1}{H_x}\Big)^{\lceil n^b/2 \rceil-1}( n^b+H_x)f^{n,|x|}(\bm{V}_x),
\end{align}
and    
\begin{align}\label{DefX_3,n}
    \mathcal{X}_{3,n} := \sum_{x\in\mathcal{O}_n}\un_{\{V(x)\geq \mathfrak{z}_n\}}\frac{e^{-V(x)}}{H_x}\frac{\tilde{H}_x}{1+\tilde{H}_x}\Big(1-\frac{1}{1+\tilde{H}_x}\Big)^{\lceil n^b/2 \rceil-1}( n^b+1+\tilde{H}_x)f^{n,|x|}(\bm{V}_x),
\end{align}
recall the definition of  $\tilde{H}_x$ in Lemma \ref{LawGeo}.
\end{lemm}

\begin{proof}
Since $g_n(0)=0$, we have, by Markov inequality, that $\P^{\mathcal{E}}(\mathcal{R}_{T^n}(g_n,\mathbf{f}^n)>u_n,\mathcal{A}_n)$ is smaller than
\begin{align*}
     \frac{2}{u_n}\Big(&\sum_{x\in\mathcal{O}_n}\un_{\{V(x)<\mathfrak{z}_n\}}\E^{\mathcal{E}}\big[g_n\big(\mathcal{L}_x^{T^n}\big)\big]f^{n,|x|}(\bm{V}_x) \\ & + \sum_{x\in\mathcal{O}_n}\un_{\{V(x)\geq \mathfrak{z}_n\}}\E^{\mathcal{E}}\big[g_n\big(\mathcal{L}_x^{T^n}\big)\un_{\{E^{n}_x,\tilde{E}^{n}_x\in\{0,1\}\}}\big]f^{n,|x|}(\bm{V}_x)\Big).
\end{align*}
The first part in the above sum is the easiest to deal with. Indeed, the application $t\in [1,\infty)\mapsto \varphi(t)/t$ is non-increasing so $g_n(t)\leq tn^{-b}\varphi(n^b)$ and we have
\begin{align*}
    \sum_{x\in\mathcal{O}_n}\un_{\{V(x)<\mathfrak{z}_n\}}\E^{\mathcal{E}}\big[g_n\big(\mathcal{L}_x^{T^n}\big)\big]f^{n,|x|}(\bm{V}_x)&\leq n^{1-b}\varphi(n^b)\sum_{x\in\mathcal{O}_n}\un_{\{V(x)<\mathfrak{z}_n\}}\E^{\mathcal{E}}\big[\mathcal{L}_x^{T^1}\big]f^{n,|x|}(\bm{V}_x) \\ & =n^{1-b}\varphi(n^b)\mathcal{X}_{1,n}.
\end{align*}
We have used that for all $1\leq i\leq n$, $\mathcal{L}^{T^i}_x-\mathcal{L}^{T^{i-1}}_x$ is distributed as $\mathcal{L}^{T^1}_x$ under $\P^{\mathcal{E}}$ with mean $e^{-V(x)}+\sum_{y;y^*=x}e^{-V(y)}$ by Lemma \ref{LawGeo}. \\ \\
We then move to the high potential part. Assume $E^n_x\in\{0,1\}$ and $\tilde{E}^n_x\in\{0,1\}$. If $E^n_x = 0$, then the vertex $x$ is never visited during any of the first $n$ excursions and $\tilde{E}^n_x = 0$. Thus, $g_n\big(\mathcal{L}_x^{T^n}\big)=g_n(0)=0$. If $E^n_x=1$ and $\tilde{E}^n_x=0$, then there exists $i\in\{1,\ldots,n\}$ such that $N^{T^i}_x-N^{T^{i-1}}_x\geq 1$ and $\forall j\not= i$, $N^{T^j}_x-N^{T^{j-1}}_x= 0$ and $\forall m\in\{1,\ldots,n\}$, $\sum_{y;y^*=x} N^{T^m}_y-N^{T^{m-1}}_y\leq N^{T^m}_x-N^{T^{m-1}}_x$. In particular, since, starting from the root $e$, $\mathcal{L}^{T^n}_x = \sum_{j=1}^{n} \big(N^{T^j}_x-N^{T^{j-1}}_x + \sum_{y;y^*=x} N^{T^j}_y-N^{T^{j-1}}_y\big)$, we have, on $\{E^n_x=1,\tilde{E}^n_x=0\}$
\begin{align}\label{Exc1_0}
    \mathcal{L}^{T^n}_x = N^{T^i}_x-N^{T^{i-1}}_x + \sum_{y;y^*=x}N^{T^i}_y-N^{T^{i-1}}_y \leq 2\big(N^{T^i}_x-N^{T^{i-1}}_x\big).
\end{align}

\noindent Otherwise, if $E^n_x=1$ and $\tilde{E}^n_x=1$, then there exists $ i\in\{1,\ldots,n\}$ such that $N^{T^i}_x-N^{T^{i-1}}_x\geq 1$ and $\forall j\not= i$, $N^{T^j}_x-N^{T^{j-1}}_x= 0$ and $\exists m'\in \{1,\ldots,n\}$ such that $\sum_{y;y^*=x} N^{T^{m'}}_y-N^{T^{m'-1}}_y> N^{T^{m'}}_x-N^{T^{m'-1}}_x$ and $\forall m\not= m'$, $\sum_{y;y^*=x} N^{T^m}_y-N^{T^{m-1}}_y\leq N^{T^m}_x-N^{T^{m-1}}_x$. So we have necessarily $m'=i$ and, on $\{E^n_x=1,\tilde{E}^n_x=1\}$
\begin{align}\label{Exc1_1}
    \mathcal{L}^{T^n}_x = N^{T^i}_x-N^{T^{i-1}}_x + \sum_{y;y^*=x}N^{T^i}_y-N^{T^{i-1}}_y \leq 2\sum_{y;y^*=x}N^{T^i}_y-N^{T^{i-1}}_y.
\end{align}

\noindent $g_n$ is non-decreasing so \eqref{Exc1_0} and \eqref{Exc1_1} give, when $E^n_x\in\{0,1\}$ and $\tilde{E}^n_x\in\{0,1\}$
\begin{align*}
    g_n\big(\mathcal{L}^{T^n}_x\big)&\leq \sum_{i=1}^n g_n\big(2\big(N^{T^i}_x-N^{T^{i-1}}_x\big)\big) + \sum_{i=1}^n g_n\big(2\sum_{y;y^*=x}N^{T^i}_y-N^{T^{i-1}}_y\big).
\end{align*}
From this inequality, it follows that $\E^{\mathcal{E}}\big[g_n\big(\mathcal{L}_x^{T^n}\big)\un_{\{E^{n}_x,\tilde{E}^{n}_x\in\{0,1\}\}}\big]$ is smaller than
\begin{align*}
    n\E^{\mathcal{E}}\big[g_n\big(2N_x^{T^1}\big)\big] + n\E^{\mathcal{E}}\big[g_n\big(2\sum_{y;y^*=x}N^{T^1}\big)\big]\leq & n^{1-b}\varphi(n^b)\E^{\mathcal{E}}\big[N_x^{T^1}\un_{\{N_x^{T^1}\geq \lceil n^b/2 \rceil\}}\big] \\ & + n^{1-b}\varphi(n ^b)\E^{\mathcal{E}}\Big[\sum_{y;y^*=x}N_y^{T^1}\un_{\{\sum_{y;y^*=x}N_y^{T^1}\geq \lceil n^b/2 \rceil\}}\Big].
\end{align*}
We have used that for all $1\leq i\leq n$, $N^{T^i}_x-N^{T^{i-1}}_x$(resp. $\sum_{y;y^*=x}N^{T^i}_y-N^{T^{i-1}}_y$) is distributed as $N^{T^1}_x$ (resp. $\sum_{y;y^*=x}N^{T^1}_y$) under $\P^{\mathcal{E}}$ and the fact that the application $t\in [1,\infty)\mapsto \varphi(t)/t$ is non-increasing. Then, by Lemma \ref{LawGeo}
\begin{align*}
    \E^{\mathcal{E}}\big[N_x^{T^1}\un_{\{N_x^{T^1}\geq \lceil n^b/2 \rceil\}}\big] \leq \frac{e^{-V(x)}}{H_x}\Big(1-\frac{1}{H_x}\Big)^{\lceil n^b/2 \rceil-1}(n^b+H_x),
\end{align*}
and
\begin{align*}
    \E^{\mathcal{E}}\Big[\sum_{y;y^*=x}N_y^{T^1}\un_{\{\sum_{y;y^*=x}N_y^{T^1}\geq \lceil n^b/2 \rceil\}}\Big]\leq \frac{e^{-V(x)}}{H_x}\frac{\tilde{H}_x}{1+\tilde{H}_x}\Big(1-\frac{1}{1+\tilde{H}_x}\Big)^{\lceil n^b/2 \rceil-1}( n^b+1+\tilde{H}_x),
\end{align*}
which ends the proof.
\end{proof}

\begin{lemm}\label{RangeUpBound2}
Let $b\in[0,1)$. For $n$ large enough
\begin{align*}
    \Eb[\mathcal{X}_{1,n}+\mathcal{X}_{2,n}+\mathcal{X}_{3,n}]\leq 3(\log n)^2 u_{2,n}.
\end{align*}
where we recall
$u_{2,n}=\sum_{k\geq 1}\big(\Psi^{k}_{n} \big(f^{n,k}\un_{\R^k\setminus \mathcal{H}^k_{\mathfrak{z}_n}}\big)+\Psi^{k}_{n,n^b/(\log n)^2} (f^{n,k})+\Eb[W\Psi^k_{n,n^b/(W(\log n)^2)}(f^{n,k})]\big)$, with $W=\sum_{|z|=1}e^{-V(z)}$. 
\end{lemm}

\begin{proof}
We start with the easiest part, that is the expression of $\Eb[\mathcal{X}_{1,n}]$. Thanks to hypothesis \eqref{hyp0}
\begin{align*}
    \Eb[\mathcal{X}_{1,n}] & = \Eb\Big[\sum_{x\in\mathcal{O}_n}\un_{\{V(x)<\mathfrak{z}_n\}}\Big(e^{-V(x)}+e^{-V(x)}\sum_{y;y^*=x}e^{-V_x(y)}\Big)f^{n,|x|}(\bm{V}_x)\Big] \\ & = 2\Eb\Big[\sum_{x\in\mathcal{O}_n}\un_{\{V(x)<\mathfrak{z}_n\}}e^{-V(x)}f^{n,|x|}(\bm{V}_x)\Big] = 2\sum_{k\geq 1}\Psi^k_n\big(f^{n,k}\un_{\R^k\setminus \mathcal{H}^k_{\mathfrak{z}_n}}\big).
\end{align*}
 Let 
\begin{align*}
    \Tilde{\lambda}_n:=\frac{\lceil n^b/2\rceil-1}{\log q_n}\;\;\textrm{ with }\;\;q_n:=\frac{4C_{\infty}\ell_nn^b}{\sum_{k\geq 1}\Psi^k_{n,n^b/(\log n)^2}\big(f^{n,k}\big)}, 
\end{align*}
and let us find an upper bound for $\Eb[\mathcal{X}_{2,n}]$. For that, we decompose $\mathcal{X}_{2,n}$ into two parts according to the value of $H_x$:
\begin{align*} 
    \mathcal{X}_{2,n} &\leq\sum_{x\in\mathcal{O}_n}(\un_{\{H_x\leq\Tilde{\lambda}_n\}}+\un_{\{H_x>\Tilde{\lambda}_n\}})\frac{e^{-V(x)}}{H_x}\Big(1-\frac{1}{H_x}\Big)^{\lceil n^b/2 \rceil-1}(n^b+H_x)f^{n,|x|}(\bm{V}_x) \\ & \leq C_{\infty}\big(n^b+\Tilde{\lambda}_n\big)\big(1-\frac{1}{\Tilde{\lambda}_n}\big)^{\lceil n^b/2 \rceil-1}\sum_{x\in\mathcal{O}_n}e^{-V(x)}+\big(1+\frac{n^b}{\Tilde{\lambda}_n}\big)\sum_{x\in\Hline{n}{\Tilde{\lambda}_n}}e^{-V(x)}f^{n,|x|}(\bm{V}_x).
\end{align*}
By definition of $\Tilde{\lambda}_n$ and $q_n$ (see above), $(1-1/\Tilde{\lambda}_n)^{\lceil n^b/2 \rceil-1}\leq 1/q_n$. Moreover, by Remark \ref{Rem1}, $\Eb[\sum_{x\in\mathcal{O}_n}e^{-V(x)}]\leq\ell_n$ and $\Eb[\sum_{k\geq 1}\Psi^k_{n,n^b}\big(f^{n,k}\big)] \leq C_{\infty}  \Eb[\sum_{x\in\mathcal{O}_n}e^{-V(x)} ] \leq C_{\infty} \ell_n$ so for $n$ large enough ($q_n\geq 4n^b$  implying $\Tilde{\lambda}_n\leq n^b$),  we obtain 
\begin{align*}
    \Eb[\mathcal{X}_{2,n}]\leq\frac{1}{2}\sum_{k\geq 1}\Psi^k_{n,n^b/(\log n)^2}\big(f^{n,k}\big)+\big(1+\frac{n^b}{\Tilde{\lambda}_n}\big)\sum_{k\geq 1}\Psi^k_{n,\Tilde{\lambda}_n}\big(f^{n,k}\big).
\end{align*}
For $\Eb[\mathcal{X}_{3,n}]$, we decompose $\mathcal{X}_{3,n}$ into two parts according to the value of $\tilde{H}_x$: $\mathcal{X}_{3,n}$ is smaller than
\begin{align*}
     &\sum_{x\in\mathcal{O}_n}(\un_{\{1+\tilde{H}_x\leq\Tilde{\lambda}_n\}}+\un_{\{1+\tilde{H}_x>\Tilde{\lambda}_n\}})\frac{e^{-V(x)}}{H_x}\frac{\tilde{H}_x}{1+\tilde{H}_x}\Big(1-\frac{1}{1+\tilde{H}_x}\Big)^{\lceil n^b/2 \rceil-1}( n^b+1+\tilde{H}_x)f^{n,|x|}(\bm{V}_x) \\ & \leq C_{\infty}\big(n^b+\Tilde{\lambda}_n\big)\big(1-\frac{1}{\Tilde{\lambda}_n}\big)^{\lceil n^b/2 \rceil-1}\sum_{x\in\mathcal{O}_n}e^{-V(x)}+\big(1+\frac{n ^b}{\Tilde{\lambda}_n}\big)\sum_{x\in\mathcal{O}_n}e^{-V(x)}\un_{\{1+\tilde{H}_x>\Tilde{\lambda}_n\}} \\ & \times \sum_{y;y^*=x}e^{-V_x(y)}f^{n,|x|}(\bm{V}_x).
\end{align*}
%Again by definition of $\Tilde{\lambda}_n$ and $q_n$ together with the fact that $\Tilde{\lambda}_n\leq n^b$ and Remark \ref{Rem1}, 
Then, as above, $C_{\infty}\big(n^b+\Tilde{\lambda}_n\big)\big(1-1/\Tilde{\lambda}_n\big)^{\lceil n^b/2 \rceil-1}\leq\sum_{k\geq 1}\Psi^k_{n,n^b/(\log n)^2}\big(f^{n,k}\big)/2$, also recall that $\tilde{H}_x=H_x\sum_{y;y^*=x}e^{-V_x(y)}$ so by conditional independence of $H_x$ and $\sum_{y;y^*=x}e^{-V_x(y)}$  together with the fact that this random variable has the same law as $W=\sum_{|x|=1}e^{-V(x)}$, 
\begin{align*}
    \Eb\Big[\sum_{x\in\mathcal{O}_n}e^{-V(x)}\un_{\{1+\tilde{H}_x>\Tilde{\lambda}_n\}}
    \sum_{y;y^*=x}e^{-V_x(y)}f^{n,|x|}(\bm{V}_x)\Big]=\sum_{k\geq 1}\Eb\Big[W\Psi^k_{n,(\Tilde{\lambda}_n-1)/W}(f^{n,k})\Big].
\end{align*}
Hence
\begin{align*}
    \Eb[\mathcal{X}_{3,n}]\leq\frac{1}{2}\sum_{k\geq 1}\Psi^k_{n,n^b/(\log n)^2}\big(f^{n,k}\big)+\big(1+\frac{n ^b}{\Tilde{\lambda}_n}\big)\sum_{k\geq 1}\Eb\big[W\Psi^k_{n,(\Tilde{\lambda}_n-1)/W}(f^{n,k})\big].
\end{align*}
Finally, note that $\Psi^k_{n,n^b}(f^{n,k})\leq \Psi^k_{n,n^b/(\log n)^2}\big(f^{n,k}\big)$ so using assumption \eqref{Hypothese1}, we get $q_n\leq 4C_{\infty}\ell_nn^{1+b}$ thus giving $\Tilde{\lambda}_n-1\geq n^b(\log n)^{-2}$ for all $b\in(0,1)$ and $n$ large enough. Hence, for all $b\in[0,1)$ and $n$ large enough, $(1+n^b/\Tilde{\lambda}_n)\leq 2(\log n)^2$ and $\Psi^k_{n,(\Tilde{\lambda}_n-1)/W}(f^{n,k})$ (resp. $\Psi^k_{n,\Tilde{\lambda}_n}(f^{n,k})$) is smaller than $\Psi^k_{n,n^b/(W(\log n)^2)}(f^{n,k})$ (resp. $\Psi^k_{n,n^b/(\log n)^2}(f^{n,k})$) so we obtain the result.
\end{proof}
\noindent We are now ready to prove the upper bound in Proposition \ref{Prop1}. Recall $\eqref{A_n}$ and let $\varepsilon>0$
\begin{align*}
     \P\Big(\frac{\mathcal{R}_{T^n}(g_n,\mathbf{f}^n)}{n^{1-b}\varphi(n^b)u_{2,n}}>e^{\varepsilon h_n}\Big)\leq\P\Big(\frac{\mathcal{R}_{T^n}(g_n,\mathbf{f}^n)}{n^{1-b}\varphi(n^b)u_{2,n}}>e^{\varepsilon h_n},\mathcal{A}_n\Big)+1-\P(\mathcal{A}_n),
\end{align*}
where $u_{2,n}=\sum_{k\geq 1}(\Psi^{k}_{n} \big(f^{n,k}\un_{\R^k\setminus \mathcal{H}^k_{\mathfrak{z}_n}}\big)+\Psi^{k}_{n,n^b/(\log n)^2} (f^{n,k})+\Eb\big[W\Psi^k_{n, n^b/(W(\log n)^2)}(f^{n,k})\big])$. By Lemma \ref{RangeUpBound1} with $u_n=e^{\varepsilon h_n}n^{1-b}\varphi(n^b)u_{2,n}$ and Lemma \ref{RangeUpBound2}, for $n$ large enough
\begin{align*}
    \P\Big(\frac{\mathcal{R}_{T^n}(g_n,\mathbf{f}^n)}{n^{1-b}\varphi(n^b)u_{2,n}}>e^{\varepsilon h_n},\mathcal{A}_n\Big)&\leq \frac{2e^{-\varepsilon h_n}}{u_{2,n}} \Eb[\mathcal{X}_{1,n}+\mathcal{X}_{2,n}+\mathcal{X}_{3,n}]\leq 6(\log n)^2e^{-\varepsilon h_n},
\end{align*}
and then for $n$ large enough
\begin{align*}
    \P\Big(\frac{\mathcal{R}_{T^n}(g_n,\mathbf{f}^n)}{n^{1-b}\varphi(n^b)u_{2,n}}>e^{\varepsilon h_n}\Big)\leq 6(\log n)^2e^{-\varepsilon h_n}+1-\P(\mathcal{A}_n). 
\end{align*}
Finally, observe (see Remark \ref{rem0}) that $(\log n)^2=o(e^{\varepsilon h_n})$ and we complete the proof of the upper bound recalling (see below \eqref{A_n}) that $1-\P(\mathcal{A}_n)=o(1)$.

 \section{Technical estimates for one-dimensional random walk \label{sec4bis}}

In this section, we prove some technical expressions involving sums of i.i.d. random variables introduced via the \mto\ at the beginning of Section \ref{sec2}.
Recall that $(S_i-S_{i-1}, i \geq 1)$ is a sequence of i.i.d. random variables such that $\E(S_1)=0$, there exists $\eta >0$ for which $\E(e^{\eta S_1})<+ \infty$. Also we denote $\sigma^2=\psi''(1)=\E(S_1^2)$. We also use the following notations : for any $a$, $\tau_a:= \inf\{k>0,\ S_k \geq a\}$, $\tau_a^-:= \inf\{k>0,\ S_k \leq a\}$ and $\tau^{\bar S-S}_{a}:= \inf\{k>0,\overline{S}_k-S_k \geq a\}$ with $\overline{S}_k:= \max_{1 \leq m \leq k} S_m$ and $H_j^S:=\sum_{i=1}^j e^{S_i-S_j}$. %We start with some facts which proof can be found explicitly in the literature : \\

%\noindent {\bf Fact 1}

%Ici on montre ce dont on a besoin sur l'environnement.

%On remarque que cette derni√É¬É√Ç¬É√É¬Ç√Ç¬É√É¬É√Ç¬Ç√É¬Ç√Ç¬Ç√É¬É√Ç¬É√É¬Ç√Ç¬Ç√É¬É√Ç¬Ç√É¬Ç√Ç¬ère chose est quasiment le th√É¬É√Ç¬É√É¬Ç√Ç¬É√É¬É√Ç¬Ç√É¬Ç√Ç¬Ç√É¬É√Ç¬É√É¬Ç√Ç¬Ç√É¬É√Ç¬Ç√É¬Ç√Ç¬éor√É¬É√Ç¬É√É¬Ç√Ç¬É√É¬É√Ç¬Ç√É¬Ç√Ç¬Ç√É¬É√Ç¬É√É¬Ç√Ç¬Ç√É¬É√Ç¬Ç√É¬Ç√Ç¬ème de Caravenna, mais pas tout √É¬É√Ç¬É√É¬Ç√Ç¬É√É¬É√Ç¬Ç√É¬Ç√Ç¬Ç√É¬É√Ç¬É√É¬Ç√Ç¬Ç√É¬É√Ç¬Ç√É¬Ç√Ç¬à fait 

%\begin{align*}
%& \P( \forall m \leq n , 0 \leq  S_m  \leq  c, a \leq S_{n} < a+s )  \\ 
%&  \geq  \sum_{k} \P(T_k=n, S_n \in [a,a+s),  \underline{S}_n \geq 0 )
%\end{align*}
\subsection{Two Laplace transforms}
In this section, we deal with Laplace transforms which appear when we study the range with underlying constraint on $V$.% of the reflecting barrier and also when a penalization via cumulative downfalls of $V$ is introduced. % $\E \left[e^{-\frac{c \sigma^2}{2m^2} \tau_{r} } \un_{\tau_{r} \leq \tau^{\bar S-S}_{m}} \right]$ when $m$ and $r=r(m)$ go to infinity. It appears when studying the range of high potential with the underlying constraint of the reflecting barrier (see \textbf{[faire r\'ef\'erence aux \'equations]}). %The proof of the result is quite usual (no !), we present them however for completion and also because we did not find it as we need it in the literature (especially with constraint $\tau_{r} \leq \tau^{\bar S-S}_m$).
%We treat the lower and upper bound in two seperate lemmata.
 
%\begin{align}
%\E \left[e^{-c \tau_{r}/ m^2} \un_{\tau_{r} \leq \tau^{\bar S-S}_{m}} \right] \geq e^{-\sqrt{c} \frac{r}{m}}. \end{align} 
%So we are left to deal with the overshoots, and also see if we can have the upper bound too.
%But we knows that $\frac{\tau_{-m}^-}{2 m^2}$ converge almost surely to a finite positive constant(pas vraiment n√É¬É√Ç¬É√É¬Ç√Ç¬É√É¬É√Ç¬Ç√É¬Ç√Ç¬Ç√É¬É√Ç¬É√É¬Ç√Ç¬Ç√É¬É√Ç¬Ç√É¬Ç√Ç¬écessaire mais bon ...),  so for large $m$ 
%\begin{align}
%\E \left(e^{ -\frac{c \sigma^2}{2 m^2} \tau_{a} } \un_{\tau_{a} < \tau_{-m}^-} \right) \geq C  e^{-\sqrt{c} \frac{a}{m}} , 
%\end{align}
%with $1>C>0$. 

%\red{[ A mettre au bon endroit : 
%About the correct constant, in the paper of Faraud-Hu-Shi, we have at some point (Proposition 3.1 page 15)  : 
%\begin{align}
%\lim_{n\rightarrow + \infty }\frac{ a_n^2}{n} \log \P(G_0(n))=-\frac{\pi^2 \sigma^2}{8} * \int_0^1 \frac{ds}{f^2(s)},
%\end{align}
%here $f=1$, so considering above Lemma, $\sqrt{c} = \sqrt{\frac{\pi^2}{4}}= \frac{ \pi}{2}$. Je suis g\^en\'e par le fait que cela d\'epende de $\pi$. We have to check the other computations and instead of taking half the path (for the condition of high potential) take an arbitrary proportion of the generation. }

\begin{lemm} \label{Laplace2} Let $r:=r(\ell)$ such that $\lim_{\ell \rightarrow + \infty}r(\ell)/\ell=+ \infty$, then for any $\varepsilon>0$
\begin{align*}
e^{-(1+\sqrt{c}- \rho(c)) \frac{r}{\ell}(1+ \varepsilon)} \leq \E \left[e^{-\frac{c \sigma^2}{2\ell^2} \tau_{r} } \un_{\tau_{r} \leq \tau^{\bar S-S}_{\ell}} \right] \leq e^{-(1+\sqrt{c}- \rho(c)) \frac{r}{\ell}(1- \varepsilon)},
\end{align*}
with $\rho(c)= \frac{c \sigma }{\sqrt{2\pi}} \int_{0}^{+ \infty} {e^{- \frac{c \sigma^2}{2} u} f(u)}du$, and $f(u)=\frac{2}{u^{1/2}}\P(\overline{\mathfrak{m}}_1>1/\sqrt{u \sigma^2})  - \frac{1}{2}\int_u^{+\infty} \frac{1}{ y^{3/2}} \P(\overline{\mathfrak{m}}_1>1/\sqrt{y \sigma^2})dy$. Note that $\rho$ can be explicitly calculated : for any $c>0$
\begin{align*}
\rho(c)& =2 \sqrt{c}\Big ( \frac{1-e^{-\sqrt{c}}}{\sinh(\sqrt{c})} \Big) -2\Big( \sqrt c-\log( (e^{\sqrt{c}}+1)/2) \Big) .
% & = 2 \sqrt{c}\Big ( \frac{1-e^{-\sqrt{c}}}{\sinh(\sqrt{c})} \Big)-2\Big( \frac{\sqrt c}{2}-\log(\cosh({\sqrt{c}/2})) \Big)  
\end{align*}
%\red{Si $c \rightarrow 0$,  apr\`es un truc vite fait j'obtiens que cela mange la $\sqrt c$ dans l'exponentiel et je me retrouve avec $1+9c/8+o(c)$ ... bref on a continuit\'e ! A voir aussi avec $c=1$ ($\rho(1) \sim 0.316$), je ne vois pas de $\coth$ non plus surtout avec un $\log$ au milieu mais peut \^etre en d\'eveloppant en s\'erie enti\`ere ... }
\end{lemm}

\begin{proof} %\red{Attention rien n'est finaliser ! m\^eme si les id\'ees de preuves sont ok maintenant.}
We start with the \textit{upper bound}. \\
 \noindent Let us introduce the usual strict ladder epoch sequence $(T_k:= \inf\{i>T_{k-1}, S_i > S_{T_{k-1}}\},k;\\T_0=0)$. Then for any $k$
 \begin{align}
\E \left[e^{-\frac{c \sigma^2}{2\ell^2} \tau_{r} } \un_{\tau_{r} \leq \tau^{\bar S-S}_{\ell}} \right] & \leq \E \left[e^{-\frac{c \sigma^2}{2\ell^2} \tau_{r} } \un_{S_{T_k}<r} \un_{\tau_{r} \leq \tau^{\bar S-S}_{\ell}} \right]+ \P(S_{T_k}\geq r) \nonumber \\
%& \leq \E \left[e^{-\frac{c \sigma^2}{2m^2} T_k } \un_{ \max_{1 \leq  m \leq k}  (S_{T_{m-1}}- \min_{T_{m-1} \leq j \leq T_m } {S_j}) \leq m} \right]+ \P(S_{T_k}\geq r) \\
&  \leq  \left(\E \left[e^{-\frac{c \sigma^2}{2\ell^2} \tau_0 }\un_{\tau_0 \leq \tau_{-\ell}^-}\right] \right)^{k} + \P(S_{T_k}\geq r), \label{Eq8.1}
\end{align}
where the last equality comes from the strong Markov property and equality $T_1= \tau_0:= \inf\{m>0, S_m >0\}$. From here we need the asymptotic in $\ell$ of $ \E \Big[e^{-\frac{c \sigma^2}{2\ell^2} \tau_0 }\un_{\tau_0 \leq \tau_{-\ell}^-}\Big]$. First we use following identity 
\begin{align}
\E \left[e^{-\frac{\lambda}{\ell^2} \tau_0 }\un_{\tau_0 \leq \tau_{-\ell}^-}\right]= \E [e^{-\frac{\lambda}{\ell^2} \tau_0 }  ]- \P(\tau_0> \tau_{-\ell}^-) +\E\left ( (1 - e^{\frac{-\lambda }{\ell^2}\tau_0})\un_{\tau_0>\tau_{-\ell}}\right), \label{ident}
\end{align}
and then give an upper bound for each of the three terms. Lemma 2.2 in  \cite{Aidekon} gives for $m$ large enough 
\begin{align}
\P(\tau_0> \tau_{-\ell}^-) =\frac{\E(S_{\tau_0})}{\ell}+ o\left(\frac{1}{\ell}\right),  \label{aid}
\end{align}
Both of the other terms can be obtained with a Tauberian theorem, we give here some details for the third one which is more delicate. Let $dH_\ell(u)$ the measure defined by  $\P( \tau_0 > z \ell^2, \tau_0> \tau_{-\ell}^-)=\int_{z}^{\infty}dH_\ell(u) $, integration by part gives
$ \E\left ( (1 - e^{\frac{-\lambda }{\ell^2}\tau_0})\un_{\tau_0>\tau_{-\ell}^-}\right)= \int_{0}^{+ \infty}(1- e^{-\lambda u}) dH_\ell(u)  
%&= \left[  (1-e^{-\lambda/\varepsilon})\P(\tau_0/m^2>1/√É¬É√Ç¬É√É¬Ç√Ç¬É√É¬É√Ç¬Ç√É¬Ç√Ç¬É√É¬É√Ç¬É√É¬Ç√Ç¬Ç√É¬É√Ç¬Ç√É¬Ç√Ç¬ä\varepsilon)-(1-e^{-\lambda \varepsilon})\P(\tau_0/m^2>\varepsilon) \right]+√É¬É√Ç¬É√É¬Ç√Ç¬É√É¬É√Ç¬Ç√É¬Ç√Ç¬É√É¬É√Ç¬É√É¬Ç√Ç¬Ç√É¬É√Ç¬Ç√É¬Ç√Ç¬ä\lambda \int_{\varepsilon}^{1/ \varepsilon} e^{-\lambda u} \P(\tau_0/m^2>u)du \\
= \lambda \int_{0}^{+ \infty} e^{-\lambda u}$ $ \P( \tau_0 > u \ell^2, \tau_0> \tau_{-\ell}^-)du$. So we need an asymptotic in $\ell$ of the tail probability  $ \P( \tau_0 > u \ell^2, \tau_0> \tau_{-\ell}^-)$.
%for that first notice that 
% \begin{align}
% \E \Big[e^{-\frac{c \sigma^2}{2m^2} \tau_0 }\un_{\tau_0 \leq \tau_{-m}^-}\Big]=  \E \Big[e^{-\frac{c \sigma^2}{2m^2} \tau_0 } \Big]- \E \Big[e^{-\frac{c \sigma^2}{2m^2} \tau_0 }\un_{\tau_0 > \tau_{-m}^-}\Big],
%\end{align}
%and 
%\begin{align*}
%\E \Big[e^{-\frac{c \sigma^2}{2m^2} \tau_0 }\un_{\tau_0 > \tau_{-m}^-}\Big]  & \geq \E \Big[e^{-\frac{c \sigma^2}{2m^2} \tau_0 }\un_{\tau_0 > \tau_{-m}^-} \un_{ \tau_0 \leq m^{1+ \varepsilon} } \Big]   
%& \geq \Big(1-\frac{Cte}{m^{1- \varepsilon}}\Big) \P \Big[{\tau_0 > \tau_{-m}^-},{ \tau_0 \leq m^{1+ \varepsilon} } \Big] \geq  \Big(1-\frac{Cte}{m^{1- \varepsilon}}\Big) \left ( \P \Big[{\tau_0 > \tau_{-m}^-} \Big]- \P({ \tau_0 > m^{1+ \varepsilon} }) \right)
%\end{align*}
%use following equality of \cite{Aidekon} : for large $y$, $\P(\tau_0> \tau_{-y})= \E(S_{\tau_0})/y + o(1/y)$.
Let us decompose this probability as follows
\begin{align}
\P( \tau_0 > z \ell^2, \tau_0< \tau_{-\ell}^-)&=\P( \tau_0 >   \tau_{-\ell}^-> z\ell^2)+\P( \tau_0 > z \ell^2,  \tau_{-\ell}^- \leq z \ell^2) \nonumber \\
&= \P(\tt_0^- >   \tt_{\ell} > z\ell^2)+\P( \tt_0 >z\ell^2,  \tt_{\ell} \leq z \ell^2)=:P_1+P_2. \label{56}
\end{align}
where $\tt_0^-:=\inf \{k>0,\ \tS_k < 0\}$ with for any $k$, $\tS_k =-S_k$ and similarly $\tt_{\ell}:=\inf \{k>0,\ \tS_k \geq \ell\}$. \\
For $P_2$, we just use Donsker's theorem for conditioned random walk to remain positive  obtain in \cite{Bolth} which gives $\lim_{\ell  \rightarrow +\infty} \P(\tt_{\ell} \leq z \ell^2|\tt_0 > z \ell^2)=\P(\overline{\mathfrak{m}}_1 > 1/\sigma \sqrt z )$, where $\mathfrak{m}$ is the Brownian meander and $\overline{\mathfrak{m}}_1= \sup_{s \leq 1} \mathfrak{m}_s$. 
 Also we know from Feller \cite{Feller} (see the first equivalence page 514 of  Caravenna \cite{Carav} for the expression we use here) that for any $z>0$ :
\begin{align}
 \lim_{\ell \rightarrow  \infty} \ell  \P( \tt_0 > z \ell^2)=\sqrt{\frac{2}{\pi}} \frac{\E(S_{\tau_0}) }{\sqrt{ z \sigma^2}},  \label{lim45}
\end{align}
so 
\begin{align}
 \lim_{\ell \rightarrow  \infty} \ell P_2=\sqrt{\frac{2}{\pi}} \frac{\E(S_{\tau_0}) }{\sqrt{ z \sigma^2}} \P(\overline{\mathfrak{m}}_1> 1/\sigma \sqrt z ).  \label{45p}
\end{align}
For $P_1$ we use a similar strategy, for any $A>x$, $\varepsilon>0$ and $\ell$ large enough
\begin{align}
P_1 & \leq  \P( z \ell^2 \leq \tt_\ell \leq  A \ell^2, \ \tt_0 > \tt_\ell)+  \P(\tt_0 > A \ell ^2) \nonumber  \\ %\leq \sum_{ k = xm ^2 }^{A m^2} \P(\tt_m=k, \ \tt_0 >k) + C / m*A^{1/2} \\
& \leq \sum_{ k = z\ell ^2 }^{A \ell^2} \P(\overline \tS_{k-1} \leq \ell,  \tS_{k} > \ell | \ \tt_0 >k)\P(\tt_0 >k)  + \P(\tt_0 > A \ell ^2)  \nonumber \\
& \leq (1+ \varepsilon)\sqrt{\frac{2}{\pi}} \frac{\E(S_{\tau_0}) }{{\ell \sigma}}  \sum_{ k = z\ell ^2 }^{A \ell^2}\P(\overline \tS_{k-1} \leq \ell,  \tS_{k} > \ell | \ \tt_0 >k) \frac{\ell}{k^{1/2}}  + \frac{C}{\ell A^{1/2}}, \nonumber
\end{align}
where we have used \eqref{lim45} for the last inequality and $C>0$ is a constant. Also functional limit theorem \cite{Bolth} implies  that $ \lim_{\ell \rightarrow +\infty} \sum_{ k = z\ell ^2 }^{A \ell^2}\P(\overline \tS_{k-1} \leq \ell,  \tS_{k} > \ell | \ \tt_0 >k) \frac{\ell}{k^{1/2}} = -\int_z^A \frac{1}{y^{1/2}} d\P(\overline{\mathfrak{m}}_1>1/\sqrt{y \sigma^2}) $. We deduce from that, taking limits $A \rightarrow +\infty$ and $\varepsilon \rightarrow 0$, 
\begin{align*}
& \lim_{\ell \rightarrow  \infty} \ell*P_1 \\ 
& \leq -\sqrt{\frac{2}{\pi}} \frac{\E(S_{\tau_0}) }{ \sigma}  \int_z^{+\infty} \frac{1}{y^{1/2}} d\P(\overline{\mathfrak{m}}_1>1/\sqrt{y \sigma^2})\\
 & =  \sqrt{\frac{2}{\pi}} \frac{\E(S_{\tau_0}) }{\sigma }\Big (  \frac{1}{z^{1/2}}\P(\overline{\mathfrak{m}}_1>1/\sqrt{z \sigma^2})  - \frac{1}{2}\int_z^{+\infty} \frac{1}{ y^{3/2}} \P(\overline{\mathfrak{m}}_1>1/\sqrt{y \sigma^2})dy \Big ). 
\end{align*}
Note that just by noticing that $P_1 \geq  \P( z \ell^2 \leq \tt_\ell \leq  A \ell^2, \ \tt_0 > \tt_\ell)$, above expression is also a lower bound for $\lim_{\ell \rightarrow  \infty} \ell*P_1$. Considering this, \eqref{45p} and \eqref{56}, we obtain 
\begin{align}
 \lim_{\ell \rightarrow  \infty} \ell  \P( \tau_0 > z \ell^2, \tau_0> \tau_{-\ell}^-) = \sqrt{\frac{2}{\pi}} \frac{\E(S_{\tau_0}) }{\sigma}f(z) \label{lim45b}
\end{align}
where $f$ is the function given in the statement of the Lemma. 
\noindent Note that this convergence is uniform on any compact set in $(0,\infty)$ by monotonicity of $z \rightarrow \ell  \P( \tau_0 > z \ell^2, \tau_0< \tau_{-\ell}^-)$, continuity of the limit and Dini's theorem. 
From here we follow the same lines of the proof of a Tauberian theorem (Feller \cite{Feller}) for completion we recall the main lines for our particular case.
%Let $dH_m(u)$ the measure defined by  $\P( \tau_0 > z m^2, \tau_0> \tau_{-m}^-)=\int_{z}^{\infty}dH_m(u) $, integration by part gives
%$ \int_{0}^{+ \infty}(1- e^{-\lambda u}) dH_m(u)  
%&= \left[  (1-e^{-\lambda/\varepsilon})\P(\tau_0/m^2>1/√É¬É√Ç¬É√É¬Ç√Ç¬É√É¬É√Ç¬Ç√É¬Ç√Ç¬É√É¬É√Ç¬É√É¬Ç√Ç¬Ç√É¬É√Ç¬Ç√É¬Ç√Ç¬ä\varepsilon)-(1-e^{-\lambda \varepsilon})\P(\tau_0/m^2>\varepsilon) \right]+√É¬É√Ç¬É√É¬Ç√Ç¬É√É¬É√Ç¬Ç√É¬Ç√Ç¬É√É¬É√Ç¬É√É¬Ç√Ç¬Ç√É¬É√Ç¬Ç√É¬Ç√Ç¬ä\lambda \int_{\varepsilon}^{1/ \varepsilon} e^{-\lambda u} \P(\tau_0/m^2>u)du \\
%= \lambda \int_{0}^{+ \infty} e^{-\lambda u}$ $ \P( \tau_0 > u m^2, \tau_0> \tau_{-m}^-)du$.
For any $\varepsilon>0$, by the uniform convergence we have talked about just above, 
\begin{align*}
\lim_{\ell \rightarrow + \infty} \ell \int_{\varepsilon}^{1/ \varepsilon}e^{-\lambda u} \P( \tau_0 > u \ell^2, \tau_0> \tau_{-\ell}^-)du=\sqrt{\frac{2}{\pi}} \frac{\E(S_{\tau_0}) }{\sigma} \int_{\varepsilon}^{1/ \varepsilon} {e^{- \lambda u} f(u)}du.
\end{align*}
By \eqref{lim45}, we also have for any $\ell$ and $z >0$, $  \P( \tau_0 > z \ell^2, \tau_0> \tau_{-\ell}^-) \leq \frac{Const}{z^{1/2} \ell}$ and as $\int_{0}^{+\infty}e^{-\lambda u} u^{-1/2}du<+\infty$, we get 
%\begin{align*}
$\lim_{\varepsilon  \rightarrow 0} \lim_{\ell \rightarrow +\infty} \int_{0}^{\varepsilon}e^{-\lambda u}\ell \P(\tau_0/\ell^2>u) =0$. 

\noindent Similarly $\lim_{\varepsilon  \rightarrow 0} \lim_{\ell \rightarrow +\infty} \int_{1/ \varepsilon }^{+\infty}e^{-\lambda u}\ell \P(\tau_0/\ell^2>u,\tau_0> \tau_{-\ell}^-)du =0. $
Finally  
\begin{align}
 \lim_{\ell \rightarrow +\infty } \ell \int_{0}^{+\infty}(1- e^{-\lambda u}) dH_\ell(u) &=  \lim_{\ell \rightarrow +\infty } \ell \E\left ( (1 - e^{\frac{-\lambda }{\ell^2}\tau_0})\un_{\tau_0>\tau_{-\ell}}\right) \nonumber \\
  & =\lambda \sqrt{\frac{2}{\pi}} \frac{\E(S_{\tau_0}) }{\sigma} \int_{0}^{+ \infty} {e^{- \lambda u} f(u)}du. \label{intf}
 \end{align}
 Note also that just by using \eqref{lim45} we also have $\lim_{\ell \rightarrow +\infty }  \ell\E [1-e^{-\frac{\lambda}{\ell^2} \tau_0 }  ] = {\sqrt{2 \lambda} \E(S_{\tau_0})}\sigma^{-1}$. Then collecting \eqref{ident}, \eqref{aid} and \eqref{intf} and taking $\lambda=c \sigma^2/2$  we obtain for $\ell$ large enough
\begin{align}
\E \left[e^{-\frac{c \sigma^2}{2\ell^2} \tau_0 }\un_{\tau_0 \leq \tau_{-\ell}^-}\right] = 1-  \frac{\E(S_{\tau_0}) }{\ell} \left(1+ {\sqrt{c}}-   \frac{c \sigma }{\sqrt{2\pi}} \int_{0}^{+ \infty} {e^{-\frac{c \sigma^2 u}{2}} f(u)}du\right) +o \Big (\frac{1}{\ell} \Big). \label{Ltr1}
\end{align}
To obtain an explicit expression for the above integral, we integrate by parts  
\begin{align*} 
&\int_{0}^{+ \infty} {e^{- \lambda u} f(u)}du \\
&=2 \int_{0}^{+ \infty} \frac{e^{- \lambda u} }{u^{1/2}} \P(\overline{\mathfrak{m}}_1>1/\sqrt{u \sigma^2}) du-\frac{1}{2 \lambda} \int_0^{+\infty} \frac{1}{u^{3/2}}  (1-e^{-\lambda u})\P(\overline{\mathfrak{m}}_1>1/\sqrt{u \sigma^2})  du, \end{align*}
then using the expression of $\P(\overline{\mathfrak{m}}_1>u):= -2\sum_{k=1}(-1)^k \exp(-(ku)^2/2),\ \forall u>0$, and elementary computations 
\begin{align} 
\int_{0}^{+ \infty} {e^{- \lambda u} f(u)}du =  2 \sqrt{\frac{\pi}{\lambda}}\Big ( \frac{1}{\sinh(\sqrt{2 \lambda}/ \sigma)}-\frac{e^{-\sqrt{2 \lambda}/ \sigma}}{\sinh(\sqrt{2 \lambda}/ \sigma)} \Big) - \frac{\sigma\sqrt{2 \pi}}{\lambda}\Big( \frac{\sqrt{2 \lambda}}{\sigma}-\log((e^{\sqrt{2 \lambda}/\sigma}+1)/2) \Big). \label{eq59}
\end{align}
Now we deal with the probability $\P(S_{T_k}\geq r)$ in the same way as  \cite{HuShi15b}. As $T_k$ can be written 	as a sum of i.i.d random variables with common law given by $\tau_0$, the exponential Markov property gives for any $\eta>0$, $\P(S_{T_k}\geq r) \leq e^{-\eta r} (\E(e^{\eta S_{\tau_0}}))^k$. Taking $k=(1- \varepsilon)r/\E(S_{\tau_0})$ we can find constants $c'$ and $c"$ such that $\P(S_{T_k}\geq r) \leq c' e^{- c" r}$ for any $r \geq 1$. So replacing this and \eqref{Ltr1} in \eqref{Eq8.1}, we finally get for any $m$ large enough
\begin{align*}
\E \left[e^{-\frac{c \sigma^2}{2\ell^2} \tau_{r} } \un_{\tau_{r} \leq \tau^{\bar S-S}_{\ell}} \right]  & \leq  \left(\E \left[e^{-\frac{c \sigma^2}{2\ell^2} \tau_0 } \un_{ \tau_0 \leq \tau_{-\ell}} \right] \right)^{k} + \P(S_{T_k}\geq r)  \\
& \leq \left(1-  \frac{\E(S_{\tau_0}) }{\ell} \Big(1+ {\sqrt{c}}-   \frac{c \sigma }{\sqrt{2\pi}} \int_{0}^{+ \infty} {e^{- \frac{c \sigma^2}{2} u} f(u)}du \Big) \right)^{(1- \varepsilon)r/\E(S_{\tau_0}) }+c' e^{- c" r},
%& \leq \left(1-\frac{\sqrt{c }\cdot \E(S_{\tau_0})}{m} \right)^{(1- \varepsilon)r/\E(S_{\tau_0}) }+c' e^{- c" r},
\end{align*}
which gives the upper bound.  \\
For the \textit{lower bound} the very beginning starts with the same spirit as the proof of Lemma A.2 in \cite{HuShi15b} : let $r_k=a*k$ for $ 0 \leq k \leq N:=  \frac{r}{a}$ and $a>0$ (chosen later) then 
\[\cap_{k=0}^N \{ \inf\{ i > \tau_{r_{k}}, S_i \geq r_{k+1} \}< \inf\{ i > \tau_{r_{k}}, S_i \leq r_{k}-\ell \}  \}  \subset \{ \tau_{r} \leq \tau^{\bar S-S}_{\ell} \}, \]
then, the strong Markov property gives
\begin{align}
\E \left[e^{-\frac{c \sigma^2}{2\ell^2} \tau_{r}} \un_{\tau_{r} \leq \tau^{\bar S-S}_{\ell}} \right] & \geq \Pi_{k=0}^N \E_{r_k} \left(e^{-\frac{c \sigma^2}{2\ell^2} \tau_{r_{k+1}}} \un_{ \tau_{r_{k+1}}<\tau_{r_k-\ell}^- }  \right) \nonumber \\
& =\Pi_{k=0}^N \E \left(e^{-\frac{c \sigma^2}{2\ell^2} \tau_{r_{k+1}-r_k}} \un_{ \tau_{r_{k+1}-r_k}<\tau_{-\ell}^- }  \right) \nonumber \\ 
& = \left( \E \left(e^{-\frac{c \sigma^2}{2\ell^2}  \tau_{a} } \un_{ \tau_{a}<\tau_{-\ell}^- }  \right) \right)^{N+1}.  \nonumber
\end{align}
 So we only need a lower bound for Laplace transform of the form $\E (e^{-h \tau_{a}} \un_{ \tau_{a}<\tau_{-\ell}^- }  )$, with $h=h(\ell) \rightarrow 0$. From here we follow the same lines as for the upper bound with following differences, $\tau_0$ (resp. $\tt_0^-$)  is replaced by $\tau_a$ (resp. by $\tt_{-a}^-$), also estimation \eqref{lim45} should be replaced by following one that can be found in \cite{AidShi} : there exists $0<\theta<+ \infty$ such that uniformly in $a \in [0, a_\ell]$ with $a_\ell=o(\ell^{1/2})$
  \begin{align*}
 \ell \P(\tt_{-a}^- \geq z \ell^2) \sim \frac{\theta R(a) }{\sqrt{ z }},
 \end{align*}
 for large $\ell$, where $R$ is the usual  renewal  function (see (2.3) in \cite{AidShi}) with following property (see (2.6) together with Lemma 2.1 in \cite{AidShi})
\begin{align}
\lim_{a \rightarrow \infty} \frac{R(a)}{a} =\frac{1}{\theta} \left(\frac{2}{\pi \sigma^2} \right)^{1/2}. \label{lesconstantes}
\end{align}
Now considering \eqref{56}, with the change we have just talked above, as for any $a>0$, $\lim_{\ell  \rightarrow +\infty} \P(\tt_{\ell} \leq z \ell^2|\tt_{-a}^- > z \ell^2)=\P(\overline{\mathfrak{m}}_1> 1/\sigma \sqrt z )$, we obtain 
\begin{align*}
 \lim_{m \rightarrow  \infty} \ell P_2= \lim_{\ell \rightarrow  \infty} \ell \P( \tt_{-a}^- > z \ell^2,  \tt_{\ell} \leq z \ell^2) = \frac{\theta R(a) }{\sqrt{ z }} \P(\overline{\mathfrak{m}}_1> 1/\sigma \sqrt z ),  
\end{align*}
similarly for $P_1= \P(\tt_a^- >   \tt_{\ell} > z\ell^2)$, for $\ell$ large enough and then taking the limit $A \rightarrow +\infty$
\begin{align*}
P_1 & \geq (1- \varepsilon)\frac{\theta R(a) }{\ell}  \sum_{ k = z\ell ^2 }^{A \ell^2}\P(\overline \tS_{k-1} \leq \ell,  \tS_{k} > \ell | \ \tt_a^- >k) \frac{\ell}{k^{1/2}} \\
&  \geq (1- 2 \varepsilon)\frac{\theta R(a) }{{ \ell}}    \int_z^{+\infty} \frac{1}{y^{1/2}} d\P(\overline{\mathfrak{m}}_1>1/\sqrt{y \sigma^2}).
\end{align*}
We then obtain the equivalent of \eqref{lim45b}, that is  $ \lim_{\ell \rightarrow  \infty} \ell  \P( \tau_{a} > z \ell^2, \tau_a> \tau_{-\ell}^-) = {\theta R(a) } f(z)$ from which we deduce following lower bound for associated Laplace transform :
\begin{align*}
% \lim_{m \rightarrow  \infty} m  \P( \tau_a > z m^2, \tau_a> \tau_{-m}^-) \geq \frac{\theta R(a) } {\sigma}f(z),
\lim_{\ell \rightarrow +\infty } m \E\left ( (1 - e^{\frac{-\lambda }{\ell^2}\tau_a})\un_{\tau_a>\tau_{-\ell}}\right) 
 = \lambda{\theta R(a) } \int_{0}^{+ \infty} {e^{- \lambda u} f(u)}du. 
 \end{align*}
In the same spirit $\lim_{\ell \rightarrow +\infty }  \ell\E [1-e^{-\frac{\lambda}{\ell^2} \tau_{-a} }  ] = \sqrt{ \lambda \pi } \theta R(a) $. Also first Lemma 2.2 in  \cite{Aidekon}  gives for any $a>0$ and any $\ell$  large $\P(\tau_{-a}> \tau_{-\ell}^-)= \P_{-a}(\tau_{0}> \tau_{-\ell-a}^-) \sim \E(-S_{\tau_{-a}})/\ell  $. So finally collecting these estimates and taking $\lambda=\sigma^2 c/2$, for any $\varepsilon>0$ and $\ell$ large enough
\begin{align*}
& \E \left[e^{-\frac{c \sigma^2}{2\ell^2} \tau_{r}} \un_{\tau_{r} \leq \tau^{\bar S-S}_{\ell}} \right] \\ 
& \geq 
 \left(1- \Big(\frac{\E(-S_{\tau_{-a}})}{\ell} + \frac{ \theta R(a) }{\ell } \Big( \sqrt{\frac{ \pi}{2}}\sigma \sqrt{c} - \frac{c \sigma^2}{2}\int_{0}^{+ \infty} {e^{- \frac{c \sigma^2}{2} u} f(u)}du  \Big)\Big) (1+ \varepsilon)   \right)^{N+1}.
\end{align*}
Now recall that $N=r/a$, so let us take $a$ large enough in such a way that (using \eqref{lesconstantes}) $R(a)/a \leq \frac{1}{\theta} \left(\frac{2}{\pi \sigma^2} \right)^{1/2} (1+ \varepsilon)$. Also for large $a$, $\E(-S_{\tau_{-a}})/a \leq (1+\varepsilon)$ (this can be seen easily, noticing that undershoot $S_{\tau_{-a}}-a$ has a second moment). This finishes the proof.
\end{proof}

\begin{lemm}\label{EspMaxLine}  For any $\varepsilon>0$, $\beta>0$, any $r$ large enough uniformly in $t=t(r)$ {with $\lim_{r \rightarrow + \infty} r-t=+\infty$}, 
\begin{align*} 
%\bo{e^{-2 \sqrt{r-t}(1+ \varepsilon)} \leq \E\left(e^{-\max_{1 \leq j \leq \tau_{r-t}} \overline{S}_j-S_j}\un_{\tau_{r-t} \leq \tau_{-B}^-} \right)\leq }
\E\left(e^{-\max_{1 \leq j \leq \tau_{r-t}} \overline{S}_j-S_j} \right) \leq e^{-2 \sqrt{r-t}(1- \varepsilon)}.
\end{align*}
%\bo{ce bout n'est pas utlile car on fait la minoration directement dans la preuve des th√É¬É√Ç¬É√É¬Ç√Ç¬É√É¬É√Ç¬Ç√É¬Ç√Ç¬©or√É¬É√Ç¬É√É¬Ç√Ç¬É√É¬É√Ç¬Ç√É¬Ç√Ç¬®mes.}
\end{lemm}

\begin{proof} 
%\textit{Lower bound} is easily obtain just notice that
%\begin{align*}
%& \E\left(e^{-\max_{1 \leq j \leq \tau_{r-t}} \overline{S}_j-S_j}\un_{\tau_{r-t} \leq \tau_0} \right) \\
%& \geq \E\left(e^{-\max_{1 \leq j \leq \tau_{r-t}} \overline{S}_j-S_j} \un_{ \max_{1 \leq j \leq \tau_{r-t}} \overline{S}_j-S_j} \leq  \sqrt{r-t} , \tau_{r-t} \leq \tau_0\right) \\
%& \geq  e^{-\sqrt{r-t}} \P\left({ \max_{1 \leq j \leq \tau_{r-t}} \overline{S}_j-S_j} \leq  \sqrt{r-t},\tau_{r-t} \leq \tau_0 \right) \geq e^{-\sqrt{r-t}} e^{- \sqrt{r-t} (1+ \varepsilon)},
%\end{align*}
%where last inequality comes from \eqref{A.2_Hu_Shi}. \\ 
% \noindent The {upper bound} needs some more work
Like in the proof of  Lemma \ref{Laplace2} we use strict ladder epoch sequence $(T_k:= \inf\{s>T_{k-1}, S_s > S_{T_{k-1}}\},k;T_0=0)$, also let us introduce random variable $Y_k := \max_{T_{k-1} \leq j \leq T_k} \overline{S}_j-S_j$ for any $k\geq 1$. Let $m$ a positive integer to be chosen later, by the strong Markov property 
\begin{align*}
 \E\left(e^{-\max_{1 \leq k \leq m} Y_k} \right) 
 %e^{-(1+\varepsilon) (\E(S_{\tau_0}) m)^{1/2} }e^{-(1-\varepsilon)^{-1} (\E(S_{\tau_0}) m)^{1/2} }
 & = \sum_{k=1}^m \E(e^{-Y_k} \un_{Y_k> \max_{i \leq k-1}Y_i,\ Y_k \geq \max_{k+1\leq i \leq m } Y_i  } ) \\
 & \leq m\E(e^{-Y_2}(1- \P(Y_1>Y_2|Y_2))^{m-1}).
% & \leq m e^{- m^{1/2}(1-\varepsilon)}+ \sum_{k=1}^m \E(e^{-X_k} \un_{X_k> \max_{i \leq k-1}X_i,\ X_k \geq \max_{k+1\leq i \leq m } X_i  } ) 
\end{align*} 
At this point we need an asymptotic in $y$ of $M(y):=\P(Y_1>y)=\P(\max_{0 \leq s \leq T_0 }S_s <-y )=\P(\tau_0>\tau_{-y})$, for that we use following equality (see for example \cite{Aidekon} Lemma 2.2) : for large $y$, $\P(\tau_0> \tau_{-y})= \E(S_{\tau_0})/y + o(1/y)$. So for any large $A$, and $\varepsilon>0$
\begin{align*}
& e^{-Y_2}(1- \P(Y_1>Y_2|Y_2))^{m-1} \\
=&e^{-Y_2}(1- \P(Y_1>Y_2|Y_2))^{m-1}\un_{Y_2>A}+e^{-Y_2}(1- \P(Y_1>Y_2|Y_2))^{m-1}\un_{Y_2 \leq A} \\
\leq & e^{-Y_2}\left(1-{\E(S_{\tau_0})(1-\varepsilon)}{(Y_2)^{-1}}\right)^{m-1}\un_{Y_2>A}+(1- \P(Y_1>A))^{m-1},
\end{align*} 
For the second term above we can find constant $c=c(A)$ such that $(1- \P(Y_1>A))^{m-1}\leq e^{-c m}$. For the first term , let us introduce measure $dM$ defined as $M(x)=\int_x^{+\infty}dM(z)dz$, then integrating by parts 
\begin{align*}
 & \E(e^{-Y_2}\left(1-{\E(S_{\tau_0})(1-\varepsilon)}{(Y_2)^{-1}}\right)^{m-1}\un_{Y_2>A})  =  -\int_{A}^{+ \infty} e^{-x}\left(1-\frac{\E(S_{\tau_0})(1-\varepsilon)}{x}\right)^{m-1}dR(x) \\
& \leq  e^{-A} \Big(1-\frac{\E(S_{\tau_0})(1-\varepsilon)}{A}\Big)^{m-1} -\int_A^{+\infty}   e^{-x}\left(1-\frac{\E(S_{\tau_0})(1-\varepsilon)}{x}\right)^{m-1}R(x)dx \\
& -(m-1)S_{\tau_0}(1-\varepsilon) \int_A^{+\infty}   \frac{e^{-x}}{x^2}\left(1-\frac{\E(S_{\tau_0})(1-\varepsilon)}{x}\right)^{m-2}R(x)dx \\
& \leq e^{-2(1-4 \varepsilon) \sqrt{\E(S_{\tau_0}) m}},
%& \sum_{k=1}^m  \int_{0}^{+ \infty} e^{-x} \E( \un_{x> \max_{i \leq k-1}X_i,\ x \geq \max_{k+1\leq i \leq m } X_i  } )dR(x) \leq m  \int_{0}^{+ \infty} e^{-x} (1-R(x))^{m-1} dR(x) \\
%& = m  \Big(\int_{0}^{\varepsilon m^{1/2}} e^{-x} (1-R(x))^{m-1} dR(x)+\int_{\varepsilon m^{1/2}}^{+ \infty} e^{-x} (1-R(x))^{m-1} dR(x) \Big) \\
%&= m(I_1+I_2).
\end{align*} 
the last inequality is definitely not optimal but enough for what we need, we can obtain it easily decomposing the interval $(A,+\infty)$ on the intervals $(A,\sqrt{\E(S_{\tau_0})m}(1-\varepsilon))$, $(\sqrt{\E(S_{\tau_0})m}(1-\varepsilon),\sqrt{\E(S_{\tau_0})m}(1+\varepsilon))$ and $(\sqrt{\E(S_{\tau_0})m}(1+\varepsilon), +\infty)$. 
Collecting the above inequalities, we obtain that for any $\varepsilon>0$ and $m$ large enough 
\begin{align*}
 \E\left(e^{-\max_{1 \leq k \leq m} Y_k} \right) \leq  2 m e^{-2(1-4 \varepsilon) \sqrt{\E(S_{\tau_0}) m}}.\end{align*}  To finish the proof we follow the same lines as the end of the proof of Lemma \ref{Laplace2} (below \eqref{eq59}), that is saying that $\E\left(e^{-\max_{1 \leq j \leq \tau_{r-t}} \overline{S}_j-S_j} \right) \leq \E\left(e^{-\max_{1 \leq k \leq m} Y_k} \right)  + \P(S_{T_{k}} \geq r-t) $ then taking  $k=(1- \varepsilon)(r-t)/\E(S_{\tau_0})$.%\red{je n'ai pas ajuste les $\varepsilon$, cela reste implicite}.
 % & \leq m e^{- m^{1/2}(1-\varepsilon)}+ \sum_{k=1}^m \E(e^{-X_k} \un_{X_k> \max_{i \leq k-1}X_i,\ X_k \geq \max_{k+1\leq i \leq m } X_i  } ) 
%\red{Ce truc au-dessus m'√É¬É√Ç¬É√É¬Ç√Ç¬É√É¬É√Ç¬Ç√É¬Ç√Ç¬Ç√É¬É√Ç¬É√É¬Ç√Ç¬Ç√É¬É√Ç¬Ç√É¬Ç√Ç¬étonne un peu j'oublie un truc sans doute.}
%Then  
%for $I_1$, integrating by parts and noticing that we just notice that for $m$ large enough $I_1 \leq m(1-R(\varepsilon m^{1/2}))^{m-1} \leq e^{- \frac{ \E(S_{\tau_0}) m^{1/2}}{2\varepsilon}}$. For $I_2$
\end{proof}

\subsection{Additional technical estimates}

%Following Lemma and more especially inequality \eqref{lem4.1b} below is used when we ask for the behavior of heavy range together with high potential. 

\begin{lemm}\label{Minor_HRHP1}  
Let $(t_{\ell})$  a positive increasing sequence such that ${t_{\ell}\ell^{-1/2}}\to+\infty$ but ${t_{\ell}}{\ell^{-1}}  \to 0$. For any $B>0$ and $\ell$ large enough
\begin{align}
\P( \tau^{\overline S-S}_{{\ell^{1/2}}} \vee {\tau}_{-B}^- > \tau_{t_{\ell}}) \geq  e^{- \frac{t_{\ell}}{\sqrt \ell}(1+o(1)) } .\label{A.2_Hu_Shi}
\end{align}
 %of $H_j^S= \sum_{i=1}^j e^{S_i-S_j}$, 
Let $A>0$ large, $d\in(0,1/2)$, $a>0$, $0<b < 1$, $q \in[b,1]$, $a_b:=a(2\un_{q>b}-1)$ and $c>0$
\begin{align}
 &\sum_{{j\leq A\ell^{3/2}}}  \P\big(S_j \geq t_{\ell}, \sup_{m \leq j} H_m^S \leq e^{{q}\sqrt{\ell}- a_b{\tlp}}, {e^{b \sqrt \ell} \leq H_j^S \leq e^{b \sqrt \ell+c {\tlp}  },\underline{S}_j \geq -B }\big) \geq e^{-{\frac{t_{\ell}}{q\sqrt{\ell}}}(1+o(1))}.  \label{lem4.1b}
\end{align}
%\DO{On a besoin de prendre $q=b$ dans la preuve des th√É¬É√Ç¬©or√É¬É√Ç¬®mes. Du coup, on peut mettre $q\sqrt{\ell}-a_b\ell^d$ avec $q\in[b,1]$ et $a_b=:a(2\un_{\{q>b\}}-1)$ et $a>0$ a la place ?} 

%let $A>0$ and $(a_{\ell})$ a positive increasing sequence such that   $\log (A \ell^{2}) :=o(a_{\ell} \sqrt \ell)$ and $\lim_{\ell \rightarrow + \infty} a_{\ell}= 0$, for any $0<b<1$, \bo{$b\leq q\leq 1$, $\varepsilon>0$} and $n$ large enough
%\begin{align}
% &\sum_{\bo{j\leq A\ell^{3/2}}}  \P\big(S_j \geq t_{\ell}, {e^{b \sqrt \ell} \leq H_j \leq e^{b \sqrt \ell+\bo{\varepsilon}\ell^d}, \sup_{m \leq j} H_m \leq e^{(\bo{q}+ a_{\ell} )\sqrt{\ell}},\underline{S}_j \geq -B }\big) \geq e^{-\bo{\frac{t_{\ell}}{q\sqrt{\ell}}}(1+o(1))}.  \label{lem4.1b}
%\end{align}
%\bo{encore vrai?}
\end{lemm}

%\red{Voici les deux expressions, exactement comme elles apparaissent dans la preuve des th√É¬É√Ç¬©or√É¬É√Ç¬®mes:\begin{align*}
%    \sum_{k\geq 1}\Pb(S_k\geq \alpha_n,\max_{j\leq k}H_j^S\leq\bmlambda, n^b<H_k^S<n^be^{\varepsilon(\log n)^{\alpha-1}},\underline{S}_k\geq-B)
%\end{align*}
%et
%\begin{align*}
%     \sum_{k\leq\lfloor A\ell_n\rfloor}\Pb(S_k\geq \alpha_n,\max_{j\leq k}H_k^S\leq n^be^{\frac{\varepsilon}{3b}(\log n)^{\alpha-1}}, H_k^S>n^b,\underline{S}_k\geq-B)
%\end{align*}
%avec $b>0$, $\ell_n=(\log n)^3$, $\alpha_n=(\log n)^{\alpha}+\log n$, $\bmlambda=ne^{-6(\log n)^{\alpha-1}}$.}

\begin{proof} 
The proof of \eqref{A.2_Hu_Shi} follows the same lines as the proof of Lemma A.2 in \cite{HuShi15b}. %\red{[attention reference doit etre mise \` a jour -> pas d'autre mise \`a jour, article non publi\'e ?]} \\
For \eqref{lem4.1b}, as $j \leq A \ell^{3/2}$, for any $(d,e)$ and any $m \leq j$, $A \ell^{3/2} \exp(\overline S_m-S_m) \leq e^{d \sqrt \ell+e \tlp}$ implies $H_m^S \leq e^{d \sqrt \ell+e \tlp}$ then 
\begin{align*} 
&   \P\left[  S_j \geq t_{\ell}, {e^{b \sqrt \ell} \leq H_j^S \leq e^{b \sqrt \ell+c \tlp}, \sup_{m \leq j} H_m^S \leq e^{q*\sqrt{\ell}-a_b \tlp},\underline{S}_j \geq -B } \right] \\
& \geq \P\left[  S_j \geq t_{\ell}, {{b \sqrt \ell} \leq \overline S_j-S_j  \leq b\sqrt\ell+c' \tlp, \sup_{m \leq j} \overline S_m-S_m \leq q{\sqrt{\ell}}-a'{\tlp},\underline{S}_j \geq -B } \right] 
%& \geq \frac{1}{e^{b(1+ 2\varepsilon_n)\sqrt{n}}} \P\left[ {b \sqrt n  \leq \overline S_j-S_j \leq b (1+ \varepsilon_n) \sqrt{n}, \sup_{m \leq j} (\overline S_m-S_m) \leq \sqrt{n}, S_j \geq t_n,\underline{S}_j \geq -B } \right].
\end{align*} 
with $c'=c/2$ and $a'=a_b+1$. To obtain a lower bound for the above probability, the idea is to say that maximum of $S$ is obtained at a certain instant  $k \leq j$ and that this maximum is larger than  $t_{\ell}+b\sqrt\ell+c'\tlp+r$ for a certain $r>0$ to be chosen latter, then above probability is larger than : 
\begin{align*}
\sum_{k \leq j} &\P(\overline{S}_{k-1}<S_k, S_k \geq t_{\ell}+b\sqrt\ell+c'\tlp+r,\sup_{m \leq k} \overline{S}_m-S_m \leq \sqrt{\ell}-a'\tlp, \underline{S}_k \geq -B; S_j- S_k \geq t_{\ell}-S_k, \nonumber \\
& b \sqrt \ell \leq S_k-S_j \leq  b\sqrt\ell+c'\tlp,  \forall m \geq k+1, S_m \leq S_k, S_k-S_m \leq \sqrt{\ell}-a'\tlp , S_m -S_k \geq -B - S_k ). %\label{sum4.1}
\end{align*}
Now, the events $ \{S_m -S_k \geq -B - x\}$, as well as $\{S_j- S_k \geq t_{\ell}-x\}$ increases in $x$ and as $S_k \geq t_{\ell}+b\sqrt\ell+c'\tlp+r$ so we can replace, in the two events of the above probability,\textrm{ \guillemotleft$-S_k$\guillemotright } by $-(t_{\ell}+b\sqrt\ell+c'\tlp+r)$. This makes appear two independent events, so above probability is larger than 
\begin{align}
& \P(\overline{S}_{k-1}<S_k, S_k \geq t_{\ell}+b\sqrt\ell+c'\tlp+r, \sup_{m \leq k} \overline{S}_m-S_m \leq \sqrt{\ell}-a'\tlp, \underline{S}_k \geq -B)  \times  \nonumber \\
& \P( S_j- S_k \geq -b\sqrt\ell-c'\tlp-r, b \sqrt \ell  \leq S_k-S_j \leq  b\sqrt\ell+c'\tlp, \forall m \geq k+1, \nonumber \\  
&  -B - t_{\ell}-b\sqrt\ell+c'\tlp-r \leq  S_m -S_k \leq 0,    S_m-S_k \geq - \sqrt{\ell}+a'\tlp  )=:p_1(k)*p_2(k,j). \label{eq29}
\end{align}
Probability $p_2$ can be easily simplified, indeed as $\lim_{\ell \rightarrow + \infty} t_{\ell}/\sqrt \ell=+ \infty$ and $\ell$ large, $ -B - t_{\ell}-b\sqrt\ell+c'\tlp-r  \leq - \sqrt{\ell}$ and by taking $r=c' \tlp $, $p_2$ is smaller than  
\begin{align*}
 & \P(-b\sqrt\ell-c'\tlp  \leq S_j-S_k \leq  -b \sqrt{\ell}, \forall m \geq k+1, - \sqrt{\ell}+ a'\tlp \leq  S_m -S_k \leq 0 ) \\
 =& \P( \forall m \leq j-k , - \sqrt{\ell}+a'\tlp \leq  S_m  \leq 0,-b\sqrt\ell-c'\tlp \leq S_{j-k} \leq  -b \sqrt{\ell} ) \\
 =&\P( \forall m \leq j-k,\  \tS_m  \leq  \sqrt{\ell}-a'\tlp | \underline{\tS}_{j-k} \geq 0,  \tS_{j-k} \in  [b \sqrt{\ell},b\sqrt\ell+c'\tlp] ) \times \\ 
  & \P( \underline{\tS}_{j-k} \geq 0,  \tS_{j-k} \in  [b \sqrt{\ell},b\sqrt\ell+c'\tlp] ),
\end{align*}
with $\tS_m=-S_m$ for any $m$.  For the conditional probability we can use a similar result proved by Caravenna and Chaumont \cite{CaraChau2} telling that the distribution $\P_x(\cdot | \forall m \leq n, S_m \geq 0, S_n \in[0,h)) $ converges. Note that they need in their work additional hypothesis on the distribution of $S_1$ (more especially absolute continuity of the distribution of $S_1$) which is not necessary here as the size of interval  $[b \sqrt{\ell},b\sqrt\ell+c'\tlp] $ equals $c'\tlp \rightarrow + \infty$. So as $a'\tlp =o (\sqrt \ell)$
\begin{align*}
\lim_{\ell \rightarrow + \infty}\P( \forall m \leq \ell ,   \tS_m  \leq  c \sqrt \ell - a'\tlp|  \underline{\tS}_{\ell} \geq 0,   \tS_{\ell}  \in [b \sqrt{\ell},b\sqrt\ell+c'\tlp] )=Cte>0. %\textbf{[!! il faut expliciter un peu]} %\P(\overline R_1 \leq c/ \sigma | R_1=a/ \sigma ) \textbf{}.
\end{align*}
Moreover another work of  Caravenna (\cite{Carav} Theorem 1)  gives for large $\ell$,  $\P(  \underline{\tS}_{\ell} \geq 0,   \tS_{\ell}  \in [b \sqrt{\ell},b\sqrt\ell+c'\tlp ) \geq   b /{\ell}$. So finally when $j-k$ is of the order of $\ell$, there exists a constante $Cte>0$ such that  $p_2(k,j) \geq Cte* \ell^{-1}$. %\textbf{[check this constant and its dependence in $\varepsilon$]}. 
Turning back to \eqref{eq29} and summing over $k$ and $j$, we obtain
%and then to the sum in \eqref{sum4.1} and finally taking the sum on $j$ : 
\begin{align*}
 & \sum_{j \leq A \ell^{3/2}} \sum_{k \leq j} p_1(k)p_2(k,j) \\ 
 & = \sum_{k \leq A \ell^{3/2}} p_1(k) \sum_{j \geq  k }p_2(k,j) \geq \sum_k p_1(k) \sum_{j, j-k \sim \ell }p_2(k,j) \\ 
%& \times \sum_{j \geq k}  \P( \forall m \leq j-k , 0 \leq  \tS_m  \leq  \sqrt{n} ,  \tS_{j-k}-\log b- \log j  \in [0,\log  (h/j)  ] ) \\
%& \geq  \sum_{k  \geq 0}  \P(\overline{S}_{k-1}<S_k, S_k \geq t_n+b+r, \underline{S}_k \geq -B, \sup_{m \leq k} \overline{S}_m-S_m \leq \sqrt{n}) \\ 
%& \times \sum_{j-k \sim n}  \P( \forall m \leq j-k , 0 \leq  \tS_m  \leq  \sqrt{n} ,  \tS_{j-k}-\log b- \log j  \in [0,\log  (h/j)  ] ) \\
%& \P(\overline{S}_{k-1}<S_k, S_k \geq t_n+b (1+ \varepsilon_n) \sqrt{n}+r, \sup_{m \leq k} \overline{S}_m-S_m \leq \sqrt{n}, \underline{S}_k \geq -B)\\
& \geq \frac{Cte}{\ell}  \sum_{k  \leq  A \ell^{3/2}}  \P(\overline{S}_{k-1}<S_k, S_k \geq t_{\ell}+b (1+ 2\varepsilon_{\ell})\sqrt \ell, \sup_{m \leq k} \overline{S}_m-S_m \leq \sqrt{\ell},  \underline{S}_k \geq -B)  \\
&  \geq \frac{Cte}{\ell} \Big(\P( \tau^{\overline S-S}_{\sqrt{\ell}} \vee \underline{\tau}_{-B} > \tau_{t_{\ell}+b\sqrt\ell+c'\tlp})- \sum_{k  >  A \ell^{3/2}}  \P( \sup_{m \leq k} \overline{S}_m-S_m \leq \sqrt{\ell}) \Big)
\end{align*}
Now we can check that above sum $\sum_{k  >  A \ell^{3/2}} \cdots$ as a negligible contribution, indeed the probability $\P( \sup_{m \leq k} \overline{S}_m-S_m \leq \sqrt{\ell})$ is smaller, thanks to Proposition 3.1 in \cite{HuShi10b},  to $e^{-\pi^2 \sigma^2 j/4 \ell}$ this implies that $ \sum_{k  >  A \ell^{3/2}}  \P( \sup_{m \leq k} \overline{S}_m-S_m \leq \sqrt{\ell})   \leq e^{-\pi^2 \sigma^2 A \ell^{1/2}/2} $.
Now if we apply \eqref{A.2_Hu_Shi} to the first probability above as $b\sqrt\ell+c'\tlp=o(t_{\ell})$, this finishes the proof.%$b (1+ 2\varepsilon_n)\sqrt n=o(t_n)$, this finish the proof. \textbf{[Attention aux sommes ... celle sur k doit aller \`a l'infini.]}
\end{proof}

\noindent Lemma below is a simple extension of FKG inequality. %\red{[Je crois que Shi l'utilise juste en disant "FKG" dans le papier  "high potential" il me semble, \`a voir mais on laisse pour l'instant.]} \red{\underline{r√É¬É√Ç¬©ponse A}: effectivement, Shi l'utilise sous le nom de "FKG" mais √É¬É√Ç¬ßa ne me paraissait pas tr√É¬É√Ç¬®s clair car la vraie "FKG" s'applique seulement pour une mesure sur les bor√É¬É√Ç¬©liens de $\R$ $\mathcal{B}(\R)$ alors que nous souhaitons l'appliquer √É¬É√Ç¬† la loi de $(S_i)_{i\geq 1}$ qui est une mesures sur $\mathcal{B}(\R)^{\otimes\N}$ et pour cela on peut passer par une extension "triviale" de FKG pour un produit fini de mesures.} \\

\noindent In the following, a function $F:\R^k\longrightarrow \R$ is said to be non-decreasing if: for all $\mathbf{s}=(s_1,\ldots,s_k)\in\R^k$ and $\mathbf{t}=(t_1,\ldots,t_k)\in\R^k$, $\mathbf{s}\leq_k \mathbf{t}$ implies $F(\mathbf{s})\leq F(\mathbf{t})$ where $\mathbf{s}\leq_k \mathbf{t}$ if and only if $s_j\leq t_j$ for all $j\in\{1,\ldots,k\}$.

\begin{lemm}\label{LemmFKG}
Let $r>0$, $k\in\N^*$, $f_1,f_2:\R^k\longrightarrow\R^+$. For any $i\in\{1,2\}$, introduce $\Tilde{f}_i(u_1,\ldots,u_k):=f_i(u_1,u_1+u_2,\ldots,u_1+u_2+\ldots+u_k)$. If $\Tilde{f}_1$ and $\Tilde{f}_2$ are non-decreasing then
\begin{align*}
    \E\big[f_1(S_1,S_2,\ldots,S_k)f_2(S_1,S_2,\ldots,S_k)\big] \geq \E\big[f_1(S_1,S_2,\ldots,S_k)\big]\E\big[f_2(S_1,S_2,\ldots,S_k)\big].
\end{align*}
\end{lemm}

\begin{proof}
When $\R^k$ is a totally order set, the first inequality above is the well known regular FKG inequality. Here, we can easily extend it to the partial order $\leq_k$. Indeed, since $\Tilde{f}_i$ is non-decreasing for any $i\in\{1,2\}$, we have, by independence of increments of $S$
\begin{align*}
    \prod_{i\in\{1,2\}}\E\big[f_i(S_1,S_2,\ldots,S_k)\big]=\prod_{i\in\{1,2\}}\E\big[\Tilde{f}_i(S_1,S_2-S_1,\ldots,S_k-S_{k-1})\big]=\E[\mathbf{F}_1(S_1)]\E[\mathbf{F}_2(S_1)],
\end{align*}
with $\mathbf{F}_i(u_1):=\E\big[\Tilde{f}_i(u_1,S_2-S_1,\ldots,S_k-S_{k-1})\big]$ for any $i\in\{1,2\}$. Since $\Tilde{f}_i$ is non-decreasing, $\mathbf{F}_i$ is also non-decreasing so thanks to the regular FKG inequality, $\E[\mathbf{F}_1(S_1)]\E[\mathbf{F}_2(S_1)]\leq \E[\mathbf{F}_1\mathbf{F}_2(S_1)]$. Again, using that the increments of $S$ are independent and stationary, the result follows by induction.
\end{proof}

\begin{lemm}\label{MinorRangeHP2}
Let $(t_{\ell})$ a sequence of positive numbers such that $t_{\ell}/\ell\to 0$. For all $d\in(0,1/2]$ such that $t_{\ell}/\ell^{d}\to+\infty$ and all $\varepsilon,B>0$, $a\geq 0$ and $0\leq d'<d$ for $n$ large enough
\begin{align*}
    \sum_{k\leq\ell^2}\P\big(S_k\geq t_{\ell}, \max_{j\leq k}H_j^S\leq e^{\ell^d-a\ell^{d'}},\underline S_k\geq-B,\overline{S}_k=S_k\big)\geq e^{-\frac{t_{\ell}}{\ell^d}(1+\varepsilon)}.
\end{align*}
\end{lemm}
\begin{proof}
Recall that $\tau_{r}=\inf\{i\geq 1;\;S_i\geq r\}$. First, observe that for all $j\leq k\leq\ell^2$, $H_j^S\leq \ell^2e^{\overline{S}_j-S_j}$ so 
\begin{align*}
    &\sum_{k\leq\ell^2}\P\big(S_k\geq t_{\ell}, \max_{j\leq k}H_j^S\leq e^{\ell^d-a\ell^{d'}},\underline S_k\geq-B,\overline{S}_k=S_k\big) \\ & \geq \sum_{k\leq\ell^2}\P\big(k=\tau_{t_{\ell}}, \max_{j\leq k}\overline{S}_j-S_j\leq\ell^d-a\ell^{d'}-2\log \ell,\underline S_k\geq-B\big),
\end{align*}
which is equal to  $\P\big(\overline{S}_{\ell^2}\geq t_{\ell},\forall  j\leq\tau_{t_{\ell}}:\overline{S}_j-S_j\leq \ell^d-a\ell^{d'}-2\log\ell,S_j\geq-B\big)$. \\ 
Now let $k_{\ell}=\lfloor(e^{\ell}t_{\ell})^2\rfloor+\ell^2$. First note that, since $\ell^2\leq k_{\ell}$, we have, on $\{\overline{S}_{\ell^2}\geq t_{\ell}\}$, $\tau_{t_{\ell}}=\tau^{k_{\ell}}_{t_{\ell}}$ with $\tau^{k_{\ell}}_{t_{\ell}}:=k_{\ell}\land\inf\{i\leq k_{\ell};\;S_i\geq t_{\ell}\}$ so
\begin{align*}
    &\P\big(\overline{S}_{\ell^2}\geq t_{\ell}, \forall j\leq\tau_{t_{\ell}}:\overline{S}_j-S_j\leq\ell^d-a\ell^{d'}-2\log \ell,S_j\geq-B\big) \\ & =\P\big(\overline{S}_{\ell^2}\geq t_{\ell}, \forall j\leq\tau^{k_{\ell}}_{t_{\ell}}:\overline{S}_j-S_j\leq\ell^d-a\ell^{d'}-2\log \ell,S_j\geq-B\big).    
\end{align*}
For any $k\in\N^*$ and $r>0$, let $\mathbf{t}=(t_1,\ldots,t_k)\in\R^k$ and define the $\mathbf{t}$-version $\tau^{k,\mathbf{t}}_r$ of $\tau^k_r$ that is
\begin{align*}
    \tau^{k,\mathbf{t}}_r := k\land\inf\big\{i\leq k;\; t_i\geq r\big\},
\end{align*}
with the usual convention $\inf\varnothing=+\infty$. Then
\begin{align*}
   \P\big(\overline{S}_{\ell^2}\geq t_n, \forall j\leq\tau^{k_{\ell}}_{t_{\ell}}:\overline{S}_j-S_j\leq\ell^d-a\ell^{d'}-2\log \ell,S_j\geq-B\big)=\E\big[f_1f_2(S_1,S_2,\ldots,S_{k_{\ell}})\big],
\end{align*}
with for all $i\in\{1,2\}$, $f_i:=\un_{A^{\ell}_i}$, $f_1f_2(\mathbf{u})=f_1(\mathbf{u})f_2(\mathbf{u})$ and 
\begin{align*}
    A^{\ell}_1:=\big\{\mathbf{u}=(u_1,\ldots,u_{k_{\ell}})\in\R^{k_{\ell}};\exists\; j\leq \ell^2: u_j\geq t_{\ell}\big\},
\end{align*}
and
\begin{align*}
    A^{\ell}_2:=\big\{\mathbf{u}=(u_1,\ldots,u_{k_{\ell}})\in\R^{k_{\ell}};\forall\;j\leq\tau^{k_{\ell},\mathbf{t}}_{t_{\ell}}, \forall i<j: u_j-u_i\geq -\ell^d+a\ell^{d'}+2\log\ell,\;u_j\geq-B\big\}.
\end{align*}
Then, it is easy to see that for all $i\in\{1,2\}$, $\tilde{f}_i$ (see Lemma \ref{LemmFKG} for the definition) is non-decreasing according to the partial order $\leq_{k_{\ell}}$ defined above. Then, thanks to Lemma \ref{LemmFKG}, $\E[f_1f_2(S_1,S_2,\ldots,S_{k_{\ell}})]$ is larger than
\begin{align*}
     & \geq\P\big((S_1,S_2,\ldots,S_{k_n})\in A^{\ell}_1\big)\P\big((S_1,S_2,\ldots,S_{k_{\ell}})\in A^{\ell}_2\big) \\ & \geq\P(\overline{S}_{\ell^2}\geq t_{\ell})\P\big(\forall j\leq\tau^{k_{\ell}}_{t_{\ell}}:\overline{S}_j-S_j\leq\ell^d-a\ell^{d'}-2\log\ell,S_j\geq-B,\tau_{t_{\ell}}\leq k_{\ell}\big).
\end{align*}
Again, on $\{\tau_{t_{\ell}}\leq k_{\ell}\}$, $\tau^{k_{\ell}}_{t_{\ell}}=\tau_{t_{\ell}}$ and thanks to \cite{Kozlov1976} (Theorem A), there exists $C_K>0$ such that for $\ell$ large enough
\begin{align*}
    &\P\big(\forall j\leq\tau^{k_{\ell}}_{t_{\ell}}:\overline{S}_j-S_j\leq\ell^d-a\ell^{d'}-2\log\ell,S_j\geq-B,\tau_{t_{\ell}}\leq k_{\ell}\big) \\ & \geq \P\big(\forall j\leq\tau_{t_{\ell}}:\overline{S}_j-S_j\leq\ell^d-a\ell^{d'}-2\log\ell,S_j\geq-B\big)-\P(\tau_{t_{\ell}}>k_{\ell}) \\ & \geq \P\big(\forall j\leq\tau_{t_{\ell}}:\overline{S}_j-S_j\leq\ell^d-a\ell^{d'}-2\log\ell,S_j\geq-B\big)-C_Ke^{-\ell}.
\end{align*}
Moreover, $t_{\ell}/\ell\to 0$ so $\P(\overline{S}_{\ell^2}\geq t_{\ell})\to 1$. Finally, by \eqref{A.2_Hu_Shi} together with the fact that $\ell^d\sim\ell^d-a\ell^{d'}-2\log\ell$ (as $d>d'$) for $\ell$ large enough, $\P\big(\forall j\leq\tau_{t_{\ell}}:\overline{S}_j-S_j\leq\ell^d-a\ell^{d'}-2\log\ell,S_j\geq-B\big)\geq 2e^{-t_{\ell}\ell^{-d}(1+\varepsilon)}$ and since $t_{\ell_n}/\ell^d=o(\ell)$, $C_Ke^{-\ell}\leq e^{-t_{\ell}\ell^{-d}(1+\varepsilon)}$, the result follows.
\end{proof}

\begin{lemm}\label{LemmBorneUnif}
Let $\alpha\in(1,2)$ and $\varepsilon_{\alpha}\in[0,\alpha-1)$ and introduce $\bm{L}_{\ell}:=\lfloor \chi\ell^{1+\frac{\varepsilon_{\alpha}}{2}}\rfloor$, $\chi>0$. For all $\varepsilon>0$, $\ell$ large enough and any $k\in\{\bm{L}_{\ell},\ldots, \ell^2\}$
\begin{align}\label{LemmBorneUnif1}
    \P\big(\max_{j\leq k}\;H_j^S\leq e^{\sqrt{\ell}}\big)\leq e^{-\frac{k\pi^2\sigma^2}{8\ell}(1-\varepsilon)}, 
\end{align}
and for any $a,d,c>0$, $b\in(0,1)$, $\ell$ large enough and any $k\in\{\bm{L}_{\ell},\ldots, \ell^2\}$
\begin{align}\label{LemmBorneUnif2}
    \P\big(\max_{j\leq k} H_j^S\leq e^{\sqrt{\ell}-a\ell^d},e^{b\sqrt{\ell}}<H_k^S\leq e^{b\sqrt{\ell}+c\ell^d},\underline{S}_k\geq 0\big)\geq e^{-\frac{k\pi^2\sigma^2}{8\ell}(1+\varepsilon)}.
\end{align}
\end{lemm}
\begin{proof}
Let us start with the upper bound. Thanks to the Markov property, for any $k\in\N$, $k>\bm{L}_{\ell}$
\begin{align*}
  \P\big(\max_{j\leq k}\;H_j^S\leq e^{\sqrt{\ell}}\big)\leq\P\big(\max_{j\leq k}\;\overline{S}_j-S_j\leq\sqrt{\ell}\big)\leq \P\big(\max_{j\leq \bm{L}_{\ell}}\;\overline{S}_j-S_j\leq\sqrt{\ell}\big)^{\lfloor\frac{k}{\bm{L}_{\ell}}\rfloor},
\end{align*}
and thanks to \cite{HuShi10b}, for $\ell$ large enough, $\P\big(\max_{j\leq \bm{L}_{\ell}}\;\overline{S}_j-S_j\leq\sqrt{\ell}\big)\leq e^{-\frac{\pi^2\sigma^2\bm{L}_{\ell}}{8\ell}(1-\frac{\varepsilon}{2})}$, so for any $\varepsilon$, $\ell$ large enough and any $k>\bm{L}_{\ell}$
\begin{align*}
  \P\big(\max_{j\leq k}\;\overline{S}_j-S_j\leq\sqrt{\ell}\big)\leq e^{-(1-\frac{\varepsilon}{2})\frac{\pi^2\sigma^2\bm{L}_{\ell}}{8\ell}\lfloor\frac{k}{\bm{L}_{\ell}}\rfloor}\leq e^{-(1-\varepsilon)\frac{k\pi^2\sigma^2}{8\ell}}.
\end{align*}
For the lower bound, observe that for any $k\leq\ell^2$, $\P\big(\max_{j\leq k} H_j^S\leq e^{\sqrt{\ell}-a\ell^d},e^{b\sqrt{\ell}}<H_k^S\leq e^{b\sqrt{\ell}+c\ell^d},\underline{S}_k\geq 0\big)$ is larger than $\P\big(\max_{j\leq k} \overline{S}_j-S_j\leq \bm{\lambda}'_{\ell},b\sqrt{\ell}<\overline{S}_k-S_k\leq b\sqrt{\ell}+c\ell^d-\log\ell^2,\underline{S}_k\geq 0\big)$, where $\bm{\lambda}'_{\ell}:=\sqrt{\ell}-a\ell^d-\log\ell^2$. As $\frac{c}{2}\ell^d\geq \log\ell^2$ ($d>0$), the previous probability is larger than $\P\big(\max_{j\leq k} \overline{S}_j-S_j\leq \bm{\lambda}'_{\ell},b\sqrt{\ell}<\overline{S}_k-S_k\leq b\sqrt{\ell}+\frac{c}{2}\ell^d,\underline{S}_k\geq 0\big)$. We need independence to compute this probability so for all $k\in\N^*$, $\bm{L}_{\ell}<k\leq\ell^2$, we say that $\overline{S}_k=S_{k-\ell}\geq\bm{\lambda}'_{\ell}$ which gives that for all $k-\ell<j\leq k$, $\overline{S}_j\leq S_{k-\ell}$ and then, $\max_{k-\ell<j\leq k}S_{k-\ell}-S_j\leq\bm{\lambda}'_{\ell}$ implies that $S_j\geq S_{k-\ell}-\bm{\lambda}'_n\geq 0$ for all $k-\ell<j\leq k$. Hence
\begin{align*}
    \P\big(\max_{j\leq k} \overline{S}_j-S_j\leq \bm{\lambda}'_{\ell},b\sqrt{\ell}<\overline{S}_k-S_k\leq b\sqrt{\ell}+\frac{c}{2}\ell^d,\underline{S}_k\geq 0\big)\geq\P(A_{k,\ell}\cap B_{k,\ell})=\P(A_{k,\ell})\P(B_{k,\ell}),
\end{align*}
with 
\begin{align*}
    A_{k,\ell}:=\big\{\max_{j\leq k-\ell}\;\overline{S}_j-S_j\leq\bm{\lambda}'_{\ell},\underline{S}_{k-\ell}\geq0, S_{k-\ell}=\overline{S}_{k-\ell}\geq\bm{\lambda}'_{\ell}\big\},
\end{align*}
and 
\begin{align*}
    B_{k,\ell}:=\big\{\forall\; k-\ell<j\leq k, S_{k-\ell}-S_j\leq\bm{\lambda}'_{\ell}, S_j\leq S_{k-\ell},b\sqrt{\ell}<S_{k-\ell}-S_{k}\leq b\sqrt{\ell}+\frac{c}{2}\ell^d\big\}.
\end{align*}
Let $\bm{S}:=-S$. $\P(B_{k,\ell})$ is nothing but
\begin{align*}
    \P\big(\overline{\bm{S}}_{\ell}\leq\bm{\lambda}'_{\ell},\underline{\bm{S}}_{\ell}\geq 0,\bm{S}_{\ell}\in(b\sqrt{\ell},b\sqrt{\ell}+\frac{c}{2} \ell^d]\big)=&\P(\underline{\bm{S}}_{\ell}>0)\P\big(\bm{S}_{\ell}\in(b\sqrt{\ell},b\sqrt{\ell}+\frac{c}{2} \ell^d]|\underline{\bm{S}}_{\ell}\geq 0\big) \\ & \times \P\big(\overline{\bm{S}}_{\ell}\leq\bm{\lambda}'_{\ell}|\underline{\bm{S}}_{\ell}>0,\bm{S}_{\ell}\in(b\sqrt{\ell},b\sqrt{\ell}+\frac{c}{2}\ell^d]\big),
\end{align*}
{which is larger than $C/\ell$ for $\ell$ large enough (see Lemma \ref{Minor_HRHP1}).} \\ 
We then deal with $\P(A_{k,\ell})$. Thanks to Lemma \ref{LemmFKG}, this probability is larger than
\begin{align*}
   \P\big(\max_{j\leq k-\ell}\;\overline{S}_j-S_j\leq\bm{\lambda}'_{\ell}\big)\P\big(\overline{S}_{k-\ell}\geq\bm{\lambda}'_{\ell}\big)\P\big(\underline{S}_{k-\ell}\geq 0\big)^{2},
\end{align*}
and again, using \cite{Kozlov1976} together with the fact that $\P(\overline{S}_{\bm{L}_{\ell}}\geq\bm{\lambda}'_{\ell})\to 1$, there exists $C>0$ such that for $\ell$ large enough and any $k\in\{\bm{L}_{\ell},\ldots,\ell^2\}$,
\begin{align*}
    \P\big(\overline{S}_{k-\ell}\geq \sqrt{\ell}\big)\P\big(\underline{S}_{k-\ell}\geq 0\big)\geq \P\big(\overline{S}_{\bm{L}_{\ell}}\geq \sqrt{\ell}\big)\P\big(\underline{S}_{\ell^2}\geq 0\big)^{2}\geq\frac{C}{\ell^{2}}.
\end{align*}
We now turn to the most important part: $\P\big(\max_{j\leq k-\ell}\;\overline{S}_j-S_j\leq\bm{\lambda}'_{\ell}\big)$. We follow the same lines as the proof of \eqref{LemmBorneUnif1}: for any $k\in\{\bm{L}_{\ell},\ldots,\ell^2\}$, $k-\ell>\bm{L}_{\ell}-\ell$ so $\max_{j\leq\bm{L}_{\ell}-\ell}\;\overline{S}_j-S_j\leq\bm{\lambda}'_{\ell}$ together with $\overline{S}_{\bm{L}_{\ell}-\ell}=S_{\bm{L}_{\ell}-\ell}\leq S_j$ and $\max_{\bm{L}_{\ell}-\ell<i\leq j}S_i-S_j\leq\bm{\lambda}'_{\ell}$ for all $\bm{L}_{\ell}-\ell<j\leq k-\ell$ implies that $\max_{j\leq k-\ell}\;\overline{S}_j-S_j\leq\bm{\lambda}'_{\ell}$. It follows that $\P\big(\max_{j\leq k-\ell}\;\overline{S}_j-S_j\leq\bm{\lambda}'_{\ell}\big)$ is larger than 
\begin{align*}
    &\P\big(\max_{j\leq\bm{L}_{\ell}-\ell}\;\overline{S}_j-S_j\leq\bm{\lambda}'_{\ell},\overline{S}_{\bm{L}_{\ell}-\ell}=S_{\bm{L}_{\ell}-\ell},\max_{\bm{L}_{\ell}-\ell<i\leq j}S_i-S_j\leq\bm{\lambda}'_{\ell}, S_j\geq S_{\bm{L}_{\ell}-\ell}\;\forall\; \bm{L}_{\ell}-\ell<j\leq k-\ell\big) \\ & = \P\big(\max_{j\leq\bm{L}_{\ell}-\ell}\;\overline{S}_j-S_j\leq\bm{\lambda}'_{\ell},\overline{S}_{\bm{L}_{\ell}-\ell}=S_{\bm{L}_{\ell}-\ell}\big)\P\big(\max_{j\leq k-\ell-(\bm{L}_{\ell}-\ell)}\;\overline{S}_j-S_j\leq\bm{\lambda}'_{\ell},\underline{S}_{k-\ell-(\bm{L}_{\ell}-\ell)}\geq 0\big).
\end{align*}
Moreover, by  Lemma \ref{LemmFKG}, $\P\big(\max_{j\leq k-\ell-(\bm{L}_{\ell}-\ell)}\;\overline{S}_j-S_j\leq\bm{\lambda}'_{\ell},\underline{S}_{k-\ell-(\bm{L}_{\ell}-\ell)}\geq 0\big)$ is larger than $\P(\max_{j\leq k-\ell-(\bm{L}_{\ell}-\ell)}\;\overline{S}_j-S_j\leq\bm{\lambda}'_{\ell})\P(\underline{S}_{k-\ell-(\bm{L}_{\ell}-\ell)}\geq 0)$. By induction, we get that $\P\big(\max_{j\leq k-\ell}\;\overline{S}_j-S_j\leq\bm{\lambda}'_{\ell}\big)$ is larger than
\begin{align*}
    \P\big(\max_{j\leq\bm{L}_{\ell}-\ell}\;\overline{S}_j-S_j\leq\bm{\lambda}'_{\ell},\overline{S}_{\bm{L}_{\ell}-\ell}=S_{\bm{L}_{\ell}-\ell}\big)^{\bm{L}_{\ell}(k)}\prod_{i\leq\bm{L}_{\ell}(k)}\P\big(\underline{S}_{k-\ell-i(\bm{L}_{\ell}-\ell)}\geq 0\big)
\end{align*}
with $\bm{L}_{\ell}(k):=\lfloor (k-\ell)/(\bm{L}_{\ell}-\ell)\rfloor$. Again, by  Lemma \ref{LemmFKG}, $\P(\max_{j\leq\bm{L}_{\ell}-\ell}\;\overline{S}_j-S_j\leq\bm{\lambda}'_{\ell},\overline{S}_{\bm{L}_{\ell}-\ell}=S_{\bm{L}_{\ell}-\ell}\big)\geq\P\big(\max_{j\leq\bm{L}_{\ell}-\ell}\;\overline{S}_j-S_j\leq\bm{\lambda}'_{\ell})\Pb(\underline{S}_{\bm{L}_{\ell}-\ell}\geq 0)$ and as $k\leq\ell^2$, $\P(\underline{S}_{k-\ell-i(\bm{L}_{\ell}-\ell)}\geq 0)\geq\P\big(\underline{S}_{k}\geq 0\big)\geq\P\big(\underline{S}_{\ell^2}\geq 0\big)$. Hence, by \cite{Kozlov1976}
\begin{align*}
    \P\big(\max_{j\leq k-\ell}\;\overline{S}_j-S_j\leq\bm{\lambda}'_{\ell}\big)\geq \Big(\frac{C}{\ell\sqrt{\bm{L}_{\ell}-\ell}}\P\big(\max_{j\leq \bm{L}_{\ell}-\ell}\;\overline{S}_j-S_j\leq\bm{\lambda}'_{\ell}\big)\Big)^{\bm{L}_{\ell}(k)}
\end{align*}
for some $C>0$. Then, thanks to \cite{HuShi10b}, for all $\varepsilon>0$ and $\ell$ large enough $\P(\max_{j\leq \bm{L}_{\ell}-\ell}\;\overline{S}_j-S_j\leq\bm{\lambda}'_{\ell})\geq e^{-(1+\frac{\varepsilon}{4})\frac{\pi^2\sigma^2(\bm{L}_{\ell}-\ell)}{8}(\bm{\lambda}'_{\ell})^{-2}}$ so for $\ell$ large enough and any $k\in\{\bm{L}_{\ell},\ldots,\ell^2\}$, $\P\big(\max_{j\leq k-\ell}\;\overline{S}_j-S_j\leq\bm{\lambda}'_{\ell}\big)$ is larger than
\begin{align*}
    \Big(\frac{C}{\ell\sqrt{\bm{L}_{\ell}-\ell}}e^{-(1+\frac{\varepsilon}{4})\frac{\pi^2\sigma^2(\bm{L}_{\ell}-\ell)}{8(\bm{\lambda}'_{\ell})^2}}\Big)^{\bm{L}_{\ell}(k)}\geq e^{-(1+\frac{\varepsilon}{3})\frac{\pi^2\sigma^2(\bm{L}_{\ell}-\ell)}{8(\bm{\lambda}'_{\ell})^2}\bm{L}_{\ell}(k)}\geq e^{-(k-\ell)(1+\frac{\varepsilon}{2})\frac{\pi^2\sigma^2}{8(\bm{\lambda}'_{\ell})^2}},
\end{align*}
where we have used for the first inequality that $e^{-\eta\frac{\pi^2\sigma^2(\bm{L}_{\ell}-\ell)}{8(\bm{\lambda}'_{\ell})^2}}$ is smaller than $\frac{1}{\ell^{\eta'}}$ for any $\eta,\eta'>0$. Collecting previous inequalities, we obtain
\begin{align*}
    \P(A_{k,\ell})\geq\frac{C}{\ell^{2}}e^{-(k-\ell)(1+\frac{\varepsilon}{2})\frac{\pi^2\sigma^2}{8(\bm{\lambda}'_{\ell})^2}}.
\end{align*}
Finally, observe that $\bm{\lambda}'_{\ell}\sim\sqrt{\ell}$ and then for any $k\in\{\bm{L}_{\ell},\ldots,\ell^2\}$
\begin{align*}
    \P(A_{k,\ell})\geq e^{-\frac{k\pi^2\sigma^2}{8\ell}(1+\varepsilon)},
\end{align*}
which completes the proof.
\end{proof}

\section*{Notations \label{notations}}
In this section, we have summarized the transversal notations, give a short description of them when it is possible and the page or equation where they are introduced. \\

\noindent \textit{Sequences and constants in the statement of the main theorem}

\hspace{1cm} $\kappa_b$  (equation \eqref{kappab}), critical exponent. 

\hspace{1cm} $h_n$ (equation \eqref{Def_hn}), resume the constraint on $V$ and second order for $\mathcal{R}_n({g_n},\mathbf{f}^n)$.

\hspace{1cm} $L$ (equation \eqref{limite_L}).

\hspace{1cm} $\xi$ (equation \eqref{Limite_xi}).

\hspace{1cm}

\noindent \textit{Different form of the cumulative exponential drop of $V$}

%\begin{array}{l{1cm}}
 \hspace{1cm}   $H_x$   (below \eqref{Def_Hx}), variable appearing in the distribution under $\Pe$ of the edge local time 
 
 \hspace{1.6cm} at $x$ before the instant $T^1$.
 
\hspace{1cm}    $\tilde H_x$  (Lemma \ref{LawGeo}), variable appearing in the distribution under $\Pe$  of the sum of edge 

 \hspace{1.6cm} local  times of the descendants of $x$ before the instant $T^1$.  
 
 \hspace{1cm}   $H^S$  (page \pageref{HS}) version of above $H$ after the many to one Lemma is applied.
%\end{array}
\hspace{1cm} 

\vspace{0.3cm}
\noindent \textit{The regular lines and their possible parameters : $ \mathcal{O}_{\lambda,\lambda'} := \big\{x\in \mathbb{T};\; \underset{j\leq|x|}{\max}\; H_{x_j}\leq \lambda,\; H_x>\lambda'\big\}$ } 

\hspace{1cm} $\lambda=\lambda_n$ (above \ref{UpperProp1}), $\lambda= \bm{\lambda}_n$ (below \eqref{Fact2TH1}), $\lambda=\bm{\lambda}_{n,1}$ and $\lambda=\bm{\lambda}_{n,2}$ (in the proof of Theorem
\vspace{0.3cm}
\hspace{1.4cm} \ref{thm3}), $\lambda =\Tilde{\lambda}_n$ in the proof of Lemma \ref{RangeUpBound2}. All along the paper  $\lambda'$ is typically of order $n^b$.

%\hspace{1cm}     

%\hspace{1cm} 

%\hspace{1cm} $\bm{\tilde \lambda}_n$ (page 17) utile ?

%\hspace{1cm} 

%\hspace{1cm}

\noindent \textit{Secondary constraints on the environment : $ \mathcal{H}^k_{B,\mathfrak{z}}= \{\left(t_1,\ldots,t_k\right)\in \mathbb{R}^k;\; t_k\geq \mathfrak{z}, \min_{i \leq k} t_i \geq -B \} $}

%\hspace{1cm} 

 \hspace{1cm} $\mathfrak{z}=\mathfrak{z}_n=\ell_n^{1/3}/\delta_1$ (beginning of Section  \ref{sec3.2}).
 
 %\hspace{1cm} $\mathfrak{z}=2\mathfrak{z}_n=2\ell_n^{1/3}/\delta_1$ \RevAlexisBis{utile? $\mathfrak{z}_n$ est une notation "locale"}

\hspace{1cm} $\Upsilon_.^.$ (Proposition \ref{Prop1}) : various intersections of  conditions on $H$ and  $\mathcal{H}^k_{B,.}.$

\hspace{1cm} 

\noindent \textit{The branching function $\Psi$}

\hspace{1cm} $\Psi^{k}_{\lambda,\lambda'}$ (equation \eqref{Def_Psi}), $k$ is a generation, $\lambda$ an upper bound for $H$, $\lambda'$ a lower bound for $H$.

\hspace{1cm} $\Psi^{k}_{\lambda,\lambda'}(\cdot|\cdot)$ (equation \eqref{PsiCondi}) a conditional version of $\Psi^{k}_{\lambda,\lambda'}$.

\hspace{1cm} 

\noindent \textit{Elementary random variables related to the random walk $\X$}

 \hspace{1cm} $N_x^n$ Edge local time at $(x^*,x)$ before $n$ (equation \eqref{Nxb}).
 
 %\hspace{1cm} $T^1$ first instant of return to the root $e$. 
 
 \hspace{1cm}  $T^n$ (page \pageref{defoftn}) $n$-th instant of return to the root $e$.  

\hspace{1cm} $E_x^n$, $\tilde E_x^n$  (above \ref{A_n}). 

\hspace{1cm} 

\noindent \textit{Different ranges} 

\hspace{1cm}  $ \Ran_n(g_n,\mathbf{f}^n)$ the generalized range (equation \eqref{genRange}) with

 \hspace{2cm} $g_n$  function of constraints on the trajectory of $(X_n,n)$, 
 
\hspace{2cm} $\mathbf{f}^n$  function of constraints on the potential $V$.

\hspace{1cm} $\overline{\mathcal{R}}_{T^n}(g_n,\mathbf{f}^n)$ variant of $ \Ran_{T^n}(g_n,\mathbf{f}^n)$ with additional condition on $\overline{V}$ (page \pageref{leRbar}).

\hspace{1cm} 

%\noindent Autres

%$q_n$, $u_n$ dans preuve du Lemme \ref{RangeUpBound2}, très locale

\bibliographystyle{alpha}
\bibliography{thbiblio}

\end{document}